\documentclass[11pt,a4paper]{article}
\usepackage{ifpdf}
\usepackage[utf8]{inputenc}
\usepackage[T1]{fontenc}
\usepackage{graphicx}
\usepackage{color}
\usepackage{textcomp}
\usepackage{amsmath}
\usepackage{amssymb}
\usepackage{bbm}
\usepackage{mathrsfs}
\usepackage{dsfont}
\usepackage{amsthm}
\usepackage{geometry}
\usepackage[all]{xy}
\usepackage[usenames,dvipsnames]{pstricks}
\usepackage{epsfig}
\usepackage{pst-grad} % For gradients
\usepackage{pst-plot} % For axes
\usepackage{amsfonts}

\usepackage[normalem]{ulem}
\title{Stability in exponential time of Minkowski Space-time with a translation space-like Killing field}
\author{Cécile Huneau}

\ifpdf
\fi

\newtheorem{thm}{Theorem}[section]
\newtheorem{prp}[thm]{Proposition}
\newtheorem{cor}[thm]{Corollary}
\newtheorem{lm}[thm]{Lemma}
\newtheorem{df}[thm]{Definition}
\newtheorem{rk}[thm]{Remark}

\newcommand{\m}[1]{\mathbb{#1}}
\newcommand{\q}[1]{\mathcal{#1}}

\newcommand{\wht}[1]{\widetilde{#1}}
\newcommand{\gra}[1]{\mathbf{#1}}
\newcommand{\grat}[1]{\mathbf{\widetilde{#1}}}
\newcommand{\ba}[1]{\underline{#1}}
\newcommand{\ep}{\varepsilon}
\newcommand{\ch}{\mathbbm{1}}
\newcommand{\Up}{\Upsilon}

\begin{document}
 
\maketitle

\begin{abstract}
In this paper, we prove the nonlinear stability in exponential time of Minkowki space-time with a translation space-like Killing field. In the presence of such a symmetry, the $3+1$ vacuum Einstein equations reduce to the $2+1$ Einstein equations with a scalar field. We work in generalised wave coordinates. In this gauge Einstein equations can be written as a system of quasilinear quadratic wave equations. The main difficulty in this paper is due to the decay in $\frac{1}{\sqrt{t}}$ of free solutions to the wave equation in $2$ dimensions, which is weaker than in $3$ dimensions. As in \cite{lind}, we have to rely on the particular structure of Einstein equations in wave coordinates. We also have to carefully choose the behaviour of our metric in the exterior region to enforce convergence to Minkowski space-time at time-like infinity.
\end{abstract}

\section{Introduction}

In this paper, we address the quasi stability of the Minkowski solution to the Einstein vacuum equations with 
a translation space-like Killing field. In the presence of a translation space-like Killing field, the $3+1$ Einstein vacuum equations reduces to the following system
in the polarized case (see Appendix \ref{reduc}).
\begin{equation}\label{s1}
 \left\{ \begin{array}{l}
          \Box_g \phi = 0,\\
R_{\mu \nu}=\partial_\mu \phi \partial_\nu \phi.
         \end{array}
\right.
\end{equation}

This system has been studied by Choquet-Bruhat and Moncrief in \cite{choquet} (see also \cite{livrecb})
in the case of a
space-time of the form $\Sigma \times \mathbb{S}^1  \times \mathbb{R}$, where $\Sigma$ is a compact two dimensional manifold of genus $G\geq 2$, and $\mathbb{R}$
is the time axis, with a space-time metric independent of the coordinate on $\mathbb{S}^1$. 
They prove the existence of global solutions corresponding to the perturbation of a particular expanding universe.
This symmetry has also been studied in \cite{asht}, with an additional rotation symmetry.

In this paper, we consider a space-time of the form $\m R^2 \times \m R_{x^3} \times \m R_{t}$, for which $\partial_3$ is a Killing vector field. Minkowski space-times can be seen as a trivial solution of Einstein vacuum equations with this symmetry.
The question we address in this paper is the stability of the Minkowski solution in this framework.

In the $3+1$ vacuum case, the stability of Minkowski space-time has been proven in the celebrated work of Christodoulou and Klainerman in \cite{ck} in the maximal foliation. It has then been proven by Lindblad and Rodnianski using harmonic gauge in \cite{lind}. Their proof extends also to Einstein equations coupled to a scalar field.
In this work we will use wave coordinates. 

\subsection{Einstein equations in wave coordinates}
Wave coordinates $(x^\alpha)$ are required to satisfy $\Box_g x^\alpha=0$. In these coordinates \eqref{s1} reduces to the following system of quasilinear wave equations

\begin{equation}
 \left\{ \begin{array}{l}
          \Box_g \phi = 0,\\
\Box_g g_{\mu \nu}= -\partial_\mu \phi \partial_\nu \phi +P_{\mu \nu}(\partial g, \partial g),
         \end{array}
\right.
\end{equation}
where $P_{\mu \nu}$ is a quadratic form.
To understand the difficulty, let us first recall known results in $3+1$ dimensions. In $3+1$ dimensions, a semi linear system of wave equations of the form
\[ \Box u^i = P^i (\partial u ^j, \partial u^k) \]
is critical in the sense that if there isn't enough structure, the solutions might blow up in finite time (see the counter examples by John \cite{john}). However, if the right-hand side satisfies the null
condition, introduced by Klainerman in \cite{klai}, the system admits global solutions. This condition requires that  $P^i$ be linear combinations of the following forms
$$Q_0(u,v)=\partial_t u \partial_t v-\nabla u. \nabla v, \quad Q_{\alpha \beta}(u,v)=\partial_\alpha u \partial_\beta v-\partial_\alpha v \partial_\beta u.$$
In three dimensions, Einstein equations written in wave coordinates do not satisfy the null condition. However, this is not a necessary condition to obtain global existence. An example is provided by the system
\begin{equation}
\label{model}
 \left\{ 
\begin{array}{l}
\Box \phi_1 =0, \\
\Box \phi_2=(\partial_t \phi_1)^2.
\end{array}\right.
\end{equation}
The non-linearity does not have the null structure, but thanks to the decoupling there is nevertheless global existence.
In \cite{craslind}, Lindblad and Rodnianski showed that the non linear terms in Einstein equations in wave coordinates consists of a linear combination of  null forms with an underlying structure of the form \eqref{model}. They used the wave condition to obtain better decay for some coefficients of the metric. However the decay is slower than for the solution of the wave equation. An example of a quasilinear scalar wave equation admitting global existence without the null condition, but with a slower decay is also studied by Lindblad in \cite{lindsym} in the radial case, and by Alinhac in \cite{alin2} and Lindblad in \cite{lindblad} in the general case. In \cite{craslind}, Lindblad and Rodnianski introduced the notion of weak-null structure, which gather all these examples.

In $2+1$ dimensions, to show global existence, one has to be careful with  both quadratic and cubic terms. Quasilinear scalar wave equations in $3+1$ dimensions have been studied by Alinhac in \cite{alin}. He shows global existence for a quasilinear equation of the form
$$\Box u =g^{\alpha \beta}(\partial u)\partial_\alpha \partial_\beta u,$$
if the quadratic and cubic terms in the right-hand side satisfy the null condition.
Global existence for a semi-linear wave equation with the quadratic and cubic terms satisfying the null condition has been shown by Godin in \cite{godin} using an algebraic trick to remove the quadratic terms, which does however not extend to systems. The global existence in the case of systems of semi-linear wave equations with the null structure has been shown by Hoshiga in  \cite{hoshiga}. It requires the use of $L^\infty-L^\infty$ estimates for the inhomogeneous wave equations, introduced in \cite{kubo}.

To show the quasi global existence for our system in wave coordinates, it will therefore be necessary to exhibit structure in quadratic and cubic terms. However, as for the vacuum Einstein equations in $3+1$ dimension in wave coordinates, our system does not satisfy the null structure. It will in particular be important to understand what happens for 
a system of the form \eqref{model} in $2+1$ dimensions. For such a system, standard estimates only give an $L^\infty$ bound for $\phi_2$, without decay. Moreover, the growth of the energy of $\phi_2$ is like $\sqrt{t}$.

One can easily imagine that with more intricate a coupling than for \eqref{model}, it will be very difficult to prove stability without decay for $\phi_2$. To obtain a more useful estimate, the idea will be to exploit more precisely the fact that $\phi_1$ also satisfies a wave equation. To understand how this might help, we will look at special solutions of vacuum Einstein equations with a translation space-like Killing field : Einstein-Rosen waves. These solutions have been discovered by Beck (see \cite{beck}, and also \cite{asht} and \cite{chru} for a mathematical description).

\subsection{Einstein-Rosen waves}\label{secerw}
Einstein-Rosen waves are solutions of vacuum Einstein equations with two space-like orthogonal Killing fields : $\partial_3$ and $\partial_\theta$. The $3+1$ metric can be written
$$\gra g= e^{2\phi}(dx^3)^2+ e^{2(a-\phi)}(-dt^2+dr^2)+r^2e^{-2\phi}r^2d\theta^2.$$
The reduced equations 
\begin{equation*}
\left\{\begin{array}{l}
R_{\mu \nu}=\partial_\mu \phi \partial_\nu \phi, \\
\Box_g \phi = 0 ,
\end{array}\right.
\end{equation*}
can be written in this setting
\begin{align}
\label{wave1}R_{tt}&= \partial_r^2 a-\partial_t^2 a + \frac{1}{r}\partial_ra =2(\partial_t \phi)^2,\\
\nonumber R_{rr}&= -\partial_r^2 a+\partial_t^2 a + \frac{1}{r}\partial_r a =2(\partial_r \phi)^2,\\
\nonumber R_{tr}&=\frac{1}{r}\partial_t a = 2\partial_t \phi \partial_r \phi.
\end{align}
The equation for $\phi$ can be written, since $\phi$ is radial
$$e^{2a}\Box_g \phi = -\partial^2_t \phi+\partial_r^2 \phi +\frac{1}{r}\partial_r\phi=0,$$
where $g$ is the metric
$$g=e^{2a}(-dt^2+dr^2)+r^2d\theta^2.$$
The equation for $\phi$ decouples from the equations for the metric. Therefore we can solve the flat wave equation
$\Box \phi = 0$, with initial data $(\phi, \partial_t \phi)|_{t=0}=(\phi_0, \phi_1)$ and then solve the Einstein equations, which reduces to
\begin{equation}
\label{transport1}\partial_r a = r\left((\partial_r \phi)^2+(\partial_t \phi)^2\right),
\end{equation}
with the boundary condition $\phi|_{r=0}=0$ in order to have a smooth solution.
Since $\Box \phi= 0$, if $(\phi_0, \phi_1)$ have enough decay, we have the following decay estimate for $\phi$
$$|\partial \phi(r,t)|\lesssim \frac{1}{\sqrt{1+t+r}(1+|t-r|)^\frac{3}{2}}.$$
Therefore since
$$a= \int_0^R r\left((\partial_r \phi)^2+(\partial_t \phi)^2\right)dr$$
we have
\begin{align*}
|a|&\lesssim \frac{1}{(1+|r-t|)^2},\; for \; r<t,\\
|a-E(\phi)&|\lesssim \frac{1}{(1+|r-t|)^2}, \; for \;r>t,
\end{align*}
where the energy
$$E(\phi)=\int_0^\infty r\left((\partial_r \phi)^2+(\partial_t \phi)^2\right)dr$$
does not depend on $t$. For $r>t$, we have $a \sim E(\gamma)$ and hence is only bounded. In particular, the metric 
$$e^{2a}dr^2+r^2d\theta^2$$
exhibits an angle at space-like infinity, that is to say the circles of radius $r$ have a perimeter growth of $e^{-E(\phi)}2\pi r$ instead of $2\pi r$.
However, in the interior, the decay we get is far better than the one we could have found with standard estimates, if we had used \eqref{wave1} instead of \eqref{transport1}.

\subsection{The background metric}\label{secback}
We would like to adapt the analysis of Section \ref{secerw} in the case where we only assume one Killing field (i.e. in the case where $\partial_3$ is Killing but not $\partial_\theta$). Assume that 
$$a= \int_0^R r\left((\partial_r \phi)^2+(\partial_t \phi)^2\right)dr$$
is still an approximate solution of
\eqref{model}, which will appear to be true in Section \ref{transport}. As in this case $\phi$ also depends on $\theta$, we will have
$$\lim_{R\rightarrow \infty} a(t,R,\theta)
=\int_0^\infty r\left((\partial_r \phi)^2+(\partial_t \phi)^2\right)dr=b(t,\theta).$$
Note that we have to be careful with the dependence on $\theta$. The metric 
$$e^{2b(\theta)}(-dt^2+dr^2)+ r^2d\theta^2$$
is no longer a Ricci flat metric when $b$ depends on $\theta$.
Consequently it is not a good guess for the behavior at 
infinity of our metric solution $g$. A good candidate should be Ricci flat in the region $r>t$. Indeed if we considered compactly supported initial data for $\phi$, by finite speed propagation, $\phi$ should intuitively be supported in the region $r<t$. Consequently, the equation
$$R_{\mu \nu}=\partial_\mu \phi \partial_\nu \phi$$
implies that $g$ should be Ricci flat for $r>t$.
Consequently, we are yield to consider the following family of space-time metrics

\begin{equation}\label{gb}
g_b = -dt^2 + dr^2 + (r+\chi(q)b(\theta)q)^2 d\theta^2+J(\theta)\chi(q)dqd\theta,
\end{equation}
where $(r,\theta)$ are polar coordinates,  $q=r-t$ and $\chi$ is a cut-off function such that $\chi(q)=0$ for $q<1$ and $\chi(q)=1$ for $q>2$. 
In the coordinates $s=r+t,q,\theta$, a tedious calculation yield that all the Ricci coefficients are zero except
\begin{equation}\label{rqq}
\begin{split}
(R_b)_{qq}=&-\frac{b(\theta)\partial_q^2(q\chi(q))}{r+b(\theta)q\chi(q)}+\frac{q\chi(q)\chi'(q)J(\theta)\partial_\theta b}{(r+b(\theta)q\chi(q))^3}+
\frac{J(\theta)^2\chi(q)\chi'(q)}{4(r+b(\theta)q\chi(q))^3}-\frac{\chi'(q)\partial_\theta J(\theta)}{(r+b(\theta)q\chi(q))^2},
\\
=&-\frac{b(\theta)\partial_q^2(q\chi(q))}{r}+O\left(\frac{C(b,b',J,J')\ch_{1<q<2}}{r^2}\right),
\end{split}
\end{equation}

\begin{equation}\label{rqu}
(R_b)_{q\theta}=-\frac{J(\theta)\chi'(q)}{2(r+b(\theta)q\chi(q))}\\
=O\left(\frac{C(b,J)\ch_{1<q<2}}{r}\right).
\end{equation}
Therefore, the metrics $g_b$ are Ricci flat in the region $r>t+2$.
We will see in the next section that they are compatible with the initial data for $g$ given by the constraint equations.

This choice of background metric will force us to work in generalized wave coordinates, instead of usual wave coordinates. Indeed, for the metric
$g_b$ defined by \eqref{gb}, the coordinates $(t,x_1,x_2)$ are not wave coordinates, not even asymptotically. The generalized wave coordinate condition reads, for $g$ of the form $g=g_b + \wht g$
$$g^{\lambda \beta}\Gamma^\alpha_{\lambda \beta}=H_b^\alpha$$
where $H_b^\alpha$ is defined by
\begin{equation}
\label{defHalpha}
H_b^\alpha=\bar{H}_b^\alpha +F^\alpha,
\end{equation}
where $\bar{H}_b^\alpha$ is defined by
\begin{equation}
\label{defHbar}
\bar{H}_b^\alpha=g_b^{\lambda \beta}(\Gamma_b)^\alpha_{\lambda \beta}
\end{equation}
and $F^\alpha$ is defined by the sum of the crossed terms of the form
$\wht g \frac{\partial_\theta}{r}g_b$
in $g^{\lambda \beta}\Gamma^\alpha_{\lambda \beta}-\bar{H}^\alpha_b.$
The reason of this choice for $F^\alpha$ will be explained in next section, in the proof of Theorem \ref{thinitial}.

The form of \eqref{s1} in generalized wave coordinates is given by \eqref{gw} .

\subsection{The initial data}\label{initial}
In this section, we will explain how to choose the initial data for $\phi$ and $g$. We will note $i,j$ the space-like indices and $\alpha,\beta$ the space-time indices.

We will work in weighted Sobolev spaces.

\begin{df} Let  $m\in \m N$ and $\delta \in \mathbb{R}$. The weighted Sobolev space $H^m_\delta(\mathbb{R}^n)$ is the completion of $C^\infty_0$ for the norm 
 $$\|u\|_{H^m_\delta}=\sum_{|\beta|\leq m}\|(1+|x|^2)^{\frac{\delta +|\beta|}{2}}D^\beta u\|_{L^2}.$$
The weighted Hölder space $C^m_{\delta}$ is the complete space of $m$-times continuously differentiable functions with norm 
$$\|u\|_{C^m_{\delta}}=\sum_{|\beta|\leq m}\|(1+|x|^2)^{\frac{\delta +|\beta|}{2}}D^\beta u\|_{L^\infty}.$$
Let $0<\alpha<1$. The Hölder space $C^{m+\alpha}_\delta$ is the the complete space of $m$-times continuously differentiable functions with norm 
$$\|u\|_{C^{m+\alpha}_{\delta}}=\|u\|_{C^m_\delta} + \sup_{x \neq y, \; |x-y|\leq 1} \frac{|\partial^m u(x)-\partial^m u(y)|(1+|x|^2)^\frac{\delta}{2}}{|x-y|^\alpha}.$$
\end{df}

We  recall the Sobolev embedding with weights (see for example \cite{livrecb}, Appendix I).

\begin{prp}\label{holder} Let $s,m \in \m N$. We assume $s >1$. Let $\beta \leq \delta +1$ and $0<\alpha<min(1,s-1)$. Then, we have the continuous embedding
$$H^{s+m}_{\delta}(\m R^2)\subset C^{m+\alpha}_{\beta}(\m R^2).$$
\end{prp}

Let $0<\delta <1$.
The initial data $(\phi_0, \phi_1)$ for $(\phi,\partial_t \phi)|_{t=0}$ are freely given in $H^{N+1}_{\delta}\times H^{N}_{\delta+1}$
with $0<\delta<1$.  However
the initial data for $(g_{\mu\nu},\partial_t g_{\mu\nu})$ cannot be chosen arbitrarily.
\begin{itemize}
\item The induced metric and second fundamental form $(\bar{g}, K)$ must satisfy the constraint equations.
\item The generalized wave coordinates condition must be satisfied at $t=0$.
\end{itemize}
Moreover, we want to prescribe the asymptotic behaviour for $g$ : we want it to be asymptotic to $g_b$, where $b(\theta)$ is arbitrarily prescribed, except for its components in $1$, $\cos(\theta)$ and $sin(\theta)$.

We recall the constraint equations.  First we write the metric $g$ in the form
$$g = -N^2(dt)^2 +\bar{g}_{ij}(dx^i +\beta^i dt)(dx^j + \beta^j dt),$$
where the scalar function $N$ is called the lapse, the vector field $\beta$ is called the shift and $\bar{g}$ is a Riemannian metric on $\m R^2$. 

We consider the initial space-like surface $\m R ^2 = \{t=0\}$.
We will use the notation
$$\partial_0=\partial_t - \mathcal{L}_{\beta},$$
where $\mathcal{L}_{\beta}$ is the Lie derivative associated to the vector field $\beta$. With this notation, we have the following 
expression for the second fundamental form of $\m R^2$
$$K_{ij}=-\frac{1}{2N}\partial_0 g_{ij}.$$
We will use the notation 
$$\tau=g^{ij}K_{ij}$$
for the mean curvature. We also introduce the Einstein tensor
$$G_{\alpha \beta}=R_{\alpha \beta} - \frac{1}{2}Rg_{\alpha \beta},$$
where $R$ is the scalar curvature $R = g^{\alpha \beta}R_{\alpha \beta}$. 
The constraint equations are given by
\begin{align}
 \label{contrmom} G_{0j} &\equiv N(\partial_j \tau - D^i K_{ij})=\partial_0\phi \partial_j \phi, \; j=1,2,\\
\label{contrham} G_{00} & \equiv \frac{N^2}{2}(\bar{R}-|K|^2+ \tau^2)= (\partial_0 \phi)^2 - \frac{1}{2}{g}_{00} {g}^{\alpha \beta}\partial_\alpha\phi \partial_\beta \phi,
\end{align}
where $D$ and $ \bar{R}$ are respectively the covariant derivative and the scalar curvature associated to $\bar{g}$. 
The following result, proven in Appendix \ref{apini}, gives us the initial data we need.

\begin{thm}\label{thinitial}
Let $0<\delta<1$. Let $(\phi_0, \phi_1)\in H^{N+1}_{\delta}(\m R^ 2)\times H^{N}_{\delta+1}(\m R^ 2)$ and $\wht b(\theta)\in W^{N,2}(\m S^ 1)$ such that
$$\int \wht b d\theta= \int \wht b \cos(\theta)d\theta=\int \wht b \sin(\theta)d\theta=0.$$
We assume
$$\|\phi_0\|_{H^{N+1}_\delta}+\|\phi_1\|_{H^N_{\delta+1}}
\lesssim \ep, \quad \|\wht b\|_{W^{N,2}}\lesssim \ep^2.$$
If $\ep>0$ is small enough, there exists $b_0,b_1,b_2 \in \m R \times \m R \times \m S^1$, $J\in W^{N,2}(\m S^ 1)$ and 
$$ (g_{\alpha \beta})_0,(g_{\alpha \beta})_1 \in H^{N+1}_{\delta}\times H^{N}_{\delta+1}$$ such that the initial data for $g$ given by
$$g=g_b+g_0, \;\partial_t g = \partial_t g_b +g_1,$$
where $g_b$ is defined by \eqref{gb} with
$$b(\theta)=b_0+b_1\cos(\theta)+b_2\sin(\theta)+\wht b(\theta),$$
are such that
\begin{itemize}
\item $g_{ij}, K_{ij}=\q L_{\beta}g_{ij}$ satisfy the constraint equations \eqref{contrmom} and \eqref{contrham}.
\item the following generalized wave coordinates condition is satisfied at $t=0$
$$g^{\lambda\beta}\Gamma^\alpha_{\lambda \beta}=g_b^{\lambda\beta}(\Gamma_b)^\alpha_{\lambda \beta}+F^\alpha,$$
where $F^\alpha$ is the sum of all the crossed term of the form
$g_0\frac{\partial_\theta}{r} g_b$ in
$g^{\lambda\beta}\Gamma^\alpha_{\lambda \beta}-g_b^{\lambda\beta}(\Gamma_b)^\alpha_{\lambda \beta}$.
\end{itemize}
Moreover, we have the estimates
$$\|J\|_{W^{N,2}(\m S^1)}+\|g_0\|_{H^{N+1}_{\delta}} +\|g_1\|_{H^N_{\delta+1}} \lesssim \ep^2,$$
\begin{align*}
b_0&=\frac{1}{4\pi}\int \left(\dot{\phi}^2+|\nabla \phi|^2\right) +O(\ep^4),\\
b_1&=\frac{1}{\pi}\int \dot{\phi}\partial_1 \phi +O(\ep^4),\\
b_2&=\frac{1}{\pi}\int \dot{\phi}\partial_2 \phi +O(\ep^4),\
\end{align*}
\end{thm}

Let us make a remark on the choice of $F$
\begin{rk}
The initial data  $\partial_t \wht g_{00}$ and $\partial_t \wht g_{0i}$ are constructed so the generalized wave coordinate condition is satisfied at $t=0$.
The choice of $F$ is here to prevent terms of the form $\wht g \partial_U (g_b)$ in this condition, and therefore allow us to have
$$\partial_t \wht g_{00},\partial_t \wht g_{0i} \in H^N_{\delta+1}.$$
\end{rk}

Before stating our main result, we will recall some notations and basic tools in the study of wave equations.
\subsection{Some basic tools}

\paragraph{Coordinates and frames}
\begin{itemize}
\item We note $x^\alpha$ the standard space-time coordinates, with $t=x^0$.
We note $(r,\theta)$ the polar space-like coordinates, and $s=t+r$, $q=r-t$ the null coordinates. The associated one-forms are
$$ds=dt+dr, \quad dq=dr-dt,$$
and the associated vector fields are
$$\partial_s = \frac{1}{2}(\partial_t+\partial_r), \quad \partial_q =\frac{1}{2}(\partial_r -\partial_t).$$
\item We note $\partial$ the space-time derivatives, $\nabla$ the space-like derivatives, and by 
$\bar{\partial}$ the derivatives tangent to the future directed light-cone in Minkowski, that is to say $\partial_t +\partial_r$ and
$\frac{\partial_\theta}{r}$.
\item We introduce the null frame $L=\partial_t + \partial_r$, $\ba L = \partial_t -\partial_r$, $U =\frac{\partial_\theta}{r}$. In this frame, the Minkowski metric takes the form
$$m_{L\ba L}= -2, \; m_{UU}= 1, \; m_{LL}=m_{\ba L \ba L}=m_{L U}= m_{\ba L U}=0.$$
The collection $\q T =\{U,L\}$ denotes the vector fields of the frame tangent to the light-cone, and the collection 
$\q V = \{U,L,\ba L\}$ denotes the full null frame.
 
\end{itemize}

\paragraph{The flat wave equation}
Let $\phi$ be a solution of
\begin{equation}\label{eqphi}
 \left\{ \begin{array}{l} \Box \phi = 0,\\
	 (\phi, \partial_t \phi)|_{t=0} =(\phi_0, \phi_1).\\
        \end{array}
\right.
\end{equation}
The following proposition establishes decay for the solutions of $2+1$ dimensional flat wave equation.
\begin{prp}[Proposition 2.1 in \cite{kubota}]\label{flat1}
Let $\mu >\frac{1}{2}$. 
We have the estimate
\[|\phi(x,t)|\lesssim M_\mu(\phi_0,\phi_1)\frac{(1+|t-r|)^{[1-\mu]_{+}}}{\sqrt{1+t+r}\sqrt{1+|t-r|}}\]
where
\[M_\mu(\phi_0, \phi_1)=\sup_{y \in \m R^2}
(1+|y|)^{\mu}|\phi_0(y)|+(1+|y|)^{\mu +1}(|\phi_1(y)|+|\nabla \phi_0(y)|)\]
and where we used the notation $A^{[\alpha]_{+}}=A^{\max(\alpha,0)}$ if $\alpha \neq 0$ and 
 $A^{[0]_{+}}=\ln(A)$.
\end{prp}

\paragraph{Minkowski vector fields} We will rely in a crucial way on the
 Klainerman vector field method. We introduce the following family of vector fields
$$\q Z =\left\{\partial_\alpha, \Omega_{\alpha \beta}=-x_\alpha \partial_\beta + x_\beta \partial_\alpha, S=t\partial_t + r\partial_r \right\},$$
where $x_\alpha = m_{\alpha \beta}x^\beta$. These vector field satisfy the commutation property 
$$[\Box, Z]=C(Z)\Box,$$
where 
$$C(Z)=0,\; Z \neq S, \quad C(S)=2.$$
Moreover some easy calculations give

\begin{align*}
& \partial_t + \partial_r = \frac{S+\cos(\theta)\Omega_{0,1}+\sin(\theta)\Omega_{0,2}}{t+r},\\
&\frac{1}{r}\partial_\theta = \frac{\Omega_{1,2}}{r}=\frac{\cos(\theta) \Omega_{0,2} -\sin(\theta)\Omega_{0,1}}{t},\\
& \partial_t-\partial_r  = \frac{S-\cos(\theta)\Omega_{0,1}-\sin(\theta)\Omega_{0,2}}{t-r}.
\end{align*}
With this calculations, and the commutations properties in the region
$-\frac{t}{2}\leq r \leq 2t$
$$[Z,\partial]\sim \partial, \; [Z,\bar{\partial}] \sim \bar{\partial},$$
we obtain
\begin{equation}
\label{important}
|\partial^k \bar{\partial^l}u|\leq \frac{1}{(1+|q|)^k (1+s)^l}|Z^{k+l} u|,
\end{equation}
where here and in the rest of the paper,
$Z^I u$ denotes any product of $I$  of the vector fields of $\q Z$.
Estimates \eqref{important} and Proposition \ref{flat1} yield

\begin{cor}\label{cordec}Let $\phi $ be a solution of \eqref{eqphi}. We have the estimate
\[|\partial^k \bar{\partial}^l \phi(x,t)|
\lesssim M^{k+l}_\mu(\phi_0, \phi_1) 
\frac{(1+|t-r|)^{[1-\mu]_{+}}}{(1+t+r)^{l+\frac{1}{2}}(1+|t-r|)^{k+\frac{1}{2}}}\]
where 
\[M^j_\mu(\phi_0, \phi_1)  =\sup_{y \in \m R^2}
(1+|y|)^{\mu+j}|\nabla^s \phi_0(y)|+(1+|y|)^{\mu +1+j}(|\nabla^s \phi_1(y)|+|\nabla^{1+j} \phi_0(y)|).\]
\end{cor}

\paragraph{Weighted energy estimate}
We consider a weight function $w (q)$, where $q=r-t$, such that $w'(q)>0$ and 
$$\frac{w(q)}{(1+|q|)^{1+\mu}}\lesssim w'(q)\lesssim \frac{w(q)}{1+|q|},$$
for some $0<\mu<\frac{1}{2}$.

\begin{prp}\label{energy}
We assume that $\Box \phi= f$. Then we have
\begin{align*}
&\frac{1}{2}\partial_t \int w(q)\left((\partial_t \phi)^2+|\nabla \phi|^2 \right)
+\frac{1}{2}\int w'(q)\left((\partial_s \phi)^2 + \left(\frac{\partial_\theta u}{r}\right)^2\right) \\
&\lesssim \int w(q)|f\partial_t \phi|. 
\end{align*}
\end{prp}
For the proof of Proposition \ref{energy}, we refer to the proof of Proposition \ref{prpweighte} which is the quasilinear equivalent of Proposition \ref{energy}.

\paragraph{Weighted Klainerman-Sobolev inequality} The following proposition allows us to obtain $L^\infty$ estimates from the energy estimates. It is proved in Appendix \ref{weigklai}. The proof is inspired from the corresponding $3+1$ dimensional proposition (Proposition 14.1 in \cite{lind}).
\begin{prp}\label{prpweight}
We denote by $v$ any of our weight functions.
 We have the inequality
$$|f(t,x)v^\frac{1}{2}(|x|-t)|\lesssim \frac{1}{\sqrt{1+t+|x|}\sqrt{1+||x|-t|}}\sum_{|I|\leq 2}\|v^\frac{1}{2}(.-t) Z^I f\|_{L^2}.$$
\end{prp}

\paragraph{Weighted Hardy inequality}
If $u$ is solution of $\Box u=f$, the energy estimate allows us to estimate the $L^2$ norm of $\partial u$. To estimate the $L^2$ norm of $u$, we will use a weighted Hardy inequality.
\begin{prp}\label{hardy}Let $\alpha<1$ and $\beta>1$. We have, with $q=r-t$
$$\left\|\frac{v(q)^\frac{1}{2}}{(1+|q|)}f\right\|_{L^2} \lesssim\|v(q)^\frac{1}{2}\partial_r f\|_{L^2}, $$
where
\begin{align*}
v(q)&=(1+|q|)^\alpha, \; for\; q<0,\\
v(q)&=(1+|q|)^\beta, \; for\; q>0.\\
\end{align*}
\end{prp}
This is proven in Appendix \ref{hardyin}. The proof is inspired from the $3+1$ dimensional analogue (Lemma 13.1 in \cite{lind}).

\paragraph{$L^\infty-L^\infty$ estimate}  With the condition $w'(q)>0$ for the energy inequality, we are not allowed to take weights of the form $(1+|q|)^\alpha$, with $\alpha>0$ in the region $q<0$. Therefore, Klainerman-Sobolev inequality cannot give us more than the estimate
$$|\partial u|\lesssim \frac{1}{\sqrt{1+|q|}\sqrt{1+s}},$$
in the region $q<0$, for a solution of $\Box u= f$. However, we know that for suitable initial data, the solution of the wave equation $\Box u=0$ satisfies
$$|u|\lesssim \frac{1}{\sqrt{1+|q|}\sqrt{1+s}}, \; 
|\partial u|\lesssim \frac{1}{(1+|q|)^\frac{3}{2}\sqrt{1+s}}.$$
To recover some of this decay we will use the following proposition
 
\begin{prp}\label{inhom}
Let u be a solution of
\begin{equation*}
 \left\{ \begin{array}{l} 
         \Box u = F, \\
	 (u, \partial_t u)|_{t=0}=(0,0).
        \end{array}
\right.
\end{equation*}
For $\mu>\frac{3}{2} , \nu >1$ we have the following $L^\infty-L^\infty$ estimate
$$|u(t,x)|( 1+t+|x|)^\frac{1}{2} \leq C(\mu, \nu) M_{\mu, \nu}(F) (1+|t-|x|||)^{-\frac{1}{2} + [2-\mu]_{+}},$$
where 
$$M_{\mu, \nu}(F) = \sup (1+|y|+s)^\mu (1+|s-|y||)^\nu F(y,s),$$
and where we used the convention $A^{[\alpha]_+}=A^{\max(\alpha,0)}$ if $\alpha \neq 0$ and $A^{[0]_{+}}=\ln(A)$. 
\end{prp}
This is proven in Appendix \ref{linflinf}. This inequality has been introduced by Kubo and Kubota in \cite{kubo}.

\paragraph{An integration lemma} The following lemma will be used many times in the proof of Theorem \ref{main}, to obtain estimates for $u$ when we only have estimates for $\partial u$.
\begin{lm}
\label{lmintegration}
Let $\alpha,\beta,\gamma \in \m R$ with $\beta<-1$. We assume that the function $u: \m R^{2+1} \rightarrow \m R$ satisfies
$$|\partial u|\lesssim (1+s)^\gamma(1+|q|)^\alpha, \; for\; q<0, \quad
|\partial u|\lesssim (1+s)^\gamma (1+|q|)^\beta \; for \; q>0,$$
and for $t=0$
$$|u|\lesssim (1+r)^{\gamma+\beta}.$$
Then we have the following estimates
$$| u|\lesssim (1+s)^\gamma\max(1,(1+|q|)^{\alpha+1}), \; for\; q<0, \quad
|u|\lesssim (1+s)^\gamma (1+|q|)^{\beta+1} \; for \; q>0.$$
\end{lm}
\begin{proof}
We assume first $q>0$. We integrate the estimate
$$|\partial_q u|\lesssim (1+s)^\gamma (1+|q|)^\beta,$$
from $t=0$. We obtain, since $\beta<-1$, for $q>0$
$$|u|\lesssim (1+s)^\gamma (1+|q|)^{\beta+1}.$$
Consequently, we have, for $q=0$, $|u|\lesssim (1+s)^\gamma$.
We now assume $q<0$. We integrate
$$ |\partial_q u|\lesssim (1+s)^\gamma(1+|q|)^\alpha,$$
from $q=0$. We obtain
$$| u|\lesssim (1+s)^\gamma\max(1,(1+|q|)^{\alpha+1}).$$
This concludes the proof of Lemma \ref{lmintegration}.
\end{proof}

\subsection{Main Result}
We introduce an other cut-off function $\Up:\m R_+ \rightarrow \m R_+$ such that $\Up(\rho)=0$ for $\rho\leq \frac{1}{2}$ and $\rho \geq 2$ and $\Up=1$ for 
$\frac{3}{4}\leq \rho \leq \frac{3}{2}$. 
Theorem \ref{main} is our main result, in which we prove stability of Minkowski space-time with a translational symmetry in exponential time $T\lesssim\exp\left( \frac{1}{\sqrt{\ep}}\right)$ where $\ep>0$ is the size of the small initial data.
\begin{thm}\label{main}Let $0<\ep<1$.
Let $\frac{1}{2}<\delta<1$ and $N \geq 40$.
Let $(\phi_0,\phi_1) \in H^{N+1}_{\delta}(\m R^2)\times H^N_{\delta+1}(\m R^2)$. 
We assume
$$\|\phi_0\|_{ H^{N+1}_{\delta}}+\|\phi_1\|_{H^N_{\delta+1}}
\leq \ep.$$
Let $T\lesssim\exp(\frac{1}{\sqrt{\ep}})$. Let $0<\rho\ll \sigma \ll \mu\ll \delta$.
If $\ep$ is small enough,
there exists $b(\theta),J(\theta)\in W^{N,2}(\m S^1)$ and
there exists a global coordinate chart $(t,x_1,x_2)$ such that, for $t\leq T$, there exists a solution $(\phi,g)$ of \eqref{s1}  that we can write 
$$g=g_b +
\Up\left(\frac{r}{t}\right)\left(\frac{g_{\ba L \ba L}}{4}dq^2
+\frac{g_{U\ba L}}{2}rdqd\theta\right) +\tilde{g}$$
such that we have the estimates
\begin{equation*}
\begin{split}&\sum_{|I|\leq N}\bigg(\|\alpha_2w_0^\frac{1}{2}(q)\partial Z^I \phi\|_{L^2}
+ \frac{1}{\sqrt{1+t}}\|\alpha_2w_3^\frac{1}{2}(q)\partial Z^I g_{\ba L \ba L}\|_{L^2}+\frac{1}{\sqrt{1+t}}\|\alpha_2w_3^\frac{1}{2}(q)\partial Z^I g_{\ba L U}\|_{L^2}\\
&+\|\alpha_2w_2^\frac{1}{2}(q)\partial Z^I \wht g\|_{L^2}\bigg) \lesssim \ep (1+t)^{C\sqrt{\ep}}.
\end{split}
\end{equation*}
with
$$\left\{ \begin{array}{l}
           w_0(q)= (1+|q|)^{2+ 2\delta}, \; q>0\\
w_0(q)= 1+ \frac{1}{(1+|q|)^{2\mu}}, \; q<0,
          \end{array}\right.$$
$$\left\{ \begin{array}{l}
           w_2(q)= (1+|q|)^{2+2\delta}, \; q>0\\
 w_2(q)= \frac{1}{(1+|q|)^{1 +2\mu}}, \; q<0,
          \end{array}\right.$$
$$\left\{ \begin{array}{l}
                               w_3(q)= (1+|q|)^{3+2\delta}, \; q>0\\
                     w_3(q)= 1+ \frac{1}{(1+|q|)^{2\mu}}, \; q<0,
                              \end{array}\right.$$
$$\left\{ \begin{array}{l}
                               \alpha_2(q)= (1+|q|)^{-2\sigma}, \; q>0\\
                     \alpha_2(q)= 1, \; q<0,
                              \end{array}\right.$$

Moreover, for all $\rho>0$, we have the $L^\infty$ estimate, for $|I|\leq \frac{N}{2}+2$ and $r<t$
\begin{align*}
|Z^I \phi (x,t)|&\lesssim \frac{\ep C(\rho)}{(1+t+r)^\frac{1}{2}(1+|t-r|)^{\frac{1}{2}-4\rho}},\\
|Z^I g_{\ba L \ba L}|&\lesssim \frac{\ep C(\rho)}{(1+|t-r|)^{\frac {1}{2}-\rho}},\\
|Z^I g_{\ba L U}|+|Z^I \wht g|&\lesssim \frac{\ep C(\rho)}{(1+t+r)^{\frac {1}{2}-\rho}}.\\
\end{align*}
and we have the estimate for $b$
$$\left|b(\theta)+\int_{\Sigma_{T,\theta}}(\partial_q \phi)^2rdr\right|\lesssim \frac{\ep^2}{\sqrt{T}},$$
where we have used the notation
\begin{equation}
\label{intsig}\int_{\Sigma_{T,\theta}}(\partial_q \phi)^2rdr=\int_0^\infty \left(\partial_q \phi(T,r,\theta)\right)^2rdr.
\end{equation}
\end{thm}
\paragraph{Comments on Theorem \ref{main}}
\begin{itemize}
\item We consider perturbations of $3+1$ dimensional Minkowski space-time with a translational space-like Killing field. These perturbations are not asymptotically flat in $3+1$ dimensions, therefore the result of Theorem \ref{main} does not follow from the stability of Minkowski space-time by Christodoulou and Klainerman  \cite{ck}.
\item As our gauge, we choose  the generalized wave coordinates, which are picked such that the generalized wave coordinates condition is satisfied by $g_b$. Therefore, the method we use has a lot in common with the method of Lindblad and Rodnianski in \cite{lind} where they proved the stability of Minkowski space-time in harmonic gauge. It is an interesting problem to investigate the stability of Minkowski with a translation symmetry using a strategy in the spirit of \cite{ck} or \cite{nicolo}. 
\item The function $J(\theta)$, and the quantities
$$\int b(\theta)d\theta, \quad \int b(\theta)\cos(\theta)d\theta, \quad
\int b(\theta)\sin(\theta)d\theta$$
are imposed by the constraint equations for the initial data (see Theorem \ref{thinitial}). The quantity
$\int b(\theta)d\theta$ is called angle, and the vector
$(\int b(\theta)\cos(\theta)d\theta, 
\int b(\theta)\sin(\theta)d\theta)$ is called linear momentum. We can make a rapprochement of these quantities with the ADM mass and linear momentum.
The remaining Fourier coefficients of $b$ are chosen to ensure the convergence to Minkowski in the direction of time-like infinity, and is an essential element in the proof of the quasi stability.
\item The logarithmic growth of $\|w^\frac{1}{2}(q)\partial Z^N \phi\|_{L^2}$, and the condition
\begin{equation}
\label{condab}b(\theta) \sim \int_{\Sigma_{T,\theta}}\left(\partial_q \phi\right)^2rdr,
\end{equation}
give the estimate $|\partial^N b |\lesssim \ep^2(1+T)^{C\ep}$. 
To avoid factors of the form $(1+T)^{C\ep}$ in all our estimate, we are forced to assume $(1+T)^{C\ep}\lesssim 1$. This is the only place where we need $(1+T)^{C\ep}\lesssim 1$, and this is what prevents us to prove the stability.
\item The condition \eqref{condab} is not necessary to control the metric in the exterior region $r>t$. For this reason we believe that the stability holds in the exterir region, without the condition $T\lesssim \exp\left(\frac{1}{\sqrt{\ep}}\right)$.
\end{itemize}

As we said in the second comment, we use a method similar than Lindblad and Rodnianski method in \cite{lind}. Let us list some of the similarities and differences with their method.
\paragraph{Similarities with \cite{lind}}
\begin{itemize}
\item We use the vector field method. The vector fields we use are Klainerman vector fields of Minkowski space-time.
\item We use the wave coordinate condition to obtain more decay on the coefficients $\wht g_{\q T \q T}$ of the metric.
\item We exhibit the structure corresponding to the model problem \eqref{model}.
\end{itemize}
\paragraph{Differences with \cite{lind}}
\begin{itemize}
\item The asymptotic behaviour given by the solutions of the constraint equations prevent us to work in wave coordinates. Instead we work in generalised wave coordinates.
\item In the exterior region, our solution do not converge to Minkowski, but to a family of Ricci flat metrics $g_b$.
\item The decay of the free wave is weaker in $2+1$ dimension. Consequently, the coefficient $g_{\ba L \ba L}$ of the metric does not have any decay near the light cone.
We have to rely on the null decomposition at all steps in our proof to isolate this behaviour, even in the $L^2$ estimates.
\item We have to fit $b(\theta)$ so that the condition \eqref{condab} is satisfied. This lead to regularity issues for $b$, which prevent us from proving the global existence.
\end{itemize}

The structure of the paper is as followed. In Section \ref{struct} we describe the structure of the equations \eqref{s1} in generalized wave coordinates. We exhibit the structure of our system in Section \ref{struct}. We also describe the interactions between $g_b$ and $\wht g$. In Section \ref{model} we outline the main issues of the proof by discussing some model problems. In section \ref{boot} we give our bootstrap assumption. In section \ref{secwave} we derive preliminaries estimates thanks to the wave coordinate condition. In section \ref{secangle} we derive preliminaries  estimate for the angle and the linear momentum. In section \ref{transport}, we will exploit the analysis begun in section \ref{secerw}. In section \ref{linf} we will improve the $L^\infty$ estimate. In section \ref{weighte} we will derive the weighted energy estimate. In section \ref{secl2} we will improve the $L^2$ estimates and in section \ref{secproof} we will adjust the parameter $b(\theta)$.

\section{Structure of the equations}\label{struct}

In this section, we provide a discussion of the specific features of the structure of the equations, which will be relevant for the proof of Theorem \ref{main}.
 
\subsection{The generalized wave coordinates}

Wave coordinates allow to recast Einstein equations as a system of non-linear wave equations.
The wave coordinates condition, which consists in choosing coordinates such that $\Box_g x^\alpha=0$ can be  rewritten as
$$ g^{\lambda \beta}\Gamma^\alpha_{\lambda \beta}=0.$$
However, for the metric
$g_b$ defined by \eqref{gb}, the coordinates $(t,x_1,x_2)$ are not wave coordinates, not even asymptotically.
We will therefore work with generalized wave coordinates. We will impose that our metric satisfies
\[g^{\lambda \beta}\Gamma^\alpha_{\lambda \beta}=H_b^\alpha\]
where $H^\alpha_b$ is defined by \eqref{defHalpha}
$$H_b^\alpha=(g_b)^{\lambda \beta}(\Gamma_b)^\alpha_{\lambda \beta}+F^\alpha,$$ 
with $F^\alpha$ of the form 
$$\wht g\frac{q\chi(q)\partial_\theta b}{r^2}.$$
The role of $F^\alpha$ was explained in section \ref{initial}.
In generalized wave coordinates, the expression \eqref{calcricci} of Appendix  \ref{genw}  allow us to write the system \eqref{s1} under the form
\begin{equation}\label{gw}
 \left\{ \begin{array}{l}
          \Box_g \phi = 0\\
\Box_g g_{\mu \nu}= -2\partial_\mu \phi \partial_\nu \phi +P_{\mu \nu}(\partial g, \partial g)
+ g_{\mu \rho}\partial_\nu H^\rho + g_{\nu \rho}\partial_\mu H^\rho ,
         \end{array}
\right.
\end{equation}
where
\begin{equation}\label{quadr}
\begin{split}
 P_{\mu \nu}(g)(\partial g, \partial g)
=&\frac{1}{2}g^{\alpha \rho}g^{\beta \sigma}\left(\partial_\mu g_{\rho \sigma}\partial_\alpha g_{\beta \nu}+\partial_\nu g_{\rho \sigma}\partial_\alpha g_{\beta \mu}
-\partial_\beta g_{\mu \rho}\partial_\alpha g_{\nu \sigma}-\frac{1}{2}\partial_\mu g_{\alpha \beta} \partial_\nu g_{\rho \sigma}\right)\\
&+\frac{1}{2}g^{\alpha \beta}g^{\lambda \rho}\partial_\alpha g_{\nu \rho}\partial_\beta g_{\mu \rho}.
\end{split}
\end{equation}

\begin{rk}
 In generalized wave coordinates, the wave operator can be expressed as
$$\Box_g = g^{\alpha \rho}\partial_\alpha \partial_\rho - H_b^\rho\partial_\rho.$$
\end{rk}

The expression \eqref{calcricci} yields also
\begin{equation}
\label{calcrb}
(R_b)_{\mu\nu}=-\frac{1}{2}\Box_{g_b} (g_b)_{\mu \nu}+\frac{1}{2}P_{\mu \nu}(g_b)(\partial g_b,\partial g_b)+\frac{1}{2}\left((g_b)_{\mu \rho}\partial_\nu \bar{H}_b^\rho + (g_b)_{\mu \rho}\partial_\mu \bar{H}^\rho_b\right).
\end{equation}
Therefore, subtracting twice the equation \eqref{calcrb} to the second equation of \eqref{gw} we obtain
\begin{equation}\label{s2}
 \left\{ \begin{array}{l}
          \Box_g\phi = 0,\\
\Box_g\wht g_{\mu \nu} =-2 \partial_\mu \phi\partial_\nu \phi
+2(R_b)_{\mu \nu} + P_{\mu \nu}(g)(\partial \wht g, \partial \wht g) + \wht P_{\mu \nu} (\wht g, g_b),
         \end{array}
\right.
\end{equation}
where $P_{\mu \nu}(g)(\partial \wht g, \partial \wht g) $ is defined by \eqref{quadr} and
\begin{equation}
\label{ptilde}
\begin{split}
\wht P_{\mu \nu} (\wht g, g_b)=&
\left(g_b^{\alpha\beta}-g^{\alpha \beta}\right)\partial_\alpha \partial_\beta (g_b)_{\mu \nu}
+F^\rho\partial_\rho (g_b)_{\mu \nu}\\
&+P_{\mu \nu}(g)(\partial g,\partial_g)-P_{\mu \nu}(g)(\partial \wht g,\partial \wht g)-P_{\mu \nu}(g_b)(\partial g_b,\partial g_b)\\
&+(g_b)_{\mu \rho}\partial_\nu F^\rho + (g_b)_{\nu \rho}\partial_\mu F^\rho
+\wht g_{\mu \rho}\partial_\nu H_b^\rho + \wht g_{\nu \rho}\partial_\mu H_b^\rho.
\end{split}
\end{equation}

Let us note that $\wht P_{\mu \nu} (\wht g, g_b)$ contains only crossed terms between $g_b$ and $\wht g$.
\subsection{The weak null structure}

To exhibit the main terms in the structure of \eqref{s2}, let us neglect for a moment $P_{\mu \nu}$, $\wht P_{\mu \nu}$, $H_b$. We will see in the next section that this approximation is relevant. Let us also neglect the nonlinear terms involving $\bar{\partial}$ derivatives.
Then we obtain the following approximate system
\begin{align*}
\Box \phi + g_{LL}\partial^2_q \phi &= 0,\\
\Box  g_{\q T \q V} + g_{LL}\partial^2_q g_{\q T \q V}& =0,\\
\Box  g_{\ba L \ba L} + g_{LL}\partial^2_q g_{\ba L \ba L}& =4\left(-2(\partial_q \phi)^2-2b(\theta)\frac{\partial^2_q (\chi(q)q)}{r}\right),
\end{align*}
where we also have used the approximation 
$$(R_b)_{qq}\sim-\frac{b(\theta)\partial_q^2(q\chi(q))}{r}+O\left(\frac{C(b,b',J,J')\ch_{1<q<2}}{r^2}\right),$$
as shown in \eqref{rqq}.
In $2+1$ dimensions, a term of the form $g_{LL}\partial^2_q \phi$ is impossible to handle if one only relies on the decay for $g_{LL}$ provided by the fact of being a solution of a wave equation.
However, as in \cite{lind}, we can exploit the wave condition to obtain better decay for some coefficients of the metric.
More precisely, we have roughly 
\[\partial g_{\q T \q T} \sim \bar{\partial} g. \]
This is done properly in Proposition \ref{estLL} for the coefficient $g_{LL}$ and in Proposition \ref{estLU} for the coefficients $g_{LU}$ and $g_{UU}$.
Therefore, the $ g_{\q T \q T}$ coefficients have a better decay in $t$ than the solutions of the wave equation
(the challenges of the quasilinear terms of the form
$g_{LL}\partial_q^2 \phi$, $ g_{LL}\partial_q^2g_{\q T \q V}$ are presented in Section \ref{modqs}).  

\begin{rk}\label{termql} The other quasilinear terms are of the form
$$g_{\q T \q V}\partial_{T}\partial_{V} \phi, \quad g_{\q T \q V}\partial_{T}\partial_{V} \wht g.$$
Consequently, they involved at least one "good derivative" of $\phi, \wht g$.
Thus, they are easier to estimate, and we can always focus on the terms
$$g_{LL}\partial^2_q \phi, \quad g_{LL}\partial^2_q \wht g.$$
\end{rk}

Assuming that we can also neglect the terms involving $g_{LL}$, we are reduced to the following system
\begin{equation}\label{sysgll}
\left\{\begin{array}{l}
\Box \phi = 0,\\
\Box  g_{\ba L \ba L} =4\left(-2(\partial_q \phi)^2-2b(\theta)\frac{\partial^2_q (\chi(q)q)}{r}\right),\\
\end{array}\right.
\end{equation}
which is a system of the form \eqref{model} and displays the weak null structure.

The second component of the solution of \eqref{model} do not have any decay near the light cone in $2+1$ dimensions (see Section \ref{secerw} for the radial case).
Therefore, the coefficient $g_{\ba L \ba L}$ will not display any decay at all near the light cone (see the estimates of Theorem \ref{main}). To obtain decay for $g_{\ba L \ba L}$ in the $q$ variable, we will approximate $\frac{g_{\ba L\ba L}}{4}$ by the solution $h_0$ of the following transport equation
$$\partial_q h_0 = -2r(\partial_q \phi)^2-2b(\theta)\partial^2_q(q\chi(q)).$$
The ideas of this approximation are presented in Section \ref{secgll}, and are exploited in Section \ref{transport}. 

\subsection{Non-commutation of the wave operator with the null frame}
The structure of Einstein equations can only be seen in the null frame.
However it is well known that the wave operator does not commute with the null frame. In Theorem \ref{main} we have decomposed our metric in the following way
$$g=g_b + \wht g + \Up\left(\frac{r}{t}\right)\left(\frac{g_{\ba L \ba L}}{4}dq^2+\frac{g_{U\ba L}}{2}rdqd\theta\right).$$
The problems of non-commutation induced by $g_{\ba L \ba L}$ and 
$g_{U\ba L}$ are totally similar. Consequently, we can neglect the second one.
We expressed the 2-forms $dq^2$ in the coordinate $(t,x_1,x_2)$
$$
dq^2= (dr-dt)^2=(\cos(\theta)dx^1+\sin(\theta)dx^2-dt)^2\\
$$
Therefore, we will have, in the coordinates
$x_1,x_2$
\begin{equation}
\label{noncom}\Box \left( \Up\left(\frac{r}{t}\right)g_{\ba L \ba L}dq^2\right)_{\mu \nu}
-\Box\left(\Up\left(\frac{r}{t}\right)g_{\ba L \ba L}\right)(dq^2)_{\mu \nu}
=\Up\left(\frac{r}{t}\right)\frac{1}{r^2}\left(u^1_{\mu \nu}(\theta)g_{\ba L \ba L}
+u^2_{\mu \nu}(\theta)\partial_\theta g_{\ba L \ba L}\right)
\end{equation}
where $u^1_{\mu \nu}$ and $u^2_{\mu \nu}$ are some trigonometric functions.
The challenges of the terms involving $u^ 1_{\mu \nu}$ and $u^ 2_{\mu \nu}$ are explained in Section \ref{secnon}. 

\subsection{The semi linear term $P_{\mu \nu}(g)(\partial \wht g, \partial \wht g)$.}
Recall the form of the term $P_{\mu \nu}(g)(\partial \wht g, \partial \wht g)$.
\begin{equation*}
\begin{split}
 P_{\mu \nu}(g)(\partial \wht g, \partial \wht g)
=&\frac{1}{2}g^{\alpha \rho}g^{\beta \sigma}\left(\partial_\mu \wht g_{\rho \sigma}\partial_\alpha\wht g_{\beta \nu}+\partial_\nu \wht g_{\rho \sigma}\partial_\alpha \wht g_{\beta \mu}
-\partial_\beta \wht g_{\mu \rho}\partial_\alpha \wht g_{\nu \sigma}-\frac{1}{2}\partial_\mu \wht g_{\alpha \beta} \partial_\nu \wht g_{\rho \sigma}\right)\\
&+\frac{1}{2}g^{\alpha \beta}g^{\lambda \rho}\partial_\alpha \wht g_{\nu \rho}\partial_\beta \wht g_{\mu \rho}.
\end{split}
\end{equation*}
\paragraph{The quadratic terms}
In the null frame $(L,\ba L,U)$ the only non zero coefficients of the Minkowski metric are $m^{L \ba L} =-\frac{1}{2}$ and $m^{UU}=1$. Thanks to this remark, we can describe the terms appearing in the different components of $P_{\mu \nu}$.
\begin{itemize}
\item In $P_{\q T \q T}(g)(\partial \wht g, \partial \wht g)$, there can not be strictly more than $2$ occurrences of the vector field $\ba L$.
Therefore, the quadratic terms are of one of these form
\begin{equation}
\label{termpll}\partial_{\q V}\wht  g_{\q V \q T} \partial_{\q T} \wht g_{\q T \q T}, \partial_{\q T} \wht g_{\q V \q V}\partial_{\q T} \wht g_{\q T \q T},
\end{equation}
where we have used the fact, proved in Section \ref{secwave} that 
$$\partial_{\q V}\wht g_{\q T \q T}\sim\partial_{\q T}\wht g_{\q V \q V}.$$
These terms all have the classical null structure. How this structure can be used to show global existence is explained in Section \ref{secsemi}. Since they are by far easier to handle than the one we will describe in the following, they will be neglected in the proof of Theorem \ref{main}.
\item  In $P_{\q T \q V}(g)(\partial \wht g, \partial \wht g)$, there
 can not be strictly more than $3$ occurrences of the vector field $\ba L$. Therefore, the quadratic terms are of one of these form
$$ \partial_{\q V} \wht g_{\q T \q V} \partial_{\q T} \wht g_{\q T\q V}, \quad \partial_{\q V} \wht g_{\q V \q V} \partial_{\q T} \wht g_{\q T \q T}, \quad \partial_{\q T} \wht g_{\q V \q V} \partial_{\q V} \wht g_{\q T \q T}, \quad
\partial_{\q T} \wht g_{\q T \q V}\partial_{\q T} \wht g_{\q V \q V}
$$
where we have used the fact, proved in Section \ref{secwave} that 
$$\partial_{\q V}\wht g_{\q T \q T}\sim\partial_{\q T}\wht g_{\q V \q V}.$$
These terms all have the null structure. 
However, since $g_{\ba L \ba L}$ does not decay at all in $t$ (see the estimates of Theorem \ref{main}), one has to be
more careful with the terms of the form
$$\partial_{\q T} g_{\q T \q T}\partial_{\ba L} g_{\ba L \ba L}$$
These terms have a good structure since $\partial_{\q T} g_{\q T \q T}$ is a "good derivative" of a "good component". However, one needs two steps to exploit this structure, which can be difficult to achieve if there is no regularity left. Thankfully, these terms have three occurrences of $\ba L$, therefore they can only intervene in $P_{\q T \ba L}$. 
\begin{itemize}
\item  In $P_{\ba L L}$ we will have to be careful with
$$\partial_L \wht g_{LL}\partial_{\ba L} \wht g_{\ba L \ba L}.$$
This term can be converted in $\partial_{\ba L}\wht g_{LL}\partial_{L}\wht g_{\ba L \ba L}$ with the help of the algebraic trick
$$\Box (uv)=u\Box v +v\Box u+\partial_L u \partial_{\ba L} v+\partial_L v \partial_{\ba L} u +\partial_U u \partial_U v.$$ 
This fact will be used only in the proof of Lemma \ref{lemmeg}.
\item In $P_{\ba L U}$ we will have to be careful with
$$\partial_U g_{LL}\partial_{\ba L} g_{\ba L \ba L}.$$
This term can not be removed with the previous trick.
 We will have to single out its influence thanks to the decomposition
$$g=g_b+\chi\left(\frac{r}{t}\right)h dq^2+
\chi\left(\frac{r}{t}\right)k rdqd\theta
+\widetilde{g_4},$$
where $k$ satisfies
$$\Box_g k = \partial_U \wht g_{LL}\partial_{\ba L} \wht g_{\ba L \ba L}.$$
This will also be used only in the proof of Lemma \ref{lemmeg}.
\end{itemize}
\item
The terms in $P_{\ba L \ba L}$ which are not of the previous form can be written
\begin{equation}
\label{termbal}\partial_{\ba L} g_{LL}\partial_{\ba L}g_{\ba L \ba L}, \quad \partial_{\ba L} g_{L \ba L}\partial_L g_{\ba L \ba L}.
\end{equation}
We note the crucial cancellation of terms of the form $(\partial_{\ba L} g_{L \ba  L})^2$ in $P_{\ba L \ba L}$. The contributions \eqref{termbal} will be single out in \eqref{qll}.
\end{itemize}

\paragraph{The cubic terms}
In two dimensions, cubic terms could be troublesome. However, in the form $P_{\q V \q T}$, if there are $4$ occurrences of the vector field $\ba L$,
or in $P_{\ba L \ba L}$ if there are $5$ occurrences of the vector field $\ba L$,
 then we have a factor $g^{\ba L \ba L}$, which has a decay equivalent to $g_{LL}$. Therefore we can neglect the cubic terms in this nonlinearity.

\subsection{The crossed terms}\label{cros}
In this section, we discuss the structure of the crossed terms between $b$ and $(\wht g,\phi)$.
\paragraph{The crossed terms involving two derivatives of $b$ are absent}
In the expression
$$\Box_g g_{\mu \nu}-\left(g_{\mu \rho}\partial_\nu H_b^\rho+g_{\nu \rho}\partial_\mu H_b^\rho\right),$$
there could be terms involving two derivatives of $b(\theta)$, which would be troublesome since 
they would lead to  a loss of a derivative (recall that we only have the regularity $b\in W^{N,2}$). However, the terms involving two derivatives of $b$ in this expression, are the same than the terms involving two derivatives of $b$ in $R_{\mu \nu}(g)$. Thus, these terms cancel in the expression
$$\left(g_b^{\alpha\beta}-g^{\alpha \beta}\right)\partial_\alpha \partial_\beta (g_b)_{\mu \nu}+(g_b)_{\mu \rho}\partial_\nu F^\rho + (g_b)_{\nu \rho}\partial_\mu F^\rho
+\wht g_{\mu \rho}\partial_\nu H_b^\rho + \wht g_{\nu \rho}\partial_\mu H_b^\rho,$$
which appears in $\wht P_{\mu \nu}(\wht g,g_b)$ defined by \eqref{ptilde}.
These cancellations can be checked for example with Mathematica.

\paragraph{The crossed terms in $\wht P_{\mu \nu}$}
We recall from \eqref{gb} that
$$g_b = dsdq + (r+ \chi(q)qb(\theta))^2d\theta^2+J(\theta)\chi(q)dqd\theta.$$
Therefore in $\wht P_{\mu \nu}$ we can find terms involving
$$(g_b)_{UU}= \left(1+ \frac{\chi(q)qb(\theta)}{r}\right)^2 \quad and \quad (g_b)_{U\ba L}=-2\frac{J(\theta)\chi(q)}{r},$$
Since $(g_b)_{U\ba L}$ decays faster than $(g_b)_{UU}$ let us focus on the crossed terms between $(g_b)_{UU}$ and $\wht g$.
The problem with the term $(g_b)_{UU}$  is that far from the light cone, it does not decay at all. This is one of the causes of the logarithmic growth of the energy in the statement of Theorem \ref{main}. However, these terms are present only in the exterior region. Moreover they display also a special structure. 
Since the terms involving two derivatives of $b$ are absent, and the terms 
involving two derivatives of $\wht g$ are only present in $\Box_g \wht g$,
the terms in $\wht P_{\mu \nu}$ are of the form
$$g^{--}\partial_{-} (g_b)_{UU}\partial_{-}g_{--}.$$
\begin{itemize}
\item In $\wht P_{\q T \q V}$ the crossed terms involving $\partial_{\ba L} (g_b)_{UU}$ can not contain more than two occurrences of $\ba L$. They must be of the following form
\begin{equation*}
%\label{cross}
\partial_{\ba L} (g_b)_{UU}\partial_{\q T}\wht g_{\q T \q V},\quad \partial_{\q T} (g_b)_{UU} \partial_{\q V} \wht g_{\q T \q V},\quad  \partial_{\q T} (g_b)_{UU} \partial_{\q T} \wht g_{\q V \q V},
\end{equation*}
where we have used the wave coordinates condition $\partial_{\q V}\wht g_{\q T \q T}\sim \partial_{\q T} \wht g_{\q T \q V}.$
We have the following inequalities, thanks to \eqref{important}
\begin{align*}
|\partial_{\ba L} (g_b)_{UU}\partial_{\q T}\wht g_{\q T \q V}|&\lesssim
\frac{\ch_{q>0}(|b|+|\partial_\theta b|)}{1+r}|\partial_{\q T}\wht g_{\q T \q V}|\lesssim \frac{\ch_{q>0}(|b|+|\partial_\theta b|)}{(1+r)^2}|Z^1\wht g_{\q T \q V}|,\\
|\partial_{\q T} (g_b)_{UU}\partial_{\q V}\wht g_{\q T \q V}|&\lesssim
\frac{\ch_{q>0}(1+|q|)(|b|+|\partial_\theta b|)}{(1+r)^2}|\partial_{\q V}\wht g_{\q T \q V}|\lesssim \frac{\ch_{q>0}(|b|+|\partial_\theta b|)}{(1+r)^2}|Z^1\wht g_{\q T \q V}|.\\
\end{align*}
These two contributions are therefore quite similar. In the following, it will be sufficient to study the term
\begin{equation}
\label{crossg}
\partial_{\ba L} (g_b)_{UU}\partial_{\q T}\wht g_{\q T \q V}.
\end{equation}
The challenges of this terms will be discussed in Section \ref{secinter}
\item
In $\wht P_{\ba L \ba L}$, we may have three occurrences of $\ba L$. Therefore there are terms of the form
$$\partial_{\ba L} (g_b)_{UU}\partial_{\q T} g_{\ba L \ba L}, \quad
\partial_{\ba L} (g_b)_{UU}\partial_{\ba L} g_{\ba L L}, \quad \partial_{\q T}g_{UU}\partial_{\ba L}g_{\ba L \ba L}.$$
We have the following inequalities, thanks to \eqref{important}
\begin{align*}
|\partial_{\ba L} (g_b)_{UU}\partial_{\q T} g_{\ba L \ba L}| &\lesssim \frac{\ch_{q>0}(|b|+|\partial_\theta b|)}{1+r}|\partial_{\q T}\wht g_{\ba L \ba L}|\lesssim \frac{\ch_{q>0}(|b|+|\partial_\theta b|)}{(1+r)^2}|Z^1\wht g_{\ba L \ba L}|\\
|\partial_{\ba L} (g_b)_{UU}\partial_{\ba L} g_{\ba L L}|&\lesssim \frac{\ch_{q>0}(|b|+|\partial_\theta b|)}{1+r}|\partial_{\ba L}\wht g_{ \ba L}|\lesssim \frac{\ch_{q>0}(|b|+|\partial_\theta b|)}{(1+r)(1+|q|)}|Z^1\wht g_{L \ba L}|\\
|\partial_{\q T}(g_b)_{UU}\partial_{\ba L}g_{\ba L \ba L}|&\lesssim \frac{\ch_{q>0}(1+|q|)(|b|+|\partial_\theta b|)}{(1+r)^2}|\partial_{\ba L}\wht g_{\ba L \ba L}|\lesssim \frac{\ch_{q>0}(|b|+|\partial_\theta b|)}{(1+r)^2}|Z^1\wht g_{\ba  L \ba  L}|.
\end{align*}
Consequently, the worst term is
\begin{equation}
\label{crossbaL}
\partial_{\ba L} (g_b)_{UU}\partial_{\ba L} g_{\ba L L}.
\end{equation}
\end{itemize}
We introduce the following notation, to single out the contributions of \eqref{crossbaL} and \eqref{termbal}
\begin{equation}\label{qll}
Q_{\ba L \ba L}(h,\wht g)=
\partial_{\ba L} g_{LL}\partial_{\ba L}h+\partial_{\ba L} g_{L \ba L}\partial_L h+
\partial_{\ba L} (g_b)_{UU}\partial_{\ba L} g_{\ba L L}.
\end{equation}

\paragraph{The crossed terms involving two derivatives of $\wht g$}
With our choice of coordinates, these terms only appear in $\Box_g \wht g$. They are of the form
$$\ch_{q>0}\frac{b(1+|q|)}{1+r}\partial_U^2 \wht g.$$
Their contribution is most of the time similar than the one of \eqref{crossg}, except in the energy estimate, where they require a special treatment because of their lack of decay far from the light cone (see Section
\ref{weighte}).

\paragraph{The crossed terms in $\Box_g \phi$}
The crossed terms between $g_b$ and $\partial \phi$ are of the form
$$g^{- -}\partial_{-}(g_b)_{UU}\partial_{-}\phi.$$
Consequently, they must be of the following form
$$\partial_{\q V} (g_b)_{UU} \partial_{\q T}\phi, \quad \partial_{\q T}(g_b)_{UU} \partial_{\q V}\phi.$$
Like for $\wht P_{\q V  \q T}$, it will be sufficient to study
\begin{equation}
\label{crossphi}
\partial_{\q V} (g_b)_{UU} \partial_{\q T}\phi.
\end{equation}
The crossed terms between $g_b$ and $\partial^2 \phi$ are of the form
$$\ch_{q>0}\frac{b(1+|q|)}{1+r}\partial_U^2 \wht \phi.$$
As for $\wht g$, their contribution is most of the time similar than the one of \eqref{crossphi}, except in the energy estimate, where they require a special treatment because of their lack of decay far from the light cone (see Section
\ref{weighte}).

\begin{rk}\label{remcros}
In the region $q>0$ it is generally sufficient to study the crossed terms. Indeed, the crossed terms are the one  presenting the less decay far from light cone.
\end{rk}

\section{Model problems}\label{modele}

The proof relies on a bootstrap scheme.  Roughly speaking, we will assume some estimates on the coefficients $Z^I \phi,Z^I g_{\ba L \ba L}$ and $Z^I g_{\q T \q V}$ :
\begin{itemize}
\item $L^\infty$ estimates for $I\leq \frac{N}{2}$, 
\item $L^2$ estimates for $I\leq N$.
\end{itemize}
We rewrite the bootstrap assumptions in the condensed form
$$|\phi|_{X_1}\leq 2C_0\ep, \quad |g|_{X_2} \leq 2C_0\ep,$$
where $C_0$ is a constant depending only on the quantities $\rho,\sigma,\mu,\delta,N$ introduced in the statement of Theorem \ref{main} and such that at $t=0$
$$ |\phi|_{X_1}\leq C_0\ep, \quad |g|_{X_2} \leq C_0\ep.$$
Thanks to the $L^\infty-L^\infty$ estimate and the energy estimate, we will be able to prove
$$|\phi|_{X_1}\leq C_0\ep +C\ep^2, \quad |g|_{X_2} \leq C_0 \ep+C\ep^2.$$
Therefore, for $\ep$ chosen small enough so that $C\ep \leq \frac{C_0}{2}$, this improves the bootstrap assumptions.

We will first consider a toy model, which exhibits some of the mechanisms involved in the proof.

\subsection{Global well posedness for a semi linear wave equation with the null structure}\label{secsemi}
We consider the following $2+1$ dimensional semi-linear wave equation
\begin{equation}
\label{toy}
\left\{
\begin{array}{l}
\Box u =\partial u \bar{\partial}u,\\
(u,\partial_tu)|_{t=0}=(u_0,u_1).
\end{array}
\right.
\end{equation}
Note that the nonlinearity satisfies the null condition. Consequently, this model will show us how to treat the terms of the form \eqref{termpll}.
The following result is proved in \cite{hoshiga}. We will give a proof of it for sake of completeness, and because it exhibits some of the mechanisms involved in the proof of Theorem \ref{main}.
\begin{prp}
Let $0<\delta<\frac{1}{2}$. Let $N\geq 8$. Let $u_0,u_1 \in \times H^{N+1}_{-\frac{1}{2}+\delta} \times H^N_{\delta+\frac{1}{2}}$ 
such that
$$\|u_0\|_{H^{N+1}_{-\frac{1}{2}+\delta}} +\|u_1 \|_{ H^N_{\delta+\frac{1}{2}}}\leq \ep.$$
If $\ep>0$ is small enough, the equation \eqref{toy} has a global solution $u$. 
\end{prp}

\begin{proof}
Let $0<\mu<\frac{1}{4}$.
We introduce the weight function
$$\left\{\begin{array}{l}
w(q)=1+\frac{1}{(1+|q|)^{2\mu}}, \;q<0,\\
w(q)=(1+|q|)^{1+2\delta} \; q>0.\\
\end{array}\right.$$
Let $0< \rho<\frac{\delta}{2}$. To prove global existence for equation \eqref{toy}, we consider a time $T>0$ such that, on $0\leq t \leq T$
\begin{align}
\label{boot1}|Z^I u|&\leq 2C_0 \frac{\ep}{\sqrt{1+s}(1+|q|)^\delta},\: I\leq \frac{N}{2},\\
\label{boot2}|Z^I u|&\leq 2C_0 \frac{\ep}{\sqrt{1+s}(1+|q|)^\frac{\delta}{2}},\: I\leq \frac{N}{2}+1,\\
\label{boot3}\|w^\frac{1}{2}\partial Z^I u\|_{L^2}&\leq 2C_0(1+t)^\rho \ep, \; I\leq N.
\end{align}
Thanks to Klainerman-Sobolev inequality, the assumption \eqref{boot3} yields, for $I\leq N-2$
\begin{equation}
\label{wks}|\partial Z^I u|\lesssim \frac{\ep(1+t)^\rho }{\sqrt{1+s}\sqrt{1+|q|}},\; for\; q<0,\quad
|\partial Z^I u|\lesssim \frac{\ep(1+t)^\rho }{\sqrt{1+s}(1+|q|)^{1+\delta}},\; for\; q>0.
\end{equation}
and consequently, thanks to Lemma \ref{lmintegration}
\begin{equation}\label{wks2}
|Z^I u|\lesssim \frac{\ep\sqrt{1+|q|}}{(1+s)^{\frac{1}{2}-\rho}}, \; for \;
q<0, \quad |Z^Iu|\lesssim \frac{\ep }{(1+s)^{\frac{1}{2}-\rho}(1+|q|)^\delta}, \; for q>0.
\end{equation}
We use the $L^\infty-L^\infty$ estimate to ameliorate
eqtimates \eqref{boot1} and \eqref{boot2}. We write
\begin{equation}
\label{eqmod}\Box Z^I u =\sum_{I_1+I_2\leq I} \partial Z^{I_1}u \bar{\partial} Z^{I_2}u.
\end{equation}
We first treat the case $I\leq \frac{N}{2}$. We assume $I_1\leq \frac{N}{4}$ (the case $I_2\leq \frac{N}{4}$ can be treated in the same way). 
Therefore, we can estimate thanks to \eqref{important}
$$|\partial Z^{I_1}u|\leq \frac{1}{1+|q|}|Z^{I_1+1}u|.$$
Since $\frac{N}{4}+1\leq \frac{N}{2}$ we obtain thanks to \eqref{boot1}
$$|\partial Z^{I_1}u|\lesssim \frac{\ep}{(1+|q|)^{1+\delta}\sqrt{1+s}}.$$
To estimate $\bar{\partial} Z^{I_2}u$ we use \eqref{important} and the bootstrap assumption \eqref{boot2} to obtain
$$|\bar{\partial} Z^{I_2}u|\lesssim  \frac{1}{1+s}|Z^{I_2+1}u|\lesssim \frac{\ep}{(1+s)^\frac{3}{2}(1+|q|)^\frac{\delta}{2}}.$$
This yields
$$|\Box Z^I u|\lesssim \frac{\ep^2}{(1+s)^2(1+|q|)^{1+\frac{3\delta}{2}}}.$$
We can now use the $L^\infty-L^\infty$ estimate of Proposition \ref{inhom}, together with the estimate of Proposition \ref{flat1} and the Sobolev injection of 
Proposition \ref{holder}, which gives
$$|Z^I u|\leq \frac{C_0 \ep}{\sqrt{1+s}(1+|q|)^\delta}+\frac{C\ep^2\ln(1+|q|)}{\sqrt{1+s}\sqrt{1+|q|}}.$$
This implies, since $\ln(1+|q|)\lesssim (1+|q|)^{\frac{1}{2}-\delta}$
\begin{equation}
\label{mod1}|Z^I u|\leq \frac{C_0 \ep}{\sqrt{1+s}(1+|q|)^\delta}+\frac{C\ep^2}{\sqrt{1+s}(1+|q|)^\delta}.
\end{equation}
We now treat the case $I=\frac{N}{2}+1$. We assume $I_1\leq \frac{N+2}{4} \leq \frac{N}{2}$ so we have the same estimate as before for $\partial Z^{I_1}u$. To estimate $\bar{\partial} Z^{I_2}u$, since $\frac{N}{2}+2\leq N-2$ we use \eqref{wks2}. We obtain
$$|\bar{\partial} Z^{I_2}u|\lesssim  \frac{1}{1+s}|Z^{I_2+1}u|\lesssim \frac{\ep\sqrt{1+|q|}}{(1+s)^{\frac{3}{2}-\rho}}.$$
Therefore we obtain
$$|\Box Z^I u|\lesssim \frac{\ep^2}{(1+s)^{2-\rho}(1+|q|)^{\frac{1}{2}+\delta}}
\lesssim \frac{\ep^2}{(1+s)^{\frac{3}{2}+\frac{\delta}{2}}(1+|q|)^{1+\frac{\delta}{2}-\rho}}
.$$
Therefore, like for \eqref{mod1}, the $L^\infty-L^\infty$ estimate yields
\begin{equation}
\label{mod2} |Z^I u|\leq \frac{C_0 \ep}{\sqrt{1+s}(1+|q|)^\delta}+\frac{C\ep^2}{\sqrt{1+s}(1+|q|)^\frac{\delta}{2}}.
\end{equation}
We now use the weighted energy estimate to ameliorate \eqref{boot3}. Let $I\leq N$. In view of \eqref{eqmod}, it implies
\begin{equation}
\label{toye}
\frac{d}{dt}\|w(q)^\frac{1}{2}\partial Z^I u\|_{L^2}^2
+\|w'(q)^\frac{1}{2}\bar{\partial }Z^I u\|_{L^2}^2\lesssim \sum_{I_1+I_2\leq I}
\|w^\frac{1}{2}\partial Z^{I_1}u \bar{\partial} Z^{I_2}u\|_{L^2}\|w^\frac{1}{2}\partial Z^I u\|_{L^2}.
\end{equation}
We first assume $I_2\leq \frac{N}{2}$. Then we
estimate
$$|\bar{\partial}Z^{I_2}u|\lesssim \frac{\ep}{(1+s)^\frac{3}{2}(1+|q|)^\frac{\delta}{2}}.$$
This yields
$$\|w^\frac{1}{2}\partial Z^{I_1}u \bar{\partial} Z^{I_2}u\|_{L^2}\lesssim \frac{\ep}{(1+t)^\frac{3}{2}}\|w^\frac{1}{2}\partial Z^{I_1}u\|_{L^2}.$$
We now assume $I_1\leq \frac{N}{2}$. Then, we estimate
$$|\partial Z^{I_1}u |\lesssim \frac{\ep}{\sqrt{1+s}(1+|q|)^{1+\frac{\delta}{2}}}.$$
Therefore we obtain
$$ \|w^\frac{1}{2}\partial Z^{I_1}u \bar{\partial} Z^{I_2}u\|_{L^2}\lesssim\frac{\ep}{\sqrt{1+t}}\left\|\frac{w^\frac{1}{2}}{(1+|q|)^{1+\frac{\delta}{2}}}\bar{\partial} Z^{I_2}u\right\|_{L^2}.$$
Since 
$$\frac{w^\frac{1}{2}}{(1+|q|)^{1+\frac{\delta}{2}}}\leq w'(q)^\frac{1}{2},$$
we infer
$$\|w^\frac{1}{2}\partial Z^{I_1}u \bar{\partial} Z^{I_2}u\|_{L^2}\|w^\frac{1}{2}\bar{\partial} Z^I u\|_{L^2}
\leq \frac{\ep}{1+t}\|w^\frac{1}{2}\partial Z^I u\|^2_{L^2}
+\ep\|w'(q)^\frac{1}{2}\bar{\partial} Z^{I_2}u\|^2_{L^2}.$$
Therefore $\eqref{toye}$ writes
$$\frac{d}{dt}\|w(q)^\frac{1}{2}\partial Z^I u\|_{L^2}^2
+\|w'(q)^\frac{1}{2}\bar{\partial Z^I u}\|_{L^2}^2\lesssim
\frac{\ep}{1+t}\|w^\frac{1}{2}\partial Z^I u\|^2_{L^2}
+\ep\|w'(q)^\frac{1}{2}\bar{\partial} Z^{I_2}u\|^2_{L^2},$$
so for $\ep$ small enough
$$\frac{d}{dt}\|w(q)^\frac{1}{2}\partial Z^I u\|_{L^2}^2
+\frac{1}{2}\|w'(q)^\frac{1}{2}\bar{\partial} Z^I u\|_{L^2}^2\lesssim
\frac{\ep}{1+t}\|w^\frac{1}{2}\partial Z^I u\|^2_{L^2}.$$
We obtain
\begin{equation}
\label{mod3}\|w(q)^\frac{1}{2}\partial Z^I u\|_{L^2}\leq C_0\ep(1+t)^{C\ep}.
\end{equation}
For $\ep$ small enough so that
$$C\ep\leq \frac{C_0}{2}, \quad (1+t)^{C\ep}\leq \frac{3}{2}(1+t)^\rho,$$
we have proved, in view of \eqref{mod1}, \eqref{mod2} and \eqref{mod3} that for $t\leq T$ we have
\begin{align*}
|Z^I u|&\leq \frac{3}{2}C_0 \frac{\ep}{\sqrt{1+s}(1+|q|)^\delta},\: |I|\leq \frac{N}{2},\\
|Z^I u|&\leq \frac{3}{2}C_0 \frac{\ep}{\sqrt{1+s}(1+|q|)^\frac{\delta}{2}},\: |I|\leq \frac{N}{2}+1,\\
\|w^\frac{1}{2}\partial Z^I u\|_{L^2}&\leq \frac{3}{2}C_0(1+t)^\rho \ep, \; |I|\leq N, 
\end{align*}
which concludes the proof.
\end{proof}
\begin{rk} Actually, only the highest order energy $\|w^\frac{1}{2}\partial Z^Nu\|_{L^2}$ grows in $t$.
To see this, we estimate
$$\|w^\frac{1}{2}\partial Z^{I_1}u \bar{\partial} Z^{I_2}u\|_{L^2}$$
for $I_1\leq \frac{N}{2}$ and $I_2\leq N-1$.
Since 
$$|\bar{\partial} Z^{I_2}u|\leq \frac{1}{1+s}|Z^{I_2+1}|,$$
we obtain, together with the weighted Hardy inequality
$$\|w^\frac{1}{2}\partial Z^{I_1}u \bar{\partial} Z^{I_2}u\|_{L^2}\lesssim \frac{\ep}{(1+t)^\frac{3}{2}}\left\|
\frac{w^\frac{1}{2}}{(1+|q|)}Z^{I_2+1}u\right\|_{L^2} \lesssim \frac{\ep}{(1+t)^\frac{3}{2}}\|w^\frac{1}{2}\partial Z^{I_2+1} u\|_{L^2}.$$
Therefore, the weighted energy estimate yields, for $|I|\leq N-1$
$$\frac{d}{dt}\|w^\frac{1}{2}\partial Z^Iu\|^2_{L^2} \lesssim \frac{\ep^2}{(1+t)^{\frac{3}{2}-C\ep}},$$
and hence
$$\|w^\frac{1}{2}\partial Z^Iu\|_{L^2}\lesssim \ep.$$
\end{rk}

\begin{rk}
The use of the term $\|w'(q)^\frac{1}{2}\bar{\partial} Z^I u\|_{L^2}^2$ to exploit the structure in the energy estimate has been 
introduced by Alinhac in \cite{alin} and is sometimes called Alinhac ghost weight method. It has also been used in the case of Einstein equations in wave coordinates
in \cite{lind}.
\end{rk}

Unfortunately, Einstein equations do not have the null structure, but only a weak form of it. In the next sections, we will see what problems this creates and the method we used to tackle them. We will be less precise than in this first example, since full details will be provided when we proceed with the proof of Theorem \ref{main}.

\subsection{The coefficient $g_{\uline{L}\uline{L}}$}\label{secgll}
To understand how to deal with $g_{\ba L \ba L}$, let us consider the question of global existence for the following system, which is of the form \eqref{sysgll}
\begin{equation}
\label{eqdeh}
\left\{\begin{array}{l}\Box \phi=0,\\
\Box h=-2(\partial_q \phi)^2-2\frac{b(\theta)\partial^2_q(q\chi(q))}{r}.\\
\end{array}\right.
\end{equation}
with initial data for $\phi$ of size $\ep$ and zero initial data for $h$.
We recall $\|b\|_{L^2(\m S^1)}\lesssim \ep^2$.
We have the following estimates for $\phi$
$$\|w^\frac{1}{2}\partial \phi\|_{L^2} \lesssim \ep, \quad |\partial \phi|\lesssim  \frac{\ep}{\sqrt{1+s}(1+|q|)^{1+\delta}}.$$
Therefore, the energy estimate for $h$ writes
$$\frac{d}{dt}\|w^\frac{1}{2}\partial h\|^2_{L^2} \lesssim
\left(\|w^\frac{1}{2}(\partial_q \phi)^2\|_{L^2}
+\left\|w^\frac{1}{2}\frac{b(\theta)\partial^2_q(q\chi(q))}{r}\right\|_{L^2}\right)
\|w^\frac{1}{2}\partial h\|_{L^2},$$
and thus
$$\frac{d}{dt}\|w^\frac{1}{2}\partial h\|_{L^2} \lesssim \left(\frac{\ep}{\sqrt{1+t}}\|w^\frac{1}{2}\partial \phi\|_{L^2}
+\frac{\ep^2}{\sqrt{1+t}} \right)\lesssim \frac{\ep^2}{\sqrt{1+t}}.$$
We infer
\begin{equation}
\label{esthmod}\|w^\frac{1}{2}\partial h\|_{L^2} \leq \ep^2\sqrt{1+t}.
\end{equation}
This estimate is not sufficient. To obtain more information on $h$, we will approximate it by the solution $h_0$ of the following transport equation ( this procedure will be made more precise in Section \ref{transport})
\begin{equation}
\label{trans1}
\partial_q h_0= -2r(\partial_q \phi)^2 -2b(\theta)\partial^2_q (q\chi(q)),
\end{equation}
with initial data $h_0=0$ at $t=0$.
The $L^\infty$ estimate for $\phi$, and the fact that $\chi'$ is supported in $[1,2]$ yield
$$|\partial_q h_0|\lesssim \frac{\ep^2}{(1+|q|)^{2+2\delta}}.$$
To estimate $h_0$ we write
$$h_0(Q,s,\theta) = \int_s^Q\left(-2(\partial_q \phi)^2r -2b(\theta)\partial_q^2(q\chi(q))\right)dq,$$
so we obtain 
\begin{align*}
h_0(s,Q,\theta)& = O\left(\frac{\ep^2}{(1+|Q|)^{1+2\delta}}\right), \; Q>0,\\
h_0(s,Q,\theta)& = \int_s^{-s}\left(-2(\partial_q \phi)^2r -2b(\theta)\partial_q^2(q\chi(q))\right)dq+ O\left(\frac{\ep^2}{(1+|Q|)^{1+2\delta}}\right), \; Q<0.
\end{align*}
Therefore, since
$$\int_s^{-s}\partial_q^2(q\chi(q))dq=-1, \; for  \; s\geq 2$$
to maximize the decay in $q$ for $h_0$ (and hence for $h$, provided one has a suitable control over $h-h_0$) we will choose $b$ such that
\begin{equation}
\label{appro}
b(\theta)\simeq \int_s^{-s}(\partial_q \phi)^2rdq.
\end{equation}

\begin{rk}
$b(\theta)$ is a free parameter, except from $\int b(\theta)$, $\int b(\theta)\cos(\theta)$ and $\int b(\theta)\sin(\theta)$ which are prescribed by the resolution of the constraint equations, and correspond intuitively to the ADM angle (energy) and linear momentum. Let $\Pi $ be the projection
defined by \eqref{projection}.
%$$\Pi :W^{N,2}(\m S^1) \rightarrow \left\{u \in  W^{N,2}(\m S^1), \; \int u= \int \cos(\theta)u=\int \sin(\theta)u=0\right\}.$$
Then
$$\Pi (b(\theta))\simeq \Pi\left(\int_s^{-s}(\partial_q \phi)^2rdq\right),$$
will be forced in the course of the  bootstrap procedure. On the other hand, the fact that
\begin{align*}
\int b(\theta) &\simeq \int\int_s^{-s}(\partial_q \phi)^2rdqd\theta,\\
\int b(\theta)\cos(\theta) &\simeq \int\int_s^{-s}(\partial_q \phi)^2\cos(\theta)rdqd\theta,\\
\int b(\theta)\sin(\theta) &\simeq \int\int_s^{-s}(\partial_q \phi)^2\sin(\theta)rdqd\theta.\\
\end{align*}
will be obtained by integrating the constraint equations at any time t (see Section \ref{transport}).
\end{rk}

\subsection{Non commutation of the wave operator with the null frame}\label{secnon}
In this section, we will discuss the influence of the terms
appearing in \eqref{noncom}.
We have seen in the previous section that $h_0$ does not decay at all with respect to the $s$ variable. In turn, we will show that this is also the case for $h$, and finally for the coefficient $g_{\ba L \ba L}$. We do not want this behavior to propagate to the other coefficients of the metric. To this end, we will rely on a decomposition of the type
$$g= g_b +\Up\left(\frac{r}{t}\right)\frac{g_{\ba L \ba L}}{4}dq^2+\wht g_i.$$
However, since the wave operator does not commute with the null decomposition, we have to control the solution $\wht g_i$ of an equation of the form
$$\Box \wht g_i = \Up \left(\frac{r}{t}\right) \frac{\bar{\partial} h}{r},$$
where $h$ is the solution of \eqref{eqdeh}. The term
$\Up\left(\frac{r}{t}\right) \frac{\bar{\partial} h}{r}$ has the form of the terms appearing in \eqref{noncom}.

Provided
we can approximate $h$ by the solution $h_0$ of the transport equation \eqref{trans1}, we obtain decay with respect to $q$ for $h$. The decay we will be able to get is 
$$|h|\lesssim \frac{\ep^2}{\sqrt{1+|q|}}.$$
With this decay we infer
$$ |\Box \wht g_i|\lesssim \frac{\ep^2}{(1+s)^2\sqrt{1+|q|}},$$
and therefore, with the $L^\infty-L^\infty$ estimate, we deduce
$$|\wht g_i|\lesssim \frac{\ep}{(1+s)^{\frac{1}{2}-\rho}},$$
for all $\rho>0$. 

On the other hand, assume we are only allowed to use the energy estimate for $h$, which is the case when deriving $L^2$ type estimates for $\wht g_i$ at the level of the highest energy. When applying the weighted energy estimate for $\wht g_i$, we obtain
$$\frac{d}{dt}\|w(q)^\frac{1}{2} \partial \wht g_i\|^2_{L^2}
\leq \left\|w(q)^\frac{1}{2}\Up\left(\frac{r}{t}\right) \frac{\bar{\partial} h}{r}\right\|_{L^2}\|w(q)^\frac{1}{2} \partial \wht g_i\|.$$
We estimate
\begin{equation}
\label{estdeh}
\left\|w(q)^\frac{1}{2}\Up\left(\frac{r}{t}\right) \frac{\bar{\partial} h}{r}\right\|_{L^2}\lesssim \frac{1}{1+t}\|w(q)^\frac{1}{2}\partial h\|_{L^2}\lesssim \frac{\ep^2}{\sqrt{1+t}},
\end{equation}
where we have used the estimate \eqref{esthmod} of the previous section for $h$.
This yields
$$\frac{d}{dt}\|w(q)^\frac{1}{2} \partial \wht g_i\|_{L^2} \leq 
\frac{\ep^2}{\sqrt{1+t}}.$$
So
$$\|w(q)^\frac{1}{2} \partial \wht g_i\|_{L^2} \leq \ep^2\sqrt{1+t},$$
which is precisely the behaviour we are trying to avoid with this decomposition ! However
we have not been able to exploit all the structure in \eqref{estdeh}. In order to do so, we will use different weight functions for $\wht g_i$ and for $h$. If we set 
$$\wht w(q)=(1+|q|)^{1+2\mu} w(q),$$
with $0<\mu \leq \frac{1}{4}$ and we assume that we have
$$\|\wht w(q)^\frac{1}{2}\partial h \|_{L^2}\lesssim \ep^2\sqrt{1+t},$$
then we can estimate
$$\left\|w(q)^\frac{1}{2}\Up\left(\frac{r}{t}\right) \frac{\bar{\partial} h}{r}\right\|_{L^2}\lesssim \frac{1}{1+t} \left\|\wht w(q)^\frac{1}{2}\Up\left(\frac{r}{t}\right) \frac{\bar{\partial} h}{(1+|q|)^{\frac{1}{2}+\mu}}\right\|_{L^2}.$$
We write
$$|\bar{\partial} h|\lesssim \frac{1}{1+s}|Z h|
\lesssim \frac{1}{(1+s)^{\frac{1}{2}+\mu}(1+|q|)^{\frac{1}{2}-\mu}}|Z h|
,$$
so we obtain
$$\left\|w(q)^\frac{1}{2}\Up\left(\frac{r}{t}\right) \frac{\bar{\partial} h}{r}\right\|_{L^2}\lesssim \frac{1}{(1+t)^{\frac{3}{2}+\mu}} \left\|\wht w(q)^\frac{1}{2} \frac{Zh }{1+|q|}\right\|_{L^2}
\lesssim \frac{1}{(1+t)^{\frac{3}{2}+\mu}} \|\wht w(q)^\frac{1}{2} \partial Z h\|_{L^2},$$
where we used the weighted Hardy inequality.
Consequently, the energy inequality for $\wht g_i$ yields
$$\frac{d}{dt}\|w(q)^\frac{1}{2} \partial \wht g_i\|_{L^2} \lesssim \frac{\ep^2}{(1+t)^{1+\mu}},$$
and therefore
$$\|w(q)^\frac{1}{2} \partial \wht g_i\|_{L^2} \lesssim \ep^2.$$
Recall that the weighted energy inequality forbids weights of the form $(1+|q|)^\alpha$ with $\alpha>0$ in the region $q<0$. Therefore we are forced to make the following choice in the region $q<0$
$$\wht w(q)= O(1), \quad w(q)= \frac{1}{(1+|q|)^{1+2\mu}}.$$
Thus, for $\wht g_i$, the $\sqrt{t}$ loss has been replaced by a loss in $(1+|q|)^{\frac{1}{2}+\mu}$.

\subsection{The quasilinear structure}\label{modqs}
In this section we will discuss the challenges of the quasilinear structure. We will take as an example the equation for $\phi$, $\Box_g \phi= 0$.
Following Remark \ref{termql}, 
we can focus on the terms of the form $g_{LL}\partial^2_q \phi$. The wave coordinates condition yields
$$\partial g_{LL} \sim \bar{\partial} g.$$
If $g$ satisfied
$\Box g=0$,
the $L^\infty$ estimates for $g$ given by Corollary \ref{cordec} for suitable initial data would imply 
$$|\partial g_{LL} |\leq \frac{\ep}{(1+s)^{\frac{3}{2}}\sqrt{1+|q|}},$$
We would like to keep this decay in $\frac{1}{(1+s)^\frac{3}{2}}$ after integrating with respect to $q$. However, we are not in the range of application of Lemma \ref{lmintegration}.
To this end, we will assume more decay on the initial data. As stated in Theorem \ref{main}, we take $(g,\partial_t g)\in H^{N+1}_{\delta}\times H^N_{\delta+1}$ with $\frac{1}{2}<\delta<1$. Then, with the weight $w_0$
stated in Theorem \ref{main}, the weighted energy inequality yields
$$\|w_0(q)\partial Z g\|_{L^2}\lesssim \ep,$$
and consequently, for $q>0$, the weighted Klainerman-Sobolev inequality yields
$$|\partial Z g|\lesssim \frac{\ep}{\sqrt{1+s}(1+|q|)^{\frac{3}{2}+\delta}}.$$
If we integrate from $t=0$, we obtain for $q>0$
$$|Zg|\lesssim \frac{\ep}{\sqrt{1+s}(1+|q|)^{\frac{1}{2}+\delta}}.$$
By writing $|\bar{\partial}g|\lesssim \frac{1}{1+s}|Zg|$, we obtain
$$|\partial g_{LL}|\lesssim \frac{\ep}{(1+s)^\frac{3}{2}(1+|q|)^\frac{1}{2}}, \; for\; q<0, \quad |\partial g_{LL}|\lesssim \frac{\ep}{(1+s)^\frac{3}{2}(1+|q|)^{\frac{1}{2}+\delta}}, \; for \; q>0.$$
Since $\delta>\frac{1}{2}$ we can apply Lemma \ref{lmintegration}, which yields
$$| g_{LL}|\lesssim \frac{\ep\sqrt{1+|q|}}{(1+s)^\frac{3}{2}}, \; for q<0, \quad | g_{LL}|\lesssim \frac{\ep}{(1+s)^\frac{3}{2}(1+|q|)^{\delta-\frac{1}{2}}}, \; for \; q>0.$$
Consequently we easily estimate
$$\|w^\frac{1}{2} g_{LL}\partial^2_q Z^I\phi \|_{L^2} \lesssim \frac{\ep}{(1+t)^\frac{3}{2}}\|w^\frac{1}{2}\partial_q Z^{I+1}\phi\|_{L^2}.  $$

This strong decay in the region $q>0$ is also needed when estimating
$$\|w_0^\frac{1}{2} Z^I g_{LL}\partial^2_q \phi \|_{L^2}.$$
The idea will be first to use the weighted Hardy inequality to derive
$$\|w_0^\frac{1}{2} Z^I g_{LL}\partial^2_q \phi \|_{L^2}
\lesssim \frac{\ep}{\sqrt{1+t}}\left\|\frac{w_0^{\frac{1}{2}}}{(1+|q|)^2}Z^I g_{LL}\right\|_{L^2}\lesssim
\frac{\ep}{\sqrt{1+t}}\left\|\frac{w_0^\frac{1}{2}}{(1+|q|)}\partial Z^I g_{LL}\right\|_{L^2}.$$
Then we rely on the wave coordinates condition, which yields
$$|\partial Z^I g_{LL}|\lesssim |\bar{\partial }Z^Ig|\lesssim \frac{1}{1+s}|Z^{I+1} g|,$$
and then use the weighted Hardy inequality again. However, one has to be careful when using the weighted Hardy inequality. In the region $q>0$ the weight must be sufficiently large to allow to perform it twice. This is an other reason why we work with initial data in $H^N_\delta$ with $\delta>\frac{1}{2}$, which is more than the decay which is necessary
to prove the global well posedness of a semi linear wave equation with null structure.

\subsection{Interaction with the metric $g_b$}\label{secinter}
In this section we want to discuss the influence of the crossed terms between $g_b$ and $\phi, \wht g$. We will take as an example the equation for $\phi$,
$ \Box_g \phi= 0$. As discussed in Section \ref{cros}, we can focus on the term \eqref{crossphi}.
We may look at the following model problem
$$\Box \phi = \frac{\ep}{r}\chi(q)\bar{\partial}\phi.$$
If we perform the weighted energy estimate, we obtain
$$\frac{d}{dt}\|w_0(q)^\frac{1}{2}\partial Z^I \phi\|^2
+\|w_0'(q)^\frac{1}{2}\bar{\partial}Z^I \phi\|^2_{L^2}
\lesssim  \frac{\ep}{1+t}\|w_0^\frac{1}{2}\partial Z^I \phi\|_{L^2}^2.$$
Therefore
$$
\|w_0(q)^\frac{1}{2}\partial Z^I \phi\|_{L^2}\leq C_0 \ep (1+t)^{C\ep} $$
and for all $\sigma>0$
\begin{equation}
\label{modb2} \int_0^T \frac{1}{(1+t)^\sigma}\|w_0'(q)^\frac{1}{2}\bar{\partial}Z^I \phi\|^2_{L^2}dt
\lesssim  \ep^2.
\end{equation}

To avoid this logarithmic loss, we need to exploit more the structure of the equation. To this end we introduce the weight modulator
$$\left\{ \begin{array}{l}
           \alpha(q)= \frac{1}{(1+|q|)^{\sigma}}, \; q>0,\\
\alpha(q)= 1, \; q<0,
          \end{array}\right.$$
for $0<\sigma<\frac{1}{2}$.
Then the energy inequality yields
$$\frac{d}{dt}\|\alpha w_0(q)^\frac{1}{2}\partial Z^I \phi\|_{L^2}^2\leq \ep\left\|\ch_{q>0}\frac{\alpha w_0^\frac{1}{2}}{1+s}\bar{\partial} Z^I \phi\right\|_{L^2}\|\alpha w_0(q)^\frac{1}{2}\partial Z^I \phi\|_{L^2}.$$
We estimate, for $q>0$
$$\frac{\alpha(q)}{1+s}\lesssim \frac{1}{(1+t)^{\frac{1}{2}+\sigma}(1+|q|)^\frac{1}{2}} .$$
And therefore, we obtain
\begin{align*}
\frac{d}{dt}\|\alpha w_0(q)^\frac{1}{2}\partial Z^I \phi\|_{L^2}^2&\lesssim
\frac{\ep}{(1+t)^{\frac{1}{2}+\sigma}}
\left\|\ch_{q>0}\frac{ w_0^\frac{1}{2}}{\sqrt{1+|q|}}\bar{\partial} Z^I \phi\right\|_{L^2}\|\alpha w_0(q)^\frac{1}{2}\partial Z^I \phi\|_{L^2}\\
&\lesssim \frac{\ep}{(1+t)^\sigma}\|w_0'(q)^\frac{1}{2}\bar{\partial} Z^I \phi\|_{L^2}
+\frac{\ep}{(1+t)^{1+\sigma}}\|\alpha w_0(q)^\frac{1}{2}\partial Z^I \phi\|_{L^2}.
\end{align*}
and consequently in view of \eqref{modb2} we obtain
$$\|\alpha w_0(q)^\frac{1}{2}\partial Z^I \phi\|_{L^2} \leq C_0\ep+C\ep^2.$$
With this technique, the logarithmic loss in $t$ has been replaced by a small loss in $q$.

\section[Bootstrap assumptions and proof of the main theorem]{Bootstrap assumptions and proof of Theorem \ref{main}}\label{boot}
\subsection{Bootstrap assumptions}
Let $\frac{1}{2}<\delta<1$.
In view of the assumptions of Theorem \ref{main},
the initial data $(\phi_0,\phi_1)$ for $\phi$ are given in $H^{N+1}_\delta(\m R^2)\times H^N_{\delta+1}(\m R^2).$

For $\wht b \in W^{2,N}$ such that
$$\int_{\m S^1} \wht b=\int_{\m S^1} \wht b \cos(\theta)= \int_{\m S^1} \wht b \sin(\theta)=0,$$
and
$$\|\wht b\|_{W^{2,N}}\lesssim 2C_0 \ep^2,$$
Theorem \ref{thinitial} allows us to find initial data $g$ and $\partial_t g$ such that
 \begin{itemize}
 \item $g_{ij}$, $K_{ij}$ satisfy the constraint equations,
 \item $g$ and $\partial_t g$ are compatible with the decomposition $g=g_b+\wht g$, where 
 \begin{equation}
 \label{defdeb}b(\theta)=\wht b(\theta)+b_0+b_1\cos(\theta)+b_2\sin(\theta)
 \end{equation}
with $b_0,b_2,b_2, J(\theta)$  given by Theorem \ref{thinitial},
 \item the generalized wave coordinate condition given by $H_b$ is satisfied at $t=0$.
 \end{itemize}
 The system $\ref{gw}$ being a standard quasilinear system of wave equations, we know that there exists a solution until a time $T$. Moreover with our conditions on the initial data, our solution $(g,\phi)$ is solution of the Einstein equations \eqref{s1}, and the wave coordinate condition is satisfied for $t\leq T$ (see Appendix \ref{genw}).
 
 \begin{rk}
 Our choice of generalized wave coordinates does not change the hyperbolic structure because $H_b$ does not contain derivatives of $\wht g$.
 \end{rk}

We take three parameters $\rho,\sigma,\mu$ such that
\begin{equation}
\label{defpar}
0<\rho \ll \sigma \ll \mu \ll \delta,
\end{equation}
\begin{equation}
\label{codsigma}
\sigma+\rho <\delta-\frac{1}{2}.
\end{equation}
We consider a time $T>0$ such that there exists $ b(\theta)\in W^{N,2}(\m S^1)$ 
and a solution $(\phi,\wht g)$ of \eqref{s2} on $[0,T]$, associated to initial data for $g$.
%We recall that the initial data for $g$ given by Theorem \ref{thinitial}
%\begin{itemize}
%\item satisfies the constraint equations,
%\item satisfies the generalized wave coordinate condition at $t=0$
%\item is compatible with the decomposition $g=g_b+ \wht g$ with $\wht g\in H^{N+1}_\delta(\m R^2)$.
%\end{itemize}
We assume that on $[0,T]$, the following estimates hold.
\paragraph{Bootstrap assumptions for $b$}
%\begin{align}
%\label{estX}
%& |\phi|_{X_1}\lesssim \ep, \quad |\wht g|_{X_2}\lesssim \ep,\\
\begin{align}
\label{bootb1}
& \left\|\partial_\theta^I \left(\Pi b(\theta)+\Pi\int_{\Sigma_{T,\theta}} (\partial_q \phi)^2 r dq\right)\right\|_{L^2(\m S^1)}\leq B\frac{\ep^2}{\sqrt{T}}, \;for\; I\leq N-4\\
\label{bootb2}
& \|\partial_\theta^I b(\theta)\|_{L^2(\m S^1)}\leq B \ep^2 , \;for\; I\leq N
\end{align}
where $\Pi$ is the projection defined by \eqref{projection},
$\int_{\Sigma_{T,\theta}}$ is defined by \eqref{intsig} and $B$ is a constant depending on $\rho,\sigma,\mu, \delta, N$.
%The bootstrap assumptions $|\wht g|_{X_2}\lesssim \ep$ and $|\phi|_{X_1}\lesssim \ep$ are made more precise in the following of this section. The strategy to prove Theorem \ref{main} is to ameliorate the estimates.

We introduce four decomposition of the metric $g$
\begin{align}
\label{dec1}g=&g_b+\Up\left(\frac{r}{t}\right)h_0 dq^2+\widetilde{g_1},\\
\label{dec2}g=&g_b+\Up\left(\frac{r}{t}\right)(h_0+\wht h) dq^2+\widetilde{g_2},\\
\label{dec3}g=&g_b+\Up\left(\frac{r}{t}\right)h dq^2+\widetilde{g_3},\\
\label{dec4}g=&g_b+\Up\left(\frac{r}{t}\right)h dq^2+
\Up\left(\frac{r}{t}\right)k rdqd\theta+
\widetilde{g_4},
\end{align}
where $h_0$ is the solution of the transport equation
\begin{equation}\label{eqh}
 \left\{ \begin{array}{l} 
         \partial_q h_0 =-2r(\partial_q \phi)^2-2b(\theta)\partial^2_q(\chi(q)q), \\
	 h_0|_{t=0}=0,
        \end{array}
\right.
\end{equation}
$\wht{h}$ is solution of the linear wave equation
\begin{equation}\label{eqht}
 \left\{ \begin{array}{l} 
         \Box \wht h= \Box\left( \Up\left(\frac{r}{t}\right)h_0\right)+\Up\left(\frac{r}{t}\right)g_{LL}\partial^2_q h_0 +2\Up\left(\frac{r}{t}\right)(\partial_q \phi)^2 - 2(R_b)_{qq}+\Up\left(\frac{r}{t}\right)\wht Q_{\ba L \ba L}(h_0,\wht g), \\
	 (\wht h, \partial_t \wht h)|_{t=0}=(0,0),
        \end{array}
\right.
\end{equation}
where
\begin{equation}\label{qtildell}
\wht Q_{\ba L \ba L}(h_0,\wht g)=
\partial_{\ba L} g_{LL}\partial_{\ba L}h_0+
\partial_{\ba L} (g_b)_{UU}\partial_{\ba L} g_{\ba L L}.
\end{equation}

\begin{equation}\label{eqa}
 \left\{ \begin{array}{l} 
         \Box_g h= -2(\partial_q \phi)^2 + 2(R_b)_{qq}
         +Q_{\ba L \ba L}(h,\wht g), \\
	 (k, \partial_t k)|_{t=0}=(0,0) ,
        \end{array}
\right.
\end{equation}
and $k$ is the solution of
\begin{equation}\label{eqk}
 \left\{ \begin{array}{l} 
         \Box_g k= \partial_U g_{LL}\partial_q h, \\
	 (h, \partial_t h)|_{t=0}=(0,0).
        \end{array}
\right.
\end{equation}

\paragraph{$L^\infty$-based bootstrap assumptions}
For $I\leq N-14$ we assume
\begin{align}
\label{bootphi1}|Z^I \phi| &\leq \frac{2C_0\ep}{\sqrt{1+s}(1+|q|)^{\frac{1}{2}-4\rho}},\\  
\label{bootg1}|Z^I \wht g_1| &\leq \frac{2C_0\ep}{(1+s)^{\frac{1}{2}-\rho}},
\end{align}
where here and in the following, $C_0$ is a constant depending on $\rho,\sigma,\mu, \delta, N$ such that the inequalities
are satisfied at $t=0$ with $2C_0$ replaced by $C_0$.
For $I\leq N-12$ we assume
\begin{align}
\label{bootphi2}|Z^I \phi| &\leq \frac{2C_0\ep}{(1+s)^{\frac{1}{2}-2\rho}},\\ 
\label{bootg2}|Z^I \wht g_1| &\leq \frac{2C_0\ep}{(1+s)^{\frac{1}{2}-2\rho}}.
\end{align}
We assume the following estimate for $h_0$ for $I\leq N-7$ and $q<0$
\begin{equation}
\label{booth1}|Z^I h_0|\leq \frac{2C_0\ep}{(1+s)^\frac{1}{2}}+\frac{2C_0\ep}{(1+|q|)^{1-4\rho}}.
\end{equation}
and for $q>0$
\begin{equation}
\label{booth2}
|Z^I h_0|\leq \frac{2C_0\ep}{(1+|q|)^{2+2(\delta-\sigma)}}.
\end{equation}
We also assume the following for $\wht h$ and $I\leq N-7$
\begin{equation}
\label{booth3}|Z^I \wht h|\leq \frac{2C_0\ep}{(1+|q|)^{\frac{1}{2}-\rho}}.
\end{equation}
\paragraph{$L^2$-based bootstrap assumptions}
We introduce four weight functions
$$\left\{ \begin{array}{l}
           w_0(q)= (1+|q|)^{2+ 2\delta}, \; q>0,\\
w_0(q)= 1+ \frac{1}{(1+|q|)^{2\mu}}, \; q<0,
          \end{array}\right.$$
$$\left\{ \begin{array}{l}
           w_1(q)= (1+|q|)^{2+2\delta}, \; q>0,\\
w_1(q)= \frac{1}{(1+|q|)^{\frac{1}{2} }}, \; q<0,
          \end{array}\right.$$
$$\left\{ \begin{array}{l}
           w_2(q)= (1+|q|)^{2+2\delta}, \; q>0,\\
 w_2(q)= \frac{1}{(1+|q|)^{1 +2\mu}}, \; q<0,
          \end{array}\right.$$
          
$$\left\{ \begin{array}{l}
                     w_3(q)= (1+|q|)^{3+2\delta}, \; q>0,\\
           w_3(q)= 1+ \frac{1}{(1+|q|)^{2\mu}}, \; q<0.
                    \end{array}\right.$$
We also introduce weight modulators
\begin{equation}
\label{wmod}\left\{ \begin{array}{l}
           \alpha(q)= \frac{1}{(1+|q|)^{\sigma}}, \; q>0,\\
\alpha(q)= 1, \; q<0,
          \end{array}\right.
\end{equation}
$$\left\{ \begin{array}{l}
           \alpha_2(q)= \frac{1}{(1+|q|)^{2\sigma}}, \; q>0,\\
\alpha_2(q)= 1, \; q<0.
          \end{array}\right.$$
We assume the following estimate for $I\leq N$
\begin{equation}
\label{bootl21}  \begin{split}&\|\alpha_2w_0(q)^\frac{1}{2}\partial Z^I \phi\|_{L^2}+\|\alpha_2 w_2(q)^\frac{1}{2}\partial Z^I \wht g_4 \|_{L^2}\\+&\frac{1}{\sqrt{1+t}}  \|\alpha_2(q)^\frac{1}{2}\partial Z^I h\|_{L^2} +\frac{1}{\sqrt{1+t}} \|\alpha_2w_3(q)^\frac{1}{2}\partial Z^I k\|_{L^2} \leq 2C_0 \ep (1+t)^\rho.
\end{split}
\end{equation}
 for $I\leq N-1$
\begin{equation}
\label{bootl22}   \|w_0(q)^\frac{1}{2}\partial Z^I \phi\|_{L^2}+\|w_2(q)^\frac{1}{2}\partial Z^I \wht g_3 \|_{L^2}+\frac{1}{\sqrt{1+t}}  \|w_3(q)^\frac{1}{2}\partial Z^I h\|_{L^2}  \leq 2C_0\ep(1+t)^\rho
\end{equation}
and for $I\leq N-2$
\begin{equation}
\label{bootl23} \|\alpha(q)w_0(q)^\frac{1}{2}\partial Z^I \phi\|_{L^2}+\|\alpha(q) w_2(q)^\frac{1}{2}\partial Z^I \wht g_3 \|_{L^2}+\frac{1}{\sqrt{1+t}}  \|\alpha(q)w_3(q)^\frac{1}{2}\partial Z^I h\|_{L^2}  \leq 2C_0\ep. 
\end{equation}
In addition, for $I\leq N-8$ we assume
\begin{equation}
\label{bottl24} \|w_1(q)^\frac{1}{2}\partial Z^I \wht{g_2}\|_{L^2} \leq 2C_0 \ep(1+t)^\rho, \quad \|\alpha(q)w_1(q)^\frac{1}{2}\partial Z^I \wht{g_2}\|_{L^2} \leq 2C_0\ep
\end{equation}
and for  $I\leq N-9$ we assume
\begin{equation}
\label{bootl25} \|w_0(q)^\frac{1}{2}\partial Z^I \wht{g_2}\|_{L^2} \leq 2C_0\ep(1+t)^\rho,\quad \|\alpha(q)w_0(q)^\frac{1}{2}\partial Z^I \wht{g_2}\|_{L^2} \leq 2C_0 \ep.
\end{equation} 

Let us do two remarks to justify our different decompositions of the metric, and our different weight functions.
\begin{rk}
We use the decomposition \eqref{dec1} instead of \eqref{dec2} to avoid a logarithmic loss when we want to improve \eqref{bootg1} and \eqref{bootg2} with the $L^\infty-L^\infty$ estimate. This loss would have been due to the terms coming from the non commutation of the wave operator with the null decomposition \eqref{dec2}.
However,  we use the decomposition \eqref{dec2} instead of  \eqref{dec1} to avoid a logarithmic loss in the energy estimate due to the term $\wht Q_{\ba L \ba L}$.  

When $h_0$ is a good approximation for $h$, we use the decomposition \eqref{dec2} instead of \eqref{dec3} in the energy estimate. This allow us to have a better control on the terms coming from the non commutation of the wave operator with the null decomposition.
When $h_0$ is no longer a good approximation for $h$, we use the decomposition \eqref{dec3}. Finally, the decomposition \eqref{dec4}
allow us to isolate the term
$Z^N\partial_U g_{LL}\partial_{\ba L }g_{\ba LL}$ on which we do not have a good control.
\end{rk}

\begin{rk}
The weight $w_2$ is introduced to deal with the non commutation of the wave operator with the null decomposition (see Section \ref{secnon}). The weight $w_1$ is a transition weight between $w_0$ and $w_2$.
The weight $w_3$ allows us to compensate the loss in $\sqrt{1+t}$ for $g_{\ba L \ba L}$ by an additional decay in $\sqrt{1+|q|}$ in the exterior region.

The weight modulators $\alpha_1$ and $\alpha_2$ are introduced to transform the logarithmic loss due to the interaction with the metric $g_b$ in a small loss in $q$ (see Section \ref{secinter}).
\end{rk}

\subsection{Proof of Theorem \ref{main}}
We have the following improvement for the bootstrap assumptions. The constant $C$ will denote a constant depending only on $\rho,\sigma,\mu,\delta,N$. The proof of Proposition \ref{prpbooth1} is the object of Section \ref{transport}.
\begin{prp}
\label{prpbooth1}
Let $I\leq N-5$. We have the estimates
$$|Z^I h_0|\leq \frac{C\ep^2}{\sqrt{1+s}}+\frac{C\ep^2}{(1+|q|)^{1-4\rho}},\; for  \; q<0, \quad
|Z^I h_0|\leq \frac{C\ep^2}{(1+|q|)^{2+2(\delta-\sigma)}}, \; for \; q>0.$$
Let $I\leq N-7$. We have the estimate
$$|Z^I \wht h|\leq\frac{C\ep^2}{(1+s)^{\frac{1}{2}-\rho}}.$$
\end{prp}
The proof of Proposition \ref{prpboot1} is the object of Section \ref{seclinf}.
\begin{prp}
\label{prpboot1}
Let $I \leq N-14$. We have the estimates
$$
|Z^I \wht g_1| \leq \frac{C_0\ep+C\ep^2 }{(1+s)^{\frac{1}{2}-\rho}},\quad
|Z^I \phi| \leq \frac{C_0\ep+C\ep^2 }{\sqrt{1+s}(1+|q|)^{\frac{1}{2}-4\rho}}.
$$
Let $I\leq N-12$. We have the estimates
$$
| Z^I \phi| \leq \frac{C_0\ep+C\ep^2}{(1+s)^{\frac{1}{2}-2\rho}}, \quad
| Z^I \wht g_1| \lesssim \frac{C_0\ep+C\ep^2}{(1+s)^{\frac{1}{2}-2\rho}}. 
$$
\end{prp}
The proof of Proposition \ref{prpboot2} is the object of Section \ref{secl2}.
\begin{prp}\label{prpboot2}
We have the estimates for $I\leq N$
$$\|\alpha_2w_0(q)^\frac{1}{2}\partial Z^I \phi\|_{L^2}+\|\alpha_2 w_2(q)^\frac{1}{2}\partial Z^I \wht g_4 \|_{L^2}
\leq (C_0\ep+\ep)(1+t)^{C\sqrt{\ep}},$$
$$ \|\alpha_2(q)^\frac{1}{2}\partial Z^I h\|_{L^2} +\|\alpha_2w_3(q)^\frac{1}{2}\partial Z^I k\|_{L^2} \leq C\ep^2(1+t)^{\frac{1}{2}+C\sqrt{\ep}},$$
for $I\leq N-1$
$$ \|w_0(q)^\frac{1}{2}\partial Z^I \phi\|_{L^2}+\|w_2(q)^\frac{1}{2}\partial Z^I \wht g_3 \|_{L^2}\leq (C_0\ep+\ep)(1+t)^{C\sqrt{\ep}},$$
$$ \|w_3(q)^\frac{1}{2}\partial Z^I h\|_{L^2} \leq C\ep^2(1+t)^{\frac{1}{2}+C\sqrt{\ep}},$$
for $I\leq N-2$
$$ \|\alpha(q)w_0(q)^\frac{1}{2}\partial Z^I \phi\|_{L^2}+\|\alpha(q) w_2(q)^\frac{1}{2}\partial Z^I \wht g_3 \|_{L^2} \leq C_0\ep + C\ep^\frac{5}{4},$$
$$\|\alpha(q)w_3(q)^\frac{1}{2}\partial Z^I h\|_{L^2}  \leq C\ep^\frac{3}{2},$$
for $I\leq N-7$
$$\|w_1(q)^\frac{1}{2}\partial Z^I \wht{g_2}\|_{L^2} \leq C_0\ep(1+t)^{C\sqrt{\ep}}+\ep, \quad \|\alpha(q)w_1(q)^\frac{1}{2}\partial Z^I \wht{g_2}\|_{L^2} \leq C_0\ep+C\ep^\frac{5}{4},$$
and for $I\leq N-8$
$$\|w_0(q)^\frac{1}{2}\partial Z^I \wht{g_2}\|_{L^2} \leq C_0\ep(1+t)^{C\sqrt{\ep}}+\ep,\quad \|\alpha(q)w_0(q)^\frac{1}{2}\partial Z^I \wht{g_2}\|_{L^2} \leq C_0 \ep +C\ep^\frac{5}{4}.$$
\end{prp}
The proof of Proposition \ref{prpbootb} is the object of Section \ref{secproof}
\begin{prp}
\label{prpbootb}
We assume that the time $T$ satisfies
$$T\leq \exp\left(\frac{C}{\sqrt{\ep}}\right).$$
There exists $b^{(2)}(\theta)\in W^{N,2}(\m S^1)$ and $(\phi{(2)},g^{(2)})$ solution of \eqref{s1} in the generalized wave coordinates $H_{b^{(2)}}$, such that, if we write
$g^{(2)}=g_{b^{2}}+\wht g$, then $(\phi{(2)},\wht g^{(2)})$ satisfies the same estimate as $(\phi, \wht g)$, and 
we have the estimates for $ b^{(2)}$
\begin{align*}
& \left\|\partial_\theta^I \left(\Pi b^{(2)}(\theta)+\Pi\int_{\Sigma_{T,\theta}} (\partial_q \phi)^2 r dq\right)\right\|_{L^2}\leq C
\frac{\ep^4}{\sqrt{T}}, \;for\; I\leq N-4,\\
& \|\partial_\theta^I b(\theta) \|_{L^2}\leq 2C_0^2 \ep^2 , \;for\; I\leq N.
\end{align*}
\end{prp}
We may now prove Theorem \ref{main}.
\begin{proof}[Proof of Theorem \ref{main}]
We may choose $C_0$ such that $C_0\geq 2$, and $B$ such that $B\geq 4C_0^2$.
We take $\ep$ small enough so that 
$$C\ep^\frac{1}{4} \leq \frac{C_0}{2}, \quad C\sqrt{\ep} \leq \rho, \quad C\ep \leq \frac{B}{2}.$$
Then Propositions \ref{prpbooth1}, \ref{prpboot1}, \ref{prpboot2} imply that the bootstrap assumptions for $(\phi, \wht g)$ are true with the constant $2C_0$ replaced by $\frac{3C_0}{2}$. Moreover Proposition \ref{prpbootb} yields the existence of $b^{(2)}$ and $\phi^{(2)},g^{(2)}=g_{b^{(2)}}+\wht g^{(2)}$ solution of \eqref{s1}, such that the bootstrap assumptions are satisfied by $(\phi^{(2)},\wht g^{(2)})$ with the constant $2C_0$ replaced by $\frac{3C_0}{2}$, and $b^{(2)}$ satisfy
\begin{align*}
& \left\|\partial_\theta^I \left(\Pi b^{(2)}(\theta)+\Pi\int_{\Sigma_{T,\theta}} (\partial_q \phi^{(2)})^2 r dq\right)\right\|_{L^2}\leq B
\frac{\ep^2}{2\sqrt{T}}, \;for\; I\leq N-4,\\
& \|\partial_\theta^I b(\theta)\|_{L^2}\leq \frac{B}{2}\ep^2 , \;for\; I\leq N.
\end{align*}
This concludes the proof of Theorem \ref{main}.
\end{proof}

Let us note that the only place where we use the assumption $T\leq \exp\left(\frac{C}{\sqrt{\ep}}\right)$ is in the proof of Proposition \ref{prpbootb}.

%\begin{equation
%\label{impr1}|Z^I h_0|\leq \frac{C_0\ep}{2\sqrt{1+|q|}},\; for  \; q<0, \quad
%|Z^I h_0|\lesssim \frac{C_0\ep}{2(1+|q|)^{2+2(\delta-\sigma)}}, \; for \; q>0,
%\end{equation}

\subsection{First consequences of the bootstrap assumptions}
Thanks to the weighted Klainerman-Sobolev inequality the bootstrap assumptions immediately imply the following proposition.
\begin{prp}\label{estks}
We assume $I\leq N-4$ we have the estimates, for $q<0$
\begin{align}
 \label{ks1}|\partial Z^I \phi(t,x)|&\lesssim \frac{\ep}{\sqrt{1+|q|}\sqrt{1+s}},\\
\label{ks2}|\partial Z^I \wht{g_3}(t,x)|&\lesssim \frac{\ep(1+|q|)^\mu}{\sqrt{1+s}},\\
 \label{ks3}|\partial Z^I h|&\lesssim \frac{\ep}{\sqrt{1+|q|}},
\end{align}
and for $q>0$ 
\begin{align}
 \label{ks4}|\partial Z^I \phi(t,x)|&\lesssim \frac{\ep}{(1+|q|)^{\frac{3}{2}+\delta-\sigma}\sqrt{1+s}},\\
\label{ks5}|\partial Z^I \wht{g_3}(t,x)|&\lesssim \frac{\ep}{(1+|q|)^{\frac{3}{2}+\delta-\sigma}\sqrt{1+s}},\\
\label{ks6} |\partial Z^I h|&\lesssim \frac{\ep}{(1+|q|)^{2+\delta-\sigma}}.
\end{align}
Moreover, for $I\leq N-11$ we have for $q<0$
 \begin{equation}
 \label{ks7} |\partial Z^I \wht g_2(t,x)|\lesssim \frac{\ep}{\sqrt{1+|q|}\sqrt{1+s}}
 \end{equation}
\end{prp}
Thanks to Lemma \ref{lmintegration} we deduce the following corollary.
\begin{cor}\label{estksg}
 We assume $I\leq N-4$ we have the estimates, for $q<0$
\begin{align}
 \label{iks1}| Z^I \phi(t,x)|&\lesssim \frac{\ep\sqrt{1+|q|}}{\sqrt{1+s}},\\
\label{iks2}| Z^I \wht{g_3}(t,x)|&\lesssim \frac{\ep(1+|q|)^{1+\mu}}{\sqrt{1+s}},\\
 \label{iks3}| Z^I h|&\lesssim \ep \sqrt{1+|q|}.
\end{align}
and for $q>0$ 
\begin{align}
\label{iks5} | Z^I \phi(t,x)|&\lesssim \frac{\ep}{(1+|q|)^{\frac{1}{2}+\delta-\sigma}\sqrt{1+s}},\\
\label{iks6}| Z^I \wht{g_3}(t,x)|&\lesssim \frac{\ep }{(1+|q|)^{\frac{1}{2}+\delta-\sigma}\sqrt{1+s}},\\
 \label{iks7}| Z^I h|&\lesssim \frac{\ep}{(1+|q|)^{1+\delta-\sigma}}.
\end{align}
Moreover, for $I\leq N-11$ we have for $q<0$
\begin{equation}\label{iks8}
|Z^I \wht g_2(t,x)|\leq \frac{\ep\sqrt{1+|q|}}{\sqrt{1+s}}.
\end{equation}
\end{cor}

The following remark allow us to compare the different decompositions of the metric $g$.
\begin{rk}\label{comp}
We have the following relations
$$\wht g_{\q T \q T}=(\wht g_1)_{\q T \q T}=(\wht g_2)_{\q T \q T}=(\wht g_3)_{\q T \q T}=(\wht g_4)_{\q T \q T},$$
$$\wht g_{L \ba L}=(\wht g_1)_{L \ba L}=(\wht g_2)_{L \ba L}=(\wht g_3)_{L \ba L}=(\wht g_4)_{L \ba L},$$
$$\wht g_{U \ba L}=(\wht g_1)_{U \ba L}=(\wht g_2)_{U \ba L}=(\wht g_3)_{U \ba L}.$$
\end{rk}

The following corollary allow us to estimate $\wht g$, independently of the chosen decomposition \eqref{dec1}, \eqref{dec2}, \eqref{dec3} or \eqref{dec4}.
\begin{cor} We have the following estimates
\label{esttildg}
\begin{align}
\label{est1} |Z^I \wht g|&\lesssim \frac{\ep}{(1+|q|)^{\frac{1}{2}-\rho}}, \; for \; I\leq N-14,\\
\label{est2} |Z^I \wht g|&\lesssim \frac{\ep}{(1+|q|)^{\frac{1}{2}-2\rho}}, \; for \; I\leq N-12,\\
\label{est3} |Z^I \wht g|&\lesssim \ep, \quad |\partial Z^I \wht g|\lesssim \frac{\ep}{1+|q|},\; for \; I\leq N-11,\\
\label{est4} |Z^I \wht g|&\lesssim \ep(1+|q|)^{\frac{1}{2}+\mu},\quad |\partial Z^I \wht g|\lesssim \ep(1+|q|)^{-\frac{1}{2}+\mu}, \; for \; I\leq N-4.\\
\end{align}
Moreover, for $q>0$ we have the following estimate
\begin{equation}
\label{est5}|Z^I \wht g|\lesssim \frac{\ep}{(1+|q|)^{1+\delta-\sigma}}, \; for \; I\leq N-4.
\end{equation}
\end{cor}
\begin{proof}
Estimate \eqref{est1} is obtained by using the decomposition \eqref{dec1} and
taking the maximum of the bounds given by \eqref{booth1} and \eqref{bootg1}. Estimate \eqref{est2} is obtained by using the decomposition \eqref{dec1} and
taking the maximum of the bounds given by \eqref{booth1} and \eqref{bootg2}. Estimate \eqref{est3} is obtained by using the decomposition \eqref{dec2} and
taking the maximum of the bounds given by \eqref{booth1}, \eqref{booth3} and \eqref{iks7}. Estimate \eqref{est4} is obtained by using the decomposition \eqref{dec3} and
taking the maximum of the bounds given by \eqref{iks3} and \eqref{iks2}.
Estimate \eqref{est5} is obtained by using the decomposition \eqref{dec3} and
taking the maximum of the bounds given by \eqref{iks7} and \eqref{iks6}.
\end{proof}

The rest of the paper is as followed
\begin{itemize}
\item In Section \ref{secwave}, we use the wave coordinates condition to obtain better decay on the coefficients $g_{\q T \q T}$ of the metric. The strategy is similar to the one introduced in \cite{lind}.
\item In Section \ref{secangle}, we obtain the missing estimates for the angle and linear momentum, namely the three first Fourier coefficient of $b$ which correspond to $b-\Pi b$, in order to get
$$\left|b(\theta)+\int_{\Sigma_{T,\theta}}(\partial_q \phi(q,s=T,\theta))^ 2 r dq\right| \lesssim \frac{\ep^2}{T^\frac{1}{2}},$$
by relying in particular on the constraint equations
\item In Section \ref{transport}, we improve the estimates for $h_0$, and show that it is indeed a good approximation for the coefficient $g_{\ba L \ba L}$. We also obtain estimates for $\wht h$. We prove Proposition \ref{prpbooth1}.
\item In Section \ref{seclinf} we prove Proposition \ref{prpboot1} thanks to the $L^\infty-L^\infty$ estimate.
\item In Section \ref{weighte} we derive a weighted energy estimate for an equation of the form $\Box_g u=f$, where $g$ satisfies the bootstrap assumptions.
\item In Section \ref{secl2}, we prove Proposition \ref{prpboot2} thanks to the  weighted energy estimate.
\item In Section \ref{secproof}, we prove Proposition \ref{prpbootb} by picking the right $\wht b=\Pi b$.
\end{itemize}

\section{The wave coordinates condition}\label{secwave}
The wave coordinates condition yields better decay properties in $t$ for some components of the metric. Since far from a conical neighborhoud of the light cone, we have $|q|\sim s$, this condition will only be relevant near the light cone. It is given by
$$H_b^\alpha = -\frac{1}{\sqrt{|det(g)|}}\partial_\mu (g^{\mu \alpha}\sqrt{|\det(g)}|).$$
\begin{prp}\label{estLL}
We have the following estimate, in the region $\frac{t}{2}\leq r \leq 2t$,
$$
|\partial_q Z^I \wht g_{LL}|\lesssim 
\sum_{J\leq I}\left(|\bar{\partial}Z^J \wht g_{\ba L L}|+|\bar{\partial}Z^J \wht g_{\q T \q T}|\right)
+\frac{1}{1+s}\sum_{J\leq I}
\left(|Z^I \wht g_{L\ba L}|+|Z^I \wht g_{\q T \q T}|\right).$$
\end{prp}
\begin{proof}
The wave coordinate condition implies
\begin{align*}
-\ba L_{\alpha}H_b^\alpha &=\ba L_{\alpha}\left(\frac{1}{\sqrt{|\det(g)|}}\partial_\mu (g^{\mu \alpha}\sqrt{\det(g)})\right)\\
=&\frac{ g^{\mu \alpha}}{\sqrt{\det(|g|)}}\ba L_{\alpha}\partial_\mu \sqrt{\det(g)}
+\partial_\mu (\ba L_\alpha g^{\mu \alpha})-g^{\mu \alpha}\partial_\mu( \ba L_\alpha)\\
=&\frac{g^{\ba L \mu}}{\sqrt{\det(|g|)}}\partial_{\mu}\sqrt{\det (g)}
+\partial_\mu (g^{\ba L \mu})
-\frac{1}{r}g^{UU}\\
=&\frac{g^{\ba L \ba L}}{\sqrt{\det(|g|)}}\partial_{\ba L} \sqrt{\det (g)}
+\frac{g^{\ba L \q T}}{\sqrt{\det(|g|)}}\partial_{\q T} \sqrt{\det (g)}
+\partial_{\ba L} g^{\ba L \ba L}+\partial_U g^{\ba L U}+\partial_L g^{\ba L L}\\
&+\frac{1}{r}g^{\ba L R}-\frac{1}{r}g^{UU},
\end{align*}
where we have denoted by $R$ the vector field $\partial_r$, and used the following calculations
\begin{align*}
g^{\mu \alpha}\partial_\mu(\ba L_\alpha)=&-g^{\mu \alpha}\partial_\mu( R_\alpha)\\
=&-g^{11}\partial_1 \cos(\theta)-g^{12}(\partial_2 \cos(\theta)-\partial_1 \sin(\theta))-g^{22}\partial_2 \sin(\theta)\\
=&-\frac{g^{UU}}{r},
\end{align*}
\begin{align*}
\partial_\mu g^{\ba L \mu}=&\partial_0 g^{\ba L0} +\partial_1 g^{\ba L 1}
+\partial_2 g^{\ba L 2}\\
=&\partial_0 g^{\ba L0}+\partial_R g^{\ba L R}+\partial_U g^{\ba L U}
+g^{\ba L R}(\partial_1 \cos(\theta)+\partial_2 \sin(\theta))
+g^{\ba L U}(-\partial_1\sin(\theta)+\partial_2 \cos(\theta))\\
=&\partial_{\ba L} g^{\ba L \ba L}+\partial_U g^{\ba L U}+\partial_L g^{\ba L L}+\frac{g^{\ba L R}}{r}.
\end{align*}
Consequently
\begin{equation}\label{condonde1}
\begin{split}
\partial_{\ba L} g^{\ba L \ba L}
=& -\ba L_{\alpha}\left(\bar{H}_b^\alpha+F^\alpha\right)-\frac{g^{\ba L \ba L}}{\sqrt{\det(|g|)}}\partial_{\ba L} \sqrt{\det (g)}
-\frac{g^{\ba L \q T}}{\sqrt{\det(|g|)}}\partial_{\q T} \sqrt{\det (g)}\\
&-\partial_U g^{\ba L U}-\partial_L g^{\ba L L}
-\frac{1}{r}g^{\ba L R}-\frac{1}{r}g^{UU},
\end{split}
\end{equation}
where we have used \eqref{defHalpha}.
Also we have
$$\det(g)=g_{LL}(g_{\ba L \ba L}g_{UU}-(g_{U \ba L})^2) -g_{L\ba L}(g_{L \ba L}g_{UU}-g_{LU}g_{\ba L U} ) + 
g_{LU}(g_{\ba L L}g_{U \ba L}-g_{\ba L \ba L}g_{LU}).
$$
Therefore
$$| \sqrt{\det(g)}-\sqrt{\det(g_b)}|\lesssim 
|\wht g_{\ba L L}|+|\wht g_{\q T \q T}|.$$
We can express
\begin{align*}
g^{\ba L \ba L}=\frac{1}{\det(g)}(g_{ L  L}g_{UU}-(g_{U  L})^2)&=-\frac{1}{4}\wht g_{LL}+ O(\wht g_{\q T \q T})O(g ),\\
g^{\ba L U}=\frac{1}{\det(g)}(g_{\ba L  L}g_{LU}-g_{U \ba L}g_{LL})&=\frac{1}{2}g_{LU}+ O(\wht g_{\q T \q T})O(g ),\\
g^{\ba L L}=\frac{1}{\det(g)}(g_{\ba L  L}g_{UU}-g_{U \ba L}g_{UL})&=\frac{1}{4}(g_b)_{UU}g_{L\ba L}+ O(\wht g_{\q T \q T}),\\
\end{align*}
where we have used the notation $O(g)=O(g-m)$ where $m$ is the Minkowski metric.
Since in \eqref{condonde1}, by definition of $\bar{H}^\alpha$ (see \eqref{defHbar}) the terms involving only $g_b$ compensate, we have
$$|\partial_q \wht g_{LL}|\lesssim (|\bar{\partial} \wht g_{L \ba L}|
+| \bar{\partial} \wht g_{\q T \q T}|)
+\frac{1}{1+s} (|\wht g_{L \ba L}|+|\wht g_{\q T \q T}|)
+s.t..$$
where $s.t$ denotes similar terms (here these terms are quadratic terms with a better or similar decay), and we have used the fact that in the region $\frac{t}{2}\leq r \leq 2t$, we have $r\sim s$.
Since $[Z,\partial_q]\sim \partial_q$ and $[Z,\bar{\partial}]\sim \bar{\partial}$ we have
$$|\partial_q Z^I \wht g_{LL}|\lesssim \sum_{J\leq I-1}|Z^J\wht g_{LL}|+|\bar{\partial}Z^I \wht g_{L \ba L}|+|\bar{\partial}Z^I \wht g_{\q T \q T}|+
\frac{1}{1+s}\sum_{J\leq I} (|Z^J \wht g_{L \ba L}|+|Z^J \wht g_{\q T \q T}|).$$
This concludes the proof of Proposition \ref{estLL}.
\end{proof}

The other two contractions of the wave condition yield better decay on a conical neighbourhood of the light cone for $\wht g_{U L}$ and $\wht g_{UU}$.
\begin{prp}\label{estLU}
We have the following property
\begin{align*}
|\partial_q Z^I \wht g_{UL}|&\lesssim \sum_{J\leq I} |\overline{\partial}Z^J \wht g_{\q T \q V}| + \frac{1}{1+s}\sum_{J \leq I} |Z^J \wht g_{\q T \q V}|,\\
|\partial_q Z^I \wht g_{UU}|&\lesssim \sum_{J\leq I} |\overline{\partial}Z^J \wht g| + \frac{1}{1+s}\sum_{J \leq I} |Z^J \wht g|.
%+\sum_{J\leq I} |Z^K g_{LL}||\partial Z^{J-K}g_{\ba L \ba L}| + |Z^K g_{\ba L \ba L}||\partial Z^{J-K} g_{LL}|
\end{align*}
\end{prp}

\begin{proof}
To obtain the first estimate, we contract the wave coordinate condition with the vector field $U$.
\begin{align*}
-U_{\alpha}H_b^\alpha &= \frac{1}{\sqrt{|\det(g)|}}U_{\alpha}\partial_\mu (g^{\mu \alpha})\sqrt{\det(g)}\\
=&\frac{ g^{\mu \alpha}}{\sqrt{|\det(g)|}}U_{\alpha}\partial_\mu \sqrt{|\det(g)|}
+\partial_\mu (U_\alpha g^{\mu \alpha})+g^{\mu \alpha}\partial_\mu( U_\alpha)\\
=&\frac{g^{U \mu}}{\sqrt{|\det(g)|}}\partial_{\mu}\sqrt{|\det (g)|}
+\partial_\mu (g^{U \mu})
+\frac{1}{r}g^{UR}\\
=&\frac{g^{U \ba L}}{\sqrt{|\det(g)|}}\partial_{\ba L} \sqrt{|\det (g)|}
+\frac{g^{U \q T}}{\sqrt{|\det(g)|}}\partial_{\q T} \sqrt{|\det (g)|}
+\partial_{\ba L} g^{U \ba L}+\partial_U g^{U U}+\partial_L g^{U L}+\frac{1}{r}g^{UR}. 
\end{align*}
Therefore 
\begin{equation*}
\partial_{\ba L} g^{U \ba L}=-U_{\alpha}H_b^\alpha-\frac{g^{U \ba L}}{\sqrt{|\det(g)|}}\partial_{\ba L} \sqrt{|\det (g)|}
-\frac{g^{U \q T}}{\sqrt{|\det(g)|}}\partial_{\q T} \sqrt{|\det (g)|}
-\partial_U g^{U U}-\partial_L g^{U L}-\frac{1}{r}g^{UR}.
\end{equation*}
and arguing as in Proposition \ref{estLL} we infer
$$|\partial_q\wht g_{UL}|\lesssim |\bar{\partial} \wht g_{\q T \q V}|+\frac{1}{1+s}|\wht g_{\q T \q V}|+s.t.$$
Commuting with the vector fields $Z$ as before, we obtain the desired estimate.
To obtain the second one, we contract the wave coordinate condition with $L$
\begin{equation}\label{wcl}\begin{split}
L_{\alpha}H_b^\alpha = &\frac{1}{\sqrt{|\det{g}|}}L_{\alpha}\partial_\mu (g^{\mu \alpha})\sqrt{|\det(g)|}.\\
=&\frac{1}{\sqrt{|\det{g}|}}\partial_{\ba L}\left(\sqrt{|\det(g)|}g^{L\ba L}\right) +\frac{1}{\sqrt{|\det{g}|}}\partial_{\q T}\left(\sqrt{|\det(g)|}g^{L\q T}\right)-g^{\mu \alpha}\partial_\mu (L_\alpha). \end{split}
\end{equation}
We note that 
\begin{align*}
\sqrt{|\det(g)|}g^{L\ba L}=&\frac{1}{\sqrt{|\det(g)|}}(g_{L\ba L}g_{UU}-g_{U\ba L}g_{UL})\\
=&\frac{g_{L\ba L}g_{UU}}{\sqrt{g_{L\ba L}^2g_{UU}+O(\wht g_{\q T \q T})O(g)}}
+O(\wht g_{\q T \q T})O(g)\\
=&\sqrt{g_{UU}}+ +O(\wht g_{\q T \q T})O(g).
\end{align*}
Therefore \eqref{wcl} yields
$$|\partial_q \wht g_{UU}|\lesssim |\bar{\partial}\wht g|
+\frac{1}{1+s}|\wht g|.$$
We commute with the vector fields $Z$ to conclude.
\end{proof}

Thanks to the bootstrap assumptions, we obtain the following corollary.
\begin{cor}\label{estV}
We have the estimates for $q<0$
\begin{align}
\label{dwc1}&|\partial  Z^I \wht g_{UU}| \lesssim \frac{\ep}{(1+s)^{\frac{3}{2}-\rho}},
\quad |\partial  Z^I \wht g_{L\q T}| \lesssim \frac{\ep}{(1+s)(1+|q|)^{\frac{1}{2}-\rho}},
 \quad for\; I \leq N-15,\\
\label{dwc2}&|\partial \wht Z^I g_{L \q T}| \lesssim \frac{\ep}{(1+s)^{\frac{3}{2}-2\rho}} ,
\quad |\partial  Z^I \wht g_{UU}| \lesssim \frac{\ep}{(1+s)(1+|q|)^{\frac{1}{2}-2\rho}},
\quad for\; I \leq N-13,\\
\label{dwc3}&|\partial Z^I \wht g_{L \q T}| \lesssim \frac{\ep \sqrt{1+|q|}}{(1+s)^\frac{3}{2}}, 
\quad |\partial  Z^I \wht g_{UU}| \lesssim \frac{\ep}{1+s},
\quad for\; I \leq N-12 ,\\
\label{dwc4}&|\partial Z^I \wht g_{L \q T}| \lesssim \frac{\ep (1+|q|)^{1+\mu}}{(1+s)^\frac{3}{2}},
\quad |\partial Z^I \wht g_{UU}| \lesssim \frac{\ep (1+|q|)^{\frac{1}{2}+\mu}}{1+s},
 \quad for\; I \leq N-5,
\end{align}
and for $q>0$
$$
|\partial Z^I \wht g_{L \q T}| \lesssim \frac{\ep }{(1+|q|)^{\frac{1}{2}+\delta-\sigma}(1+s)^\frac{3}{2}}, \quad 
|\partial Z^I \wht g_{UU}| \lesssim \frac{\ep }{(1+|q|)^{1+\delta-\sigma}(1+s)} ,\quad
for\; I \leq N-5.
$$
\end{cor}
\begin{proof}
As mentioned in Remark \ref{comp}, the metric coefficients $\wht g_{\q V \q T}$ do not depend on the choice of decomposition between \eqref{dec1}, \eqref{dec2} and \eqref{dec3}.
Thanks to Proposition \ref{estLL} and \ref{estLU}, and the fact that
$$|\bar{\partial} u|\leq \frac{1}{1+s}|Zu|,$$
we may write
\begin{equation}
\label{esti1}|\partial Z^I \wht g_{L\q T}|\lesssim \frac{1}{1+s}|Z^{I+1} \wht g_{\q T \q V}|.
\end{equation}
The bootstrap assumptions \eqref{bootg1} and \eqref{bootg2} in the region $q<0$ yield
\begin{align*}
|Z^J \wht g_{\q T \q V}|&\lesssim \frac{\ep}{(1+s)^{\frac{1}{2}-\rho}}, \quad for\; J\leq N-14,\\
|Z^J \wht g_{\q T \q V}|&\lesssim \frac{\ep}{(1+s)^{\frac{1}{2}-2\rho}}, \quad for\; J\leq N-12.
\end{align*}
Therefore we obtain, in view of \eqref{esti1}
\begin{align*}
&|\partial Z^I \wht g_{L\q T}| \lesssim \frac{\ep}{(1+s)^{\frac{3}{2}-\rho}}, \quad for\; I \leq N-15,\\
&|\partial Z^I \wht g_{L \q T}| \lesssim \frac{\ep}{(1+s)^{\frac{3}{2}-2\rho}} \quad for\; I \leq N-13.
\end{align*}
Corollary \ref{estksg} yields the following estimate for $q<0$
\begin{align*}
|Z^J \wht g_{\q T \q V}|&\lesssim \frac{\ep\sqrt{1+|q|}}{\sqrt{1+s}}, \quad for\; J\leq N-11,\\
|Z^J \wht g_{\q T \q V}|&\lesssim \frac{\ep(1+|q|)^{1+\mu}}{\sqrt{1+s}}, \quad for\; J\leq N-4.
\end{align*}
Therefore we obtain in view of \eqref{esti1}
\begin{align*}
&|\partial Z^I \wht g_{L \q T}| \lesssim \frac{\ep \sqrt{1+|q|}}{(1+s)^\frac{3}{2}} \quad for\; I \leq N-12 ,\\
&|\partial Z^I \wht g_{L \q T}| \lesssim \frac{\ep (1+|q|)^{1+\mu}}{(1+s)^\frac{3}{2}} \quad for\; I \leq N-5.
\end{align*}
For $q>0$ and $I\leq N-4$, we have in view of Corollary \ref{estksg}
$$|Z^I \wht g_{\q T \q V}|\lesssim \frac{\ep}{\sqrt{1+s}(1+|q|)^{\frac{1}{2}+\delta-\sigma}}$$
which together with \eqref{esti1} yields
$$
|\partial_q Z^I \wht g_{L \q T}| \lesssim \frac{\ep }{(1+|q|)^{\frac{1}{2}+\delta-\sigma}(1+s)^\frac{3}{2}} \quad for\; I \leq N-5.
$$

We now estimate $Z^I \wht g_{UU}$. As for $Z^I \wht g_{L \q T}$, Proposition \ref{estLU} yields
$$|\partial Z^I  \wht g_{UU}|\lesssim \frac{1}{1+s}|Z^{I+1} \wht g|.$$
Therefore, the estimates of Corollary \ref{estV} are a direct consequence of the estimates of Corollary \ref{esttildg}.
\end{proof}

Thanks to Lemma \ref{lmintegration}, since $\delta-\sigma >\frac{1}{2}$ we obtain the following corollary
\begin{cor}\label{estV2}
We have the estimates for $q<0$
\begin{align}
\label{wc1}&|Z^I\wht g_{L \q T}| \lesssim \frac{\ep (1+|q|)}{(1+s)^{\frac{3}{2}-\rho}}, \quad |Z^I\wht g_{UU}| \lesssim \frac{\ep (1+|q|)^{\frac{1}{2}+\rho}}{1+s}, \quad 
for\; I \leq N-15,\\
\label{wc2}&| Z^I \wht g_{L \q T}| \lesssim \frac{\ep (1+|q|)}{(1+s)^{\frac{3}{2}-2\rho}}
 \quad |Z^I\wht g_{UU}| \lesssim \frac{\ep (1+|q|)^{\frac{1}{2}+2\rho}}{1+s}, \quad for\; I \leq N-13,\\
\label{wc3}&|Z^I \wht g_{L \q T}| \lesssim \frac{\ep (1+|q|)^\frac{3}{2}}{(1+s)^\frac{3}{2}} \quad
|Z^I\wht g_{UU}| \lesssim \frac{\ep (1+|q|)}{1+s}, \quad  for \;I \leq N-12,\\
\label{wc4}&|Z^I \wht g_{L\q T}| \lesssim \frac{\ep (1+|q|)^{2+\mu}}{(1+s)^\frac{3}{2}} \quad
|Z^I\wht g_{UU}| \lesssim \frac{\ep (1+|q|)^{\frac{3}{2}+\mu}}{1+s}, \quad  for\; I\leq N-5,
\end{align}
and for $q>0$
\begin{equation}
\label{wc5}|Z^I\wht g_{L \q T}| \lesssim \frac{\ep (1+|q|)^{\frac{1}{2}+\sigma-\delta}}{(1+s)^\frac{3}{2}}, \quad 
|Z^I\wht g_{UU}| \lesssim \frac{\ep }{(1+s)(1+|q|)^{\delta-\sigma}}, \quad 
for\; I \leq N-5.
\end{equation}
\end{cor}

\section{Angle and linear momentum}\label{secangle}
We call angle and linear momentum the three first coefficients of $b$, $b_0,b_1,b_2$. These coefficients can not be prescribed arbitrarily, they are given by the resolution of the constraint equations (see Theorem \ref{thinitial}). We need $b$ to satisfy
\begin{equation}
\label{estpourb}\left\|\partial_\theta^I \left( b(\theta)+\int_{\Sigma_{T,\theta}} (\partial_q \phi)^2 r dq\right)\right\|_{L^2}\lesssim \frac{\ep^2}{\sqrt{T}}, \;for\; I\leq N-4.
\end{equation}
This is used crucially to estimate $h_0$ in the proof of Proposition \ref{estzh}. The heuristic of it is discussed in Section \ref{secgll} (see \eqref{appro}). The estimate \eqref{estpourb} is satisfied with $b$ replaced by $\Pi b$ thanks to the bootstrap assumption \eqref{bootb1}. For the angle and linear momentum, this is the object of the following proposition,
which says that the relations of Theorem \ref{contrainte} are asymptotically conserved by the flow of the Einstein equations.
\begin{prp}\label{prpangle}
We have
\begin{align*}
\left|\int b(\theta)d\theta +\frac{1}{2}\int_{\m R^2}\left((\partial_t \phi)^2
+|\nabla \phi|^2\right)(t,x)dx\right|&\lesssim \frac{\ep^2}{\sqrt{1+t}},\\ 
\left|\int b(\theta)\cos(\theta)d\theta -\int_{\m R^2}\left(\partial_t \phi \partial_1 \phi\right)(t,x) dx \right|&\lesssim \frac{\ep^2}{\sqrt{1+t}},\\ 
\left|\int b(\theta)\sin(\theta)d\theta -\int_{\m R^2}\left(\partial_t \phi \partial_2 \phi\right)(t,x) dx\right| &\lesssim \frac{\ep^2}{\sqrt{1+t}}.\\ 
\end{align*}
\end{prp}

To prove this proposition, we need the following lemma.
\begin{lm}\label{lemmeangle}
The equation for $g_{\mu \nu}$ can be written under the form
\begin{equation}
\label{onde}
\Box \wht g_{\mu\nu}=-2\partial_\mu\phi \partial_\nu \phi -2b(\theta)\frac{\partial^2_q(\chi(q)q)}{r}M_{\mu \nu}+O\left(\frac{\ep^2}{(1+t)^\frac{3}{2}(1+|q|)^{\frac{3}{2}-2\rho}}\right),
\end{equation}
where the tensor $M_{\mu \nu}$ corresponds to $dq^2$.
\end{lm}
\begin{proof}[Proof of Lemma \ref{lemmeangle}]
We recall the quasilinear equation for $\wht g_{\mu\nu}$ (see \eqref{s2})
$$g^{\alpha \beta}\partial_{\alpha}\partial_\beta \wht g_{\mu \nu} -H_b^\rho \partial_\rho \wht g_{\mu \nu}=-2 \partial_\mu \phi\partial_\nu \phi
+2(R_b)_{\mu \nu} + P_{\mu \nu}(\partial \wht g, \partial \wht g) + \wht P_{\mu \nu} (\wht g, g_b).$$
The worst term in 
$$g^{\alpha \beta}\partial_{\alpha}\partial_\beta \wht g_{\mu \nu} -\Box \wht g_{ \mu \nu}$$ is, according to Remark \ref{termql}, 
$$g_{LL} \partial_q^2 \wht g_{\mu \nu}.$$
We distinguish two kinds of contributions : 
$$g_{LL} \partial_q^2 \wht g_1\quad and \quad
g_{LL} \partial_q^2 h_0.$$
To estimate the first term, we use \eqref{wc1} of 
Corollary \ref{estV2}, which gives
$$|g_{LL}|\lesssim \frac{\ep (1+|q|)}{(1+s)^{\frac{3}{2}-\rho}}.$$
We estimate then
$$|\partial^2_q \wht g_1|\leq \frac{1}{(1+|q|)^2}\sum_{I\leq 2}|Z^I \wht g_1|,$$
and we use the bootstrap assumption \eqref{bootg1} for $I\leq N-14$
$$ |Z^I \wht g_1|\leq \frac{\ep}{(1+s)^{\frac{1}{2}-\rho}},$$
to obtain
\begin{equation}
\label{am1}|g_{LL}\partial^2_q \wht g_1|\lesssim \frac{\ep^2}{(1+s)^{2-2\rho}(1+|q|)}.
\end{equation}
We now estimate the second term.
To estimate $\partial_q^2 h_0$, we recall  \eqref{booth1} for $I\leq N-6$
$$|Z^I h_0|\lesssim \frac{\ep}{\sqrt{1+s}} + \frac{\ep}{(1+|q|)^{1-4\rho}}.$$
Consequently
$$|\partial^2_q h_0|\lesssim \frac{\ep}{(1+|q|)^2\sqrt{1+s}}+ \frac{\ep}{(1+|q|)^{3-4\rho}}.$$
The first contribution can be estimated like \ref{am1}.
To tackle the second contribution we need to use the estimate for $g_{LL}$ which gives the most decay in $s$ : we use \eqref{wc3} of Corollary \ref{estV2}, which yields
$$|g_{LL}|\lesssim \frac{\ep(1+|q|)^\frac{3}{2}}{(1+s)^\frac{3}{2}}.$$
This, together with the estimate \eqref{am1}, yields
\begin{equation}
\label{am2}|
g_{LL} \partial_q^2 h_0|\lesssim \frac{\ep^2}{(1+s)^\frac{3}{2}(1+|q|)^{\frac{3}{2}-4\rho}}+\frac{\ep^2}{(1+s)^{2-2\rho}(1+|q|)}.
\end{equation}
The semi linear terms $ P_{\mu \nu}(\partial \wht g, \partial \wht g) $ are estimated similarly. We now turn to the crossed terms. Thanks to Section \ref{cros}, the worst contribution is
\eqref{crossbaL}, which gives a contribution of the form
$\frac{\ep}{r}\partial \wht g_{L\ba L}$ in the region $q>0$. We estimate
thanks to \eqref{ks5} of Corollary \ref{estksg} in the region $q>0$
$$|\partial  \wht g_{L\ba L}|\lesssim \frac{\ep}{(1+s)^\frac{1}{2}(1+|q|)^{\frac{3}{2}+\delta-\sigma}}.$$
Therefore we obtain
\begin{equation}
\label{am3}\left|\ch_{q>0}\frac{\ep}{r} \partial \wht g_{\ba L L}\right|\lesssim \frac{\ep}{(1+s)^2(1+|q|)^{\frac{3}{2}+\delta-\sigma}}.
\end{equation}
We now estimate $(R_b)_{\mu \nu}$. Thanks to \eqref{rqq} and \eqref{rqu}, we may write
\begin{equation}
\label{am5}(R_b)_{\mu \nu}=-\frac{b(\theta)\partial_q^2(q\chi(q))}{r}M_{\mu \nu}+O\left(\frac{\ch_{1\leq q \leq 2}\ep^2}{(1+r)^2}\right).
\end{equation}
Thanks to \eqref{am1}, \eqref{am2}, \eqref{am3} and \eqref{am5} we conclude the proof of Lemma \ref{lemmeangle}.
\end{proof}
\begin{proof}[Proof of Proposition \ref{prpangle}]
We want to integrate equation \eqref{onde} for $(\mu,\nu)=0,0$ over the space-like hypersurfaces of $t$ constant. To deal with the term $\partial^2_t g_{00}$, we use the wave coordinate condition
$$g^{\alpha \beta}\partial_{\beta}g_{\alpha 0} =\frac{1}{2}g^{\alpha \beta}\partial_t g_{\alpha \beta}+(H_b)_0.$$
We can rewrite it, by definition of $(H_b)_0$
$$(g^{\alpha \beta}-(g_b)^{\alpha \beta})\partial_{\beta}g_{\alpha 0}
+g_b^{\alpha \beta}(\partial_{\beta}g_{\alpha 0}-\partial_{\beta}(g_b)_{\alpha 0})
=\frac{1}{2}(g^{\alpha \beta}-g_b^{\alpha \beta})\partial_t g_{\alpha \beta}+ \frac{1}{2}g_b^{\alpha\beta}(\partial_t g_{\alpha \beta}-\partial_t (g_b)_{\alpha \beta})+F_0.$$
By definition, $F$ contains only terms of the form $\wht g\partial_U g_b$, so we can estimate
\begin{equation}
\label{estF}
|ZF|\lesssim \frac{\ep\ch_{q>0}(1+|q|)}{r^2}|Z \wht g|\lesssim \frac{\ep^2}{(1+s)^2(1+|q|)^{\delta-\sigma}},
\end{equation}
where we have used \eqref{est5} to estimate $|Z \wht g|$. We note
$$m^{\alpha \beta}\partial_\beta \wht g_{\alpha 0}-\frac{1}{2}m^{\alpha \beta}\partial_t \wht g_{\alpha \beta}
=\frac{1}{2}(-\partial_t\wht g_{00}-\partial_t\wht g_{11}-\partial_t \wht g_{22})+\partial_1\wht g_{01}+\partial_2 \wht g_{02},$$
and we estimate
\begin{align*}
(g^{\alpha \beta}-(g_b)^{\alpha \beta})\partial_{\beta}g_{\alpha 0}&=(g^{L \ba L}-m^{L \ba L})\partial_{\ba L} g_{L 0} +f_1,\\
\frac{1}{2}(g^{\alpha \beta}-g_b^{\alpha \beta})\partial_t g_{\alpha \beta}&=(g^{L \ba L}-m^{L \ba L})\partial_t \wht g_{L \ba L}+f_2,\\
(m^{\alpha \beta}-g_b^{\alpha \beta})\partial_\beta \wht g_{\alpha 0}&=f_4\\
(m^{\alpha \beta}-g_b^{\alpha \beta})\partial_t \wht g_{\alpha \beta}&=f_5,
\end{align*}
where the $f_i$ contain terms of the form
$$\wht g_{LL}\partial \wht g_{\q V \q V}, \quad \wht g_{\q V \q V}\partial_{\q T} g_{\q T \q V}, \quad \frac{b\chi(q)}{r}\partial_U \wht g_{U\q V},\quad...$$
They satisfy the following estimate
\begin{equation}
\label{estfi}|Zf_i|\lesssim \frac{\ep^2}{(1+s)^\frac{3}{2}(1+|q|)^{\frac{1}{2}-2\rho}}.
\end{equation}
We note $2\partial_t \wht g_{L \ba L}=\partial_L \wht g_{L \ba L}+\partial_{\ba L}\wht g_{L \ba L}$ and
$2g_{L0}=g_{L \ba L}+g_{LL}$.
Consequently
$$(g^{L \ba L}-m^{L \ba L})(\partial_{\ba L} g_{L 0} -\partial_t \wht g_{L \ba L})=O\left(\wht g_{L \ba L}\partial_L \wht g_{L \ba L} +\wht g_{L \ba L}\partial_{\ba L} \wht g_{L  L} \right)$$
satisfies the same estimate \eqref{estfi} than the $f_i$. 
 Therefore the wave coordinate condition gives
$$
\frac{1}{2}(-\partial_t \wht g_{00}-\partial_t \wht g_{11}-\partial_t \wht g_{22})+\partial_1\wht g_{01}+\partial_2 \wht g_{02}\\
=f_5$$
where $f_5$ satisfies \eqref{estfi}.
Therefore, differentiating this equation with respect to $t$, and using \eqref{onde} for $(\mu,\nu)=(0,0),(1,1),(2,2)$ we obtain
\begin{align*}
&\Delta \wht g_{00} +\Delta \wht g_{11}
+\Delta \wht g_{22}-2\partial_1\partial_t\wht g_{01}-2\partial_2\partial_t\wht g_{02}\\
=&-2((\partial_0 \phi)^2+(\partial_1 \phi)^2+(\partial_2 \phi)^2)-4b(\theta)\frac{\partial^2_q(\chi(q)q)}{r} +O\left(\frac{\ep^2}{(1+s)^\frac{3}{2}(1+|q|)^{\frac{3}{2}-2\rho}}\right).
\end{align*}
Integrating on the space-like hypersurface $t$ constant we obtain, since 
$\int_0^\infty \partial^2(q\chi(q))dr=1$,
\begin{equation}
\label{amb1}-\frac{1}{2}\int (\partial_t \phi)^2 +|\nabla \phi|^2=\int b(\theta)d\theta +O\left(\frac{\ep^2}{\sqrt{1+t}}\right).
\end{equation}
To obtain the next relation we do the same reasoning but with \eqref{onde} for $(\mu,\nu)=(0,1)$ and $(\mu, \nu)=(0,2)$. We only detail the case $(\mu,\nu)=(0,1)$ as the other one is treated in the same way.
Recall the wave coordinates condition 
$$g^{\alpha \beta}\partial_{\beta}g_{\alpha 1} =\frac{1}{2}g^{\alpha \beta}\partial_1 g_{\alpha \beta}+(H_b)_1.$$
We can rewrite it, by definition of $(H_b)_1$
$$(g^{\alpha \beta}-(g_b)^{\alpha \beta})\partial_{\beta}g_{\alpha 1}
+g_b^{\alpha \beta}(\partial_{\beta}g_{\alpha 1}-\partial_{\beta}(g_b)_{\alpha 1})
=\frac{1}{2}(g^{\alpha \beta}-g_b^{\alpha \beta})\partial_1 g_{\alpha \beta}+ \frac{1}{2}g_b^{\alpha\beta}(\partial_1 g_{\alpha \beta}-\partial_1 (g_b)_{\alpha \beta})+F_1$$
We note
$$m^{\alpha \beta}\partial_\beta \wht g_{\alpha 1}-\frac{1}{2}m^{\alpha \beta}\partial_1 \wht g_{\alpha \beta}
=-\partial_t \wht g_{01}+\partial_1\wht g_{11}+\partial_2 \wht g_{12}-\frac{1}{2}m^{\alpha \beta}\partial_1 \wht g_{\alpha \beta}$$
and we estimate
\begin{align*}
(g^{\alpha \beta}-(g_b)^{\alpha \beta})\partial_{\beta}g_{\alpha 1}&=(g^{L \ba L}-m^{L \ba L})\partial_{\ba L} g_{L 1} +f_6,\\
\frac{1}{2}(g^{\alpha \beta}-g_b^{\alpha \beta})\partial_1 g_{\alpha \beta}&=(g^{L \ba L}-m^{L \ba L})\partial_1 \wht g_{L \ba L}+f_7,\\
(m^{\alpha \beta}-g_b^{\alpha \beta})\partial_\beta \wht g_{\alpha 1}&=f_8,\\
(m^{\alpha \beta}-g_b^{\alpha \beta})\partial_1 \wht g_{\alpha \beta}&=f_9,
\end{align*}
where the quantities $f_i$ satisfy \eqref{estfi}.
We note $2\partial_1 \wht g_{L\ba L}=-\cos(\theta)\partial_{\ba L} \wht g_{L \ba L}+\bar{\partial} \wht g_{L \ba L}$ and
$2\partial_{\ba L}\wht g_{L1}=-\partial_{\ba L} (\cos(\theta)g_{L \ba L})
+g_{L \q T}.$ Therefore we obtain
$$
-\partial_t \wht g_{01}+\partial_1 \wht g_{11}+\partial_2 \wht g_{12}-\frac{1}{2}m^{\alpha \beta}\partial_1 \wht g_{\alpha \beta}\\
=f_{10},
$$
where $f_{10}$ satisfies \eqref{estfi}.
Differentiating with respect to $t$ and using \eqref{onde} for $(\mu,\nu)=(0,1)$ we obtain
\begin{align*}
&\Delta \wht g_{01}+\partial_1\partial_t \wht g_{11}+\partial_2\partial_t \wht g_{12}-\frac{1}{2}m^{\alpha \beta}\partial_1\partial_t \wht g_{\alpha \beta}\\
=&-2\partial_t \phi \partial_1 \phi+2 b(\theta)\cos(\theta)\frac{\partial^2_q(\chi(q)q)}{r}+O\left(\frac{\ep^2}{(1+s)^{\frac{3}{2}}(1+|q|)^{\frac{3}{2\rho}}}\right).
\end{align*}
Integrating on the space-like hypersurface $t$ constant we obtain
\begin{equation}\label{amb2}
\int \partial_t \phi \partial_1 \phi =\int b(\theta)\cos(\theta)d\theta +O\left(\frac{\ep^2}{\sqrt{1+t}}\right),
\end{equation}
and similarly
\begin{equation}\label{amb3}
\int \partial_t \phi \partial_2 \phi =\int b(\theta)\sin(\theta)d\theta +O\left(\frac{\ep^2}{\sqrt{1+t}}\right).
\end{equation}
Estimates \eqref{amb1}, \eqref{amb2} and \eqref{amb3} conclude the proof of Proposition \ref{prpangle}
\end{proof}
\begin{cor}\label{energies}
We have the estimates
\begin{align*}
\left|\int b(\theta)d\theta +\int_{\Sigma_{T}}(\partial_q \phi)^2 rdrd\theta\right|&\lesssim \frac{\ep^2}{\sqrt{T}}\\ 
\left|\int b(\theta)\cos(\theta)d\theta +\int_{\Sigma_T}\cos(\theta)(\partial_q \phi)^2rdrd\theta \right|&\lesssim \frac{\ep^2}{\sqrt{T}}\\ 
\left|\int b(\theta)\sin(\theta)d\theta +\int_{\Sigma_T}\sin(\theta)(\partial_q \phi)^2rdrd\theta \right| &\lesssim \frac{\ep^2}{\sqrt{T}}\\ 
\end{align*}
\end{cor}

\begin{proof}
We may write
\begin{align*}
\partial_t \phi&= -\partial_q \phi +\partial_s \phi,\\
\partial_1 \phi&= \cos(\theta)\partial_q \phi +\cos(\theta)\partial_s \phi-\sin(\theta)\partial_U \phi\\
\partial_1 \phi&= \sin(\theta)\partial_q \phi +\sin(\theta)\partial_s \phi+\cos(\theta)\partial_U \phi.
\end{align*}
Moreover, thanks to the bootstrap assumption \eqref{bootphi1}
$$| \partial \phi \bar{\partial}\phi|\lesssim \frac{1}{(1+|q|)(1+s)}|Z\phi|^2
\lesssim \frac{\ep^2}{(1+s)^2(1+|q|)^{2-8\rho}},$$
and consequently
$$\left|\int  \left( \partial \phi \bar{\partial}\phi\right)(t,x)dx\right|\lesssim \frac{\ep^2}{1+t}.$$
Therefore
\begin{align*}
&\left|2\int_{\Sigma_{T}}(\partial_q \phi)^2dx-\int_{\Sigma_{T}}((\partial_t \phi)^2 +|\nabla \phi|^2)dx\right| \lesssim \frac{\ep^2}{1+T},\\
&\left|\int_{\Sigma_{T}}\cos(\theta)(\partial_q \phi)^2dx+\int_{\Sigma_{T}}\partial_t \phi\partial_1 \phi dx\right| \lesssim \frac{\ep^2}{1+T},\\
&\left|\int_{\Sigma_{T}}\sin(\theta)(\partial_q \phi)^2dx+\int_{\Sigma_{T}}\partial_t \phi\partial_2 \phi dx\right| \lesssim \frac{\ep^2}{1+T}.
\end{align*}
This concludes the proof of Corollary \ref{energies}.
\end{proof}
Corollary \ref{energies} and the bootstrap assumption \ref{bootb1} directly imply the following corollary.
\begin{cor}
\label{corb}
We have, for $I\leq N-4$
$$\left|\partial_\theta^I \left( b(\theta)+\int_{\Sigma_{T,\theta}} (\partial_q \phi)^2 r dq\right)\right|\lesssim\frac{\ep^2}{\sqrt{T}}.$$
\end{cor}

\section[The transport equation]{The transport equation \eqref{eqh}}\label{transport}
In this section we will estimate $h_0$, $\Box h_0$ and $\wht h$.
\subsection{Estimations on $h_0$}
We recall the equation \eqref{eqh}
\begin{equation*}
 \left\{ \begin{array}{l} 
         \partial_q h_0 =-2r(\partial_q \phi)^2-2b(\theta)\partial_q^2(\chi(q)q), \\
	 h_0|_{t=0}=0.
        \end{array}
\right.
\end{equation*}
The solution of this equation is
\begin{equation}\label{ho}
h_0(s,Q,\theta) = \int_s^Q \left(-2(\partial_q \phi)^2 -2\frac{b(\theta)\partial_q^2(q\chi(q))}{r}
\right)r dq.
\end{equation}
All the estimates we will perform in this section take place in the region $r>\frac{t}{2}$ since we will always apply the cut-off function $\Up\left(\frac{r}{t}\right)$
to $h_0$.
\begin{prp}In the region $r>\frac{t}{2}$
we have the estimates on $h_0$, for $q<0$
$$|\partial_s h_0|\lesssim \frac{\ep^2}{(1+s)^{\frac{3}{2}}}, \quad |h_0|\lesssim \frac{\ep^2}{\sqrt{1+s}}+\frac{\ep^2}{(1+|q|)^{2-8\rho}}$$
and for $q>0$
$$|\partial_s h_0|\lesssim \frac{\ep^2}{(1+s)^{\frac{3}{2}}(1+|q|)^{\frac{3}{2}+2(\delta-\sigma)}}, \quad |h_0|\lesssim \frac{\ep^2}{(1+|q|)^{2+2(\delta-\sigma)}}.$$
\end{prp}
\begin{proof}
We write the wave operator in coordinates $(s,q,\theta)$
\begin{equation}
\label{formbox}\Box=4\partial_s \partial_q + \frac{1}{r}(\partial_s + \partial_q) + \frac{1}{r^2}\partial^2_\theta.
\end{equation}
We calculate
\begin{equation}\label{dsh}
\partial_s\partial_q h_0
= 
\partial_s (-2r(\partial_q\phi)^2)
= -r\partial_q \phi \left(4\partial_s \partial_q \phi + \frac{1}{r}\partial_q \phi\right)
= -r\partial_q \phi \left( \Box \phi-\frac{1}{r}\partial_s \phi -\frac{1}{r^2}\partial_\theta^2 \phi\right),
\end{equation}
where we have used
$$\partial_s(-2b(\theta)\partial^2_q(q\chi(q)))=0.$$
Therefore we have
\begin{equation}\label{formuleds}
\partial_s h_0=  \int_s^Q \left(-\Box \phi +\frac{1}{r}\partial_s \phi+ \frac{1}{r^2}\partial^2_\theta \phi \right)\partial_q \phi r dq+O\left(\frac{\ep^2}{(1+s)^{3+2\delta}}\right),
\end{equation}
where we have used
$$\partial_s h_0|_{t=0}=-\partial_q h_0|_{t=0}=\left(2r(\partial_q \phi)^2+2b(\theta)\partial_q^2(\chi(q)q)\right)|_{t=0}=O\left(\frac{\ep^2}{(1+s)^{3+2\delta}}\right).$$
The bootstrap assumption \eqref{bootphi1} gives
$$\left|\frac{1}{r}\partial_s \phi\right|+ \left|\frac{1}{r^2}\partial^2_\theta \phi \right|\lesssim \frac{1}{(1+s)^2}|Z^2\phi|\lesssim \frac{\ep}{(1+s)^{\frac{5}{2}}(1+|q|)^{\frac{1}{2}-4\rho}},$$
and
$$|\partial_q \phi| \lesssim\frac{1}{1+|q|}|Z\phi|\lesssim \frac{\ep}{(1+s)^\frac{1}{2}(1+|q|)^{\frac{3}{2}-4\rho}}.$$
Therefore
\begin{equation}
\label{transp1}\left|\left(\frac{1}{r}\partial_s \phi+ \frac{1}{r^2}\partial^2_\theta \phi\right)\partial_q \phi r \right|\lesssim \frac{\ep^2}{(1+s)^2 (1+|q|)^{2-8\rho}}.
\end{equation}
To estimate $\Box \phi$ we write $\Box \phi =(\Box -\Box_g)\phi$. 
Thanks to Remark \ref{termql}, in the region $q<0$ it is sufficient to estimate $g_{LL}\partial^2_q \phi$. We start with the region $q<0$.
To obtain all the possible decay in $s$, we use the estimate \eqref{wc3} of Corollary \ref{estV2} for $I\leq N-11$, which gives, for $q<0$
$$|g_{LL}|\lesssim \frac{\ep (1+|q|)^\frac{3}{2}}{(1+s)^\frac{3}{2}}.$$
The bootstrap assumption \eqref{bootphi1} imply
$$|\partial^2_q \phi|\lesssim \frac{\ep}{(1+|q|)^{\frac{5}{2}-4\rho}\sqrt{1+s}},$$
therefore
$$|g_{LL}\partial^2_q \phi \partial_q \phi|\lesssim \frac{\ep^3(1+|q|)^\frac{3}{2}}{(1+s)^\frac{5}{2}(1+|q|)^{4-8\rho}},$$
and we obtain
\begin{equation}\label{transp2}
|(\Box \phi)\partial_q \phi r|\lesssim \frac{\ep^3 }{(1+s)^\frac{3}{2}(1+|q|)^{\frac{5}{2}-8\rho}}.
\end{equation}
Thanks to \eqref{transp1} and \eqref{transp2}, in the region $q<0$ we have 
\begin{equation}\label{transp3}
\left| \left(-\Box \phi +\frac{1}{r}\partial_s \phi+ \frac{1}{r^2}\partial^2_\theta \phi \right)\partial_q \phi r\right|\lesssim \frac{\ep^3 }{(1+s)^\frac{3}{2}(1+|q|)^{\frac{5}{2}-8\rho}}.
\end{equation}
We now estimate the integrand in the region $q>0$. Estimate \eqref{iks5} yields, for $q>0$ and $I\leq N-3$
$$|Z^I\phi|\lesssim \frac{\ep}{\sqrt{1+s}(1+|q|)^{\frac{1}{2}+\delta-\sigma}},$$
and estimate \eqref{wc5} yields for $q>0$
$$|g_{LL}|\lesssim \frac{(1+|q|)^{\frac{1}{2}+\sigma-\delta}}{(1+s)^\frac{3}{2}}.$$
In the region $q>0$, $\Box \phi-\Box_g \phi$ contains also terms of the form
$\frac{\ep\chi(q)}{r}\bar{\partial}\phi$ (see \eqref{crossphi} in the discussion of Section \ref{cros}). We can neglect them since we already take into account terms of the form $\frac{1}{r}\partial_s \phi+\frac{1}{r^2}\partial^2_\theta \phi$ in \eqref{dsh}. 
Consequently for $q>0$
\begin{equation}\label{transp4}
\left| \left(-\Box \phi +\frac{1}{r}\partial_s \phi+ \frac{1}{r^2}\partial^2_\theta \phi \right)\partial_q \phi r\right|\lesssim  \frac{\ep ^2}{(1+s)^\frac{3}{2}(1+|q|)^{\frac{5}{2}+2\delta-2\sigma}}.
\end{equation}
Therefore, \eqref{formuleds} and \eqref{transp4} yield for $q>0$
\begin{equation}
\label{ests1}
|\partial_s h_0|\lesssim \frac{\ep^2}{(1+s)^\frac{3}{2}(1+|q|)^{\frac{3}{2}+2(\delta-\sigma)}},
\end{equation}
and \eqref{formuleds}, \eqref{transp3} and \eqref{transp4} yield for $q<0$,
since $\frac{1}{(1+|q|)^{\frac{5}{2}-8\rho}}$ is integrable,
\begin{equation}\label{ests}
|\partial_s h_0|\lesssim \frac{\ep^2}{(1+s)^\frac{3}{2}}.
\end{equation}

\begin{figure}[ht]\label{cone}
\centering
\includegraphics[width=15cm,height=8cm]{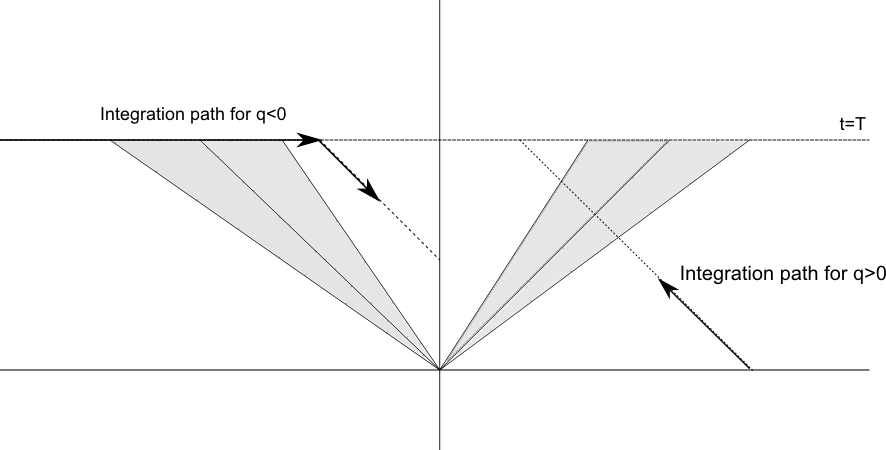}
\caption{Integration of $h_0$}
\end{figure}

Thanks to Corollary \ref{corb} we have
$$\left|b(\theta)+\int_{\Sigma_{T,\theta}} (\partial_q \phi)^ 2 r dr\right| \lesssim \frac{\ep^2}{T^\frac{1}{2}}.$$
Moreover $\partial_r h_0=\partial_q h_0 +\partial_s h_0$ and therefore \eqref{ests} and \eqref{ests1} yield
\begin{equation}
\label{aintegrer}
\partial_r h_0= -2r(\partial_q \phi)^2-2b(\theta)\partial_r^2(\chi(q)q)+O\left(\frac{\ep^2}{(1+s)^{\frac{3}{2}}}\right).
\end{equation}
Therefore, on the line $t=T$, with fixed $\theta$ we obtain the following estimate for $h_0$ in the region $r<t$ by integrating \eqref{aintegrer}
\begin{align*}
h_0(T,R,\theta) =& -\int_R^\infty \left(-2r(\partial_q \phi)^2+O\left(\frac{\ep^2}{(r+T)^\frac{3}{2}}\right)\right) +2b(\theta)\\
=&\int_0^R2r(\partial_q \phi)^2dr+O\left(\frac{\ep^2}{\sqrt{1+T}}\right)\\
=& O\left(\frac{\ep^2}{(1+T)^\frac{1}{2}}\right)+O\left(\frac{\ep^2}{(1+q)^{2-8\rho}}\right).
\end{align*}
To estimate $h_0$ elsewhere in the region $r<t$, we can integrate the estimate \eqref{ests}, at fixed $q$, as shown in left of the figure \ref{cone}. 
To estimate $h_0$ in the region $r>t$ we integrate the transport equation from $t=0$, as shown in the right of the figure \ref{cone} : we rely on formula \eqref{ho} and the estimate for $q>0$
$$|\partial_q \phi|\lesssim \frac{\ep}{\sqrt{1+s}(1+|q|)^{\frac{3}{2}+\delta-\sigma}}.$$
We obtain
\begin{align*}
 h_0&= O\left(\frac{\ep^2}{(1+q)^{2+2(\delta-\sigma)}}\right), \; q>0,\\
h_0& =  O\left(\frac{\ep^2}{(1+s)^\frac{1}{2}}\right)+O\left(\frac{\ep^2}{(1+q)^{2-8\rho}}\right), \; q<0.
\end{align*}

\end{proof}

Next we derive an estimate for $Z^I h_0$.
\begin{prp}\label{estzh}
Let $I\leq N-5$. We have the estimate for $q<0$
$$|Z^I h_0|\lesssim \frac{\ep^2}{\sqrt{1+s}}+\frac{\ep^2}{(1+|q|)^{1-4\rho}}, \quad|\partial_s Z^I h_0|\lesssim \frac{\ep^2}{(1+s)^\frac{3}{2}},
\quad |\partial_q Z^I h_0|\lesssim \frac{\ep^2}{(1+|q|)^{2-4\rho}},
$$
and for $q>0$
$$|Z^I h_0|\lesssim \frac{\ep^2}{(1+|q|)^{2+2(\delta-\sigma)}}, \quad|\partial_s Z^I h_0|\lesssim \frac{\ep^2}{(1+s)^\frac{3}{2}(1+|q)^{1+2(\delta-\sigma)}}.$$
\end{prp}
Observe that 
$$S=s\partial_s +q\partial_q, \; \Omega_{12}=\partial_\theta,\;\Omega_{01}=
\cos(\theta)(s\partial_s -q\partial_q)-\frac{t}{r}\sin(\theta)\partial_\theta,\;\Omega_{02}=
\sin(\theta)(s\partial_s -q\partial_q)+\frac{t}{r}\cos(\theta)\partial_\theta.$$
Hence Proposition \ref{estzh} is an immediate consequence of Proposition \eqref{estjkl}.
\begin{prp}\label{estjkl} We assume
Let $j+k+l\leq N-5$ then in the region $r>\frac{t}{2}$ we have the estimates on $h_0$, for $q<0$, if $j,k\geq 1$
$$
|\partial^{j}_s \partial^k_q \partial^l_\theta h_0|\lesssim \frac{\ep^2}{s^{j+\frac{1}{2}}(1+|q|)^{k+1-4\rho}}
$$
and 
$$|\partial^k_q \partial^l_\theta h_0|\lesssim \frac{\ep^2}{(1+|q|)^{k}}\left(\frac{1}{\sqrt{1+s}}+\frac{1}{(1+|q|)^{1-4\rho}}\right), \quad |\partial^j_s \partial^l_\theta h_0|\lesssim \frac{\ep^2}{(1+s)^{\frac{3}{2}+j}}.$$
For $q>0$ we have, with $j\geq1$
$$
|\partial^{j}_{s} \partial^k_q \partial^l_\theta h_0|\lesssim \frac{\ep^2}{s^{j+\frac{1}{2}}(1+|q|)^{k+\frac{3}{2}+2(\delta-\sigma)}},\; 
|\partial^k_q \partial^l_\theta h_0|\lesssim \frac{\ep^2}{(1+|q|)^{k+2+2(\delta-\sigma)}}.$$
\end{prp}

\begin{proof}
 We assume first $j=0$ and $k\geq 1$. We assume $l+k\leq N-3$. Then we can write
$$\partial^k_q \partial^l_\theta h_0=-2\partial^{k-1}_q\partial^l_\theta \left(r(\partial_q \phi)^2+\partial_q^2(q\chi(q))b(\theta)\right).$$
Therefore we can estimate
$$|\partial^k_q \partial^l_\theta h_0|\lesssim \frac{r}{(1+|q|)^{k-1}}\sum_{J\leq k+l-1}|Z^J(\partial_q \phi)^2|+\frac{1}{(1+|q|)^{k+1}}|\partial_\theta^l b|.
$$
The terms in $Z^J(\partial_q \phi)^2$ are of the form $\partial_q Z^{J_1} \phi \partial_q Z^{J_2} \phi$, where $J_1 \leq \frac{k+l}{2}\leq N-15$
therefore we can estimate, thanks to the bootstrap assumption \eqref{bootphi1}
$$|\partial_q Z^{J_1} \phi|\lesssim \frac{\ep}{(1+|q|)^{\frac{3}{2}-4\rho}\sqrt{1+s}},$$
and we estimate $\partial_q Z^{J_2} \phi$ thanks to \eqref{ks1} of Proposition \ref{estks} since $J_2\leq l+k-1\leq N-4$
$$|\partial_q Z^{J_2} \phi|\lesssim \frac{\ep}{\sqrt{1+|q|}\sqrt{1+s}}.$$
Consequently we have shown that for $k+l\leq N-3$, $k\geq1$
\begin{equation}
\label{kl}
|\partial^k_q \partial^l_\theta h_0|\lesssim \frac{\ep^2}{(1+|q|)^{k+1-4\rho}}.
\end{equation}
In the region $q>0$ we have the better estimate for $i=1,2$ thanks to \eqref{ks4} of Proposition \ref{estks} 
$$|\partial_q Z^{J_i} \phi|\lesssim \frac{\ep}{(1+|q|)^{\frac{3}{2}+\delta-\sigma}\sqrt{1+s}},$$
so
\begin{equation}
\label{kl2}|\partial^k_q \partial^l_\theta h_0|\lesssim \frac{\ep^2}{(1+|q|)^{k+2+2\delta-2\sigma}}.
\end{equation}

We now assume $k \geq 1$, $j\geq 1$ and estimate $\partial^{j}_s \partial^k_q \partial^l_\theta h_0$, for 
$j+k+l\leq N-4$. Thanks to \eqref{dsh}, we can write
$$
\partial^{j}_s \partial^k_q \partial^l_\theta h_0=-\partial^{j-1}_s \partial^{k-1}_q \partial^l_\theta 
 \left(r\partial_q \phi \left( \Box \phi-\frac{1}{r}\partial_s \phi -\frac{1}{r^2}\partial_\theta^2 \phi\right)\right).$$
We estimate
\begin{align*}\left|\partial^{j-1}_s \partial^{k-1}_q \partial^l_\theta \left(r\partial_q \phi\left(\frac{1}{r}\partial_s \phi +\frac{1}{r^2}\partial_\theta^2 \phi\right)\right)\right|&\lesssim
\frac{1}{(1+|q|)^{k-1}(1+|s|)^{j-1}}\left|Z^{j+k+l-2}\left(r\partial_q \phi\left(\frac{1}{r}\partial_s \phi +\frac{1}{r^2}\partial_\theta^2 \phi\right)\right)\right|\\
&\lesssim\frac{1}{(1+|q|)^{k}(1+|s|)^{j}}
\sum_{\substack{J_1+J_2\\ \leq j+k+l-2}}|Z^{J_1+2}\phi Z^{J_2+1} \phi|
\end{align*}
We can assume $J_1\leq \frac{ j+k+l-2}{2}$. In the region $q<0$, \eqref{bootphi1} and \eqref{iks1} yield
$$|Z^{J_1+2 }\phi|\lesssim \frac{\ep}{\sqrt{1+s}(1+|q|)^{\frac{1}{2}-4\rho}},\quad | Z^{J_2} \phi| \lesssim \frac{\ep\sqrt{1+|q|}}{\sqrt{1+s}}.$$
Consequently, for $q<0$
\begin{equation}
\label{tran1}
\left|\partial^{j-1}_s \partial^{k-1}_q \partial^l_\theta \left(r\partial_q \phi\left(\frac{1}{r}\partial_s \phi +\frac{1}{r^2}\partial_\theta^2 \phi\right)\right)\right|\lesssim \frac{\ep^2}{(1+|q|)^{k-4\rho}(1+s)^{j+1}}.
\end{equation}
To estimate the contribution of $\Box \phi$, we write as before $\Box \phi  =(\Box -\Box_g)\phi$.
Following Remark \ref{termql}, it is sufficient to estimate 
\begin{align*}
\left|\partial^{j-1}_s \partial^{k-1}_q \partial^l_\theta\left(rg_{LL}\partial_q \phi\partial_q^2\phi \right)\right|
&\lesssim \frac{1}{(1+|q|)^{k-1}(1+|s|)^{j-1}}\left|Z^{k+j+l-2}\left(rg_{LL}\partial_q \phi\partial_q^2\phi \right)\right|\\
&\lesssim \frac{1}{(1+|q|)^{k}(1+|s|)^{j-2}}\sum_{\substack{J_1+J_2+J_3\\ \leq j+k+l-2}}|Z^{J_1}g_{LL} \partial_q Z^{J_2+1}\phi \partial_q Z^{J_3} \phi|.
\end{align*}
We have
$J_1+J_2+J_3 \leq j+k+l-2\leq N-5$.
We separate in two cases
\begin{itemize}
\item $J_1\leq \frac{N}{2}-2$ and $J_2\leq \frac{N}{2}-2$ : then we have
thanks to \eqref{wc3}, \eqref{bootphi1} and \eqref{ks1}
$$|Z^{J_1}g_{LL}|\lesssim \frac{\ep (1+|q|)^{\frac{3}{2}}}{(1+s)^\frac{3}{2}},$$
$$|\partial_q Z^{J_2+1}\phi|\lesssim \frac{\ep}{\sqrt{1+s}(1+|q|)^{\frac{3}{2}-4\rho} },\quad
|\partial_q Z^{J_3}\phi|\lesssim \frac{\ep}{\sqrt{1+|q|}\sqrt{1+s}}.$$
The case  $J_1\leq \frac{N}{2}-2$ and $J_3\leq \frac{N}{2}-2$
can be treated in the same way.
\item $J_2\leq \frac{N}{2}-2$ and $J_3\leq \frac{N}{2}-2$
then, since $|J_1|\leq j+k+l-2\leq N-4$ we have thanks to \eqref{wc4} and \eqref{bootphi1}
$$|Z^{J_1}g_{LL}|\lesssim \frac{\ep (1+|q|)^{2+\mu}}{(1+s)^\frac{3}{2}},
\quad |\partial_q Z^{J}\phi|\lesssim \frac{\ep}{\sqrt{1+s}(1+|q|)^{\frac{3}{2}-4\rho} },\;
for \;J=J_2+1,J_3.$$
\end{itemize}
In the first case we obtain
\begin{equation}
\label{truc}
|Z^{J_1}g_{LL} \partial_q Z^{J_2+1}\phi \partial_q Z^{J_3} \phi|
\lesssim \frac{\ep^3}{(1+s)^{\frac{5}{2}}(1+|q|)^{\frac{1}{2}-4\rho}},\end{equation}
and in the last case we obtain
$$|Z^{J_1}g_{LL} \partial_q Z^{J_2+1}\phi \partial_q Z^{J_3} \phi|
\lesssim \frac{\ep^3}{(1+s)^{\frac{5}{2}}(1+|q|)^{1-8\rho-\mu}}.$$
We have $\mu +4\rho\leq \frac{1}{2}$. Consequently, we have 
in the region $q<0$
\begin{equation}
\label{tran2}
\left|\partial^{j-1}_s \partial^{k-1}_q \partial^l_\theta\left(rg_{LL}\partial_q \phi\partial_q^2\phi \right)\right|
\lesssim \frac{\ep^3}{(1+|q|)^{k+\frac{1}{2}-4\rho}(1+|s|)^{j-2}}.
\end{equation}
Estimates \eqref{tran1} and \eqref{tran2} yield, in the region $q<0$
for $j+k+l\leq N-4$, $j,k\geq 1$
\begin{equation}
\label{jkl}
|\partial^{j}_s \partial^k_q \partial^l_\theta h_0|
\lesssim \frac{\ep^2}{(1+s)^{j+\frac{1}{2}}(1+|q|)^{k+\frac{1}{2}-4\rho}}.
\end{equation}
In the region $q>0$, thanks to \eqref{ks4} and \eqref{wc5} we have the better estimate, for $J\leq N-5$
$$|\partial_q Z^J \phi|\leq \frac{\ep}{\sqrt{1+s}(1+|q|)^{\frac{3}{2}+\delta-\sigma}}, \; |Z^J g_{LL}|\leq \frac{\ep(1+|q|)^{\frac{1}{2}-\delta+\sigma}}{(1+s)^\frac{3}{2}},$$
so
we have
\begin{equation}
\label{jkl2}
|\partial^{j}_s \partial^k_q \partial^l_\theta h_0|
\lesssim \frac{\ep^2}{(1+s)^{j+\frac{1}{2}}(1+|q|)^{k+\frac{3}{2}+2(\delta-\sigma)}}.
\end{equation}

We now assume $k=0$ and $j\geq 1$. We obtain an estimate on
$\partial^{j}_s  \partial^l_\theta h_0$ for $q>0$ by integrating \eqref{jkl2} for $k=1$ with respect to $q$, from the hypersurface $t=0$. We obtain for $j+l \leq N-4$, $j\geq 1$, $q>0$
\begin{equation}
\label{jl2}
|\partial^{j}_s  \partial^l_\theta h_0|\lesssim
\frac{\ep^2}{(1+s)^{j+\frac{1}{2}}(1+|q|)^{\frac{3}{2}+2(\delta-\sigma)}}.
\end{equation}
For $q<0$, we integrate \eqref{jkl} from $q=0$. We obtain for $j+l \leq N-4$, $j\geq 1$,
\begin{equation}
\label{jl}
|\partial^{j}_s  \partial^l_\theta h_0|\lesssim
\frac{\ep^2}{(1+s)^{j+\frac{1}{2}}}.
\end{equation}

We now estimate $\partial^l_\theta h_0$ for $l\leq N-5$. Recall from Corollary \ref{corb} that
$$\left|\partial^l_\theta \left(b(\theta)+\int_{\Sigma_{T,\theta}} (\partial_q \phi)^ 2 r dr\right)\right| \lesssim \frac{\ep^2}{T^\frac{1}{2}}.$$
Moreover, we can write, thanks to the estimate \eqref{jl}
\[\partial_r \partial_\theta^lh_0=\partial_q \partial_\theta^l h_0+\partial_s \partial_\theta^l h_0
=\partial_\theta^l \left( -2r(\partial_q \phi)^2\right) -2\partial_q^2(q\chi(q))\partial^l_\theta b
+O\left(\frac{\ep^2}{(1+s)^\frac{3}{2}}\right).
\]
Therefore, by integrating this 
on the line $t=T$, we have 
$$\partial_\theta^l h_0(T,R,\theta)=\int_0^R \partial_q \partial^l_\theta h_0 dr+O\left(\frac{\ep^2}{\sqrt{1+T}}\right),$$
and consequently, thanks to \eqref{kl}
 we have the estimate, for $l\leq N-4$ and $q<0$
$$
|\partial^l_\theta h_0(T,r,\theta)|\lesssim \frac{\ep^2}{\sqrt{1+T}}+\frac{\ep^2}{(1+|q|)^{1-4\rho}}.
$$
To have an estimate everywhere, we integrate \eqref{jl} for $j=0$ with respect to $s$, as shown in the figure \ref{cone}.
We obtain, for $l\leq N-5$
\begin{equation}\label{l}
|\partial^l_\theta h_0|\lesssim \frac{\ep^2}{\sqrt{1+s}}+\frac{\ep^2}{(1+|q|)^{1-4\rho}}.
\end{equation}
In the region $q>0$, we just integrate \eqref{jl2} from $t=0$, and we obtain
\begin{equation}
\label{l2}|\partial^l_\theta h_0|\lesssim \frac{\ep^2}{(1+|q|)^{2+2(\delta-\sigma)}}.
\end{equation}
In view of \eqref{jkl}, \eqref{jkl2}, \eqref{kl}, \eqref{kl2}, \eqref{jl}, \eqref{jl2},  \eqref{l}, \eqref{l2}
we conclude the proof of Proposition \ref{estjkl}.
\end{proof}

\subsection{Estimation of $\Box \Up(\frac{r}{t})h_0$}

%We will now estimate $\Box_g \chi\left(\frac{r}{t}\right)h_0$
%\begin{prp}\label{estboxg}
%Let $I\leq N-6$. We have
%$$|Z^I\left(\Box_g \chi\left(\frac{r}{t}\right)h_0
%-\chi(\frac{r}{t})\left( -2(\partial_q \phi)^2 %-4\frac{b(\theta)\partial_q^2(q\chi(q))}{r}\right)\right)|
%\lesssim \frac{\ep^2}{s^2\sqrt{1+|q|}}$$
%\end{prp}
\begin{prp}\label{boxh}Let $I\leq N-7$
We have the estimate for $q<0$
$$\left|Z^I \left(\Box \left(\Up\left(\frac{r}{t}\right)h_0\right)
-\Up\left(\frac{r}{t}\right)\left( -2(\partial_q \phi)^2 -2\frac{b(\theta)\partial_q^2(q\chi(q))}{r}\right) \right)\right|
\lesssim \frac{\ep^2}{(1+s)^\frac{3}{2}(1+|q|)},$$
and for $q>0$
$$\left|Z^I \left(\Box \left(\Up\left(\frac{r}{t}\right)h_0\right)
-\Up\left(\frac{r}{t}\right)\left( -2(\partial_q \phi)^2 -2\frac{b(\theta)\partial_q^2(q\chi(q))}{r}\right) \right)\right|
\lesssim \frac{\ep^2}{(1+s)^\frac{3}{2}(1+|q|)^{2+2(\delta-\sigma)}}.$$
\end{prp}

\begin{proof} We have in view of \eqref{eqh}, \eqref{formbox} and \eqref{dsh},
\begin{align*}\Box \left(\Up\left(\frac{r}{t}\right)h_0\right)
=&\Up\left(\frac{r}{t}\right)\left(4\partial_s \partial_q h_0 + \frac{1}{r}(\partial_s h_0 + \partial_q h_0) + \frac{1}{r^2}\partial^2_\theta h_0\right)
+\nabla \Up\left(\frac{r}{t}\right).\nabla h_0 + h_0 \Box \Up\left(\frac{r}{t}\right)\\
=&\Up\left(\frac{r}{t}\right)\left(
 -4r\partial_q \phi \left( \Box \phi-\frac{1}{r}\partial_s \phi -\frac{1}{r^2}\partial_\theta^2 \phi\right)
+\frac{1}{r}\left(
-2(\partial_q \phi)^2r 
-2b(\theta)\partial_q^2(q\chi(q))\right)\right)\\
&+\Up\left(\frac{r}{t}\right)\left( \frac{1}{r}\partial_s h_0  + \frac{1}{r^2}\partial^2_\theta h_0\right)
+\nabla \Up\left(\frac{r}{t}\right).\nabla h_0 + h_0 \Box \Up\left(\frac{r}{t}\right)\\
=&\Up\left(\frac{r}{t}\right)\left( -2(\partial_q \phi)^2 -4\frac{b(\theta)\partial_q^2(q\chi(q))}{r}\right) 
-4r\Up\left(\frac{r}{t}\right)\partial_q \phi \Box \phi +
 f(s,q,\theta),
\end{align*}
where
$$f(s,q,\theta)=\Up\left(\frac{r}{t}\right)\left(4\partial_q \phi \left(\partial_s \phi +\frac{\partial^2_\theta \phi}{r}\right)+\frac{1}{r}\partial_s h_0+\frac{1}{r^2}\partial^2_\theta h_0\right)+\nabla \Up\left(\frac{r}{t}\right).\nabla h_0 + h_0 \Box \Up\left(\frac{r}{t}\right).$$
We can estimate $Z^I f$, noticing that when $\Up'\left(\frac{r}{t}\right)\neq 0$ we have $r\sim t \sim |q|$. We obtain
$$|Z^I f|\lesssim \frac{1}{(1+s)^2}
\sum_{J\leq I+2} |Z^I h_0|
+\frac{1}{1+s}\sum_{I_1+I_2\leq I}|Z^{I_1+2}\phi|
|\partial_q Z^{I_2}\phi|.$$
Proposition \ref{estzh} yields, for $I\leq N-7$
$$\frac{1}{(1+s)^2}
\sum_{J\leq I+2} |Z^I h_0|\lesssim \frac{\ep^2}{(1+s)^2\sqrt{1+|q|}},$$
and as usual we may estimate, thanks to \eqref{ks1} and \eqref{bootphi1},
$$|Z^{I_1+2}\phi \partial_q Z^{I_2}\phi|\lesssim \frac{\ep^2}{(1+s)(1+|q|)^{1-4\rho}},$$
therefore we obtain
$$|Z^If|\lesssim \frac{\ep^2}{(1+s)^2\sqrt{1+|q|}}.$$
In the region $q>0$, we have the better estimate
$$|Z^If|\lesssim \frac{\ep^2}{(1+s)^2(1+|q|)^{2+2(\delta-\sigma)}}.$$
To estimate $\Box \phi$ we write, as before
$$\Box \phi = \Box \phi-\Box_g \phi.$$
It is sufficient to estimate a term of the form
$g_{LL}\partial^2_q\phi$. Therefore we write, like in estimate \eqref{truc},
$$|Z^I(rg_{LL}\partial_q \phi \partial^2_q \phi)|
\lesssim \frac{r}{1+|q|}\sum_{J_1+J_2+J_3\leq I}
|Z^{J_1}g_{LL}||\partial Z^{J_2+1}\phi||\partial Z^{J_3}\phi|
\lesssim \frac{\ep^3}{(1+s)^\frac{3}{2}(1+|q|)^{\frac{3}{2}-4\rho}}.$$
In the region $q>0$, we have the better estimate
$$|Z^I(r\partial_q \phi\Box \phi)|
\lesssim \frac{\ep^3}{(1+s)^\frac{3}{2}(1+|q|)^{2+2(\delta-\sigma)}}.$$
This concludes the proof of Proposition \ref{boxh}.
\end{proof}

\subsection{Estimation on $\wht{h}$}
We recall that $\wht h$ satisfies the equation
\begin{equation*}
 \left\{ \begin{array}{l} 
         \Box \wht h= \Box\left( \Up\left(\frac{r}{t}\right)h_0\right)+\Up\left(\frac{r}{t}\right)g_{LL}\partial^2_q h_0 +2\Up\left(\frac{r}{t}\right)(\partial_q \phi)^2 - 2(R_b)_{qq}+\Up\left(\frac{r}{t}\right)\wht Q_{\ba L \ba L}(h_0,\wht g), \\
	 (\wht h, \partial_t \wht h)|_{t=0}=(0,0),
        \end{array}
\right.
\end{equation*}
where $\wht Q_{\ba L \ba L}$ is defined by \eqref{qtildell}.
\begin{prp}\label{whth}
$\wht h$ satisfies, for $I\leq N-7$ 
$$|Z^I \wht h|\lesssim \frac{\ep^2}{(1+s)^{\frac{1}{2}-\rho}}.$$
\end{prp}
\begin{proof}
Proposition \ref{boxh} gives for $I\leq N-7$ and $q<0$
\begin{equation}
\label{tildeh1}\left|Z^I\left(\Box \Up\left(\frac{r}{t}\right)h_0
+2\Up\left(\frac{r}{t}\right)(\partial_q \phi)^2 - 2(R_b)_{qq}\right)\right|\lesssim \frac{\ep^2}{(1+s)^\frac{3}{2}(1+|q|)},
\end{equation}
where we have used that thanks to \eqref{rqq}
$$\left|(R_b)_{qq}-2\frac{b(\theta)\partial_q^2(q\chi(q))}{r}\right|\lesssim
\frac{\ch_{1\leq q \leq 2}\ep^2}{(1+s)^2}.$$
To estimate
$Z^I (g_{LL}\partial^2_q h_0)$
we use the transport equation for $h_0$
$$
g_{LL}\partial^2_q h_0
= g_{LL}\partial_q(-2
r(\partial_q \phi)^2-2b(\theta)\partial_q^2(q\chi(q))
$$
We estimate the first term as in the proof of Proposition \ref{boxh}.
$$|Z^I(rg_{LL}\partial_q \phi \partial^2_q \phi)|
\lesssim \frac{\ep^2}{(1+s)^\frac{3}{2}(1+|q|)^{\frac{3}{2}-4\rho}}.$$
To estimate the second term, we note that the terms of the form
$\chi^{(j)}(q)$ decay faster than any power of $q$, so thanks to \eqref{wc4},
$$|Z^I\left(g_{LL}b(\theta)\partial^2_q(q\chi(q))\right)|\lesssim \frac{\ep^2}{s^\frac{3}{2}(1+|q|)^3}.$$
Consequently we have proved
\begin{equation}
\label{tildeh2}
\left|Z^i\left(\Up\left(\frac{r}{t}\right)  g_{LL}\partial^2_q h_0\right)\right|\lesssim \frac{\ep^2}{(1+s)^\frac{3}{2}(1+|q|)^{\frac{3}{2}-4\rho}}.
\end{equation}
We now estimate $\wht Q_{\ba L \ba L}(h_0,\wht g)$. We note than in the region $q<0$ the only term is $\partial_{\ba L}\wht g_{L L}\partial_{\ba L}h_0$. We use again the transport equation for $h_0$
$$\partial_q g_{LL}\partial_q h_0
=\partial_q g_{LL}(-2
r(\partial_q \phi)^2-2b(\theta)\partial_q^2(q\chi(q)).$$
Consequently, for similar reasons than for \eqref{tildeh2}, we obtain in the region $q<0$
\begin{equation}
\label{tildeh3}
\left|Z^i\left(\Up\left(\frac{r}{t}\right)  \partial_q g_{LL}\partial_q h_0\right)\right|\lesssim \frac{\ep^2}{(1+s)^\frac{3}{2}(1+|q|)^{\frac{3}{2}-4\rho}}.
\end{equation}
Thanks to \eqref{tildeh1}, \eqref{tildeh2} and \eqref{tildeh3}, we have in the region $q<0$ for $I\leq N-7$
\begin{equation}
\label{waveh}|\Box Z^I \wht h|\lesssim \frac{\ep^2}{(1+s)^\frac{3}{2}(1+|q|)}.
\end{equation}
In the region $q>0$, we have 
to estimate in $\wht Q_{\ba L \ba L}(h_0,\wht g)$ the term $\partial_{\ba L}(g_b)_{UU}\partial_{\ba L}g_{\ba L L}$, which is of the form
$\frac{\chi(q)b(\theta)}{r}\partial_q g_{\ba L L}$.
 Thanks to \eqref{ks5} we have
\begin{equation}
\label{estcross}\left|Z^I \left(\frac{\chi(q)b(\theta)}{r}\partial_q g_{\ba L L}\right)\right|\lesssim \frac{\ep^2}{(1+s)^{\frac{3}{2}}(1+|q|)^{\frac{3}{2}+\delta-\sigma}}.
\end{equation}
The other terms give contributions similar to the one of Proposition \ref{boxh}. Consequently, for $q>0$ we have the
better estimate for $I\leq N-7$
\begin{equation}
\label{waveh2}|\Box Z^I\wht h|\lesssim \frac{\ep^2}{(1+s)^\frac{3}{2}(1+|q|)^{\frac{3}{2}+\delta-\sigma}}.
\end{equation}
We now use lemma \ref{linf2}, whose proof is given at the end of this section, to conclude.
\begin{lm}
\label{linf2}Let $\beta,\alpha \geq 0$, such that $\beta-\alpha \geq \rho>0$.
Let $u$ be such that
$$|\Box u|\lesssim \frac{1}{(1+s)^{\frac{3}{2}-\alpha}(1+|q|)}, \; for \;q<0
\quad |\Box u|\lesssim \frac{1}{(1+s)^{\frac{3}{2}-\alpha}(1+|q|)^{1+\beta}},
\; for \; q>0,$$
and $(u,\partial_t u)|_{t=0}=0.$
Then we have the estimate
$$|u|\lesssim \frac{(1+t)^{\alpha+\rho}}{\sqrt{1+s}}.$$
\end{lm}
Thanks to \eqref{waveh} and \eqref{waveh2}, the conditions of Lemma \ref{linf2} are satisfied with $\alpha= 0$ and $\beta=\frac{1}{2}+\delta-\sigma$. Moreover, the initial data for $Z^I \wht h$ are given by the right-hand side of \eqref{eqht} (i.e. they are quadratic), therefore, for $I\leq N-7$ at $t=0$ we have
$$|Z^I \wht h|+(1+r)|\partial_t Z^I \wht h|\lesssim \frac{\ep^2}{(1+r)^{1+\delta}}.$$
Consequently, Lemma \ref{linf2} and Proposition \ref{flat1} yield for $I\leq N-7$
$$|Z^I \wht h|\lesssim \frac{\ep^2(1+t)^\rho}{\sqrt{1+s}}.$$
This concludes the proof of Lemma \ref{whth}.
\end{proof}

\begin{proof}[Proof of Lemma \ref{linf2}]
Let $t_0 >0$. We consider times $t\leq t_0$.
In the region $r\leq 2t$ we have $|q|\leq t \leq t_0$ and $s \leq 3t \leq 3t_0$. Therefore
$$|\Box u|\lesssim \frac{(1+t_0)^{\alpha+\rho}}{(1+|q|)^{1+\frac{\rho}{2}}(1+s)^{\frac{3}{2}+\frac{\rho}{2}}}.$$
In the region $r\leq 2t$, we have 
$\frac{r}{2}\leq |q|\leq r$ and $r\leq s\leq \frac{3r}{2}$, therefore
$$|\Box u|\lesssim \frac{1}{(1+r)^{\frac{3}{2}-\alpha+1+\beta}}
\lesssim \frac{(1+t_0)^{\alpha+\rho}}{(1+r)^{\frac{5}{2}-\alpha+\beta}}
\lesssim\frac{(1+t_0)^{\alpha+\rho}}{(1+|q|)^{1+\frac{\rho}{2}}(1+s)^{\frac{3}{2}+\frac{\rho}{2}}},$$
provided $\frac{5}{2}+\rho \leq \frac{5}{2}+\beta-\alpha$, i.e. $\beta -\alpha\geq \rho$.
Consequently, the $L^\infty-L^\infty$ estimate yields, for $t\leq t_0$
$$|u|\lesssim \frac{(1+t_0)^{\alpha+\rho}}{\sqrt{1+s}}.$$
If we take $t=t_0$ we have proved
$$|u|\lesssim \frac{(1+t)^{\alpha+\rho}}{\sqrt{1+s}},$$
which concludes the proof of Lemma \ref{linf2}.
\end{proof}

\section{Commutation with the vector fields and $L^\infty$ estimates}\label{seclinf}
\subsection{Estimates for $I\leq N-14$}

\begin{prp}\label{linfN}
We have the estimates for for $I \leq N-14$
\begin{align*}
|Z^I \wht g_1| &\leq \frac{C_0\ep+C\ep^2 }{(1+s)^{\frac{1}{2}-\rho}},\\
|Z^I \phi| &\leq \frac{C_0\ep+C\ep^2 }{\sqrt{1+s}(1+|q|)^{\frac{1}{2}-4\rho}}.
\end{align*}
\end{prp}
This proposition is a consequence of $L^\infty-L^\infty$ estimates and the following propositions.

\begin{prp}\label{linfphi}
We have the estimate for $I \leq N-14$ 
\begin{align*}
|\Box Z^I \phi| &\lesssim \frac{\ep^2 } {(1+s)^{2-3\rho}(1+ |q|)} , \; q<0,\\
|\Box Z^I \phi| &\lesssim \frac{\ep^2 } {(1+s)^{2}(1+ |q|)^{1+\delta-\sigma}} ,\;q>0.
\end{align*}
\end{prp}

\begin{prp}\label{linfwhtg}
We have the estimate for $I \leq N-14$ 
\begin{align*}
|\Box Z^I \wht g_1| &\lesssim \frac{\ep^2 } {(1+s)^{\frac{3}{2}}(1+ |q|)}, \; q<0,\\
|\Box Z^I \wht g_1| &\lesssim \frac{\ep^2}{(1+s)^\frac{3}{2}(1+|q|)^{\frac{3}{2}+\delta-\sigma}},\; q>0.\\ 
\end{align*}
\end{prp}
We first assume Proposition \ref{linfphi} and \ref{linfwhtg}, and prove Proposition \ref{linfN}.

\begin{proof}[Proof of Proposition \ref{linfN}]
We have 
$$|\Box Z^I \phi| \lesssim \frac{\ep^2 } {(1+s)^{2-3\rho}(1+ |q|)}
\lesssim \frac{\ep^2}{(1+s)^{2-4\rho}(1+|q|)^{1+\rho}},$$
therefore the $L^\infty-L^\infty$ estimate, combined with Proposition \ref{flat1} for the contribution of the initial data yields
$$|Z^I \phi|\leq \frac{C_0\ep}{\sqrt{1+s}\sqrt{1+|q|}}+\frac{C\ep^2}{\sqrt{1+s}(1+|q|)^{\frac{1}{2}-4\rho}},$$
where $C$ is a constant depending on $\rho$.

The estimate $\wht g_1$ follows from Lemma \ref{linf2} with $\alpha= 0$, $\beta=\frac{3}{2}+\delta-\sigma$ combined with Proposition \ref{flat1} 
$$| Z^I \wht g_1|\leq \frac{C_0\ep}{\sqrt{1+s}\sqrt{1+|q|}}+\frac{C\ep^2}{(1+s)^{\frac{1}{2}-\rho}},$$
which concludes the proof of Proposition \ref{linfN}.
\end{proof}

\begin{proof}[Proof of Proposition \ref{linfphi}]
We first estimate $\Box Z^I \phi$ in the region $q<0$
\[
Z^I \Box \phi =  Z^I \left(\Box \phi -\Box_g \phi\right).\]
In the region $q<0$, thanks to Remark \ref{termql}, it is sufficient to estimate
$Z^I \left( g_{LL}\partial_q^2 \phi\right)$
\[|Z^{I-J} g_{LL} \partial^2_q Z^{J} \phi|\lesssim \frac{1}{(1+|q|)^2}|Z^{I-J} g_{LL}| |Z^{J+2} \phi|. \]
If $J \leq  \frac{N-14}{2}$ we have $J+2\leq \frac{N-14}{2}+2\leq N-14$ so,
thanks to \eqref{bootphi1}
$$| Z^{J+2} \phi|\lesssim \frac{\ep}{(1+s)^{\frac{1}{2}}(1+|q|)^{\frac{1}{2}-4\rho}},$$
and since $I-J\leq N-14$ we have thanks to \eqref{wc2}
$$|Z^{I-J} g_{LL}|\lesssim \frac{\ep (1+|q|) }{(1+s)^{\frac{3}{2}-2\rho}}.$$
Therefore
$$|Z^{I-J} g_{LL} \partial^2_q Z^{J} \phi|\lesssim \frac{\ep^2}{(1+s)^{2-2\rho}(1+|q|)^{\frac{3}{2}-4\rho}}.$$
If $I-J \leq  \frac{N-14}{2} \leq N-15$ we have thanks to \eqref{wc1}
$$|Z^{I-J} g_{LL}|\lesssim \frac{\ep (1+|q|) }{(1+s)^{\frac{3}{2}-\rho}},$$
and since $J+2\leq N-12$ we have thanks to \eqref{bootphi2}
$$| Z^{J+2} \phi |\lesssim \frac{\ep}{(1+s)^{\frac{1}{2}-2\rho}}.$$
In the two cases, we have for $q<0$
\begin{equation}
\label{philinf1}|Z^{I-J} g_{LL} \partial^2_q Z^{J} \phi|\lesssim \frac{\ep^2}{(1+s)^{2-3\rho}(1+|q|)}.
\end{equation}
In the region $q>0$ we have the better estimate thanks to \eqref{wc5} and \eqref{ks5}
\begin{equation}
\label{philinf2}|Z^{I-J} g_{LL} \partial^2_q Z^{J} \phi|\lesssim \frac{\ep^2}{(1+s)^{2}(1+|q|)^{2+2(\delta-\sigma)}}.
\end{equation}
In the region $q>0$ we also have to take into account the crossed term. These terms are described by \eqref{crossphi} in Section \ref{cros}. It is sufficient to estimate
$$Z^I \left(b(\theta)\frac{\chi(q)}{r}\partial_s \phi\right).$$
Since they occur only in the region $q>0$, we can estimate, thanks to \eqref{iks5}
$$|Z^I \phi| \lesssim \frac{\ep }{\sqrt{1+s}(1+|q|)^{\frac{1}{2}+\delta-\sigma}}.$$
Therefore
\begin{equation}
\label{philinf3}\left|Z^I b(\theta)\frac{\partial_q(q\chi(q))}{r}\partial_s \phi\right|\lesssim \frac{\ep^2}{(1+s)^{\frac{5}{2}}(1+|q|)^{\frac{1}{2}+\delta-\sigma}}
\lesssim \frac{\ep^2}{(1+s)^2(1+|q|)^{1+\delta-\sigma}}.
\end{equation}
Estimates \eqref{philinf1}, \eqref{philinf2} and \eqref{philinf3} conclude the proof of Proposition \ref{linfphi}.
\end{proof}

\begin{proof}[Proof of Proposition \ref{linfwhtg}]
We write the equation for
$\wht g_1$. We have, thanks to \eqref{s2} and \eqref{noncom}
\begin{equation}\label{eqg1}
\begin{split}
\Box (\wht g_1)_{\mu \nu} =& -2 \partial_\mu \phi\partial_\nu \phi
+2(R_b)_{\mu \nu} + (dq^2)_{\mu \nu}\Box \Up\left(\frac{r}{t}\right)h_0\\
&+\Up\left(\frac{r}{t}\right)\frac{1}{r^2}\left(u^1_{\mu \nu}(\theta)h_0
+u^2_{\mu \nu}(\theta)\partial_\theta h_0\right)\\
&+P_{\mu \nu}(g)(\partial \wht g, \partial \wht g) + \wht P_{\mu \nu} (\wht g, g_b),
\end{split}
\end{equation}
and therefore
$\Box Z^I (\wht g_1)_{\mu \nu} = f_{\mu \nu},$
where the terms in $f_{\mu \nu}$ are of the forms
\begin{itemize}
\item the quasilinear terms : thanks to Remark \ref{termql} it is sufficient to study $Z^I(g_{LL}\partial_q^2\wht g_1)$,
\item the terms coming from the non commutation of the wave operator with the null decomposition: they are calculated in \eqref{noncom} and they are of the form $\Up(\frac{r}{t})\frac{1}{r^2}\partial_\theta Z^I h_0$,
\item the semi-linear terms: following section the worst term is the term   $Z^I\left(\partial_{\ba L}g_{\ba L \ba L}\partial_{\ba L}g_{LL}\right)$
appearing in $Z^I P_{\ba L \ba L}$ (see \eqref{termbal}).
\item the crossed terms with the background metric $g_b$: the worst term is the term $Z^I\left(\partial_{\ba L}(g_b)_{UU}\partial_{\ba L}g_{L\ba L}\right)$ appearing in $Z^I \wht P_{\ba L \ba L}$ (see \eqref{crossbaL}).
\end{itemize}

\paragraph{The quasilinear terms}
We estimate 
$$Z^I \left( g_{LL}\partial_q^2 \wht g_1\right)=\sum_{J\leq I}Z^{I-J}g_{LL}Z^J\partial_q^2 \wht g_1.$$
We have
\[|Z^{I-J} g_{LL} \partial^2_q Z^{J} \wht g_1|\lesssim \frac{1}{(1+|q|)^2}||Z^{I-J} g_{LL}|| Z^{J+2} \wht{g_1}| .\]
If $J\leq  \frac{N-14}{2}$ we have $J+2\leq \frac{N-14}{2}+2\leq N-14$ so
thanks to \eqref{bootg1}
$$| Z^{J+2} \wht{g_1}|\lesssim \frac{\ep}{(1+s)^{\frac{1}{2}-\rho}},$$
and since $I-J\leq N-14$ we have thanks to \eqref{wc2}
$$|Z^{I-J} g_{LL}|\lesssim \frac{\ep (1+|q|) }{(1+s)^{\frac{3}{2}-2\rho}}.$$
If $I-J \leq  \frac{N-14}{2} \leq N-15$ we have thanks to \eqref{wc1}
$$|Z^{I-J} g_{LL}|\lesssim \frac{\ep (1+|q|) }{(1+s)^{\frac{3}{2}-\rho}},$$
and since $J+2\leq N-12$ we have thanks to \eqref{bootg2}
$$| Z^{J+2} \wht{g_1}|\lesssim \frac{\ep}{(1+s)^{\frac{1}{2}-2\rho}}$$
In the two cases, we have
\begin{equation}
\label{linfg1}|Z^{I-J} g_{LL} \partial^2_q Z^{J} \wht g_1|\lesssim \frac{\ep^2}{(1+s)^{2-3\rho}(1+|q|)}.
\end{equation}

\paragraph{The term coming from the non commutation of the wave operator with the null structure}
We have to estimate 
$$\Up\left(\frac{r}{t}\right)\frac{\partial_\theta Z^I h_0}{r^2}.$$
Since $I\leq N-14$, we have $I+1\leq N-5$ so thanks to Proposition \ref{estzh}
\begin{equation}
\label{linfg4}\left|\Up\left(\frac{r}{t}\right)\frac{\partial_\theta Z^I h_0}{r^2}\right|\lesssim \frac{\ep^2}{(1+s)^2\sqrt{1+|q|}}
\lesssim \frac{\ep^2}{(1+s)^\frac{3}{2}(1+|q|)}.
\end{equation}

\paragraph{The semi-linear terms}
We estimate $Z^I\left(\partial_{\ba L}g_{\ba L \ba L}\partial_{\ba L}g_{LL}\right)$. For this, we have to estimate, using the decomposition \eqref{dec1}
$$Z^I\left(\partial_{\ba L}h_0\partial_{\ba L}g_{LL}\right) \quad and \quad Z^I\left(\partial_{\ba L}\wht g_1\partial_{\ba L}g_{LL}\right) $$
The first term has been estimated in 
\eqref{tildeh1}. For the second term, we write
$$|Z^I\left(\partial_{\ba L}\wht g_1\partial_{\ba L}g_{LL}\right) |\lesssim \frac{1}{1+|q|}\sum_{J\leq I}|Z^{J+1}\wht g_1||\partial Z^{I-J}g_{LL}|,$$
and we estimate
if $J\leq \frac{N-14}{2}$ thanks to \eqref{bootg1} and \eqref{dwc2}
$$|Z^{J+1}\wht g_1|\lesssim \frac{\ep}{(1+s)^{\frac{1}{2}-\rho}}\quad and \quad
|\partial Z^{I-J}g_{LL}|\lesssim \frac{\ep}{(1+s)^{\frac{3}{2}-2\rho}}.$$
If $I-J \leq \frac{N-14}{2}$  thanks to \eqref{bootg2} and \eqref{dwc1} we have
$$|Z^{J+1}\wht g_1|\lesssim \frac{\ep}{(1+s)^{\frac{1}{2}-2\rho}}\quad and \quad
|\partial Z^{I-J}g_{LL}|\lesssim \frac{\ep}{(1+s)^{\frac{3}{2}-\rho}}.$$
In the two cases we have
$$|Z^I\left(\partial_{\ba L}\wht g_1\partial_{\ba L}g_{LL}\right) |\lesssim
\frac{\ep^2}{(1+s)^{2-3\rho}(1+|q|)}.$$
This estimate and \eqref{tildeh1} yields for $I\leq N-14$
\begin{equation}\label{linfg3}
|Z^I\left(\partial_{\ba L}\wht g_1\partial_{\ba L}g_{LL}\right) |
\lesssim \frac{\ep^2}{(1+s)^{\frac{3}{2}}(1+|q|)^{\frac{3}{2}-4\rho}}.
\end{equation}

We have now estimated $\Box Z^I (\wht g_1)_{\mu \nu}$ in the region $q<0$. Thanks to  \eqref{linfg1}, \eqref{linfg4} and \eqref{linfg3} we have, for $q<0$ and $I\leq N-14$
\begin{equation}
\label{lingfA}|\Box Z^I \wht g_1| \lesssim \frac{\ep^2 } {(1+s)^{2-3\rho}(1+ |q|)}  
\end{equation}

\paragraph{The crossed terms}
The crossed term are only present in the region $q>0$.
The estimate of
$$Z^I\left(\partial_q (g_b)_{UU}\partial_q \wht g_{\ba L L}\right)$$
is done in \eqref{estcross}. The other terms give better contributions in the region $q>0$ (see Remark \ref{remcros}). Therefore we have
for $q<0$ and $I\leq N-4$
\begin{equation}
\label{lingfB}|\Box Z^I \wht g_1| \lesssim \frac{\ep^2 } {(1+s)^{\frac{3}{2}}(1+ |q|)^{\frac{3}{2}+\delta-\sigma}}.  
\end{equation}
The estimates \eqref{lingfA} and \eqref{lingfB} conclude the proof of Proposition \ref{linfwhtg}.
\end{proof}

\subsection{Estimates for $I\leq N-12$}
\begin{prp}\label{linf}
We have the estimates for $I \leq N-12$
\begin{align*}
| Z^I \phi| &\leq \frac{C_0\ep+C\ep^2}{(1+s)^{\frac{1}{2}-2\rho}}, \\
| Z^I \wht g_1| &\lesssim \frac{C_0\ep+C\ep^2}{(1+s)^{\frac{1}{2}-2\rho}}. 
\end{align*}
\end{prp}
This proposition is a straightforward consequence of Lemma \ref{linf2}, Proposition \ref{flat1} and the following propositions.

\begin{prp}\label{linfphi2}
We have the estimate for $I \leq N-12$
\begin{align*}
 |\Box Z^I  \phi| &\lesssim \frac{\ep^2}{(1+s)^{\frac{3}{2}-\rho}(1+|q|)},\; q<0, \\
|\Box Z^I \phi| &\lesssim \frac{\ep^2 } {(1+s)^{2}(1+ |q|)^{1+\delta-\sigma}} ,\;q>0.
\end{align*}
\end{prp}

\begin{prp}\label{linfg2}
We have the estimate for $I \leq N-12$
\begin{align*}
 |\Box Z^I \wht g_1| &\lesssim \frac{\ep^2}{(1+s)^{\frac{3}{2}-\rho}(1+|q|)} , \; q<0,\\
|\Box Z^I \wht g_1| &\lesssim \frac{\ep^2}{(1+s)^\frac{3}{2}(1+|q|)^{\frac{3}{2}+\delta-\sigma}},\; q>0.\\ 
\end{align*}
\end{prp}

%\begin{proof}[Proof of Proposition \ref{linf}]
%We have, for $I\leq \frac{N}{2}+4$ and $t\leq t_0$
%$$|\Box Z^I \phi|\lesssim %\frac{\ep^2(1+t_0)^{2\rho}}{(1+s)^{\frac{3}{2}+\frac{\rho}{2}}(1+|q|)^{1+\frac{\rho}{2}}},$$
%therefore the $L^\infty-L^\infty$ estimate and Proposition \ref{flat1} yield
%$$|Z^I \phi|\leq \frac{C_0 \ep}{\sqrt{1+|q|}\sqrt{1+s}}+\frac{C\ep^2}{(1+s)^{\frac{1}{2}-2\rho}}.$$
%It is exactly the same for $\wht g_1$.

%\end{proof}

\begin{proof}[Proof of Proposition \ref{linfphi2}]
We first estimate $\phi$
\[
Z^I \Box \phi =  Z^I \left(\Box \phi -\Box_g \phi\right).\]
In the region $q<0$, it is sufficient to estimate
$Z^I \left( g_{LL}\partial_q^2 \phi\right)$
\[|Z^{I-J} g_{LL} \partial^2_q Z^{J} \phi|\lesssim \frac{1}{1+|q|}|Z^{I-J} g_{LL}| |\partial_q Z^{J+1} \phi| \]
If $J \leq  \frac{N-12}{2}$ we have $J+1\leq N-14$ so thanks to \eqref{bootphi1}
$$|\partial Z^{J+1} \phi|\lesssim \frac{\ep}{(1+s)^{\frac{1}{2}}(1+|q|)^{\frac{3}{2}-4\rho}},$$
and since $I-J\leq N-12$ we have thanks to \eqref{wc3}
$$|Z^{I-J} g_{LL}|\lesssim \frac{\ep (1+|q|) }{1+s}.$$
Therefore
$$|Z^{I-J} g_{LL} \partial^2_q Z^{J} \phi|\lesssim \frac{\ep^2}{(1+s)^\frac{3}{2}(1+|q|)^{\frac{3}{2}-4\rho}}.$$
If $I-J \leq  \frac{N-12}{2} \leq N-15$ we have thanks to \eqref{wc1}
$$|Z^{I-J} g_{LL}|\lesssim \frac{\ep (1+|q|) }{(1+s)^{\frac{3}{2}-\rho}},$$
and since $J+1\leq N-12\leq N-4$ we have thanks to \eqref{ks1}
$$| \partial Z^{J+1} \phi |\lesssim \frac{\ep}{\sqrt{1+s}\sqrt{(1+|q|)}}.$$
In the two cases, we have
\begin{equation*}
|Z^{I-J} g_{LL} \partial^2_q Z^{J} \wht \phi|\lesssim \frac{\ep^2}{(1+s)^{\frac{3}{2}-\rho}(1+|q|)}.
\end{equation*}
The main contribution in the region $q>0$ is like \eqref{philinf3} in the proof of Proposition \ref{linfphi}. This concludes the proof of Proposition \ref{linfphi2}.
\end{proof}
\begin{proof}[Proof of Proposition \ref{linfg2}]
We estimate $\wht g_1$. We only deal with the quasilinear and semilinear terms in the region $q<0$, as the control obtain in the proof of Proposition \ref{linfwhtg} is sufficient to deal with the others (see \eqref{linfg4} and \eqref{estcross}).

\paragraph{The semi-linear terms}
We estimate $Z^I\left(\partial_{\ba L}g_{\ba L \ba L}\partial_{\ba L}g_{LL}\right)$. For this, we have to estimate
$$Z^I\left(\partial_{\ba L}h_0\partial_{\ba L}g_{LL}\right) \quad and \quad Z^I\left(\partial_{\ba L}\wht g_1\partial_{\ba L}g_{LL}\right) $$
The first term has been estimated in 
\eqref{tildeh1}. For the second term, we write
$$|Z^I\left(\partial_{\ba L}\wht g_1\partial_{\ba L}g_{LL}\right) |\lesssim \frac{1}{1+|q|}\sum_{J\leq I}|Z^{J+1}\wht g_1||\partial Z^{I-J}g_{LL}|,$$
and we estimate
if $J\leq \frac{N-12}{2}$ thanks to \eqref{bootg1} and \eqref{dwc3}
$$|Z^{J+1}\wht g_1|\lesssim \frac{\ep}{(1+s)^{\frac{1}{2}-\rho}}\quad and \quad
|\partial Z^{I-J}g_{LL}|\lesssim \frac{\ep\sqrt{1+|q|}}{(1+s)^{\frac{3}{2}}}.$$
If $I-J \leq \frac{N-12}{2}$  thanks to \eqref{est3} and \eqref{dwc1} we have
$$|Z^{J+1}\wht g_1|\lesssim \ep \quad and \quad
|\partial Z^{I-J}g_{LL}|\lesssim \frac{\ep}{(1+s)^{\frac{3}{2}-\rho}}.$$
In the two cases we have
$$|Z^I\left(\partial_{\ba L}\wht g_1\partial_{\ba L}g_{LL}\right) |\lesssim
\frac{\ep^2}{(1+s)^{\frac{3}{2}-\rho}(1+|q|)}.$$
This estimate and \eqref{tildeh1} yields for $I\leq N-12$
\begin{equation}\label{linfg7}
|Z^I\left(\partial_{\ba L}\wht g_1\partial_{\ba L}g_{LL}\right) |
\lesssim \frac{\ep^2}{(1+s)^{\frac{3}{2}-\rho}(1+|q|)}.
\end{equation}

\paragraph{The quasilinear terms}
We estimate $Z^I \left( g_{LL}\partial_q^2 \wht g_1\right)$. We have
\[|Z^{I-J} g_{LL} \partial^2_q Z^{J} \wht g_1|\lesssim \frac{1}{1+|q|} |Z^{I-J} g_{LL}| |\partial_q Z^{J+1} \wht{g_1}|. \]
If $J \leq  \frac{N-12}{2}$ we have $J+2\leq \frac{N-12}{2}+2\leq N-14$ so thanks to \eqref{bootg1}
$$| \partial_q Z^{J+1} \wht{g_1}|\lesssim \frac{\ep}{(1+s)^{\frac{1}{2}-\rho}(1+|q|)},$$
and since $|I-J|\leq N-12$ we have thanks to \eqref{wc3}
$$|Z^{I-J} g_{LL}|\lesssim \frac{\ep (1+|q|) }{1+s}.$$
If $|I-J| \leq  \frac{N-12}{2} \leq N-15$ we have thanks to \eqref{wc1}
$$|Z^{I-J} g_{LL}|\lesssim \frac{\ep (1+|q|) }{(1+s)^{\frac{3}{2}-\rho}}$$
and since $J+1\leq N-11$ we have thanks to \eqref{est3}
$$|\partial_q Z^{J+1} \wht{g_1}|\lesssim \frac{\ep}{1+|q|}.$$
In the two cases, we have
\begin{equation}
\label{linfg8}|Z^{I-J} g_{LL} \partial^2_q Z^{J} \wht g_1|\lesssim \frac{\ep^2}{(1+s)^{\frac{3}{2}-\rho}(1+|q|)}.
\end{equation}
The equation \eqref{linfg7} and \eqref{linfg8}, together with \eqref{linfg4} proved during the proof of Proposition \ref{linfwhtg} conclude the proof of Proposition \ref{linfg2} for $q<0$.
The estimate for $\Box Z^I \wht g_1$ in the region $q>0$ is given by \eqref{lingfB}. This conclude the proof of Proposition \ref{linfg2} for $q>0$.
\end{proof}

\section{Weighted energy estimate}\label{weighte}
We consider the equation 
$$\Box_g u=f,$$
where $g=g_b +\wht g$ is our space-time metric, satisfying the bootstrap assumptions.
We introduce the energy-momentum tensor associated to $\Box_g$
$$Q_{\alpha \beta}=\partial_\alpha u \partial_\beta u-\frac{1}{2}g_{\alpha \beta}g^{\mu \nu}\partial_\mu u \partial_\nu u.$$
We have
$$D^\alpha Q_{\alpha \beta}=f\partial_\beta u.$$
We also note $T=\partial_t$, and introduce the deformation tensor of $T$
$$\pi_{\alpha \beta}=D_\alpha T_\beta +D_\beta T_\alpha$$
where $D$ is the covariant derivative.
We have
\begin{equation}
\label{eq}
D^\alpha (Q_{\alpha \beta} T^\beta) =f\partial_t u +Q_{\alpha \beta}\pi^{\alpha \beta}.
\end{equation}

We remark that
$$Q_{TT}=\frac{1}{2}\left((\partial_t u)^2+|\nabla u|^2\right)+O(\ep(\partial u)^2).$$
\begin{prp}\label{prpweighte}Let $w$ be any of our weight functions.
 We have the following weighted energy estimate for $u$
 $$\frac{d}{dt}\left(\int Q_{TT} w(q)\right)+C\int w'(q)\left((\partial_s u)^2+\left(\frac{\partial_\theta u}{r}\right)^2\right)
 \lesssim \frac{\ep}{1+t}\int w(q)(\partial u)^2 + \int w(q)|f\partial_t u|.$$
 Moreover, if we use the weight modulator $\alpha$ defined in \eqref{wmod}, we obtain
$$\frac{d}{dt}\left(\int Q_{TT}\alpha^2 w(q)\right)+C\int \alpha^2w'(q)\left((\partial_s u)^2+\left(\frac{\partial_\theta u}{r}\right)^2\right)
  \lesssim \frac{\ep}{(1+t)^{1+2\sigma}}\int w(q)(\partial u)^2 + \int \alpha ^2w(q)|f\partial_t u|.$$
\end{prp}

\begin{proof}
We multiply \eqref{eq} by $w(q)$ and integrate it on an hypersurface of constant $t$.
We obtain
\begin{equation}
\label{energi1}-\frac{d}{dt}\left(\int Q_{TT} w(q)\right)=\int w(q)\left(f\partial_t u +Q_{\alpha \beta}\pi^{\alpha \beta}\right)+\int Q_{T\alpha}D^\alpha w.
\end{equation}
We have
$$Q_{T\alpha}D^\alpha u=
-2w'(q)g^{\alpha \ba L}Q_{T\alpha}
=w'(q)Q_{TL}+g_{L \q T}w'(q)(\partial u)^2.$$
We calculate
\begin{align*}
Q_{TL}=&\partial_t u(\partial_t u+\partial_r u)-\frac{1}{2}(-(\partial_t u)^2
+|\nabla u|^2)+g_{LL}(\partial_q u)^2+g_{\ba L \ba L}(\partial_s u)^2+s.t.\\
=&\frac{1}{2}\left((\partial_s u)^2+\left(\frac{\partial_\theta u}{r}\right)^2\right)+g_{LL}(\partial_q u)^2+g_{\ba L \ba L}(\partial_s u)^2+s.t.
\end{align*}
where $s.t.$ denotes similar terms. Consequently, with the help of the bootstrap \eqref{bootg1}, \eqref{booth1} and the estimate \eqref{wc1} we have
$$Q_{T\alpha}D^\alpha u=\left((\partial_s u)^2+\left(\frac{\partial_\theta u}{r}\right)^2\right)(1+O(\ep))w'(q)
+O\left(\frac{\ep w'(q)(1+|q|)}{(1+t)^{\frac{3}{2}-\rho}}(\partial u)^2\right),$$
and since $|w'(q)|\lesssim \frac{w(q)}{1+|q|}$
\begin{equation}
\label{qlt}
Q_{T\alpha}D^\alpha u=\left((\partial_s u)^2+\left(\frac{\partial_\theta u}{r}\right)^2\right)(1+O(\ep))w'(q)
+O\left(\frac{\ep w(q)}{(1+t)^{\frac{3}{2}-\rho}}(\partial u)^2\right).
\end{equation}

We now estimate the deformation tensor of $T$. We have
$$\pi_{\alpha \beta}= \q L_T g_{\alpha \beta}=\partial_t g_{\alpha \beta}.$$
 We obtain
\begin{align*}
\pi_{LL}&=\partial_T g_{LL}=O\left(\frac{\ep}{(1+t)^{\frac{3}{2}-\rho}}\right),\\
\pi_{UL}&=\partial_T g_{UL}=O\left(\frac{\ep}{(1+t)^{\frac{3}{2}-\rho}}\right),\\
\pi_{L \ba L}&=\partial_T g_{\ba L L}=O\left(\frac{\ep}{(1+t)^{\frac{1}{2}-\rho}(1+|q|)}\right),\\
\pi_{U \ba L}&=\partial_T g_{U \ba L}= O\left(\frac{\ep}{(1+t)^{\frac{1}{2}-\rho}(1+|q|)}\right),\\
\pi_{\ba L \ba L}&= \partial_{T}g_{\ba L \ba L}=O\left(\frac{\ep}{(1+|q|)^{\frac{3}{2}-\rho}}\right),\\
\pi_{UU}&=\partial_T g_{UU}=\frac{\partial_q(q\chi(q))b(\theta)}{r}+ O\left(\frac{\ep}{(1+s)(1+|q|)^{\frac{1}{2}-\rho}}\right),
\end{align*}
Consequently, the terms $Q^{LL}\pi_{LL}$ and $Q^{UL}Q_{UL}$ give contributions of the form
\begin{equation}\label{cont1}
\frac{\ep}{(1+t)^{\frac{3}{2}-\rho}}(\partial u)^2.
\end{equation}
We can calculate
$$Q_{L\ba L}= \partial_{L}u\partial_{\ba L}u-\frac{1}{2}g_{L\ba L}\left(2g^{L\ba L}\partial_L u \partial_{\ba L}u +(\partial_U u)^2\right)+g_{\q T \q T}(\partial u)^2+s.t.=(\partial_U u)^2+g_{\q T \q T}(\partial u)^2+s.t.$$
Consequently the term $Q^{L\ba L}\pi_{L \ba L}$ gives contributions of the form
\begin{equation}\label{cont2}
\frac{\ep}{(1+|q|)(1+t)^{\frac{1}{2}-\rho}}(\bar{\partial} u)^2.
\end{equation}
The terms $Q^{\ba L \ba L}\pi_{\ba L \ba L}$ and $Q^{\ba L U}\pi_{\ba L U}$ give contributions of the form
\begin{equation}\label{cont3}
\frac{\ep}{(1+|q|)^{\frac{3}{2}-\rho}}(\bar{\partial} u)^2,
\end{equation}
and the term  $Q^{UU}\pi_{UU}$ gives contributions of the form
\begin{equation}\label{cont4}
\frac{\partial_q(q\chi(q))b(\theta)}{r}\bar{\partial} u\partial u
, \quad\frac{\ep}{(1+s)(1+|q|)^{\frac{1}{2}-\rho}}\bar{\partial} u\partial u.
\end{equation}
Thanks to \eqref{energi1}, \eqref{qlt}, \eqref{cont1}, \eqref{cont2}, \eqref{cont3} and \eqref{cont4} what we obtain is
\begin{equation}\begin{split}\label{enee}
&\frac{d}{dt}\left(\int Q_{TT} w(q)\right)+\frac{1}{2}\int w'(q)\left((\partial_s u)^2+\left(\frac{\partial_\theta u}{r}\right)^2\right)\\
\lesssim& \frac{\ep}{(1+t)^{\frac{3}{2}-\rho}}\int w(q)(\partial u)^2
+\ep\int \frac{w(q)}{(1+|q|)^{\frac{3}{2}-\rho}}(\bar{\partial}u)^2
+\ep \int w(q)\frac{\ch_{q>1}}{r}|\partial u \bar{\partial} u|+\int w(q)|f\partial_t u |.
\end{split}
\end{equation}
In the region $q>1$, we have $\frac{1}{r}\leq\frac{1}{t+1}$. Moreover,
all our weight functions satisfy
$$\frac{w(q)}{(1+|q|)^{\frac{3}{2}-\rho}}\lesssim w'(q),$$
therefore, for $\ep$ small enough, we can subtract from our inequality the term 
$$ \ep\int \frac{w(q)}{(1+|q|)^{\frac{3}{2}-\rho}}(\bar{\partial}u)^2,$$
and we obtain
\begin{equation*}
\frac{d}{dt}\left(\int Q_{TT} w(q)\right)+C\int w'(q)\left((\partial_s u)^2+\left(\frac{\partial_\theta u}{r}\right)^2\right)
 \lesssim \frac{\ep}{1+t}\int w(q)(\partial u)^2 + \int w(q)|f\partial_t u|.
\end{equation*}
 This conclude the first part of the proof of Proposition \ref{prpweighte}.

Next, we perform the estimate with the weight modulator $\alpha$. If we replace $w$ by $\alpha^2w$ in \eqref{enee}, and we absorb as before
the term $\ep\int \frac{\alpha^2w(q)}{(1+|q|)^{\frac{3}{2}-\rho}}(\bar{\partial}u)^2$
we obtain
\begin{align*}&\frac{d}{dt}\left(\int Q_{TT}\alpha^2 w(q)\right)+\frac{1}{2}\int (\alpha^2w)'(q)\left((\partial_s u)^2+\left(\frac{\partial_\theta u}{r}\right)^2\right)\\
\lesssim& \frac{\ep}{(1+t)^{\frac{3}{2}-\rho}}\int\alpha ^2 w(q)(\partial u)^2
+ \int \frac{\alpha^2w(q)\ch_{q>1}}{r}|\partial u \bar{\partial} u|+\int\alpha^2w(q)|f\partial_t u |.
\end{align*}
We write
$$\frac{\ch_{q>1}}{r}\leq \frac{\ch_{q>1}(1+|q|)^{\sigma}}{(1+t)^{\frac{1}{2}+\sigma}(1+|q|)^\frac{1}{2}},$$
and so we estimate, since in the region $q>1$ we have $\alpha(q)=(1+|q|)^{-\sigma}$
\begin{align*}\ep\int \frac{\alpha^2(q)w(q)\ch_{q>1}}{r}|\partial u \bar{\partial} u|
&\leq\ep\int \frac{\alpha(q)w(q)\ch_{q>1}}{(1+t)^{\frac{1}{2}+\sigma}(1+|q|)^\frac{1}{2}}|\partial u \bar{\partial} u|\\
&\leq \frac{\ep}{t^{1+2\sigma}} \int  \ch_{q>1}w(q)(\partial u)^2
+\ep \int  \ch_{q>1}\frac{\alpha^2(q)w(q)}{1+|q|}(\bar{\partial}u)^2.
\end{align*}
Moreover $ \ch_{q>1}\frac{\alpha^2(q)w(q)}{1+|q|}\lesssim (\alpha^2w)'$. Therefore
\begin{align*}
&\frac{d}{dt}\left(\int Q_{TT} \alpha^2 w(q)\right)+C\int(\alpha^2 w)'(q)\left((\partial_s u)^2+\left(\frac{\partial_\theta u}{r}\right)^2\right)\\
\lesssim& \frac{\ep}{(1+t)^{\frac{3}{2}-\rho}}\int\alpha^2 w(q)(\partial u)^2
+\frac{\ep}{t^{1+2\sigma}} \int  \ch_{q>1}w(q)(\partial u)^2
+\ep \int (w\alpha^2)'(q)(\bar{\partial}u)^2
+\int\alpha^2 w(q)|f\partial_t u| .
\end{align*}
We note that with our weight functions and the definition of $\alpha$, we have $\alpha^2 w' \sim(\alpha^2w)'$.
For $\ep$ small enough, we can absorb the term
$$\ep \int w'(q)\alpha^2(q)(\bar{\partial}u)^2$$
to obtain
\begin{align*}
&\frac{d}{dt}\left(\int Q_{TT} \alpha^2 w(q)\right)+C\int\alpha^2 w'(q)\left((\partial_s u)^2+\left(\frac{\partial_\theta u}{r}\right)^2\right)\\
\lesssim& \frac{\ep}{(1+t)^{\frac{3}{2}-\rho}}\int\alpha^2 w(q)(\partial u)^2
+\frac{\ep}{t^{1+2\sigma}} \int  \ch_{q>1}w(q)(\partial u)^2
+\int\alpha^2 w(q)|f\partial_t u |,
\end{align*}
which concludes the proof of Proposition \ref{prpweighte}.
\end{proof}

\section{Commutation with the vector fields and $L^2$ estimate}\label{secl2}

\subsection{Estimation for $I\leq N$}\label{subsecl21}
We note  for $J<N$
$$E_J=\sum_{I\leq J}\|w_0(q)^\frac{1}{2}\partial Z^I \phi\|^2_{L^2}
+\|w_2(q)^\frac{1}{2}\partial Z^I \wht g_3\|^2_{L^2}
+\frac{1}{\ep (1+t)}\|w_3(q)^\frac{1}{2}\partial Z^I h\|^2_{L^2}
$$
and 
\begin{align*}
E_N=&\sum_{I\leq N}\|\alpha_2 w_0(q)^\frac{1}{2}\partial Z^I \phi\|^2_{L^2}
+\|\alpha_2 w_2(q)^\frac{1}{2}\partial Z^I \wht g_4\|^2_{L^2}\\
&+\frac{1}{\ep (1+t)}\|\alpha_2 w_3(q)^\frac{1}{2}\partial Z^I h\|^2_{L^2}
+\frac{1}{\ep (1+t)}\|\alpha_2 w_3(q)^\frac{1}{2}\partial Z^I k\|^2_{L^2}.
\end{align*}
We also note for $J<N$
\[A_J=\sum_{I\leq J}\|w_0'(q)^\frac{1}{2}\bar{\partial} Z^I \phi\|^2_{L^2}+
+\|w'_2(q)^\frac{1}{2}\bar{\partial} Z^I \wht g_3\|^2_{L^2}
+\frac{1}{\ep (1+t)}\|w_3'(q)^\frac{1}{2}\bar{\partial} Z^I h\|^2_{L^2}\]
and
\begin{align*}
A_N=&\sum_{I\leq J}\|\alpha_2 w_0'(q)^\frac{1}{2}\bar{\partial} Z^I \phi\|^2_{L^2}+
+\|\alpha_2 w'_2(q)^\frac{1}{2}\bar{\partial} Z^I \wht g_4\|^2_{L^2}\\
&+\frac{1}{\ep (1+t)}\|\alpha_2 w_3'(q)^\frac{1}{2}\bar{\partial} Z^I h\|^2_{L^2}
+\frac{1}{\ep(1+t)}\|\alpha_2w_3'(q)^\frac{1}{2}\bar{\partial} Z^I k\|^2_{L^2}.
\end{align*}
\begin{rk}
Because of the decompositions \eqref{dec3} and \eqref{dec4} for the metric, and the non commutation of the wave operator with the null decomposition, we have to deal with terms of the form $\frac{\partial_\theta h}{r^2}$ in the equation for $\wht g_4$ or $\wht g_3$. Written like this, these terms are not quadratic. However, since we choose for $h$ zero initial data, and since the equation for $h$ is quadratic, $h$ in itself is quadratic. To carry this information along the proof, we may divide in the energies $E_I$ the norms involving $h$ and $k$ by $\ep$. Since the initial data for $h$ and $k$ are zero, we have
\begin{equation}
\label{ini} E_I(0)\leq C_0^2\ep^2.
\end{equation}
\end{rk}

\begin{prp}\label{estn}
We have the estimates for $I\leq N$,
$$
E_I
\leq (C_0^2\ep^2+\ep^2)(1+t)^{C\sqrt{\ep}},
$$
and for $\kappa \gg \ep$
\[\int_0^t \frac{1}{(1+t)^{\kappa}}A_I
\lesssim \ep^2.
\]
\end{prp}
This is a straightforward consequence of the following proposition.

\begin{prp}\label{estl2N}
We have the inequality, up to some negligible terms defined in Lemmas \ref{lemmephi}, \ref{lemmeh} and \ref{lemmeg} 
 for $I\leq N$
\[\frac{d}{dt}E_I+A_I\lesssim\frac{\sqrt{\ep}}{1+t}E_I+\frac{\ep^{\frac{5}{2}}}{1+t}.\]
\end{prp}

We first prove Proposition \ref{estn}, admitting Proposition \ref{estl2N}.
\begin{proof}[Proof of Proposition \ref{estn}]
We have proved
$$\frac{d}{dt}E_I\leq C\frac{\sqrt{\ep}}{1+t}E_I +C\frac{\ep^\frac{5}{2}}{1+t},$$
therefore, if we note $E_I=G(1+t)^{C\sqrt{\ep}}$, we have
$$\frac{d}{dt}G\leq C\frac{\ep^\frac{5}{2}}{(1+t)^{1+C\sqrt{\ep}}}.$$
After integrating, we obtain
$$G(t)\leq G(0)+\ep^2-\frac{\ep^2}{(1+t)^{C\sqrt{\ep}}},$$
and hence
$$E_I\leq (E_I(0)+\ep^2)(1+t)^{C\sqrt{\ep}}
\leq(C_0^2\ep^2+\ep^2)(1+t)^{C\sqrt{\ep}}
.$$
Moreover, we have
$$\frac{d}{dt}E_I+A_I \leq C\frac{\sqrt{\ep}}{1+t}E_I+\frac{\ep^\frac{5}{2}}{1+t},$$
therefore if we multiply this inequality by $\frac{1}{(1+t)^\kappa}$ we obtain
$$\frac{d}{dt}\left(\frac{E_I}{(1+t)^\kappa}\right)+\frac{A_I}{(1+t)^\kappa}
\leq \frac{1}{(1+t)^\kappa}\left(\frac{d}{dt}E_I+A_I \right) \leq \frac{C\sqrt{\ep}}{(1+t)^{1+\kappa}}E_I +\frac{C\ep^\frac{5}{2}}{(1+t)^{1+\kappa}}\leq \frac{C\ep^\frac{5}{2}}{(1+t)^{1+\kappa-C\sqrt{\ep}}}.$$
Therefore, if $C\sqrt{\ep} <\kappa$, the right-hand side is integrable and so
$$\int \frac{1}{(1+t)^\kappa}A_I \lesssim \ep^2.$$
This concludes the proof of Proposition \ref{estn}.
\end{proof}
Proposition \ref{estl2N} is a direct consequence of the three following lemmas.

\begin{lm}
\label{lemmephi}We have the inequality, 
\[\frac{d}{dt}\|\alpha_2 w(q)^\frac{1}{2}\partial \wht{Z^N \phi}\|^2_{L^2}+\|\alpha_2 w'(q)^\frac{1}{2}\bar{\partial} \wht{Z^N \phi}\|^2_{L^2} \lesssim \frac{\ep}{1+t}E_N+\ep\|\alpha_2 w'_2(q)^\frac{1}{2}\bar{\partial} Z^N \wht g_4\|^2_{L^2}
+\frac{\ep^3}{1+t}
\]
where
$\widetilde{Z^N \phi}-Z^N \phi$ is composed of terms of the form $$\frac{\chi(q)q\partial \phi\partial^{N-1}_\theta b}{g_{UU}},$$
and we have
$$\|\alpha_2w_0^\frac{1}{2}\partial(\widetilde{Z^N \phi}-Z^N \phi)\|_{L^2}\lesssim \ep^2.$$
For $I<N$ we have
\[
\frac{d}{dt}\|w_0(q)^\frac{1}{2}\partial Z^I \phi\|^2_{L^2}+\|w_0'(q)^\frac{1}{2}\bar{\partial} Z^I \phi\|^2_{L^2}
\lesssim \frac{\ep}{1+t}E_I+\ep\|w'_2(q)^\frac{1}{2}\bar{\partial} Z^I \wht g_4\|^2_{L^2}.
\]
\end{lm}
\begin{lm}\label{lemmeh}We have the inequality, 
\begin{align*}
&\frac{d}{dt}\left(\frac{1}{\ep t}\|\alpha_2 w_3(q)^\frac{1}{2}\partial \wht{Z^N h}\|^2_{L^2}\right)
+\frac{1}{\ep (1+t)}\|\alpha_2 w_3'(q)^\frac{1}{2}\bar{\partial} \wht{Z^N h}\|_{L^2}^ 2\\
\lesssim &\frac{\sqrt{\ep}}{1+t}E_N
 +\sqrt{\ep} \|\alpha_2 w_2'^{\frac{1}{2}}(q)\bar{\partial}\wht Z^N g_4\|^2_{L^2}+\frac{\ep^\frac{5}{2}}{1+t},
\end{align*}
where
$\widetilde{Z^N h}-Z^N h$ is composed of terms of the form $$\frac{\chi(q)q\partial h\partial^{N-1}_\theta b}{g_{UU}},$$
and we have
$$\|\alpha_2w_0^\frac{1}{2}\partial(\widetilde{Z^N h}-Z^N h)\|_{L^2}\lesssim \ep^3\sqrt{1+t}.$$
We have a similar estimate for $k$
\begin{align*}
&\frac{d}{dt}\left(\frac{1}{\ep t}\|\alpha_2 w_3(q)^\frac{1}{2}\partial Z^N k\|^2_{L^2}\right)
+\frac{1}{\ep (1+t)}\|\alpha_2 w_3'(q)^\frac{1}{2}\bar{\partial} Z^N k\|_{L^2}^ 2\\
\lesssim &\frac{\sqrt{\ep}}{1+t}E_N
 +\sqrt{\ep} \|\alpha_2 w_2'^{\frac{1}{2}}(q)\bar{\partial}\wht Z^I g_4\|^2_{L^2}.
\end{align*}
Moreover for $I<N$
\begin{align*}
&\frac{d}{dt}\left(\frac{1}{\ep t}\| w_3(q)^\frac{1}{2}\partial Z^I h\|^2_{L^2}\right)
+\frac{1}{\ep (1+t)}\| w_3'(q)^\frac{1}{2}\bar{\partial}Z^I h\|_{L^2}^ 2\\
\lesssim &\frac{\sqrt{\ep}}{1+t}E_I
 +\sqrt{\ep} \| w_2'^{\frac{1}{2}}(q)\bar{\partial}\wht Z^I g_4\|^2_{L^2}+\frac{\ep^\frac{5}{2}}{1+t},
\end{align*}
\end{lm}

\begin{lm}\label{lemmeg}We have the estimate 
\begin{align*}
&\frac{d}{dt}\|\alpha_2 w_2(q)^\frac{1}{2}\partial \wht{Z^N \wht g_4}\|^2_{L^2}
+\|\alpha_2 w'_2(q)^\frac{1}{2}\partial \wht{Z^N \wht g_4}\|^2_{L^2}\\
&\lesssim \frac{\sqrt{\ep}}{1+t}E_N
+\sqrt{\ep}\frac{1}{\ep(1+t)}(\|\alpha_2 w_3'(q)^\frac{1}{2}\bar{\partial}Z^N h\|_{L^2}^ 2
+\|\alpha_2 w_3'(q)^\frac{1}{2}\bar{\partial} Z^N k\|_{L^2}^ 2)
\end{align*}
where
$\widetilde{Z^N \wht g_4}-Z^N \wht g_4$ is composed of terms of the form $$-\frac{\chi(q)q\partial \wht g_4\partial^{N-1}_\theta b}{g_{UU}}, \quad h Z^N g_{LL}dqds$$
and we have
$$\|\alpha_2w_0^\frac{1}{2}\partial(\widetilde{Z^N \wht g_4}-Z^N \wht g_4)\|_{L^2}\lesssim \ep^2 + \ep \|\alpha_2 w_2^\frac{1}{2}Z^N \wht g_{LL}\|_{L^2}.$$
For $I<N$, we have
\[
\frac{d}{dt}\|w_2(q)^\frac{1}{2}\partial Z^I \wht g_3\|^2_{L^2}
+\|w'_2(q)^\frac{1}{2}\partial Z^I \wht g_3\|^2_{L^2}
\lesssim \frac{\sqrt{\ep}}{1+t}E_I
+\sqrt{\ep}\frac{1}{\ep(1+t)}\|w_3'(q)^\frac{1}{2}\bar{\partial} Z^Ih\|_{L^2}^ 2.
\]
\end{lm}

We prove Proposition \ref{estl2N}.
\begin{proof}[Proof of Proposition \ref{estl2N}]Therefore, if we combine Lemmas \ref{lemmephi}, \ref{lemmeh} and \ref{lemmeg} we obtain
\[\frac{d}{dt}E_I + A_I \lesssim \frac{\sqrt{\ep}}{1+t}E_I +\sqrt{\ep}A_I + \frac{\ep^\frac{5}{2}}{1+t},\]
and therefore
\[\frac{d}{dt}E_I + (1-C\sqrt{\ep})A_I \lesssim \frac{\sqrt{\ep}}{1+t}E_I+ \frac{\ep^\frac{5}{2}}{1+t}.\]
If $\ep$ is small enough, we have $1-C\sqrt{\ep}\geq \frac{1}{2}$, which concludes the proof of Proposition \ref{estl2N}.
\end{proof}

It is sufficient to prove these three lemmas for $I=N$. For $I<N$ everything work in the same way. The weight modulator $\alpha_2$ is only needed to estimate a particular term for $I=N$ and is no longer needed for $I<N$.

\begin{proof}[Proof of Lemma \ref{lemmephi}]
We start with the estimates for $\phi$.
We use the weighted energy estimate for the equation
\begin{equation}
\label{equaphi}\Box_g Z^N \phi = \sum_{\substack{I+J\leq N\\J\leq N-1}}\left(Z^I g^{\alpha \beta}\right)\left(Z^J\partial_\alpha \partial_\beta \phi \right)
+\sum_{\substack{I+J\leq N\\J\leq N-1}}Z^IH_b^\rho Z^J\partial_\rho \phi.
\end{equation}
It yields
\begin{align*}
&\frac{d}{dt}\left(\|\alpha_2w_0(q)^\frac{1}{2}\partial Z^N \phi\|^2_{L^2}\right)
+\|\alpha_2w_0'(q)^\frac{1}{2}\bar{\partial} Z^N \phi\|^2_{L^2}\\
&\lesssim\left\|\alpha_2 w_0\Box_g Z^N \phi\right\|_{L^2}\|\alpha_2w_0(q)^\frac{1}{2}\partial Z^N \phi\|_{L^2}
+\frac{\ep}{1+t}\|\alpha_2w_0(q)^\frac{1}{2}\partial Z^N \phi\|^2_{L^2}.
\end{align*}
\paragraph{Estimate of the first term} Thanks to Remark \ref{termql}, it is sufficient to estimate
$$\left|\sum_{\substack{I+J\leq N\\J\leq N-1}}Z^I g_{LL}\partial^2_q Z^J\phi\right|\lesssim
\frac{1}{(1+|q|)}\sum_{\substack{I+J\leq N\\J\leq N-1}}|Z^I g_{LL}\partial_q Z^{J+1}\phi|.$$
If $I\leq \frac{N}{2}\leq N-15$, we can estimate thanks to \eqref{wc1}
$$|Z^I g_{LL}|\lesssim \frac{\ep (1+|q|)}{(1+t)^{\frac{3}{2}-\rho}}$$
so
\begin{equation}
\label{enephi1}\left\|\frac{\alpha_2w_0^\frac{1}{2}}{(1+|q|)}Z^I g_{LL}\partial_q Z^J\phi\right\|_{L^2}
\lesssim \frac{\ep}{(1+t)^{\frac{3}{2}-\rho}}\|\alpha_2w_0^\frac{1}{2}\partial_q Z^J\phi\|_{L^2}.
\end{equation}
If $J\leq \frac{N}{2}$, we can estimate
$$|\partial_q \phi|\lesssim \frac{\ep}{(1+|q|)^{\frac{3}{2}-4\rho}\sqrt{1+t}}.$$
Therefore, 
$$\left\|\frac{\alpha_2w_0^\frac{1}{2}}{(1+|q|)}Z^I g_{LL}\partial_q Z^J\phi\right\|_{L^2}
\lesssim \frac{\ep}{\sqrt{1+t}}\left\|\frac{\alpha_2w_0(q)^\frac{1}{2}}{(1+|q|)^{\frac{5}{2}-4\rho}}Z^I g_{LL}\right\|_{L^2}\lesssim
\frac{\ep}{\sqrt{1+t}}\left\|\frac{\alpha_2v(q)^\frac{1}{2}}{1+|q|}Z^I g_{LL}\right\|_{L^2},$$
where
$$\left\{\begin{array}{l}
v(q)=\frac{1}{(1+|q|)^{\frac{5}{2}-4\rho}}\; for \; q<0,\\
v(q)=\frac{w_0(q)}{(1+|q|)}=(1+|q|)^{1+2\delta}\; for \; q>0.
\\
\end{array}\right.$$
We do not keep all the decay in $q$ in the region $q>0$ in order to be in the range of application of the weighted Hardy inequality and we obtain
$$\left\|\frac{\alpha_2w_0^\frac{1}{2}}{(1+|q|)}Z^I g_{LL}\partial_q Z^J\phi\right\|_{L^2}\lesssim \frac{\ep}{\sqrt{1+t}}\|\alpha_2v(q)^\frac{1}{2}\partial_q Z^I g_{LL}\|_{L^2}.$$
We use Proposition \ref{estLL}, which gives
\begin{equation}
\label{cool}\partial_q Z^N g_{LL} \sim \bar{\partial}Z^N (\wht g_{L\ba L}+\wht g_{\q T \q T}).
\end{equation}
Consequently, thanks to Remark \ref{comp}, we have $\partial_q Z^N g_{LL} \sim \bar{\partial} Z^N \wht g_4$.
Moreover, we calculate
$$\left\{\begin{array}{l}
w_2'(q)=\frac{1+2\mu}{(1+|q|)^{2+2\mu}}\; for \; q<0,\\
w_2'(q)=(2+2\delta)(1+|q|)^{1+2\delta}\; for \; q>0.
\\
\end{array}\right.$$
Therefore, $v\lesssim w_2'$ and we obtain
\begin{equation}
\label{enephi2}\left\|\frac{\alpha_2w_0^\frac{1}{2}}{(1+|q|)}Z^I g_{LL}\partial_q Z^J\phi\right\|_{L^2}\|\alpha_2w_0(q)^\frac{1}{2}\partial Z^N \phi\|_{L^2}
\lesssim
\frac{\ep}{1+t}\|\alpha_2w_0(q)^\frac{1}{2}\partial Z^N \phi\|^2_{L^2}
+\ep\left\|\alpha_2w_2'(q)^\frac{1}{2}\bar{\partial} \wht Z^N g_4\right\|^2_{L^2}.
\end{equation}

\paragraph{Estimate of the second term} The second term contains only the crossed term, which occur only in the region $q>0$. Thanks to the discussion of Section \ref{cros}, it is sufficient to estimate \eqref{crossphi},  which gives a contribution of the form
$$Z^N\left(\partial (g_b)_{UU}\partial \phi\right).$$
For $I\leq N-2$ we have
$$|Z^I \partial (g_b)_{UU}|\lesssim \frac{\ep\ch_{q>0}}{r}$$
and consequently
\begin{equation}
\label{enephi3}\|\alpha_2w_0^\frac{1}{2}\partial Z^I  (g_b)_{UU}\partial Z^{N-I} \phi\|_{L^2}
 \lesssim \frac{\ep}{1+t}\|\alpha_2 w^\frac{1}{2}\partial Z^{N-I} \phi\|_{L^2}.
\end{equation}
In $\partial Z^I  (g_b)_{UU}\partial_\rho Z^{N-I} \phi$ with $I\geq N-2$, we have to note the presence of terms of the form
\begin{equation}
\label{problem}\frac{\chi(q)q\partial^{N+1}_\theta b(\theta)}{r^2}\partial \phi,
\end{equation}
which require a special treatment since $\partial^{N+1}_\theta b(\theta)$ does not belong to $L^2$.
%To treat these terms, we come back to the proof of the energy estimate, to see that what we want to estimate is
%\[\int_{\m R^2}w(q)\frac{\chi(q)q\partial^{N+1}_\theta b(\theta)}{r^2}\partial \phi\partial_t Z^N \phi rdrd\theta.\]
%We integrate by part with respect to $\theta$. The main term we have to deal with is of the form
%\[\int_{\m R^2}w(q)\frac{\chi(q)q\partial^{N}_\theta b(\theta)}{r^2}\partial\phi\partial_t \partial_\theta Z^N\phi rdrd\theta.\]
%We now remove the $\partial_t$ from the integral. We obtain
%\begin{equation*}
%\label{dtint}
%\partial_t \int_{\m R^2}w(q)\frac{\chi(q)q\partial^{N}_\theta b(\theta)}{r}\partial\phi\bar{\partial}Z^N \phi %rdrd\theta
%\end{equation*}
%We note this quantity $\frac{d}{dt}A$.
%and the reminder is of the form
%\begin{equation*}
%Rem=\int_{\m R^2}w(q)\frac{\chi(q)\partial^{N}_\theta b(\theta)}{r}\partial\phi\bar{\partial}Z^N \phi rdrd\theta
%\end{equation*}
%We estimate first
%\begin{align*}
%|Rem|&\lesssim \left\| w(q)^\frac{1}{2}\sqrt{1+|q|}\frac{\chi(q)\partial^{N}_\theta b(\theta)}{r}\partial\phi\right\|_{L^2}
%\left\|\frac{w(q)^\frac{1}{2}}{\sqrt{1+|q|}}\chi(q)\bar{\partial}Z^N \phi\right\|\\
%&\lesssim \left\| \frac{\ep}{(1+s)^{\frac{3}{2}-\rho}}\partial^N_\theta b\right\|_{L^2}\left\|w'(q)^\frac{1}{2}\bar{\partial}Z^N \phi\right\|\\
%&\lesssim \frac{\ep^2}{(1+t)^{1-\rho}}\left\|w'(q)^\frac{1}{2}\bar{\partial}Z^N \phi\right\|\\
%&\lesssim 
%\end{align*}
To deal with these terms we write
$$\left|\Box_g \left(\frac{\chi(q)q\partial^{N-1}_\theta b}{g_{UU}} \partial \phi \right)-\frac{\chi(q)q\partial^{N+1}_\theta b(\theta)}{r^2}\partial \phi\right|
\lesssim \frac{\chi(q)}{1+s}\left(\sum_{I\leq 2}|\partial Z^I \phi|\right)\left(|\partial^N_\theta b|+|\partial^{N-1}_\theta b|\right)+s.t.$$
We can estimate,  thanks to the estimate \eqref{ks4} for $\partial \phi$,
\begin{equation}
\label{termphi}\left\|w_0^\frac{1}{2}\partial\left(\chi(q)q\partial \phi\partial^{N-1}_\theta b\right)\right\|_{L^2}
\lesssim \left\|\frac{\ep}{\sqrt{1+s}(1+|q|)^{\frac{1}{2}+2\sigma-\sigma}}\partial^{N}_\theta b\right\|_{L^2}
\lesssim \ep^2.
\end{equation}
Therefore, we may perform the energy
estimate for $\widetilde{Z^N \phi}=Z^N \phi-\frac{\chi(q)q\partial \phi\partial^{N-1}_\theta b}{g_{UU}}$ instead of $Z^N \phi$.
We are reduced to estimate
\begin{equation}\label{enephi4}\begin{split}
\left\|\alpha_2w_0^\frac{1}{2}\frac{\chi(q)}{1+s}\left(\sum_{I\leq 2}|\partial Z^I \phi|\right)\left(|\partial^N_\theta b|+|\partial^{N-1}_\theta b|\right)\right\|_{L^2}
&\lesssim \left\|\frac{\ep}{(1+s)^\frac{3}{2}(1+|q|)^{\frac{1}{2}+\sigma}}\left(|\partial^N_\theta b|+|\partial^{N-1}_\theta b|\right)\right\|_{L^2}\\
&\lesssim \frac{\ep^3}{1+t}.\end{split}
\end{equation}
The other terms in $\partial Z^I  (g_b)_{UU}\partial Z^{N-I} \phi$ with $I\geq N-2$, give contributions similar to \eqref{enephi4}.
\begin{rk}
We introduce the weight modulator $\alpha_2$ to deal with the term
\eqref{problem} which is only present for $I=N$. It is no longer needed for $I<N$. To see this, let us estimate
$\frac{\chi(q)q\partial^{N}_\theta b(\theta)}{r^2}\partial \phi$ which is the analogue of \eqref{problem} for $I=N-1$.
$$\left\|w_0^\frac{1}{2}\frac{\chi(q)q\partial^{N}_\theta b(\theta)}{r^2}\partial \phi\right\|_{L^2}
\lesssim \sum_{I\leq 2}\|w_0^\frac{1}{2}\partial Z^I \phi\|_{L^2} \left\|\frac{\chi(q) q\partial^N_\theta b}{r^2\sqrt{1+s}\sqrt{1+|q|}}\right\|_{L^2}
\lesssim \frac{1}{1+t}\sqrt{E_2} \|\partial^N_\theta b\|_{L^2(\m S^1)},$$
where we have used the weighted Klainerman-Sobolev inequality
$$|w_0^\frac{1}{2}\partial \phi|\lesssim \frac{1}{\sqrt{1+s}\sqrt{1+|q|}}\sum_{I\leq 2}\|w_0^\frac{1}{2}\partial Z^I\phi\|_{L^2},$$
and consequently
\begin{equation}
\label{noproblem}\left\|w_0^\frac{1}{2}\frac{\chi(q)q\partial^{N}_\theta b(\theta)}{r^2}\partial \phi\right\|_{L^2}\lesssim \frac{\ep^2}{1+t}\sqrt{E_2}.
\end{equation}
\end{rk}

Thanks to \eqref{enephi1}, \eqref{enephi2}, \eqref{enephi3}, \eqref{enephi4} we obtain
\begin{equation}
\label{estphi}
\frac{d}{dt}\left(\|\alpha_2 w_0(q)^\frac{1}{2}\partial \widetilde{Z^N \phi}\|^2_{L^2}\right)
+\|\alpha w'(q)^\frac{1}{2}\bar{\partial} \widetilde{Z^N \phi}\|^2_{L^2}\lesssim
\frac{\ep}{1+t}E_N
+\ep\|\alpha_2w_2'(q)^\frac{1}{2}\bar{\partial} \wht g_4\|^2_{L^2}+\frac{\ep^3}{1+t},
\end{equation}
which, with the estimate \eqref{termphi} for $\wht{Z^N \phi}-Z^N \phi$ concludes the proof of Lemma \ref{lemmephi}.
\end{proof}
\begin{proof}[Proof of Lemma \ref{lemmeh}]
We now estimate $h$. The equation for $Z^N h$ writes
\begin{equation}\label{equah}
\begin{split}
\Box_g Z^I h =& \sum_{\substack{I+J\leq N\\J\leq N-1}}\left(Z^I g^{\alpha \beta}\right)\left(Z^J\partial_\alpha \partial_\beta h\right)+\sum_{\substack{I+J\leq N\\J\leq N-1}}Z^IH^\rho Z^J\partial_\rho h\\
&+ Z^I((\partial_q \phi)^2+(R_b)_{qq} + Q_{\ba L \ba L}(h,\wht g)).
\end{split}
\end{equation}
\paragraph{Estimate of 
the first term}
Following Remark \ref{termql}, it is sufficient to estimate
$Z^I g_{LL}\partial_q^2Z^J h$. 
For $I\leq \frac{N}{2}$, similarly than \eqref{enephi1} we have
\begin{equation}
\label{eneh1}\left\|\frac{\alpha_2w_3(q)^\frac{1}{2}}{(1+|q|)}Z^I g_{LL}\partial_q Z^J h\right\|_{L^2}
\lesssim \frac{\ep}{(1+t)^{\frac{3}{2}-\rho}}\|\alpha_2 w_3(q)^\frac{1}{2}\partial_q Z^J h\|_{L^2}.
\end{equation}
%where we have used in the region $q>0$ the estimate
%$$|Z^I g_{LL}|\lesssim \frac{\ep(1+|q|)^{\frac{1}{2}-\delta}}{t^{\frac{3}{2}-\rho}}$$
For $J\leq \frac{N}{2}$, we have the estimate, thanks to \eqref{est1},
$$|\partial_q Z^J h|\lesssim \frac{\ep}{(1+|q|)^{\frac{3}{2}-\rho}},$$
so
$$ \left\|\frac{\alpha_2w_3(q)^\frac{1}{2}}{(1+|q|)}Z^I g_{LL}\partial_q Z^J h\right\|_{L^2}
\lesssim \left\|\frac{\alpha_2}{1+|q|}\left(\frac{w_3(q)}{(1+|q|)^{3-2\rho}}\right)^\frac{1}{2}Z^I g_{LL}\right\|_{L^2}.$$
We have
$$\frac{w_3(q)}{(1+|q|)^{3-2\rho}}\leq\left\{\begin{array}{l}
\frac{1}{(1+|q|)^{3-2\rho}}\quad for \; q<0,\\
(1+|q|)^{2\delta-2\rho}\leq(1+|q|)^{1+2\delta}\quad for \; q>0.
\\
\end{array}\right.$$
This yields
$$\frac{w_3(q)}{(1+|q|)^{3-2\rho}}\lesssim w_2'(q).$$
Therefore the weighted Hardy inequality and the wave coordinate condition
give, similarly than for \eqref{enephi2},
\begin{equation}
\label{eneh2} \left\|\frac{\alpha_2w_3(q)^\frac{1}{2}}{(1+|q|)}Z^I g_{LL}\partial_q Z^J h\right\|_{L^2}
\lesssim \ep\|\alpha_2w_2'(q)^\frac{1}{2}\partial_q Z^I g_{LL}\|_{L^2}\lesssim \ep \|\alpha_2w_2'^{\frac{1}{2}}(q)\bar{\partial}\wht g_4\|_{L^2}.
\end{equation}
\paragraph{Estimate of the second term}
The second term contains crossed terms, which can be studied exactly in the same way than for $\phi$. Similarly than \eqref{enephi3}, we have for $I\leq N-2$
\begin{equation}
\label{eneh7}\|\alpha_2w_3^\frac{1}{2}\partial Z^I  (g_b)_{UU}\partial Z^{N-I} h\|_{L^2}
 \lesssim \frac{\ep}{1+t}\|\alpha_2 w^\frac{1}{2}\partial Z^{N-I} h\|_{L^2}.
\end{equation}
Like for $\phi$ the following term require a special treatment.
\begin{equation}
\label{problemh}\frac{\chi(q)q\partial^{N+1}_\theta b(\theta)}{r^2}\partial h.
\end{equation}
We have
$$\left|\Box_g \left(\frac{\chi(q)q\partial \phi\partial^{N-1}_\theta b}{g_{UU}} \partial h \right)-\frac{\chi(q)q\partial^{N+1}_\theta b(\theta)}{r^2}\partial h\right|
\lesssim \frac{\chi(q)}{1+s}\left(\sum_{I\leq 2}|\partial Z^I h|\right)\left(|\partial^N_\theta b|+|\partial^{N-1}_\theta b|\right)+s.t.$$
We can estimate,  thanks to the estimate \eqref{ks6} for $\partial h$,
\begin{equation}
\label{termh}\left\|\alpha_2w_3^\frac{1}{2}\partial\left(\chi(q)q\partial h\partial^{N-1}_\theta b\right)\right\|_{L^2}
\lesssim \left\|\frac{\ep}{(1+|q|)^{1+2\sigma-\sigma}}\partial^{N}_\theta b\right\|_{L^2}
\lesssim \ep^2\sqrt{1+t}.
\end{equation}
Therefore, we may perform the energy
estimate for $\widetilde{Z^N h}=Z^N h-\frac{\chi(q)q\partial h\partial^{N-1}_\theta b}{g_{UU}}$ instead of $Z^N h$.
We are reduced to estimate
\begin{equation}\label{eneh8}\begin{split}
\left\|\alpha_2w_3^\frac{1}{2}\frac{\chi(q)}{1+s}\left(\sum_{I\leq 2}|\partial Z^I h|\right)\left(|\partial^N_\theta b|+|\partial^{N-1}_\theta b|\right)\right\|_{L^2}
&\lesssim \left\|\frac{\ep}{(1+s)(1+|q|)^{1+\sigma}}\left(|\partial^N_\theta b|+|\partial^{N-1}_\theta b|\right)\right\|_{L^2}\\
&\lesssim \frac{\ep^3}{\sqrt{1+t}}.\end{split}
\end{equation}
The other terms in $\partial Z^I  (g_b)_{UU}\partial Z^{N-I} h$ with $I\geq N-2$, give contributions similar to \eqref{eneh8}.

\paragraph{Estimate of $Z^N (\partial_q \phi)^2$} We have
$$\|\alpha_2w_3(q)^\frac{1}{2}Z^N \left((\partial_q \phi)^2\right)\|_{L^2}
\lesssim \sum_{I+J\leq N}
\|\alpha_2w_3(q)^\frac{1}{2}\partial_q Z^I \phi \partial_q Z^J \phi\|_{L^2}.$$
We can assume $I\leq \frac{N}{2}$ and estimate thanks to \eqref{ks1}
$$|\partial_q Z^I \phi|\lesssim \frac{\ep}{\sqrt{1+|q|}\sqrt{1+t}}.$$
Then, since 
$$\frac{w_3^\frac{1}{2}}{\sqrt{1+|q|}}\leq w_0^\frac{1}{2},$$
we obtain
\begin{equation}
\label{eneh3}\|\alpha_2w_3(q)^\frac{1}{2}Z^N (\partial_q \phi)^2\|_{L^2}
\lesssim \frac{\ep}{\sqrt{1+t}}\sum_{I\leq N} \|\alpha_2w_0(q)^\frac{1}{2} \partial_q Z^J \phi\|_{L^2}.
\end{equation}

\paragraph{Estimate of $Z^N (R_b)_{qq}$}
Thanks to \eqref{rqq}, the main contribution in $(R_b)_{qq}$ is
$$\frac{\partial_q ^2(q\chi(q))b(\theta)}{r},$$
which is supported in $1\leq q\leq 2$.
We estimate
\begin{equation}
\label{eneh5}\left\|\alpha_2w_3(q)^\frac{1}{2}Z^N\left(\frac{b(\theta)\partial^2_q(q\chi(q))}{r}\right)\right\|_{L^2}\lesssim \frac{1}{\sqrt{1+t}} \sum_{I\leq N}\|\partial^I_\theta b\|_{L^2(\m S^1)}\lesssim \frac{\ep^2}{\sqrt{1+t}}.
\end{equation}
\paragraph{Estimate of $Z^N Q_{\ba L \ba L}(h,\wht g)$} We recall from
\eqref{qll} that
$$Q_{\ba L \ba L}(h,\wht g)=
\partial_{\ba L} g_{LL}\partial_{\ba L}h+\partial_{\ba L} g_{L \ba L}\partial_L h+
\partial_{\ba L} (g_b)_{UU}\partial_{\ba L} g_{\ba L L}.$$
The terms $Z^N(\partial_{\ba L} g_{LL}\partial_{\ba L}h)$ and
$Z^N(\partial_{\ba L} g_{L \ba L}\partial_L h)$ may be treated in a similar way than the quasilinear term, giving contributions similar to \eqref{eneh1} and \eqref{eneh2}.
The term $Z^N(\partial_{\ba L} (g_b)_{UU}\partial_{\ba L} g_{\ba L L})$ is a crossed term, hence it is supported only in the region $q>0$.
It is sufficient to estimate $\partial_{\ba L} (g_b)_{UU}\partial_{\ba L} Z^Ng_{\ba L L}$.
We have
$$|\partial_q (g_b)_{UU}|\lesssim  \frac{\ep \mathbbm{1}_{q>0}}{r} \lesssim \frac{\ep \mathbbm{1}_{q>0}}{\sqrt{1+t}\sqrt{1+|q|}}.$$
so we can estimate
$$\|\alpha_2 w^\frac{1}{2}_3\partial_q (g_b)_{UU}\partial Z^Ng_{L\ba L}\|_{L^2}\lesssim \frac{\ep}{\sqrt{1+t}}\left\|
\frac{\alpha_2 w_3^\frac{1}{2}\mathbbm{1}_{q>0}}{\sqrt{1+|q|}}\partial Z^Ng_{L\ba L} \right\|_{L^2}
\lesssim \frac{\ep}{\sqrt{1+t}}\left\|
\alpha_2 w_2^\frac{1}{2}\partial Z^Ng_{L\ba L} \right\|_{L^2},$$
and consequently, since $\wht g_{L \ba L}=(\wht g_4)_{L\ba L}$ we have
\begin{equation}
\label{eneh6}
\|\alpha_2w^\frac{1}{2}_3\partial_q (g_b)_{UU}\partial Z^Ng_{L\ba L}\|_{L^2}\lesssim \frac{\ep}{\sqrt{1+t}}\left\|
\alpha_2 w_2^\frac{1}{2}\partial Z^N\wht g_4\right\|_{L^2}.
\end{equation}

In view of \eqref{eneh1}, \eqref{eneh2}, \eqref{eneh7}, \eqref{eneh8}, \eqref{eneh3}, \eqref{eneh5}, \eqref{eneh6},
the energy inequality yields
\begin{align*}
&\frac{d}{dt}\|\alpha_2w_3(q)^\frac{1}{2}\partial\wht{Z^N h}\|^2_{L^2}+
\|\alpha_2 w_3'(q)^\frac{1}{2}\bar{\partial} \wht{Z^N h}\|_{L^2}^ 2\\
&\lesssim \Bigg(
\frac{\ep}{1+t}\|\alpha_2w_3^\frac{1}{2}\partial Z^N h\|_{L^2}
+\ep \|\alpha_2 w_2'(q)^{\frac{1}{2}}\bar{\partial} Z^N \wht g_4\|_{L^2}\\
&+\frac{\ep}{\sqrt{1+t}} \left(\|\alpha_2w_0^\frac{1}{2} \partial_q Z^J \phi\|_{L^2}+\|\alpha_2w_2^\frac{1}{2} \partial_q Z^J \wht g_4\|_{L^2}\right) +\frac{\ep^2}{\sqrt{1+t}}\Bigg)\|\alpha_2w_3^\frac{1}{2}\partial_q Z^N h\|_{L^2}+s.t.
\end{align*}
We note that
$$\frac{d}{dt}\left(\frac{1}{\ep (1+t)}\|\alpha_2w_3^\frac{1}{2}\partial Z^N h\|^2_{L^2}\right)\leq
\frac{1}{\ep (1+t)}\frac{d}{dt}\|\alpha_2w_3^\frac{1}{2}\partial Z^Nh\|^2_{L^2}$$
and we calculate
\begin{align*}
\frac{\ep}{\ep (1+t)^2}\|\alpha_2w_3^\frac{1}{2}\partial_q Z^N h\|^2_{L^2}&\leq
\frac{\ep}{1+t}\frac{ \|\alpha_2w_3^\frac{1}{2}\partial_q Z^J h\|^2_{L^2}}{\ep (1+t)},\\
\frac{\ep}{\ep (1+t)} \|\alpha_2w_2'^{\frac{1}{2}}(q)\bar{\partial} Z^N \wht g_4\|_{L^2}
\|\alpha_2w_3^\frac{1}{2}\partial_q Z^N h\|_{L^2}&\leq \sqrt{\ep}\|\alpha_2w_2'^{\frac{1}{2}}(q)\bar{\partial}\wht Z^I g_4\|^2_{L^2}+\frac{1}{\sqrt{\ep} (1+t)^2}\|\alpha_2w_3^\frac{1}{2}\partial_q Z^N h\|^2_{L^2},\\
\frac{\ep}{\ep (1+t)^{\frac{3}{2}}} \|\alpha_2w_0^\frac{1}{2} \partial_q Z^J \phi\|_{L^2}\|\alpha_2w_3^\frac{1}{2}\partial_q Z^N h\|_{L^2}
&\leq \frac{\sqrt{\ep}}{1+t}\|\alpha_2w^\frac{1}{2} \partial_q Z^J \phi\|^2_{L^2}
+\frac{1}{\sqrt{\ep}(1+t)^2}\|\alpha_2w_3^\frac{1}{2}\partial_q Z^N h\|^2_{L^2}.
\end{align*}
This yields
\begin{equation}
\label{esthN}\begin{split}
&\frac{d}{dt}\left(\frac{1}{\ep (1+t)}\|\alpha_2w_3(q)^\frac{1}{2}\partial \wht{Z^Nh}\|^2_{L^2}\right)
+\frac{1}{\ep (1+t)}\|\alpha_2w_3'(q)^\frac{1}{2}\bar{\partial}\wht{Z^N h}\|_{L^2}^ 2\\
&\lesssim \frac{\sqrt{\ep}}{1+t}E_N +\sqrt{\ep} \|\alpha_2w_2'^{\frac{1}{2}}(q)\bar{\partial}\wht Z^I g_4\|^2_{L^2}+\frac{\ep^\frac{5}{2}}{1+t}.
\end{split}
\end{equation}
The estimate for $Z^Nk$ is totally similar. This, with the estimate \eqref{termh} concludes the proof of Lemma \ref{lemmeh}.
\end{proof}
\begin{proof}[Proof of Lemma \ref{lemmeg}]
We now go to the estimate for $Z^N \wht{g_4}$. 
We write $\Box_g Z^N \wht g_4= f_{\mu \nu}$.
The energy estimate writes
\begin{align*}
\frac{d}{dt}\left(\|\alpha_2w_2(q)^\frac{1}{2}\partial Z^N \wht g_4\|^2_{L^2}\right)
+\|\alpha_2w_2'(q)^\frac{1}{2}\bar{\partial} Z^N \wht g_4\|^2_{L^2}
\lesssim&\|\alpha_2w_2(q)^\frac{1}{2}f_{\mu \nu}\|_{L^2} \|\alpha_2w_2(q)^\frac{1}{2}\partial Z^N \wht g_4\|_{L^2}\\
&+\frac{\ep}{(1+t)}\|\alpha_2w_2(q)^\frac{1}{2}\partial Z^N \wht g_4\|^2_{L^2}
\end{align*}
We recall that the terms in $f_{\mu \nu }$ consist of
\begin{itemize}
\item the quasilinear terms, 
\item the terms coming from the non commutation of the wave operator with the null decomposition: it will be sufficient to study the term $\Up(\frac{r}{t})\frac{1}{r^2}\partial_\theta Z^N h$,
\item the semi-linear terms: it is sufficient to study 
the term $Z^N(g^{L \ba L}\partial_L g_{LL}\partial_{\ba L}h)$.
We note that thanks to our decomposition, the term 
$Z^N(\partial_U g_{LL}\partial_{\ba L}h)$ is absent,
\item The crossed terms: their analysis is  similar to the one for $\phi$.

\end{itemize} 
\paragraph{The quasilinear terms}
We consider
$$\left|\sum_{\substack{I+J\leq N\\J\leq N-1}}Z^I g_{LL}\partial^2_q Z^J\wht g_4\right|\lesssim
\frac{1}{(1+|q|)}\sum_{\substack{I+J\leq N\\J\leq N-1}}Z^I g_{LL}\partial_q Z^{J+1}\wht g_4|.$$
If $I\leq \frac{N}{2}$, we can estimate
$$|Z^I g_{LL}|\lesssim \frac{\ep (1+|q|)}{(1+t)^{\frac{3}{2}-\rho}},$$
so
\begin{equation}
\label{eneg1}\left\|\frac{\alpha_2w_2^\frac{1}{2}}{(1+|q|)}Z^I g_{LL}\partial_q Z^J\wht g_4\right\|_{L^2}
\lesssim \frac{\ep}{(1+t)^{\frac{3}{2}-\rho}}\|\alpha_2w_2^\frac{1}{2}\partial_q Z^J\wht g_4\|_{L^2}.
\end{equation}
If $J\leq \frac{N}{2}$, we can estimate, thanks to Proposition \ref{estks} and since the difference between $\wht g_4$ and $\wht g_3$ is contained in $\wht g_{\ba L U}$, which is equal to $(\wht g_3)_{\ba L U}$,
\begin{equation}
\label{ksg4}\partial_q  Z^J \wht g_4|\lesssim \frac{\ep}{\sqrt{1+|q|}\sqrt{1+t}}.
\end{equation}
Therefore, if we apply Hardy inequality we obtain
$$\left\|\frac{\alpha_2w_2^\frac{1}{2}}{(1+|q|)}Z^I g_{LL}\partial_q Z^J\wht g_4\right\|_{L^2}
\lesssim 
\frac{\ep}{\sqrt{1+t}}\left\|\frac{\alpha_2w_2^\frac{1}{2}}{(1+|q|)^{\frac{1}{2}}}\partial_q Z^I g_{LL}\right\|_{L^2}.$$
Thanks to \eqref{cool} and the fact that $\frac{w_2^\frac{1}{2}}{(1+|q|)^{\frac{1}{2}}}\lesssim w_2'(q)^\frac{1}{2}$ we obtain
\begin{equation}
\label{eneg2}\left\|\frac{\alpha_2w_2^\frac{1}{2}}{(1+|q|)}Z^I g_{LL}\partial_q Z^J\wht g_4\right\|_{L^2}
\|\alpha_2w_2(q)^\frac{1}{2}\partial Z^N \wht g_4\|_{L^2}
\lesssim
\frac{\ep}{1+t}\|w_2(q)^\frac{1}{2}\partial Z^N \wht g_4\|^2_{L^2}
+\ep\|{\alpha_2w'_2}^\frac{1}{2}\bar{\partial} \wht g_4\|^2_{L^2}.
\end{equation}

\paragraph{The term coming from the non commutation of the wave operator with the null decomposition} We note that $\frac{\partial_\theta h}{r}$ is a tangential derivative $\bar{\partial}h$. Therefore
$$\left\|\alpha_2w^\frac{1}{2}_2(q)\Up\left(\frac{r}{t}\right)\frac{1}{r^2}\partial_\theta Z^N h\right\|_{L^2} \lesssim
\frac{1}{1+t}\|\alpha_2 w_2^\frac{1}{2}\bar{\partial}Z^N h\|_{L^2}
$$
We calculate
$$\left\{\begin{array}{l}
w_3'(q)=2\mu\frac{1}{(1+|q|)^{1+2\mu}}\; for \; q<0,\\
w_3'(q)=(3+2\delta)(1+|q|)^{2+2\delta}\; for \; q>0.
\\
\end{array}\right.$$
Therefore $w_2\lesssim w_3'$ and we obtain
$$\left\|w^\frac{1}{2}_2(q)\Up\left(\frac{r}{t}\right)\frac{1}{r^2}\partial_\theta Z^N h\right\|_{L^2} \lesssim \frac{1}{1+t}\|\alpha_2w'_3(q)^\frac{1}{2}\bar{\partial}Z^N h\|_{L^2}.$$
This yields
\begin{equation}\label{eneg3}
\begin{split}
&\left\|\alpha_2w^\frac{1}{2}_2(q)\Up\left(\frac{r}{t}\right)\frac{1}{r^2}\partial_\theta Z^N h\right\|_{L^2}\|\alpha_2w_2(q)\partial Z^N \wht g_4\|_{L^2}\\
&\lesssim \frac{1}{\sqrt{\ep}(1+t)}\|\alpha_2w_3'(q)^\frac{1}{2}\bar{\partial }Z^Nh\|^2_{L^2}+\frac{\sqrt{\ep}}{1+t}\|\alpha_2w_2(q)^\frac{1}{2}\partial Z^N \wht g_4\|^2_{L^2}.
\end{split}
\end{equation}
\paragraph{The semi-linear terms}
We now estimate $Z^N(g^{L \ba L}\partial_L g_{LL}\partial_{\ba L}h)$. We first estimate
$$\|w_2(q)^\frac{1}{2}\bar{\partial}Z^{I_1} g_{LL}\partial Z^{I_2} h\|_{L^2}$$
 for $I_1+I_2\leq N$ and $I_1\leq N-1$.
If $I_1\leq \frac{N}{2}$ we estimate
$$|\bar{\partial}Z^{I_1} g_{LL}|\lesssim \frac{1}{1+s}
| Z^{I_1+1}g_{LL}|\lesssim \frac{\ep (1+|q|)}{(1+s)^{\frac{5}{2}-\rho}}\lesssim \frac{(1+|q|)^\rho}{(1+t)^{\frac{3}{2}}}.$$
Therefore
$$\|\alpha_2 w_2(q)^\frac{1}{2}\bar{\partial}Z^{I_1} g_{LL}\partial Z^{I_2} h\|_{L^2}\lesssim
\frac{\ep}{(1+t)^\frac{3}{2}}\|\alpha_2 w_2(q)^\frac{1}{2}(1+|q|)^\rho\partial Z^{I_2}h\|_{L^2}\lesssim 
\frac{\ep}{(1+t)^\frac{3}{2}}\|\alpha_2 w_3(q)^\frac{1}{2}\partial Z^{I_2}h\|_{L^2}.$$
If $I_2\leq \frac{N}{2}$ we estimate, thanks to \eqref{est1}
$$|\partial Z^{I_2}h|\leq \frac{\ep}{(1+|q|)^{\frac{3}{2}-\rho}},$$
therefore
\begin{align*}
\|\alpha_2w_2(q)^\frac{1}{2}\bar{\partial}Z^{I_1} g_{LL}\partial Z^{I_2} h\|_{L^2}&\lesssim
\ep\left\|\frac{\alpha_2w_2^\frac{1}{2}}{(1+|q|)^{\frac{3}{2}-\rho}}\bar{\partial}Z^{I_1} g_{LL}\right\|_{L^2}\\
&\lesssim \frac{\ep}{1+t} \left\|\frac{\alpha_2w_2^\frac{1}{2}}{(1+|q|)^{\frac{3}{2}-\rho}} Z^{ I_1+1}g_{LL}\right\|_{L^2}\\
&\lesssim \frac{\ep}{1+t}\left\|\frac{\alpha_2w_2^\frac{1}{2}}{(1+|q|)^{\frac{1}{2}-\rho}}
\partial Z^{ I_1+1}g_{LL}\right\|_{L^2}
\end{align*}
where in the third inequality we have used the weighted Hardy inequality. Consequently
\begin{equation}
\label{eneg4}
\|\alpha_2w_2(q)^\frac{1}{2}\bar{\partial}Z^{I_1} g_{LL}\partial Z^{I_2} h\|_{L^2}\lesssim  \frac{\ep}{1+t}\|\alpha_2w_2(q)^\frac{1}{2}\partial Z^{ I_1+1}\wht g_4\|_{L^2}.
\end{equation}

It is not possible to do the same reasoning for $I_1=N$. To treat the term $g^{L\ba L}\partial_L Z^N g_{LL}\partial_{\ba L}h$, which appears only in
$P_{\ba L L}$ we will write
$$\Box_g (h Z^N g_{LL})=D^{\alpha}D_\alpha (hZ^Ng_{LL})
=h\Box_g Z^N g_{LL} + Z^N(g_{LL})\Box_g h+g^{\alpha \beta}
\partial_\alpha h \partial_\beta Z^N g_{LL}.$$
We estimate
\begin{equation}
\label{termg}\|w_2(q)^\frac{1}{2}\partial(h Z^N g_{LL})\|_{L^2}
\lesssim \ep \|w_2(q)^\frac{1}{2}\partial Z^N g_{LL}\|_{L^2},
\end{equation}
therefore, we can 
perform the energy estimate for $\wht{Z^N \wht g_4}=Z^N  \wht g_4-hZ^N g_{LL}-\frac{\chi(q)q\partial \wht g_4\partial^{N-1}_\theta b}{g_{UU}}$ instead of $Z^N \wht g_4$, where the last term is here to deal with the troublesome crossed term which is the equivalent of \eqref{problem}.
We have now to estimate $h\Box_g Z^N g_{LL} + Z^N(g_{LL})\Box_g h+\partial Z^N g_{LL}\bar{\partial} h$.
We estimate first
\begin{equation}
\label{eneg5}\|\alpha_2 w_2(q)^\frac{1}{2}\partial Z^N g_{LL}\bar{\partial} h\|_{L^2}\lesssim \frac{\ep}{1+t}\|w_2(q)^\frac{1}{2}\partial Z^N g_{LL}\|_{L^2}.
\end{equation}
We have $\Box_g h = -2(\partial_q \phi)^2+\partial_q h \partial_q g_{LL}+...$ therefore
$$|\Box_g h|\lesssim \frac{\ep^2}{(1+t)(1+|q|)}$$
and
\begin{equation}\label{eneg6}\|\alpha_2 w_2(q)^\frac{1}{2}Z^N g_{LL}\Box_g h\|_{L^2}
\lesssim \frac{\ep}{1+t}\left\|\frac{\alpha_2w_2(q)^\frac{1}{2}}{(1+|q|)}Z^N g_{LL}\right\|\lesssim \frac{\ep}{1+t}\|w_2(q)^\frac{1}{2}\partial Z^N \wht g_4\|_{L^2}.
\end{equation}
To estimate the last term, we have to note that since $g^{L\ba L}\partial_L Z^N g_{LL}\partial_{\ba L}h$ appears only in $P_{L \ba L}$, it is absent from $\Box_g Z^N g_{LL}$. However, we have terms appearing from the non commutation of the wave operator with the null decomposition. They are of the form
$ \frac{1}{r}h\bar{\partial} Z^N g_{LL}$. We estimate
\begin{equation}
\label{eneg7}\left\|\alpha_2w_2(q)^\frac{1}{2}\frac{1}{r}h\bar{\partial} Z^N g_{LL}\right\|_{L^2}\lesssim \frac{\ep}{1+t}\|\alpha_2w_2(q)^\frac{1}{2}\partial Z^N \wht g_4\|_{L^2}
\end{equation}
The other terms in $\Box_g Z^N g_{LL}$ have already been estimated. 

\begin{rk}This reasoning would not have been possible to treat terms of the form $\partial_U g_{LL}\partial_q h$. It is why we have introduced the function $k$, which is allowed to decay less.
\end{rk}
Thanks to \eqref{eneg1}, \eqref{eneg2}, \eqref{eneg3}, \eqref{eneg4}, \eqref{eneg5}, \eqref{eneg6}, \eqref{eneg7} the energy estimate yields

\begin{equation}
\label{estg4}
\begin{split}
&\frac{d}{dt}\left(\|\alpha_2w_2(q)^\frac{1}{2}\partial \wht{Z^N \wht g_4}\|^2_{L^2}\right)
+\|\alpha_2w_2'(q)^\frac{1}{2}\bar{\partial} \wht{Z^N \wht g_4}\|^2_{L^2}\\
\lesssim& \frac{\sqrt{\ep}}{1+t}\left(\|\alpha_2w_2(q)^\frac{1}{2}\partial Z^N \wht g_4\|^2_{L^2}+\frac{1}{\ep (1+t)}\|\alpha_2w_3^\frac{1}{2}\partial h\|^2_{L^2} \right)\\
&+\sqrt{\ep}\left(\frac{1}{\ep (1+t)}\|\alpha_2w'_3(q)^\frac{1}{2}\bar{\partial}Z^Nh\|_{L^2}^2
+\|\alpha_2w'_2(q)^\frac{1}{2}\bar{\partial} Z^N\wht g_4\|_{L^2}^2\right)+s.t.
\end{split}
\end{equation}
This, together with the estimates \eqref{termg} concludes the proof of Lemma \ref{lemmeg}.
\end{proof}

\subsection{Estimates for $I\leq N-2$}\label{subsecl2}
\begin{prp}\label{l2n1} Let $I\leq N-2$. We have the estimates
\begin{align*}
\|\alpha w_0(q)^\frac{1}{2}\partial Z^I \phi\|_{L^2}&\leq C_0\ep+C\ep^\frac{3}{2},\\
\|\alpha w(q)^\frac{1}{2}\partial Z^I h\|_{L^2}&\leq C\ep^\frac{3}{2} (1+t),\\
\|\alpha w_2(q)^\frac{1}{2}\partial Z^I \wht g_3\|_{L^2}&\leq C_0\ep+C\ep^{\frac{5}{4}}.\\
\end{align*}
Moreover
$$\int_0^t\|\alpha w_2'(q)^\frac{1}{2}\bar{\partial} Z^I \wht g_3\|^2_{L^2}\lesssim \ep^2.$$
\end{prp}
We prove the proposition by using the energy estimate for $\phi$, $h$ and $\wht g_3$.

\begin{prp}\label{prpphi}
Let $I\leq N-2$. We have
$$
\frac{d}{dt}\sum_{J\leq I}\|\alpha w_0(q)^\frac{1}{2}\partial Z^J \phi\|^2_{L^2}
+\sum_{J\leq I}\|\alpha w_0'(q)^\frac{1}{2}\bar{\partial} Z^J \phi\|^2_{L^2}
\lesssim\ep^3\frac{1}{(1+t)^{1+\sigma-C\sqrt{\ep}}}
+\frac{\ep}{(1+t)^{\sigma}}\sum_{J\leq I}\left\| w'_0(q)^\frac{1}{2}\bar{\partial}Z^{J}\phi\right\|^2_{L^2}.
$$
\end{prp}

\begin{prp}\label{prphn}
Let $I\leq N-2$.
We have the estimate
$$
\frac{d}{dt}\sum_{J\leq I}\|\alpha w_3(q)^\frac{1}{2}\partial Z^J h\|^2_{L^2}
+\sum_{J\leq I}\|\alpha w_3'(q)\bar{\partial} Z^J h\|^2_{L^2}
\lesssim \ep^3.$$
\end{prp}

\begin{prp}\label{prpwhtg}
Let $I\leq N-2$. We have
$$
\frac{d}{dt}\sum_{J\leq I}\|\alpha w_2(q)^\frac{1}{2}\partial Z^J \wht g_3\|^2_{L^2}
+\sum_{J\leq I}\|\alpha w_2'(q)^\frac{1}{2}\bar{\partial} Z^J \wht g_3\|^2_{L^2}
\lesssim\frac{\ep^\frac{5}{2}}{(1+t)^{1+\sigma}}+
\frac{\ep}{(1+t)^\sigma}\|w_2'(q)^\frac{1}{2}\bar{\partial} Z^{I+1} \wht g_3\|^2_{L^2}.
$$
\end{prp}

We admit for the moment Propositions \ref{prpphi}, \ref{prphn} and \ref{prpwhtg} and prove Proposition \ref{l2n1}.
\begin{proof}[Proof of Proposition \ref{l2n1}]
We estimate $\phi$. Since $\sigma>C\sqrt{\ep}$ for $\ep>0$ small enough, 
by integrating the inequality of Proposition \ref{prpphi} with respect to $t$ we obtain
\begin{align*}
 &\sum_{J\leq I}\|\alpha w_0(q)^\frac{1}{2}\partial Z^J \phi\|^2_{L^2}+\int_0^t\sum_{J\leq I}\|\alpha w_0'(q)^\frac{1}{2}\bar{\partial} Z^J \phi\|^2_{L^2}\\
&\leq \sum_{J\leq I}\|\alpha w_0(q)^\frac{1}{2}\partial Z^I \phi(0) \|^2_{L^2}+C\ep^3 +C\int_0^t \frac{\ep}{(1+\tau)^{\sigma}}\sum_{J\leq I}\left\| w'_0(q)^\frac{1}{2}\bar{\partial}Z^{J}\phi\right\|^2_{L^2}.
\end{align*}
Thanks to Proposition \ref{estn}, we have
$$\int_0^t \frac{\ep}{(1+\tau)^{\sigma}}\sum_{J\leq I}\left\| w'_0(q)^\frac{1}{2}\bar{\partial}Z^{J}\phi\right\|^2_{L^2}\lesssim \ep^2,$$
and therefore
$$\sum_{J\leq I}\|\alpha w_0(q)^\frac{1}{2}\partial Z^J \phi\|^2_{L^2}
\lesssim C_0^2\ep^2+C\ep^3.$$
We now estimate $h$.
We integrate the inequality of Proposition \ref{prphn} with respect to $t$. We obtain, since we take zero initial data for $h$, and therefore, initial data for $Z^I h$ of size $\ep^2$
$$\sum_{J\leq I}\|\alpha w_3(q)^\frac{1}{2}\partial Z^J h\|^2_{L^2}
+\int_0^t\sum_{J\leq I}\|\alpha w_3'(q)\bar{\partial} Z^J h\|^2_{L^2}
\lesssim \ep^3(1+t).$$
We now integrate the inequality of Proposition \ref{prpwhtg} to estimate $\wht g_3$. We obtain
\begin{align*}
&\sum_{J\leq I}\|\alpha w_2(q)^\frac{1}{2}\partial Z^J \wht g_3\|^2_{L^2}
+\int_0^t \sum_{J\leq I}\|\alpha w_2'(q)^\frac{1}{2}\bar{\partial} Z^J \wht g_3\|^2_{L^2}\\
\leq&\sum_{J\leq I}\|\alpha w_2(q)^\frac{1}{2}\partial Z^J \wht g_3(0)\|^2_{L^2}+ C\ep^\frac{5}{2}+
\int_0^t\frac{\ep}{(1+\tau)^\sigma}\|\alpha w_2'(q)^\frac{1}{2}\bar{\partial} Z^{I+1} \wht g_4\|^2_{L^2}.
\end{align*}
Proposition \ref{estn} yields
$$\int_0^t\frac{1}{(1+t)^\sigma}\| w_2'(q)^\frac{1}{2}\bar{\partial} Z^{I+1} \wht g_4\|^2_{L^2}\lesssim \ep^2,$$
Therefore
$$\sum_{I\leq N-1}\|\alpha w_2(q)^\frac{1}{2}\partial Z^I \wht g_3\|^2_{L^2}
+\int_0^t \sum_{I\leq N-1}\|\alpha w_2'(q)^\frac{1}{2}\bar{\partial} Z^I \wht g_3\|^2_{L^2}\leq C_0^2\ep^2+C\ep^\frac{5}{2}.$$
This concludes the proof of Proposition \ref{l2n1}.
\end{proof}

\begin{proof}[Proof of Proposition \ref{prpphi}]
We follow the proof of Lemma \ref{lemmephi}.
Let $I\leq N-1$.
We use the weighted energy estimate for the equation
\eqref{equaphi}.
It yields
\begin{equation}\label{benephi}\begin{split}\frac{d}{dt}\left(\|\alpha w(q)^\frac{1}{2}\partial Z^I \phi\|^2_{L^2}\right)
+\|\alpha w'(q)^\frac{1}{2}\bar{\partial} Z^I \phi\|^2_{L^2}
\lesssim&\left\|\Box_g Z^I \phi \right\|_{L^2} \|\alpha w(q)^\frac{1}{2}\partial Z^I \phi\|_{L^2}\\
&+\frac{\ep}{(1+t)^{1+\sigma}}\|w^\frac{1}{2}\partial Z^I \phi\|^2_{L^2}.
\end{split}
\end{equation}
We first estimate
$$\left|\sum_{\substack{I_1+I_2\leq I\\I_2\leq I-1}}Z^{I_1} g_{LL}\partial^2_q Z^{I_2}\phi\right|\lesssim
\frac{1}{(1+|q|)}\sum_{I_1+I_2\leq I}|Z^{I_1} g_{LL}\partial_q Z^{I_2}\phi|.$$
If $I_1\leq \frac{N}{2}$, we can estimate
$$|Z^{I_1} g_{LL}|\lesssim \frac{\ep (1+|q|)}{(1+t)^{\frac{3}{2}-\rho}}$$
so
\begin{equation}
\label{benephi1}\left\|\frac{\alpha w_0^\frac{1}{2}}{(1+|q|)}Z^{I_1} g_{LL}\partial_q Z^{I_2}\phi\right\|_{L^2}
\lesssim \frac{\ep}{(1+t)^{\frac{3}{2}-\rho}}\|\alpha w_0(q)^\frac{1}{2}\partial_q Z^{I_2}\phi\|_{L^2}.
\end{equation}
If $I_2\leq \frac{N}{2}$, we can estimate
$$|\partial_q Z^{I_2}\phi|\lesssim \frac{\ep}{(1+|q|)^{\frac{3}{2}-4\rho }\sqrt{1+t}}, \; for \;q<0,\quad
|\partial_q Z^{I_2}\phi|\lesssim \frac{\ep}{(1+|q|)^{\frac{3}{2}+\delta-\sigma }\sqrt{1+t}},\; for \;q>0.
$$
We apply the weighted Hardy inequality, but in order to be in its range, we cannot keep all the decay in $q$ in the region $q>0$. 
$$\left\|\frac{\alpha w_0(q)^\frac{1}{2}}{(1+|q|)}Z^{I_1} g_{LL}\partial_q Z^{I_2}\phi\right\|_{L^2}
\lesssim \frac{\ep}{\sqrt{1+t}}\left\|\frac{v(q)^\frac{1}{2}}{(1+|q|)}Z^I g_{LL}\right\|_{L^2}\lesssim
\frac{\ep}{\sqrt{1+t}}\|v(q)^\frac{1}{2}
\partial_q Z^{I_1} g_{LL}\|_{L^2}$$
where
\begin{equation}
\label{poid}
\left\{\begin{array}{l}
v(q)=\frac{1}{(1+|q|)^{3-4\rho}}\; for \; q<0,\\
v(q)=\frac{w_0(q)}{(1+|q|)^{1+\delta}}=(1+|q|)^{1+\delta}\; for \; q>0.
\\
\end{array}\right.
\end{equation}
We use \eqref{cool}, which gives
$\partial_q Z^{I_1} g_{LL} \sim \bar{\partial}Z^{I_1} \wht{g_4}$ so
\begin{equation}
\label{estutil}
|\partial_q Z^{I_1} g_{LL}|\lesssim \frac{1}{1+s} | Z^{I_1+1} \wht g_4|
\lesssim \frac{1}{(1+t)^{\frac{1}{2}+\sigma}(1+|q|)^{\frac{1}{2}-\sigma}}
 | Z^{I_1+1} \wht g_4|.
\end{equation}
Therefore, we obtain
\begin{align*}
\left\|\frac{\alpha w_0(q)^\frac{1}{2}}{(1+|q|)}Z^{I_1} g_{LL}\partial_q Z^{I_2}\phi\right\|_{L^2}
&\lesssim\frac{\ep}{(1+t)^{1+\sigma}}\left\|\frac{v(q)^\frac{1}{2}}{(1+|q|)^{\frac{1}{2}-\sigma}}Z^{I_1+1}\wht g_4\right\|_{L^2}\\
&\lesssim \frac{\ep}{(1+t)^{1+\sigma}}\|v^\frac{1}{2}(1+|q|)^{\frac{1}{2}+\sigma}\partial_q Z^{I_1+1} \wht g_4\|_{L^2}
\end{align*}
where we have used again the weighted Hardy inequality.
We calculate
$$\left\{\begin{array}{l}
v(q)(1+|q|)^{1+2\sigma}=\frac{1}{(1+|q|)^{2-4\rho-2\sigma}}\; for \; q<0,\\
v(q)(1+|q|)^{1+2\sigma}=(1+|q|)^{2+\delta+2\sigma}\; for \; q>0.
\\
\end{array}\right.$$
Therefore if $1-4\rho-2\sigma \geq \mu$ and $\delta+2\sigma<2\delta$ we have $v(q)(1+|q|)^{1+2\sigma}\leq w_2$ so we obtain, together with Proposition \ref{estn},
\begin{equation}
\label{benephi2}\left\|\frac{\alpha w_0(q)^\frac{1}{2}}{(1+|q|)}Z^{I_1} g_{LL}\partial_q Z^{I_2}\phi\right\|_{L^2}
\lesssim
\frac{\ep}{(1+t)^{1+\sigma}}\|w_2(q)^\frac{1}{2} \partial_q Z^{I_1+1} \wht g_4\|_{L^2}\lesssim \frac{\ep^2(1+t)^{C\sqrt{\ep}}}{(1+t)^{1+\sigma}}
\end{equation}
We now estimate the crossed terms, for which the weight modulator $\alpha$ has been introduced. They are of the form \eqref{crossphi}.
It is sufficient to estimate, for $I\leq N-1$
$$\left\|\frac{\ep\mathbbm{1}_{q>0}}{r}\alpha(q)w_0(q)^\frac{1}{2}\bar{\partial}
Z^{I} \phi\right\|_{L^2}.$$
%We estimate
%$$\frac{\ep\mathbbm{1}_{q>0}}{r}\alpha(q)|\bar{\partial}
%Z^{N-1} \phi| \lesssim \frac{\ep\mathbbm{1}_{q>0}}{r^2(1+|q|)^\sigma}\sum_{I\leq 1}|Z^{N-1+I}\phi|\lesssim \frac{\ep\mathbbm{1}_{q>0}}{t^{1+\sigma}(1+|q|)}|Z^N \phi|.$$
We obtain
\begin{equation*}
\left\|\frac{\ep\mathbbm{1}_{q>0}}{r}\alpha(q)w_0(q)^\frac{1}{2}\bar{\partial}
Z^{N-1} \phi\right\|_{L^2}
\lesssim \frac{\ep}{(1+t)^{\frac{1}{2}+\sigma}}\left\| \frac{\mathbbm{1}_{q>0}}{(1+|q|)^\frac{1}{2}}w_0(q)^\frac{1}{2}\bar{\partial}Z^{N-1}\phi\right\|_{L^2}.
\end{equation*}
and consequently, since in the region $q>0$ we have
$\frac{w_0(q)^\frac{1}{2}}{\sqrt{1+|q|}}\lesssim w_0'(q)^\frac{1}{2}$,
\begin{equation}
\label{benephi3}\left\|\frac{\ep\mathbbm{1}_{q>0}}{r}\alpha(q)w_0(q)^\frac{1}{2}\bar{\partial}
Z^{N-1} \phi\right\|_{L^2}\|\alpha w_0^\frac{1}{2}\partial Z^I \phi\|^2_{L^2}\lesssim
\frac{\ep}{(1+t)^{1+\sigma}}\|\alpha w_0^\frac{1}{2}\partial Z^I \phi\|^2_{L^2}+\frac{\ep}{(1+t)^\sigma}\|w'_0(q)^\frac{1}{2}\bar{\partial}
Z^{N-1} \phi\|^2_{L^2}
\end{equation}
The last term which appears in \eqref{benephi} can be estimated thanks to Proposition \ref{estn}
\begin{equation}
\label{benephi4}\frac{\ep}{(1+t)^{1+\sigma}}\|w_0^\frac{1}{2}\partial Z^I \phi\|^2_{L^2} \lesssim \frac{\ep^3}{(1+t)^{1+\sigma-C\sqrt{\ep}}}.
\end{equation}
The estimates \eqref{benephi}, \eqref{benephi1}, \eqref{benephi2}, \eqref{benephi3} and \eqref{benephi4}, together with the bootstrap assumption \eqref{bootl23}
which imply
$$\|\alpha w_0(q)^\frac{1}{2}\partial Z^I \phi\|_{L^2}\lesssim \ep,$$
conclude the proof of Proposition \ref{prpphi}.
\end{proof}
We now estimate $h$

\begin{proof}[Proof of \ref{prphn}]
 The equation for $Z^I h$ is given by \eqref{equah}.
We estimate first
$$\left\|\alpha w_3^\frac{1}{2}\sum_{\substack{I_1+I_2\leq I\\I_2\leq I-1}}\left(Z^{I_1} g^{\alpha \beta}\right)\left(Z^{I_2}\partial_\alpha \partial_\beta h\right)\right\|_{L^2}.$$
As before, we can estimate, for $I_1\leq \frac{N}{2}$, thanks to the bootstrap assumption \eqref{bootl23},
\begin{equation}\label{beneh1}
\left\|\frac{\alpha w_3^\frac{1}{2}}{(1+|q|)}Z^{I_1} g_{LL}\partial_q Z^{I_2} h\right\|_{L^2}
\lesssim \frac{\ep}{(1+t)^{\frac{3}{2}-\rho}}\|\alpha w_3^\frac{1}{2}\partial_q Z^{I_2} h\|_{L^2}\lesssim \frac{\ep^2}{(1+t)^{1-\rho}}.
\end{equation}
For $I_2\leq \frac{N}{2}$, we have the estimate
$$|\partial_q Z^{I_2}h|\lesssim \frac{\ep}{(1+|q|)^{\frac{3}{2}-\rho}},\; for \; q<0,  \quad
|\partial_q Z^{I_2}h|\lesssim \frac{\ep}{(1+|q|)^{2+\delta-\sigma}},\; for \; q<0, 
$$
so
$$ \left\|\frac{\alpha w_3^\frac{1}{2}}{(1+|q|)}Z^{I_1} g_{LL}\partial_q Z^{I_2} h\right\|_{L^2}
\lesssim \left\|\frac{\alpha v^\frac{1}{2}}{(1+|q|)}Z^{I_1} g_{LL}\right\|_{L^2}.$$
 where $v$ is defined by \eqref{poid} and
with Hardy inequality and the same reasoning than for $\phi$
$$ \left\|\frac{\alpha w_3^\frac{1}{2}}{(1+|q|)}Z^{I_1}g_{LL}\partial_q Z^{I_2} h\right\|_{L^2}
\lesssim \ep\|v^\frac{1}{2}\partial_q Z^{I_1} g_{LL}\|_{L^2}\lesssim\ep\frac{1}{(1+t)^{\frac{1}{2}+\sigma}}
\|w_2(q)^\frac{1}{2}\partial Z^{I_1+1} \wht g_4\|_{L^2},$$
and thanks to Proposition \ref{estn} we obtain
\begin{equation}
\label{beneh2}\left\|\frac{\alpha w_3^\frac{1}{2}}{(1+|q|)}Z^{I_1}g_{LL}\partial_q Z^{I_2} h\right\|_{L^2}\lesssim \frac{\ep^2(1+t)^{C\sqrt{\ep}}}{(1+t)^{\frac{1}{2}+\sigma}}.
\end{equation}
%We now use \eqref{estutil}, and once again Hardy inequality
%$$ \|\frac{\alpha %w_3^\frac{1}{2}}{(1+|q|)}Z^{I_1}g_{LL}\partial_q Z^{I_2} %h\|_{L^2}
%\lesssim\ep\frac{1}{t^{\frac{1}{2}+\sigma}}
% \|\frac{\alpha w_3^\frac{1}{2}}{(1+|q|)^{2+\rho-\sigma}}Z^{I_1+1} \wht g_4\|_{L^2}\lesssim
% \ep\frac{1}{t^{1}{2}+\sigma} \|\frac{\alpha w_3^\frac{1}{2}}{(1+|q|)^{1-\rho-\sigma}}\partial Z^{I_1+1} \wht g_4\|_{L^2}$$
%It yields, since $\rho+\sigma< \frac{1}{2}$
%$$ \|\frac{\alpha w_3^\frac{1}{2}}{(1+|q|)}Z^{I_1}g_{LL}\partial_q Z^{I_2} h\|_{L^2}
%\lesssim\ep\frac{1}{t^{\frac{1}{2}+\sigma}}
%\|w_2(q)^\frac{1}{2}\partial Z^{I_1+1} \wht g_4\|_{L^2}
%\lesssim \frac{\ep^2t^{C\sqrt{\ep}}}{t^{\frac{1}{2}+\sigma}}$$
We estimate the second term
$$\left\|\alpha w_3^\frac{1}{2}Z^I (\partial_q \phi)^2\right\|_{L^2}
\lesssim \sum_{I_1+I_2\leq I}
\|\alpha w_3^\frac{1}{2}\partial_q Z^{I_1} \phi \partial_q Z^{I_2} \phi\|_{L^2}
\lesssim \frac{\ep}{\sqrt{1+t}}\sum_{J\leq I} \| \alpha w_0^\frac{1}{2}\partial_q Z^I \phi\|_{L^2} $$
so thanks to \eqref{bootl23} we obtain
\begin{equation}
\label{beneh3}\left\|\alpha w_3^\frac{1}{2}Z^I (\partial_q \phi)^2\right\|_{L^2}\lesssim \frac{\ep^2}{\sqrt{1+t}}.
\end{equation}
The semi-linear term $\partial_{\ba L}g_{LL}\partial_{\ba L}h$, appearing in $Q_{\ba L \ba L}$ can be estimated in the same way at the first. The crossed term $\partial_{\ba L}(g_b)_{UU}\partial_{\ba L}g_{L\ba L}$ appearing in $Q_{\ba L \ba L}$ and
the term $(R_b)_{qq}$ can be estimated in the same way than in the case $I\leq N$. The crossed terms of $H_b^\rho \partial_\rho h$ can be estimated in the following way
\begin{equation}\label{beneh4}
\left\|\frac{\ep\ch_{q>0}}{r}\alpha w_3^\frac{1}{2}\partial Z^I h\right\|_{L^2}\lesssim \frac{\ep}{\sqrt{1+t}}\|\alpha w_3^\frac{1}{2}\partial Z^I h\|\lesssim \frac{\ep^2}{\sqrt{1+t}}.
\end{equation}
Thanks to \eqref{beneh1}, \eqref{beneh2}, \eqref{beneh3} and \eqref{beneh4}, and the bootstrap assumption \eqref{bootl23},
the energy inequality yields (we use here the first inequality of Proposition \ref{prpweighte})
$$
\frac{d}{dt}\|\alpha w_3^\frac{1}{2}\partial Z^I h\|^2_{L^2}+
\|\alpha w_3'(q)^\frac{1}{2}\bar{\partial} Z^I h\|_{L^2}^ 2
\lesssim \frac{\ep^2}{\sqrt{1+t}} \|\alpha w_3^\frac{1}{2}\partial_q Z^I h\|_{L^2}+\frac{\ep}{1+t}\|\alpha w_3^\frac{1}{2}\partial_q Z^I h\|^2_{L^2}
\lesssim \ep^3,
$$
which concludes the proof of Proposition \ref{prphn}.
\end{proof}
We now estimate $\wht g_3$

\begin{proof}[Proof of Proposition \ref{prpwhtg}]
We write $\Box_g Z^I \wht g_3= f_{\mu \nu}$.
The energy estimate yields
\begin{align*}
\frac{d}{dt}\left(\|\alpha w_2(q)^\frac{1}{2}\partial Z^I \wht g_3\|^2_{L^2}\right)
+\|\alpha w'(q)^\frac{1}{2}\bar{\partial} Z^I \wht g_3\|^2_{L^2}
\lesssim&\|\alpha w_2(q)f_{\mu \nu}\|_{L^2} \|\alpha w_2(q)\partial Z^I \wht g_3\|_{L^2}\\
&+\frac{\ep}{(1+t)^\sigma}\| w_2(q)^\frac{1}{2}\partial Z^I \wht g_3\|^2_{L^2}.
\end{align*}
We recall that the terms in $f_{\mu \nu }$ are 
\begin{itemize}
\item the quasilinear terms, 
\item the terms coming from the non commutation of the wave operator with the null decomposition: it will be sufficient to study the term $\chi(\frac{r}{t})\frac{1}{r^2}\partial_\theta Z^N h$,
\item the semi-linear terms: it is sufficient to study 
the term $Z^I\partial_U g_{LL}\partial_{\ba L}h$,
\item the crossed term: their analysis is the same than for $\phi$.
\end{itemize} 
We first estimate the quasilinear term.
$$\left|\sum_{\substack{I_1+I_2\leq I\\I_2\leq I-1}}Z^{I_1} g_{LL}\partial^2_q Z^{I_2}\wht g_4\right|\lesssim
\frac{1}{(1+|q|)}\sum_{I_1+I_2\leq I}|Z^{I_1} g_{LL}\partial_q Z^{I_2}\wht g_4|$$
If $I_1\leq \frac{N}{2}$, we can estimate
$$|Z^{I_1} g_{LL}|\lesssim \frac{\ep (1+|q|)}{(1+t)^{\frac{3}{2}-\rho}}$$
so
\begin{equation}
\label{beneg1}\left\|\frac{\alpha w_2^\frac{1}{2}}{(1+|q|)}Z^{I_1} g_{LL}\partial_q Z^{I_2}\wht g_4\right\|_{L^2}
\lesssim \frac{\ep}{(1+t)^{\frac{3}{2}-\rho}}\|\alpha w_2^\frac{1}{2}\partial_q Z^{I_2}\wht g_4\|_{L^2}
\end{equation}
If $I_2\leq \frac{N}{2}$, we can estimate, thanks \eqref{ksg4} and \eqref{est1} 
$$|\partial_q Z^{I_2}\wht g_4|\lesssim \left(\frac{\ep}{\sqrt{1+t}\sqrt{1+|q|}}\right)^\frac{1}{2}
\left(\frac{\ep}{(1+|q|)^{\frac{3}{2}-\rho}}\right)^\frac{1}{2}
\lesssim \frac{\ep}{(1+t)^\frac{1}{4}(1+|q|)^{1-\frac{\rho}{2}}}
$$
Therefore, 
$$\left\|\frac{\alpha w_2^\frac{1}{2}}{(1+|q|)}Z^{I_1} g_{LL}\partial_q Z^{I_2}\wht g_4\right\|_{L^2}
\lesssim 
\frac{\ep}{t^{\frac{1}{4}}} \left\|\frac{\alpha w_2^\frac{1}{2}}{(1+|q|)^{2-\frac{\rho}{2}}} Z^{I_1} g_{LL}\right\|_{L^2}
\lesssim \frac{\ep}{t^{\frac{1}{4}}} \left\|\frac{\alpha w_2^\frac{1}{2}}{(1+|q|)^{1-\frac{\rho}{2}}}\partial Z^{I_1} g_{LL}\right\|_{L^2},$$
%where
%$$\left\{\begin{array}{l}
%\wht v(q) = \frac{w_2}{(1+|q|)^2}\; for \; q<0,\\
%\wht v(q) = \frac{w_2}{(1+|q|)^{1+\delta}}=(1+|q|)^{1+\delta}\; for \; %q>0.
%\\
%\end{array}\right.$$
where we have used the weighted Hardy inequality, noting that in the  region $q>0$
$$\frac{\alpha^2 w_2}{(1+|q|)^{2-\rho}}=(1+|q|)^{2\delta-2\sigma-\rho},$$
so the condition $\delta>\sigma+\rho+\frac{1}{2}$ ensure that we can apply the weighted Hardy inequality. 
We use the wave coordinate condition, \eqref{cool} which gives
$\partial_q Z^{I_1} g_{LL} \sim \bar{\partial}Z^{I_1} \wht{g_4}$.
We obtain
\begin{equation*}
|\partial_q Z^{I_1} g_{LL}|
\lesssim \frac{1}{1+s}|Z^{I_1+1} \wht g_4| \frac{1}{(1+t)^{\frac{3}{4}+\sigma}(1+|q|)^{\frac{1}{4}-\sigma}}
| Z^{I_1+1} \wht g_4|.
\end{equation*} 
It yields, by using Hardy inequality again
\begin{align*}
\left\|\frac{\alpha w_2^\frac{1}{2}}{(1+|q|)}Z^{I_1} g_{LL}\partial_q Z^{I_2}\wht g_4\right\|_{L^2}
&\lesssim 
\frac{\ep}{(1+t)^{1+\sigma}}\left\|\frac{\alpha w_2^\frac{1}{2}}{(1+|q|)^{\frac{5}{4}-\frac{\rho}{2}-\sigma}}Z^{I_1+1} \wht g_4\right\|_{L^2}\\
&\lesssim
\frac{\ep}{(1+t)^{1+\sigma}}\|\frac{\alpha w_2^\frac{1}{2}}{(1+|q|)^{\frac{1}{4}-\frac{\rho}{2}-\sigma}}Z^{I_1+1} \wht g_4\|_{L^2}.
\end{align*}
%We calculate
%$$\left\{\begin{array}{l}
%\wht v(q)(1+|q|)^{1+2(\sigma+\rho)} = %\frac{w_2}{(1+|q|)^{1-2(\sigma+\rho)}}\; for \; q<0,\\
%\wht v(q) (1+|q|)^{1+2(\sigma+\rho)}= %\frac{w_2}{(1+|q|)^{\delta-2(\sigma+\rho)}}\; for \; q>0.
%\\
%\end{array}\right.$$
%Therefore, if $2\sigma+\rho\leq \delta$ (and $2(\sigma+\rho)\leq 1$) we %have 
Consequently, since $\frac{1}{4}-\frac{\rho}{2}-\sigma>0$ we have
\begin{equation}
\label{beneg2}\left\|\frac{\alpha w_2^\frac{1}{2}}{(1+|q|)}Z^{I_1} g_{LL}\partial_q Z^{I_2}\wht g_4\right\|_{L^2}
\lesssim \frac{\ep}{(1+t)^{1+\sigma}}\|w_2^\frac{1}{2}\partial Z^{I+1} \wht g_4\|_{L^2}\lesssim \frac{\ep^2(1+t)^{C\sqrt{\ep}}}{(1+t)^{1+\sigma}}, 
\end{equation}
where we have used Proposition \ref{estn}.
We now estimate the term coming from the non commutation with the wave operator 
$\|\alpha w^\frac{1}{2}_2(q)\Up\left(\frac{r}{t}\right)\frac{1}{r^2}\partial_\theta Z^I h\|_{L^2}$.
On the support of $\Up \left(\frac{r}{t}\right)$, we have $r\sim t$ and hence 
$$\frac{1}{r^2}\lesssim \frac{1}{(1+t)^{\frac{3}{2}+\sigma}(1+|q|)^{\frac{1}{2}-\sigma}}.$$ Therefore
\begin{align*}\left\|\alpha w^\frac{1}{2}_2(q)\Up\left(\frac{r}{t}\right)\frac{1}{r^2}\partial_\theta Z^I h\right\|_{L^2}
&\lesssim \frac{1}{(1+t)^{\frac{3}{2}+\sigma}}\left\|\frac{\alpha w_2^\frac{1}{2}}{(1+|q|)^{\frac{1}{2}-\sigma}}Z^{I+1}h\right\|_{L^2}\\
&\lesssim \frac{1}{(1+t)^{\frac{3}{2}+\sigma}}\left\|\alpha w_2^\frac{1}{2}(1+|q|)^{\frac{1}{2}+\sigma}\partial Z^{I+1}h\right\|_{L^2},
\end{align*}
where we have applied the weighted Hardy inequality. We calculate
$$\alpha^2 w_2(1+|q|)^{1+2\sigma}=\left\{\begin{array}{l}
(1+|q|)^{2\sigma-2\mu}\; for \; q<0,\\
(1+|q|)^{3+2\delta}\; for \; q>0.
\\
\end{array}\right.$$
If $\sigma<\mu$ we have $\alpha^2 w_2(1+|q|)^{1+2\sigma}\leq w_3$ and
\begin{equation}
\label{beneg4}\left\|\alpha w^\frac{1}{2}_2(q)\chi\left(\frac{r}{t}\right)\frac{1}{r^2}\partial_\theta Z^Ih\right\|_{L^2}
\lesssim\frac{1}{(1+t)^{\frac{3}{2}+\sigma}}\|w_3(q)^\frac{1}{2}\partial Z^{I+1}h\|_{L^2}\lesssim \frac{\ep^\frac{3}{2}(1+t)^{C\sqrt{\ep}}}{(1+t)^{1+\sigma}},
\end{equation}
where we have used Proposition \ref{estn} which yields, for $I\leq N-2$
$$\|w_3(q)^\frac{1}{2}\partial Z^{I+1}h\|_{L^2}\lesssim \ep{\frac{3}{2}}(1+t)^{\frac{1}{2}+C\sqrt{\ep}}.$$
We now estimate $Z^I(\partial_U g_{LL}\partial_{\ba L}h)$. We have
$$\|\alpha w_2(q)^\frac{1}{2}Z^I(\partial_U g_{LL}\partial_{\ba L}h)\|_{L^2}\lesssim \sum_{I_1+I_2\leq I}
\|\alpha w_2(q)^\frac{1}{2}\bar{\partial}Z^{I_1} g_{LL}\partial Z^{I_2} h\|_{L^2}.$$ 
If $I_1\leq \frac{N}{2}$ we estimate
$$|\bar{\partial}Z^{I_1} g_{LL}|\lesssim \frac{1}{1+s}
| Z^{I_1+1}g_{LL}|\lesssim \frac{\ep (1+|q|)}{(1+t)^{\frac{5}{2}-\rho}}\lesssim \frac{(1+|q|)^{\rho+\sigma}}{(1+t)^{\frac{3}{2}+\sigma}}.$$
Therefore
$$\|\alpha w_2(q)^\frac{1}{2}\bar{\partial}Z^{I_1} g_{LL}\partial Z^{I_2} h\|_{L^2}\lesssim
\frac{\ep}{(1+t)^{\frac{3}{2}+\sigma}}\|\alpha w_2^\frac{1}{2}(1+|q|)^{\rho+\sigma}\partial Z^{I_2}h\|_{L^2}\lesssim 
\frac{\ep}{t^{\frac{3}{2}+\sigma}}\|\alpha w_3(q)^\frac{1}{2}\partial Z^{I_2}h\|_{L^2},$$
and consequently
\begin{equation}
\label{beneg5}\|\alpha w_2(q)^\frac{1}{2}\bar{\partial}Z^{I_1} g_{LL}\partial Z^{I_2} h\|_{L^2}\lesssim\lesssim \frac{\ep^2}{t^{1+\sigma}}.
\end{equation}
If $I_2\leq \frac{N}{2}$ thanks to \eqref{est1} we estimate
$$|\partial Z^{I_2}h|\leq \frac{\ep}{(1+|q|)^{\frac{3}{2}-\rho}},$$
therefore
\begin{align*}
\|\alpha w_2(q)^\frac{1}{2}\bar{\partial}Z^{I_1} g_{LL}\partial Z^{I_2} h\|_{L^2}&\lesssim
\ep\left\|\frac{\alpha w_2^\frac{1}{2}}{(1+|q|)^{\frac{3}{2}-\rho}}\bar{\partial}Z^{I_1} g_{LL}\right\|_{L^2}\\
&\lesssim \frac{\ep}{(1+t)^{\frac{1}{2}+\sigma}} \left\|\frac{\alpha w_2^\frac{1}{2}}{(1+|q|)^{2-\rho-\sigma}} Z^{ I_1+1}g_{LL}\right\|_{L^2}\\
&\lesssim \frac{\ep}{(1+t)^{\frac{1}{2}+\sigma}}\left\|\frac{\alpha w_2^\frac{1}{2}}{(1+|q|)^{1-\rho-\sigma}}
\partial Z^{ I_1+1}g_{LL}\right\|_{L^2}\\
&\lesssim  \frac{\ep}{(1+t)^{\frac{1}{2}+\sigma}}\|\alpha w'_2(q)^\frac{1}{2}\bar{\partial} Z^{ I_1+1}\wht g_4\|_{L^2}
\end{align*}
where in the last inequality we have used the wave coordinate condition. 
Therefore
\begin{equation}\label{beneg6}\|\alpha w_2(q)^\frac{1}{2}\bar{\partial}Z^{I_1} g_{LL}\partial Z^{I_2} h\|_{L^2}\|\alpha w_2(q)^\frac{1}{2}\partial Z^I \wht g_3\|
\lesssim \frac{\ep}{(1+t)^\sigma}\|\alpha w'_2(q)^\frac{1}{2}\bar{\partial} Z^{ I_1+1}\wht g_4\|^2_{L^2} +\frac{\ep^3}{(1+t)^{1+\sigma}}.
\end{equation}
The estimates \eqref{beneg1},\eqref{beneg2}, \eqref{beneg4}, \eqref{beneg5} and \eqref{beneg6} 
 conclude the proof of Proposition \ref{prpwhtg}.
\end{proof}
\subsection{Estimates for $I\leq N-8$}
\begin{prp}\label{estl27}
We have for $I\leq N-8$
\begin{align}
\label{ff1}&\| w_1(q)^\frac{1}{2}\partial Z^I \wht g_2\|_{L^2} \leq C_0\ep(1+t)^{C\ep},\\
\label{ff2}&\|\alpha w_1(q)^\frac{1}{2}\partial Z^I \wht g_2\|_{L^2}\leq C_0  \ep + C\ep^\frac{3}{2},
\end{align}
and for $I\leq N-9$
\begin{align}
\label{ff3}&\| w(q)^\frac{1}{2}\partial Z^I \wht g_2\|_{L^2} \leq C_0\ep(1+t)^{C\ep},\\
\label{ff4}&\|\alpha w(q)^\frac{1}{2}\partial Z^I \wht g_2\|_{L^2}\leq C_0  \ep + C\ep^\frac{3}{2}.
\end{align}
\end{prp}
This is a consequence of the following two propositions.
\begin{prp}\label{l27}
We have for $I\leq N-8$
\begin{equation}\label{ee1}
\frac{d}{dt}\sum_{J\leq I}\| w_1(q)^\frac{1}{2}\partial Z^J \wht g_2\|^2_{L^2}
+\sum_{J\leq I}\| w'_1(q)^\frac{1}{2}\bar{\partial} Z^J \wht g_2\|^2_{L^2}\\
\lesssim\frac{\ep^3}{(1+t)^{1+\sigma}}+
\frac{\ep}{1+t}\sum_{J\leq I}\| w_1(q)^\frac{1}{2}\partial Z^J \wht g_2\|^2_{L^2},
\end{equation}
and
\begin{equation}\label{ee2}
\frac{d}{dt}\sum_{J\leq I}\|\alpha w_1(q)^\frac{1}{2}\partial Z^J \wht g_2\|^2_{L^2}
+\sum_{J\leq I}\|\alpha w'_1(q)^\frac{1}{2}\bar{\partial} Z^J \wht g_2\|^2_{L^2}
\lesssim\frac{\ep^3}{(1+t)^{1+\sigma}}+
\ep\|\alpha w_2'(q)^\frac{1}{2}\bar{\partial} Z^{I+1} \wht g_3\|^2_{L^2}.
\end{equation}

\end{prp}
\begin{prp}\label{l28}
We have for $I\leq N-9$
\begin{equation}
\label{ee3}\frac{d}{dt}\sum_{J\leq I}\| w_0(q)^\frac{1}{2}\partial Z^J \wht g_2\|^2_{L^2}
+\sum_{J\leq I}\| w_0'(q)^\frac{1}{2}\bar{\partial} Z^J \wht g_2\|^2_{L^2}
\lesssim\frac{\ep^3}{(1+t)^{1+\sigma}}+
+\frac{\ep}{1+t}\sum_{J\leq I}\| w_0(q)^\frac{1}{2}\partial Z^J \wht g_2\|^2_{L^2},
\end{equation}
and
\begin{equation}
\label{ee4}\frac{d}{dt}\sum_{J\leq I}\|\alpha w_0(q)^\frac{1}{2}\partial Z^J \wht g_2\|^2_{L^2}
+\sum_{J\leq I}\|\alpha w_0'(q)^\frac{1}{2}\bar{\partial} Z^J \wht g_2\|^2_{L^2}
\lesssim\frac{\ep^3}{(1+t)^{1+\sigma}}+
\ep\|\alpha w_1'(q)^\frac{1}{2}\bar{\partial} Z^{I+1} \wht g_2\|^2_{L^2}.
\end{equation}
\end{prp}
We assume Proposition \ref{l27} and \ref{l28} and prove Proposition \ref{estl27}.
\begin{proof}[Proof of Proposition \ref{estl27}]
The inequalities \eqref{ff1} and \eqref{ff3} are straightforward consequences of \eqref{ee1} and \eqref{ee3}.
To prove  \eqref{ff2}, we integrate \eqref{ee2}. We obtain
\begin{align*}&\sum_{J\leq I}\|\alpha w_1(q)^\frac{1}{2}\partial Z^J \wht g_2\|^2_{L^2}
+\int_0^t\sum_{J\leq I}\|\alpha w'_1(q)^\frac{1}{2}\bar{\partial} Z^J \wht g_2\|^2_{L^2}d\tau\\
&\leq \sum_{J\leq I}\|\alpha w_1(q)^\frac{1}{2}\partial Z^J \wht g_2(0)\|^2_{L^2}
+C\ep^3+
C\ep \int_0^t \|\alpha w_2'(q)^\frac{1}{2}\bar{\partial} Z^{I+1} \wht g_3\|^2_{L^2}d\tau.
\end{align*}
Thanks to Proposition \ref{l2n1}, we have
$$\int_0^t \|\alpha w_2'(q)^\frac{1}{2}\bar{\partial} Z^{I+1} \wht g_3\|^2_{L^2}\lesssim \ep^2,$$
and consequently
\begin{equation}
\label{energi7}
\sum_{J\leq I}\|\alpha w_1(q)^\frac{1}{2}\partial Z^J \wht g_2\|^2_{L^2}
+\int_0^t\sum_{J\leq I}\|\alpha w'_1(q)^\frac{1}{2}\bar{\partial} Z^J \wht g_2\|^2_{L^2}d\tau\leq C_0^2\ep^2+ C\ep^3,
\end{equation}
which proves \eqref{ff2}. To prove \eqref{ff4}, we integrate \eqref{ee4}
\begin{align*}
&\sum_{J\leq I}\|\alpha w_0(q)^\frac{1}{2}\partial Z^J \wht g_2\|^2_{L^2}
+\int\sum_{J\leq I}\|\alpha w_0'(q)^\frac{1}{2}\bar{\partial} Z^J \wht g_2\|^2_{L^2}d\tau\\
&\leq\sum_{J\leq I}\|\alpha w_0(q)^\frac{1}{2}\partial Z^J \wht g_2(0)\|^2_{L^2}+C\ep^3+
C\ep\int_0^t\|\alpha w_1'(q)^\frac{1}{2}\bar{\partial} Z^{I+1} \wht g_2\|^2_{L^2}d\tau.
\end{align*}
Thanks to \eqref{energi7}, we have for $I\leq N-9$
$$\int_0^t\|\alpha w_1'(q)^\frac{1}{2}\bar{\partial} Z^{I+1} \wht g_2\|^2_{L^2}d\tau\lesssim \ep^2,$$
and consequently
$$\sum_{J\leq I}\|\alpha w_0(q)^\frac{1}{2}\partial Z^J \wht g_2\|^2_{L^2}
\leq C^2_0\ep^2+C\ep^3,$$
which concludes the proof of Proposition \ref{estl27}.
\end{proof}

\begin{proof}[Proof of Proposition \ref{l27}]

It is sufficient to estimate the terms in the region $q<0$, since in the region $q>0$,
we have $w_0=w_1=w_2$ so the estimates
 are strictly the same than in the previous section. Once again, the weight modulator $\alpha$ is used to tackle the crossed terms, which create a logarithmic loss in the estimates. However, in the region $q<0$, since $\alpha=1$, we write everything with the weight $w_1$, and do everything as if no terms were present in the region $q>0$, since the influence of these terms have already been tackled.
 
We first estimate the term coming from the non commutation of the wave operator with the null decomposition, 
$$\Up\left(\frac{r}{t}\right)\frac{1}{r^2}\partial_\theta Z^I (h_0+\wht h).$$
Since $I+1\leq N-7$, we can use the Propositions \ref{estzh} for $Z^{I+1}h_0$ and Proposition \ref{whth} for $Z^{I+1}\wht h$. We obtain
$$|Z^{I+1}(h_0+\wht h)|\lesssim \frac{\ep^2}{(1+|q|)^{\frac{1}{2}-\rho}}.$$
Therefore
\begin{equation}
\label{7ene1}
\left\|\Up\left(\frac{r}{t}\right)\frac{1}{r^2}\partial_\theta Z^I (h_0+\wht h)\right\|_{L^2}\lesssim \frac{\ep^2}{(1+t)^{1+\sigma}}\left\|
\Up\left(\frac{r}{t}\right)\frac{1}{(1+|q|)^{\frac{1}{2}-\rho}r^{1-\sigma}}\right\|_{L^2}\lesssim  \frac{\ep^2}{(1+t)^{1+\sigma}},
\end{equation}
where we have used the calculation
\begin{align*}
\left\|
\Up\left(\frac{r}{t}\right)\frac{1}{(1+|q|)^{\frac{1}{2}-\rho}r^{1-\sigma}}\right\|^2_{L^2}
&\leq 2\pi\int\Up\left(\frac{r}{t}\right)^2 \frac{1}{(1+|q|)^{1-2\rho}r^{2-2\sigma}}rdr \\
&\leq 2\pi\int  \frac{dq}{(1+|q|)^{2-2\rho-2\sigma}} <+\infty,
\end{align*}
if $\rho +\sigma < \frac{1}{2}$.

We now estimate $Z^I(\partial_U g_{LL}\partial_{\ba L}h)$. We have
$$\|\ch_{q<0}w_1(q)^\frac{1}{2}Z^I(\partial_U g_{LL}\partial_{\ba L}h)\|_{L^2}\lesssim \sum_{I_1+I_2\leq I}
\|\ch_{q<0}w_1(q)^\frac{1}{2}\bar{\partial}Z^{I_1} g_{LL}\partial Z^{I_2} h\|_{L^2}$$ 
If $I_1\leq \frac{N}{2}$ we estimate
$$|\bar{\partial}Z^{I_1} g_{LL}|\lesssim \frac{1}{1+s}
|Z^{I_1+1}g_{LL}|\lesssim \frac{\ep (1+|q|)}{(1+s)^{\frac{5}{2}-\rho}}\lesssim \frac{(1+|q|)^{\rho+\sigma}}{(1+t)^{\frac{3}{2}+\sigma}}.$$
Therefore
\begin{align*}
\|\ch_{q<0}w_1(q)^\frac{1}{2}\bar{\partial}Z^{I_1} g_{LL}\partial Z^{I_2} h\|_{L^2}&\lesssim
\frac{\ep}{(1+t)^{\frac{3}{2}+\sigma}}\left\|\ch_{q<0}\frac{(1+|q|)^{\rho+\sigma}}{(1+|q|)^{\frac{1}{4}}}\partial Z^{I_2}h\right\|_{L^2}\\
&\lesssim 
\frac{\ep}{(1+t)^{\frac{3}{2}+\sigma}}\|\ch_{q<0}w_3(q)^\frac{1}{2}\partial Z^{I_2} h\|_{L^2}
\end{align*}
if $\rho+\sigma \leq \frac{1}{4}$, and consequently
\begin{equation}
\label{7ene2}\|\ch_{q<0}w_1(q)^\frac{1}{2}\bar{\partial}Z^{I_1} g_{LL}\partial Z^{I_2} h\|_{L^2}\lesssim \frac{\ep^2}{(1+t)^{1+\sigma}}.
\end{equation}
If $I_2\leq \frac{N}{2}$ we estimate, thanks to \eqref{est1}
$$|\partial Z^{I_2}h|\leq \frac{\ep}{(1+|q|)^{\frac{3}{2}-\rho}}$$
therefore
$$
\|\ch_{q<0}w_1(q)^\frac{1}{2}\bar{\partial}Z^{I_1} g_{LL}\partial Z^{I_2} h\|_{L^2}\lesssim
\ep\left\|\frac{\ch_{q<0}}{(1+|q|)^{\frac{1}{4}+\frac{3}{2}-\rho}}\bar{\partial}Z^{I_1} g_{LL}\right\|_{L^2}.$$
We estimate
$$|\bar{\partial}Z^{I_1} g_{LL}|\lesssim \frac{1}{1+s}
|Z^{I_1+1}g_{LL}|\lesssim \frac{1}{(1+t)^{\frac{1}{2}+\sigma}(1+|q|)^{\frac{1}{2}-\sigma}}
|Z^{I_1+1}g_{LL}|.$$
We obtain
\begin{align*}\|\ch_{q<0}w_1(q)^\frac{1}{2}\bar{\partial}Z^{I_1} g_{LL}\partial Z^{I_2} h\|_{L^2}
&\lesssim \frac{\ep}{(1+t)^{\frac{1}{2}+\sigma}} \left\|\frac{\ch_{q<0}}{(1+|q|)^{2+\frac{1}{4}-\rho-\sigma}} Z^{ I_1+1}g_{LL}\right\|_{L^2}\\
&\lesssim \frac{\ep}{(1+t)^{\frac{1}{2}+\sigma}}\left\|\frac{\ch_{q<0}}{(1+|q|)^{1+\frac{1}{4}-\rho-\sigma}}
\partial Z^{ I_1+1}g_{LL}\right\|_{L^2}\\
&\lesssim  \frac{\ep}{(1+t)^{\frac{1}{2}+\sigma}}\left\|\ch_{q<0}w'_2(q)^\frac{1}{2}\bar{\partial} Z^{ I_1+1}\wht g_3\right\|_{L^2}
\end{align*}
where in the last inequality we have used the wave coordinate condition, and the fact that, 
since for $q<0$
$$w'_2(q)= \frac{1+2\mu}{(1+|q|)^{2+2\mu}},$$
we have
$$\frac{1}{(1+|q|)^{1+\frac{1}{4}-\rho-\sigma}}\leq w'_2(q)^\frac{1}{2}$$
if $\sigma + \rho + \mu \leq \frac{1}{4}$.
Therefore
\begin{equation}
\label{7ene3}\|\ch_{q<0}w_1(q)^\frac{1}{2}\bar{\partial}Z^{I_1} g_{LL}\partial Z^{I_2} h\|_{L^2}\|w_1(q)^\frac{1}{2}\partial Z^I \wht g_2\|
\lesssim \frac{\ep}{(1+t)^\sigma}\|\ch_{q<0}w'_2(q)^\frac{1}{2}\bar{\partial} Z^{ I_1+1}\wht g_3\|^2_{L^2} +\frac{\ep^3}{(1+t)^{1+\sigma}}.
\end{equation}In view of \eqref{7ene1}, \eqref{7ene2} and \eqref{7ene3}, we conclude the proof of Proposition \ref{l27}.
\end{proof}

\begin{proof}[Proof of Proposition \ref{l28}]
We have already proved
\begin{equation}
\label{8ene1}\left\|w(q)^ \frac{1}{2}\Up\left(\frac{r}{t}\right)\frac{1}{r^2}\partial_\theta Z^I (h_0+\wht h)\right\|_{L^2}\lesssim  \frac{\ep^2}{(1+t)^{1+\sigma}}.
\end{equation}
We now estimate $Z^I(\partial_U g_{LL}\partial_{\ba L}h)$. We have
$$\|w(q)^\frac{1}{2}Z^I(\partial_U g_{LL}\partial_{\ba L}h)\|_{L^2}\lesssim \sum_{I_1+I_2\leq I}
\|w(q)^\frac{1}{2}\bar{\partial}Z^{I_1} g_{LL}\partial Z^{I_2} h\|_{L^2}.$$ 
If $I_1\leq \frac{N}{2}$ we use the estimate
$$|\bar{\partial}Z^{I_1} g_{LL}|\lesssim \frac{1}{1+s}
| Z^{I_1+1}g_{LL}|\lesssim \frac{\ep (1+|q|)}{(1+s)^{\frac{5}{2}-\rho}}\lesssim\ep \frac{(1+|q|)^{\rho+\sigma}}{(1+t)^{\frac{3}{2}+\sigma}}.$$
Instead of estimating
$\|w(q)^\frac{1}{2}\bar{\partial}Z^{I_1} g_{LL}\partial Z^{I_2} h\|_{L^2}$ we estimate
$$\|w(q)^\frac{1}{2}\bar{\partial}Z^{I_1} g_{LL}\partial Z^{I_2} (h_0+\wht h)\|_{L^2} \; and \;  \|w(q)^\frac{1}{2}\bar{\partial}Z^{I_1} g_{LL}\partial Z^{I_2} \wht g_2|_{L^2}.$$
We can also estimate since $I_2+1\leq N-7$, thanks to \eqref{booth1} and \eqref{booth3}
$$\left|\partial Z^{I_2}(h_0+\wht h)\right|\lesssim \frac{\ep}{(1+|q|)^{\frac{3}{2}-\rho}}.$$
Therefore
\begin{align*}
\|\ch_{q<0}w_0(q)^\frac{1}{2}\bar{\partial}Z^{I_1} g_{LL}\partial Z^{I_2} h\|_{L^2}&\lesssim
\ep^2\left\|\ch_{q<0}\Up\left(\frac{r}{t}\right)\frac{(1+|q|)^{\rho+\sigma}}{(1+s)^{\frac{3}{2}+\sigma}(1+|q|)^{\frac{3}{2}-\rho}}\right\|_{L^2}\\
&\lesssim 
\frac{\ep^2}{(1+t)^{1+\sigma}}\left\|\ch_{q<0}\frac{1}{\sqrt{1+s}(1+|q|)^{\frac{3}{2}-2\rho-\sigma}}\right\|_{L^2}
\end{align*}
and consequently
\begin{equation}
\label{8ene2}\|w_0(q)^\frac{1}{2}\bar{\partial}Z^{I_1} g_{LL}\partial Z^{I_2} h\|_{L^2}\lesssim \frac{\ep^2}{(1+t)^{1+\sigma}}.
\end{equation}
We estimate also
\begin{equation}
\label{8ene3}\|w_0(q)^\frac{1}{2}\bar{\partial}Z^{I_1} g_{LL}\partial Z^{I_2} \wht g_2\|_{L^2}\lesssim \frac{\ep}{(1+t)^{\frac{3}{2}-\rho}}\|
w_0(q)^\frac{1}{2}\partial Z^{I_2} \wht g_2 \|_{L^2}\lesssim \frac{\ep^2}{(1+t)^{\frac{3}{2}-\rho}}.
\end{equation}
If $I_2\leq \frac{N}{2}$ we estimate, thanks to \eqref{est1}
$$|\partial Z^{I_2}h|\leq \frac{\ep}{(1+|q|)^{\frac{3}{2}-\rho}}$$
therefore
\begin{align*}
\|\ch_{q<0}w_0(q)^\frac{1}{2}\bar{\partial}Z^{I_1} g_{LL}\partial Z^{I_2} h\|_{L^2}&\lesssim
\ep\left\|\frac{\ch_{q<0}}{(1+|q|)^{\frac{3}{2}-\rho}}\bar{\partial}Z^{I_1} g_{LL}\right\|_{L^2}\\
&\lesssim \frac{\ep}{(1+t)^{\frac{1}{2}+\sigma}} \left\|\frac{\ch_{q<0}}{(1+|q|)^{2-\rho-\sigma}} Z^{ I_1+1}g_{LL}\right\|_{L^2}\\
&\lesssim \frac{\ep}{(1+t)^{\frac{1}{2}+\sigma}}\left\|\frac{\ch_{q<0}}{(1+|q|)^{1-\rho-\sigma}}
\partial Z^{ I_1+1}g_{LL}\right\|_{L^2}\\
&\lesssim  \frac{\ep}{(1+t)^{\frac{1}{2}+\sigma}}\|w_1'(q)^\frac{1}{2}\bar{\partial} Z^{ I_1+1}\wht g_2\|_{L^2}
\end{align*}
where in the last inequality we have used the wave coordinate condition, and the fact that, since for $q<0$
$$w_1'(q)= \frac{1}{2(1+|q|)^\frac{3}{2}},$$
we have
$$\frac{1}{(1+|q|)^{1-\rho-\sigma}}\lesssim w'_1(q)^\frac{1}{2},$$
if $\sigma + \rho  \leq \frac{1}{4}$. and $q<0$.
Therefore
\begin{equation}
\label{8ene4}\|\ch_{q<0} w_0(q)^\frac{1}{2}\bar{\partial}Z^{I_1} g_{LL}\partial Z^{I_2} h\|_{L^2}\|w_0(q)^\frac{1}{2}\partial Z^I \wht g_2\|
\lesssim \ep\|\ch_{q<0}w'_1(q)^\frac{1}{2}\bar{\partial} Z^{ I_1+1}\wht g_2\|^2_{L^2} +\frac{\ep^3}{(1+t)^{1+2\sigma}}.
\end{equation}
The estimates \eqref{8ene1}, \eqref{8ene2}, \eqref{8ene3} and \eqref{8ene4} conclude the proof of Proposition \ref{l28}.
\end{proof}

\section{Improvement of the estimates for $\Pi b$}\label{secproof}

%The initial data for $\phi$ are given.
%We have considered a time $T$ such that there exists $\wht b^{(1)}(\theta)$ 
%and initial data for $g$, $\lambda^{(1)}$ and $K^{(1)}$, solutions of the constraint equations with parameter $\wht b^{(1)}$ such that if we note $b^{(1)}=\wht b^{(1)} +\rho^{(1)} \cos(\theta-\eta^{(1)})+\alpha^{(1)}$ where $\rho^{(1)},\eta^{(1)}$ and $\alpha^{(1)}$ are given by the constraint equation, there exists a solution $(\phi^{(1)},\wht g^{(1)})$ of \eqref{s2} on $[0,T]$ which satisfies
%\begin{align*}
%&|\wht g^{(1)}|_{X_1}\lesssim \ep, \quad |\phi^{(1)}|_{X_2}\lesssim \ep\\
%& |\partial_\theta^I \left(\wht b^{(1)}(\theta)-\Pi\int_{C_T} (\partial_q \phi^{(1)})^2 r dq\right)|\lesssim \frac{\ep^2}{\sqrt{T}}, \;for\; I\leq N-4\\
%& \left\|\partial_\theta^I \left(\wht b^{(1)}(\theta)-\Pi\int_{C_T} (\partial_q \phi^{(1)})^2 r dq\right)\right\|_{L^2}\lesssim \frac{\ep^2}{\sqrt{T}}, \;for\; I\leq N
%\end{align*}
In order to conclude the proof of Theorem \ref{main},
it still remains to ameliorate the bootstrap assumptions \eqref{bootb1} and \eqref{bootb2}.
To this end, we will set
\begin{equation}
\label{defb2}\wht b^{(2)}(\theta)= \Pi \int_{\Sigma_{T,\theta}} (\partial_q \phi)^2rdq.
\end{equation}

\begin{prp}
\label{prpfin}
We assume that the time $T$ satisfies
$$T\leq \exp\left(\frac{C}{\sqrt{\ep}}\right).$$
There exists $(\phi^{(2)},g^{(2)})$ solution of \eqref{s1} in $[0,T]$ in the generalized wave coordinates $H_{b^{(2)}}$, such that, if we write
$g^{(2)}=g_{b^{2}}+\wht g$, then $(\phi^{(2)},\wht g^{(2)})$ satisfies the same estimate as $(\phi, \wht g)$, and 
we have the estimates for $ b^{(2)}$
\begin{align*}
& \left\|\partial_\theta^I \left(\Pi b^{(2)}(\theta)+\Pi\int_{\Sigma_{T,\theta}} (\partial_q \phi^{(2)})^2 r dq\right)\right\|_{L^2}\leq C
\frac{\ep^4}{\sqrt{T}}, \;for\; I\leq N-4,\\
& \|\partial_\theta^I b(\theta)\|_{L^2}\leq 2C_0^2 \ep^2 , \;for\; I\leq N.
\end{align*}
\end{prp}

The rest of this section is devoted to the proof of Proposition \ref{prpfin}.

We solve the constraint equations with parameter $\wht b^{(2)}$. 
The initial data we obtain, constructed in Theorem \ref{thinitial} are of the form
$$g=g_{b^{(2)}}+\wht g^{(2)}$$
where we write
$$b^{(2)}=\wht b^{(2)} +b_0^{(2)}+b_1^{(2)}\cos(\theta)+b_2^{(2)}\sin(\theta),$$
with $b_0^{(2)},b_1^{(2)},b_2^{(2)}$ given by Theorem \ref{thinitial}.
We have the following estimates for the initial data at $t=0$
$$\|\wht g-\wht g^{(2)}\|_{H^{N-3}_\delta}+\|\partial_t\wht g-\partial_t\wht g^{(2)}\|_{H^{N-4}_{\delta+1}}\lesssim
\|\wht b-\wht b^{(2)}\|_{W^{N-4,2}}\leq C_0'\frac{\ep^2}{\sqrt{T}},$$
thanks to \eqref{bootb1},
and
\[\|\wht g-\wht g^{(2)}\|_{H^{N+1}_\delta}+\|\partial_t\wht g-\partial_t\wht g^{(2)}\|_{H^{N}_{\delta+1}}\lesssim
\ep^2.
\]
 We solve, on an interval $[0,T_2]$, the system \eqref{s2} in generalized coordinates given by $g_{b^{(2)}}$. We note $(\phi^{(2)},\wht g^{(2)})$ the solution.
 
We want to estimate the difference between $(\phi^{(2)},\wht g^{(2)})$ and $(\phi,\wht g)$. However, it will not be possible to estimate the difference with the same norms than when we estimated $\phi$ and $\wht g$. When we estimated $h_0$ we were able to use the condition
\[\left|\wht b+\Pi \int_{\Sigma_{T,\theta}} (\partial_q \phi)^2\right|\lesssim \frac{\ep^2}{\sqrt{T}},\]
to obtain decay in $\frac{1}{\sqrt{1+|q|}}$ for $h_0$. However we want to keep the factor $\frac{1}{\sqrt{T}}$ in the estimates of the difference. To this end, we will loose the decay of $h_0-h^{(2)}_0$ in $\frac{1}{\sqrt{1+|q|}}$  and consequently in $\wht g-\wht g^{(2)}$.

We will prove Proposition \ref{prpfin} with a bootstrap argument.

\subsection{Bootstrap assumptions for $\phi^{(2)}-\phi$ and $\wht g^{(2)}-\wht g$}
\paragraph{$L^\infty$ estimates}
 First some $L^\infty$ estimates on $\phi-\phi^{(2)}$.
\begin{align}
\label{2bootphi1}&|Z^I(\phi-\phi^{(2)})|\leq\frac{2C_0\ep^2}{\sqrt{T}\sqrt{1+s}(1+q)^{\frac{1}{2}-2\kappa-5\rho}},\;for\;I\leq N-20,\\
\label{2bootphi2}&|Z^I(\phi-\phi^{(2)})|\leq \frac{2C_0\ep^2}{\sqrt{T}(1+s)^{\frac{1}{2}-2\kappa-2\rho}},\;for\;I\leq N-18.
\end{align}
We use the decompositions
\begin{equation}
\label{dec22}
g^{(2)}=g_{b^{(2)}} +\Up\left(\frac{r}{t}\right)(h^{(2)}_0+\wht h^{(2)})dq^2+\wht g^{(2)}_2,
\end{equation}
where $h^{(2)}_0$ satisfies the transport equation
\begin{equation*}
 \left\{ \begin{array}{l} 
         \partial_q h^{(2)}_0 =-2r\left(\partial_q \phi^{(2)}\right)^2-2b^{(2)}(\theta)\partial^2_q(\chi(q)q), \\
	 h^{(2)}_0|_{t=0}=0,
        \end{array}
\right.
\end{equation*}
and $\wht h^{(2)}$ satisfies the linear wave equation
\begin{equation*}
 \left\{ \begin{array}{l} 
         \Box \wht h^{(2)}= \Box h^{(2)}_0+g^{(2)}_{LL}\partial^2_q h^{(2)}_0+2\left(\partial_q \phi^{(2)}\right)^2 - 2(R_{b^{(2)}})_{qq}+\wht Q_{\ba L \ba L}(h^{(2)}_0,\wht g^{(2)}), \\
	 (\wht h^{(2)}, \partial_t \wht h^{(2)})|_{t=0}=(0,0),
        \end{array}
\right.
\end{equation*}
We assume the following estimates on $h_0-h_0^{(2)}$  for $I\leq N-12$
\begin{equation}
\label{2booth2}|Z^I (h_0 - h^{(2)}_0)|\leq 2C_0\frac{\ep^2}{\sqrt{T}}.
\end{equation}
We introduce the two weight modulators
$$\left\{ \begin{array}{l}
           \beta_1(q)= 1, \; q>0,\\
\beta_1(q)=\frac{1}{ (1+|q|)^{\kappa}}, \; q<0,
          \end{array}\right.$$
          and
$$\left\{ \begin{array}{l}
           \beta_2(q)= 1, \; q>0,\\
\beta_2(q)= \frac{1}{(1+|q|)^{2\kappa }}, \; q<0,
          \end{array}\right.$$
with $0<\kappa \ll 1$.
We assume for $I\leq N-15$

\begin{align}
\label{2bootg1}\left\|\beta_1 w_0^\frac{1}{2}\partial Z^I (\wht g_2-\wht g^{(2)}_2)\right\|_{L^2}&\leq \frac{2C_0\ep^2}{\sqrt{T}}(1+t)^\rho\\
\label{2bootg1bis} \left\|\alpha \beta_1 w^\frac{1}{2}\partial Z^I (\wht g_2-\wht g^{(2)}_2)\right\|_{L^2}&\leq \frac{2C_0\ep^2}{\sqrt{T}}
\end{align}
and 
\begin{align}
\label{2bootg2}\left\|\beta_2 w_1^\frac{1}{2}\partial Z^{N-14} \left(\wht g_2-\wht g^{(2)}_2\right)\right\|_{L^2}&\leq \frac{2C_0\ep^2}{\sqrt{T}}(1+t)^\rho\\
\label{2bootg2bis}\left\|\alpha \beta_2 w_1^\frac{1}{2}\partial Z^{N-14} \left(\wht g_2-\wht g^{(2)}_2\right)\right\|_{L^2}
&\leq \frac{2C_0\ep^2}{\sqrt{T}}.
\end{align}
We use the decomposition
\begin{equation}\label{dec23}
g^{(2)}=g_{b^{(2)}} +\Up\left(\frac{r}{t}\right)h^{(2)}dq^2+\wht g^{(2)}_3,
\end{equation}
where $h^{(2)}$ is the solution of
\begin{equation*}
 \left\{ \begin{array}{l} 
         \Box_{g^{(2)}} h^{(2)}= -2(\partial_q \phi^{(2)})^2 + 2(R_{b^{(2)}})_{qq}
         +Q_{\ba L \ba L}(h^{(2)},\wht g^{(2)}), \\
	 (h^{(2)}, \partial_t h^{(2)})|_{t=0}=(0,0) .
        \end{array}
\right.
\end{equation*}
We assume for $I\leq N-6$
\begin{equation}
\label{2bootl21}\left\|\alpha \beta_2w_0^\frac{1}{2} \partial Z^I \left(\phi-\phi^{(2)}\right)\right\|_{L^2}
+\left\|\alpha \beta_2 w_2^\frac{1}{2}\partial Z^I \left(\wht g_3-\wht g^{(2)}_3\right)\right\|_{L^2}
+\frac{1}{\sqrt{1+t}}\left\|\alpha \beta_2 w_3^\frac{1}{2}\partial Z^I \left( h-h^{(2)}\right)\right\|_{L^2}\leq \frac{2C_0\ep^2}{\sqrt{T}},
\end{equation}
and for $I\leq N-5$
\begin{equation}
\label{2bootl22}\left\|\alpha \beta_2w_0^\frac{1}{2} \partial Z^I \left(\phi-\phi^{(2)}\right)\right\|_{L^2}
+\left\|\alpha \beta_2 w_2^\frac{1}{2}\partial Z^I \left(\wht g_3-\wht g^{(2)}_3\right)\right\|_{L^2}
+\frac{1}{\sqrt{1+t}}\left\|\alpha \beta_2 w_3^\frac{1}{2}\partial Z^I \left( h-h^{(2)}\right)\right\|_{L^2}\leq \frac{2C_0\ep^2}{\sqrt{T}}(1+t)^{\rho}.
\end{equation}
We use the decomposition
$$
g^{(2)}=g_{b^{(2)}}+\Up\left(\frac{r}{t}\right)h^{(2)} dq^2+
\Up\left(\frac{r}{t}\right)k^{(2)} rdqd\theta+
\wht g^{(2)}_4,
$$
where $k^{(2)}$ is the solution of
\begin{equation*}
 \left\{ \begin{array}{l} 
         \Box_g k^{(2)}= \partial_U g^{(2)}_{LL}\partial_q h^{(2)}, \\
	 (h^{(2)}, \partial_t h^{(2)})|_{t=0}=(0,0).
        \end{array}
\right.
\end{equation*}
We assume for $I\leq N-4$
\begin{equation}\label{2bootl23}\begin{split}
&\left\| \alpha_2\beta_2w_0^\frac{1}{2} \partial Z^I \left(\phi-\phi^{(2)}\right)\right\|_{L^2}
+\left\|\alpha_2 \beta_2 w_2^\frac{1}{2}\partial Z^I \left(\wht g_3-\wht g^{(2)}_3\right)\right\|_{L^2}\\
+&\frac{1}{\sqrt{1+t}}\left\| \alpha_2\beta_2 w_3^\frac{1}{2}\partial Z^I \left( h-h^{(2)}\right)\right\|_{L^2}
+\frac{1}{\sqrt{1+t}}\left\| \alpha_2\beta_2 w_3^\frac{1}{2}\partial Z^I \left(k-k^{(2)}\right)\right\|_{L^2}
\leq \frac{2C_0\ep^2}{\sqrt{T}}(1+t)^{\rho}.
\end{split}
\end{equation}

To improve the estimates,
we follow the same steps than when we ameliorated the bootstrap assumptions of Section \ref{boot}.
The difference of our new bootstrap assumptions compared with the estimates of Section \ref{boot} is at worse a factor $\frac{\ep\sqrt{1+|q|}}{\sqrt{T}}$ in the region $q<0$. In the region $q>0$ the decay is the same and we have won a factor $\frac{\ep}{\sqrt{T}}$. Therefore we can restrict our study to the region $q<0$: we will perform our estimates as if no term was present in the region $q>0$. We will follow the same steps as before, but with much less details since the mechanisms are the same. 

\begin{rk}
As long as the bootstrap estimates for $\phi^{(2)}-\phi$ and $\wht g^{(2)}-\wht g$ are satisfied, $\phi^{(2)}$ and $\wht g^{(2)}$ satisfy the same estimates as $\phi$ and $\wht g$.
\end{rk}

\paragraph{$L^\infty$ estimates using the weighted Klainerman-Sobolev inequality}
The following estimates are a direct consequence of the bootstrap assumptions and the weighted Klainerman-Sobolev inequality. 
For $I\leq N-8$ we have
\begin{align}
\label{2ksphi}\left|\partial Z^I \left(\phi^{(2)}-\phi\right)\right|&\lesssim \frac{\ep^2}{\sqrt{T}\sqrt{1+t}(1+|q|)^{\frac{1}{2}-2\kappa}},\\
\label{2ksg}\left|\partial Z^I \left(\wht g_3^{(2)}-\wht g_3\right)\right|&\lesssim \frac{\ep^2(1+|q|)^{\frac{1}{2}+\mu+2\kappa}}{\sqrt{T}\sqrt{1+s}},\\
\label{2ksh}\left|\partial Z^I \left(h^{(2)}-h\right)\right|&\lesssim \frac{\ep^2}{\sqrt{T}(1+|q|)^{\frac{1}{2}-2\kappa}},
\end{align}
and for $I\leq N-17$
\begin{equation}
\label{ks15}\left|\partial Z^I \left(\wht g_2^{(2)}-\wht g_2\right)\right|\lesssim \frac{\ep^2(1+|q|)^\kappa}{\sqrt{T}\sqrt{1+s}\sqrt{1+|q|}}.
\end{equation}

\subsection{Improvement of the estimate of $h_0-h_0^{(2)}$ and $\wht h^{(2)}-\wht h_0$}
\paragraph{Estimate of $h_0-h_0^{(2)}$}
The quantity $h_0-h_0^{(2)}$ satisfies the transport equation
\begin{equation*}
 \left\{ \begin{array}{l} 
         \partial_q\left( h^{(2)}_0-h_0\right) =-2r\left(\left(\partial_q \phi^{(2)}\right)^2-\left(\partial_q \phi\right)^2\right)-2\left(b^{(2)}(\theta)-b(\theta)\right)\partial^2_q(\chi(q)q), \\
	 (h^{(2)}_0-h_0)|_{t=0}=0.
        \end{array}
\right.
\end{equation*}
We write this equation under the form
\[
\partial_q\left( h^{(2)}_0-h_0\right)=-2r\left(\partial_q \phi^{(2)}+\partial_q \phi\right)\left(\partial_q \phi^{(2)}-\partial_q \phi\right)-2\left(b^{(2)}(\theta)-b(\theta)\right)\partial^2_q(\chi(q)q).
\]
For $k+l\leq N-7$, $k\geq 1$, the equivalent of estimate \eqref{kl}, that we obtain using \eqref{2bootphi1} and \eqref{2ksphi} to estimate
$\partial(\phi-\phi{(2)})$ and \eqref{bootphi1} and \eqref{ks1} to estimate
$\partial(\phi+\phi{(2)})$ corresponds to \eqref{kl} 
multiplied by  $\frac{\ep\sqrt{1+|q|}}{\sqrt{T}}$. 
\[\left|\partial^k_q \partial^l_\theta \left(h_0-h_0^{(2)}\right)\right|\lesssim \frac{\ep^3}{\sqrt{T}(1+|q|)^{k+\frac{1}{2}-4\rho}}.\]
We obtain the estimate for $k=0$ by integrating the previous one with respect to $q$. We obtain, for $l\leq N-8$
\[|\partial^l_\theta h_0|\lesssim \frac{\ep^3}{\sqrt{T}}.\]
For $k+l+j\leq N-8$, $k\geq 1$ and $j\geq 1$ the equivalent of \eqref{jkl} is
\[\left|\partial^{j}_s \partial^k_q \partial^l_\theta \left(h_0-h^{(2)}_0\right)\right|
\lesssim \frac{\ep^3}{\sqrt{T}(1+s)^{j+\frac{1}{2}}(1+|q|)^{k-4\rho}}.\]
Consequently we have proved that for $I\leq N-8$ we have
\begin{equation}
\label{amelioh}
\left|Z^I \left(h_0-h_0^{(2)}\right)\right|\lesssim \frac{\ep^3}{\sqrt{T}}. 
\end{equation}

\paragraph{Estimation of $\wht h^{(2)}-\wht h_0$}
The quantity $\wht h^{(2)}-\wht h_0$ satisfies the linear equation
\begin{equation*}
 \left\{ \begin{array}{l} \begin{split}
         \Box \left(\wht h^{(2)}-\wht h\right)= 
        &\Box \left(h^{(2)}_0-h_0\right)+2\left(\left(\partial_q \phi^{(2)}\right)^2-(\partial_q\phi)^2\right) - 2(R_{b^{(2)}})_{qq}+2(R_b)_{qq}\\
        &+g^{(2)}_{LL}\partial^2_q h^{(2)}_0-g_{LL}\partial_q^2 h_0+\wht Q_{\ba L \ba L}(h^{(2)}_0,\wht g^{(2)})-\wht Q_{\ba L \ba L}(h_0,\wht g),
         \end{split}\\
	 \left(\wht h^{(2)}-\wht h, \partial_t \left(\wht h^{(2)}-\wht h\right)\right)|_{t=0}=(0,0).
        \end{array}
\right.
\end{equation*}
Proceeding as for the estimate of \eqref{waveh}, and in view of the bootstrap assumptions for $\phi-\phi{(2)}$ and $\wht g-\wht g^{(2)}$ we obtain the analogue of \eqref{waveh} for  $\Box \left(Z^I\wht h^{(2)}-Z^I\wht h\right)$, where the corresponding right-hand side gets multiplied by 
$\frac{\ep\sqrt{1+|q|}}{\sqrt{T}}$. We obtain, for $I\leq N-10$ and $q<0$
\[\left|\Box \left(Z^I\wht h^{(2)}-Z^I\wht h\right)\right|\lesssim \frac{\ep^3}{\sqrt{T}(1+s)^\frac{3}{2}(1+|q|)^\frac{1}{2}}
.\]
Therefore if we perform the weighted energy estimate we obtain
\[\frac{d}{dt}\left\|w^\frac{1}{2} \partial\left(Z^I\wht h^{(2)}-Z^I\wht h\right)\right\|_{L^2}
\lesssim \left\|\frac{\ep^3}{\sqrt{T}(1+s)^\frac{3}{2}(1+|q|)^\frac{1}{2}}\right\|_{L^2}\lesssim \frac{\ep^3\ln(1+t)}{\sqrt{T}(1+t)},
\]
and therefore for $I\leq N-10$ we have
\begin{equation}
\label{l2htilde}\left\|w_0^\frac{1}{2} \partial\left(Z^I\wht h^{(2)}-Z^I\wht h\right)\right\|_{L^2}\lesssim
\frac{\ep^3}{\sqrt{T}}(1+t)^\rho.
\end{equation}
The weighted Klainerman-Sobolev inequality yields, for $I\leq N-12$
\begin{equation}
\label{esthtilde}
\left|\partial\left(Z^I\wht h^{(2)}-Z^I\wht h\right)\right|\lesssim
\frac{\ep^3(1+t)^\rho}{\sqrt{T}\sqrt{1+s}\sqrt{1+|q|}}.
\end{equation}

\subsection{Improvement of the $L^\infty$ estimate for $\phi-\phi^{(2)}$}

We write the equation satisfied by $\phi^{(2)}-\phi$ 
\[\Box_{g}\left(\phi-\phi^{(2)}\right)=\left(\left(g^{(2)}\right)^{\alpha\beta}-g^{\alpha \beta}\right)\partial_\alpha \partial_\beta \phi^{(2)}+\left(H_{b^{(2)}}-H_{b}\right)^\rho\partial_\rho \phi^{(2)}.\]
We limit ourselves to the region $q<0$.
We estimate for $I+J\leq N-20$
\[Z^I\left(g^{(2)}_{LL}-g_{LL}\right)Z^J\partial^2\phi.\]
With the wave coordinate condition and the estimate \eqref{ks15}, we obtain, for $I\leq N-17$
\begin{equation}
\label{2wc1}\left|Z^I\left(g^{(2)}_{LL}-g_{LL}\right)\right|\lesssim \frac{\ep^2(1+|q|)^{\frac{3}{2}+\kappa}}{(1+s)^\frac{3}{2}}.
\end{equation}
Moreover we have, for $J\leq N-20$ thanks to \eqref{bootphi1}
\[|Z^J \partial^2 \phi|\lesssim \frac{1}{(1+|q|)^2}|Z^{J+2}\phi|\lesssim \frac{\ep^2}{\sqrt{T}\sqrt{1+s}(1+|q|)^{\frac{5}{2}-4\rho}}.\]
Consequently
\[\left|Z^I\left(g^{(2)}_{LL}-g_{LL}\right)Z^J\partial^2\phi\right|
\lesssim \frac{\ep^3}{\sqrt{T}(1+s)^2(1+|q|)^{1-4\rho-\kappa}}
\lesssim \frac{\ep^3}{\sqrt{T}(1+s)^{2-5\rho-\kappa}(1+|q|)^{1+\rho}}.\]
We now estimate for $I+J\leq N-20$
\[Z^I g_{LL}Z^J\partial^2\left(\phi-\phi^{(2)}\right).\]
We have, thanks to \eqref{wc1} and \eqref{2bootphi2}
\[|Z^I g_{LL}|\lesssim \frac{\ep(1+|q|)}{(1+s)^{\frac{3}{2}-\rho}},\]
\[\left|Z^J\partial^2\left(\phi-\phi^{(2)}\right)\right|\lesssim \frac{1}{(1+|q|)^2}\left|Z^{J+2}\left(\phi-\phi^{(2)}\right)\right|\lesssim \frac{\ep^2}{\sqrt{T}(1+s)^{\frac{1}{2}-2\rho-2\kappa}(1+|q|)^2}.\]
Consequently
\[\left|Z^I g_{LL}Z^J\partial^2\left(\phi-\phi^{(2)}\right)\right|\lesssim \frac{\ep^3}{(1+s)^{2-5\rho-2\kappa}(1+|q|)^{1+\rho}}\]
and the $L^\infty- L^\infty$ estimate yields for $I\leq N-20$, since the initial data for $\phi-\phi^{(2)}$ are zero.
\begin{equation}\label{phiN}
\left|Z^I\left( \phi-\phi^{(2)}\right)\right|\leq \frac{ C\ep^3}{\sqrt{T}(1+s)^\frac{1}{2}(1+|q|)^{\frac{1}{2}-5\rho-2\kappa}}.
\end{equation}
We now estimate for $I+J\leq N-18$, thanks to \eqref{2wc1} an \eqref{bootphi1} for the first inequality, and \eqref{wc1} and \eqref{2ksphi} for the second inequality
\[\left|Z^I\left(g^{(2)}_{LL}-g_{LL}\right)Z^J\partial^2\phi\right|\lesssim
\left(\frac{\ep^2(1+|q|)^{\frac{3}{2}+\kappa}}{\sqrt{T}(1+s)^\frac{3}{2}}\right)
\left(\frac{\ep}{(1+|q|)^{\frac{5}{2}-4\rho}(1+s)^{\frac{1}{2}}}\right)
\lesssim
\frac{\ep^3}{\sqrt{T}(1+s)^{2-5\rho-\kappa}(1+|q|)^{1+\rho}},
\]
\[\left|Z^I g_{LL}Z^J\partial^2\left(\phi-\phi^{(2)}\right)\right|
\lesssim \left(\frac{\ep(1+|q|)}{(1+s)^{\frac{3}{2}-\rho}}\right)\left(\frac{\ep^2(1+|q|)^{2\kappa}}{\sqrt{T}(1+|q|)^\frac{3}{2}\sqrt{1+s}}\right)
\lesssim \frac{\ep^3}{\sqrt{T}(1+s)^{2-\rho}(1+|q|)^{\frac{1}{2}-2\kappa}}.
\]
Consequently, for $I\leq N-18$ and $q<0$ we have
\[\left|\Box Z^I\left(\phi-\phi^{(2)}\right)\right|\lesssim \frac{\ep^3}{\sqrt{T}(1+s)^{\frac{3}{2}-\rho-2\kappa}(1+|q|)}\]
and Lemma \ref{linf2} yields, for $I\leq  N-18$, since the initial data for $\phi-\phi^{(2)}$ are zero.
\begin{equation}
\label{phiN2}
\left| Z^I\left(\phi-\phi^{(2)}\right)\right|\leq
 \frac{C\ep^3}{\sqrt{T}(1+s)^{\frac{1}{2}-2\rho-2\kappa}}.
\end{equation}

\subsection{$L^2$ estimates}
\paragraph{$L^2$ estimate for $\partial Z^I\left(\wht g_2^{(2)}-\wht g_2\right)$ with $I\leq N-15$}
We have
\[
\Box_g\left((\wht g_2)_{\mu \nu}-\left(\wht g^{(2)}_2\right)_{\mu \nu}\right)=f_{\mu \nu},
\]
where the terms in $f_{\mu \nu}$ are
\begin{itemize}
\item the terms coming from the non commutation of the null decomposition with the wave operator: it is sufficient to study the term 
$\Up\left(\frac{r}{t}\right)\frac{1}{r^2}\partial_\theta \left(h^{(2)}_0-h_0+\wht h^{(2)}-\wht h\right)$,
\item the semi-linear terms: it is sufficient to study
$\partial_{\ba L} (h)\partial_U \left(g_{LL}-g^{(2)}_{LL}\right)$ and $\partial_U g_{LL}\partial_{\ba L} \left(h^{(2)}-h\right)$,
\item the quasilinear terms: it is sufficient to study the terms
$g_{LL}\partial^2_{\ba L} \left( \wht g^{(2)}-\wht g\right)$ and $\left( g^{(2)}_{LL}-g_{LL}\right)\partial^2_{\ba L} \wht g$,
\item the crossed terms: they do not occur in the region $q<0$.
\end{itemize}

We estimate the first term. Thanks to \eqref{esthtilde} and \eqref{amelioh} we have, for $I\leq N-15$
\[\left|\partial_\theta Z^I \left(h^{(2)}_0-h_0+\wht h^{(2)}-\wht h\right)\right|\lesssim \frac{\ep^3(1+|q|)^\rho}{\sqrt{T}}.\]
Therefore, we can estimate in the region $q<0$,
$$
\left\|\beta_1w_0^\frac{1}{2}\Up\left(\frac{r}{t}\right)\frac{1}{r^2}\partial_\theta Z^I \left(h^{(2)}_0-h_0+\wht h^{(2)}-\wht h\right)\right\|_{L^2}\lesssim \left\|\Up\left(\frac{r}{t}\right)\frac{\ep^3}{\sqrt{T}(1+s)^2(1+|q|)^{\kappa-\rho}}\right\|_{L^2}$$
and consequently
\begin{equation}
\label{he1}\left\|\beta_1w_0^\frac{1}{2}\Up\left(\frac{r}{t}\right)\frac{1}{r^2}\partial_\theta Z^I \left(h^{(2)}_0-h_0+\wht h^{(2)}-\wht h\right)\right\|_{L^2}\lesssim  \frac{\ep^3}{\sqrt{T}(1+t)^{1+\kappa-\rho}}.
\end{equation}

We now estimate the semi-linear terms. For $I\leq N-13$, we have, thanks to \eqref{est2} 
\[|\partial_{\ba L} (Z^I h)| \lesssim \frac{\ep}{(1+|q|)^{\frac{3}{2}-2\rho}},\]
Therefore we can estimate, for $I+J\leq N-15$ in the region $q<0$
\begin{align*}
\left\|\beta_1w_0^\frac{1}{2}Z^I\partial_{\ba L} (h)Z^J\partial_U \left(g_{LL}-g^{(2)}_{LL}\right)\right\|_{L^2}
&\lesssim \left\|\frac{\ep}{(1+|q|)^{\frac{3}{2}-2\rho+\kappa}(1+s)}Z^{J+1}\left(g_{LL}-g^{(2)}_{LL}\right)\right\|_{L^2}\\
&\lesssim \left \|\frac{\ep}{(1+|q|)^{\frac{1}{2}-2\rho+\kappa}(1+s)}\partial Z^{J+1}\left(g_{LL}-g^{(2)}_{LL}\right)\right\|_{L^2}\\
&\lesssim \frac{\ep}{(1+t)^{\frac{1}{2}+\sigma}}\left\|\frac{1}{(1+|q|)^{1-2\rho+\kappa-\sigma}}\bar{\partial}Z^{J+1}\left(\wht g_2-\wht g^{(2)}_2\right)\right\|_{L^2}
\end{align*}
and consequently
\begin{equation}
\label{he2}\left\|\beta_1w_0^\frac{1}{2}Z^I\partial_{\ba L} (h)Z^J\partial_U \left(g_{LL}-g^{(2)}_{LL}\right)\right\|_{L^2}\lesssim\frac{\ep}{(1+t)^{\frac{1}{2}+\sigma}}\left\|\beta_2w_1'(q)^\frac{1}{2}\bar{\partial}Z^{J+1}\left(\wht g_2-\wht g^{(2)}_2\right)\right\|_{L^2},
\end{equation}
where we have used the wave coordinate condition and the fact that, for $q<0$
$$\beta_2w_1'(q)^\frac{1}{2}=\frac{1}{4(1+|q|)^{{3}{4}+2\kappa}}\geq
\frac{1}{(1+|q|)^{1-2\rho+\kappa-\sigma}}.$$
For $I\leq N-14$ thanks to Proposition \ref{linfg2}, we have
\[|Z^I\partial_U g_{LL}|\lesssim \frac{\ep(1+|q|)}{(1+s)^{\frac{5}{2}-2\rho}}.
\]
In order to estimate
$$\left\|\beta_1w_0^\frac{1}{2}\partial_U Z^I g_{LL}\partial_{\ba L} Z^J \left(h^{(2)}-h\right)\right\|_{L^2}$$
we will perform the estimates with $\left(h^{(2)}-h\right)$ replaced by
$\left(h^{(2)}_0-h_0\right)$, $\left(\wht h^{(2)}-\wht h\right)$ and
$\left(\wht g_2^{(2)}-\wht g_2\right)$.  
We estimate, in the region $q<0$, thanks to \eqref{amelioh},
\begin{align*}
\left\|\beta_1w_0^\frac{1}{2}\partial_U Z^I g_{LL}\partial_{\ba L} Z^J \left(h_0^{(2)}-h_0\right)\right\|_{L^2}&\lesssim\frac{\ep^3}{\sqrt{T}}
\left\|\frac{\ep(1+|q|)^{1-\kappa}}{(1+s)^{\frac{5}{2}-2\rho}(1+|q|)}\right\|_{L^2}\\
&\lesssim \frac{\ep^3}{\sqrt{T}(1+t)^\frac{3}{2}},
\end{align*}
thanks to \eqref{l2htilde}
\begin{align*}
\left\|\beta_1w_0^\frac{1}{2}\partial_U Z^I g_{LL}\partial_{\ba L} Z^J \left(\wht h^{(2)}-\wht h\right)\right\|_{L^2}&\lesssim\ep
\left\|\frac{\ep(1+|q|)^{1-\kappa}}{(1+s)^{\frac{5}{2}-2\rho}}\partial_{\ba L} Z^J \left(\wht h^{(2)}-\wht h\right)\right\|_{L^2}\\
&\lesssim \frac{\ep^3}{\sqrt{T}(1+t)^{\frac{3}{2}+\kappa-2\rho}},
\end{align*}
and thanks to \eqref{2bootg1}
\begin{align*}
\left\|\beta_1w_0^\frac{1}{2}\partial_U Z^I g_{LL}\partial_{\ba L} Z^J \left(\wht g_2^{(2)}-\wht g_2\right)\right\|_{L^2}&\lesssim\ep
\left\|\frac{\ep(1+|q|)^{1-\kappa}}{(1+s)^{\frac{5}{2}-2\rho}}\partial_{\ba L} Z^J \left(\wht g^{(2)}-\wht g\right)\right\|_{L^2}\\
&\lesssim \frac{\ep^3}{\sqrt{T}(1+t)^{\frac{3}{2}-\rho}}.
\end{align*}
Consequently, we have
\begin{equation}
\label{he3}\left\|\beta_1w_0^\frac{1}{2}\partial_U Z^I g_{LL}\partial_{\ba L} Z^J \left(h^{(2)}-h\right)\right\|_{L^2}\lesssim \frac{\ep^3}{\sqrt{T}(1+t)^{\frac{3}{2}-\rho}}.
\end{equation}

The other terms are similar to estimate. Thanks to \eqref{he1}, \eqref{he2} and \eqref{he3},
the energy inequality yields for $I\leq N-15$
\begin{equation}\label{en13}\begin{split}
&\frac{d}{dt}\left\|\beta_1w_0^\frac{1}{2}\partial Z^I \left(\wht g^{(2)}_2-\wht g_2\right)\right\|^2_{L^2}
+\left\|\beta_1w_0'(q)^\frac{1}{2}\bar{\partial} Z^I \left(\wht g^{(2)}_2-\wht g_2\right)\right\|^2_{L^2}\\
\lesssim& \frac{\ep^3}{\sqrt{T}(1+t)^{1+\sigma}}
\left\|\beta_1w_0^\frac{1}{2}\partial Z^I \left(\wht g^{(2)}_2-\wht g_2\right)\right\|_{L^2}+
\frac{\ep}{(1+t)^{\sigma}}\left\|\beta_2w_1'(q)^\frac{1}{2}\bar{\partial}Z^{I+1}\left(\wht g_2-\wht g^{(2)}_2\right)\right\|^2_{L^2}\\
&+\frac{\ep}{(1+t)^{1+\sigma}}
\left\|\beta_1w_0^\frac{1}{2}\partial Z^I \left(\wht g^{(2)}_2-\wht g_2\right)\right\|^2_{L^2}.
\end{split}
\end{equation}

\paragraph{$L^2$ estimate for $\partial Z^I\left(\wht g_2^{(2)}-\wht g_2\right)$ with $I\leq N-14$.}
We follow the same steps as in the previous paragraph. First we still have
\begin{align*}
\left\|\beta_2w_1^\frac{1}{2}\Up\left(\frac{r}{t}\right)\frac{1}{r^2}\partial_\theta Z^I \left(h^{(2)}_0-h_0+\wht h^{(2)}-\wht h\right)\right\|_{L^2}&\lesssim \left\|\Up\left(\frac{r}{t}\right)\frac{\ep^3}{\sqrt{T}(1+s)^2(1+|q|)^{\frac{1}{4}+2\kappa-\rho}}\right\|_{L^2}\\
&\lesssim \frac{\ep^3}{\sqrt{T}(1+t)^{\frac{5}{4}}} .
\end{align*}
We estimate the second terms for $I+J\leq N-14$ 
\begin{align*}
\left\|\beta_2w_1^\frac{1}{2}Z^I\partial_{\ba L} (h)Z^J\partial_U \left(g_{LL}-g^{(2)}_{LL}\right)\right\|_{L^2}
&\lesssim \left\|\frac{\ep}{(1+|q|)^{\frac{3}{2}-2\rho+\frac{1}{4}+2\kappa}(1+s)}Z^{J+1}\left(g_{LL}-g^{(2)}_{LL}\right)\right\|_{L^2}\\
&\lesssim \left \|\frac{\ep}{(1+|q|)^{\frac{3}{4}-2\rho+2\kappa}(1+s)}\partial Z^{J+1}\left(g_{LL}-g^{(2)}_{LL}\right)\right\|_{L^2}\\
&\lesssim \frac{\ep}{(1+t)^{\frac{1}{2}+\sigma}}\left\|\frac{1}{(1+|q|)^{\frac{5}{4}-2\rho+2\kappa-\sigma}}\bar{\partial}Z^{J+1}\left(\wht g_3-\wht g_3^{(2)}\right)\right\|_{L^2}\\
&\lesssim \frac{\ep}{(1+t)^{\frac{1}{2}+\sigma}}\left\|\beta_2w'_2(q)^\frac{1}{2}\bar{\partial}Z^{J+1}\left(\wht g_3-\wht g_3^{(2)}\right)\right\|_{L^2},
\end{align*}
where we have used the wave coordinate condition and the fact that
$$\beta_2w_2'(q)^\frac{1}{2}=\frac{1}{(1+|q|)^{1+2\kappa+\mu}}\geq
\frac{1}{(1+|q|)^{\frac{5}{4}-2\rho+2\kappa-\sigma}}.$$ 

The other terms are similar to estimate than for $I\leq N-15$. The energy inequality yields for $I\leq N-14$
\begin{equation}\label{en12}\begin{split}
&\frac{d}{dt}\left\|\beta_2w_1^\frac{1}{2}\partial Z^I \left(\wht g^{(2)}_2-\wht g_2\right)\right\|^2_{L^2}
+\left\|\beta_2w'_1(q)^\frac{1}{2}\bar{\partial} Z^I \left(\wht g^{(2)}_2-\wht g_2\right)\right\|^2_{L^2}\\
\lesssim& \frac{\ep^3}{\sqrt{T}(1+t)^{1+\sigma}}
\left\|\beta_2w_1^\frac{1}{2}\partial Z^I \left(\wht g^{(2)}_2-\wht g_2\right)\right\|_{L^2}+
\frac{\ep}{(1+t)^{\sigma}}\left\|\beta_2w_2'(q)^\frac{1}{2}\bar{\partial}Z^{I+1}\left(\wht g_3-\wht g^{(2)}_3\right)\right\|^2_{L^2}\\
&+\frac{\ep}{(1+t)^{1+\sigma}}
\left\|\beta_2w_1^\frac{1}{2}\partial Z^I \left(\wht g^{(2)}_2-\wht g_2\right)\right\|^2_{L^2}.
\end{split}
\end{equation}

\paragraph{$L^2$ estimates for $\partial Z^I \left(\phi^{(2)}-\phi\right)$ with $I\leq N-6$.}
We estimate for $I+J\leq N-6$, $J\leq N-7$,
\[\left\| \beta_2w_0^\frac{1}{2}Z^I\left(g^{(2)}_{LL}-g_{LL}\right)\partial_q^2Z^J \phi\right\|_{L^2}.\]
If $I\leq \frac{N-7}{2}$ we can estimate, thanks to \eqref{2wc1}
\[\left|Z^I\left(g^{(2)}_{LL}-g_{LL}\right)\right|\lesssim \frac{\ep^2(1+|q|)^{\frac{3}{2}+\kappa}}{\sqrt{T}(1+s)^{\frac{3}{2}}},\]
and therefore, if we restrict our quantities to $q<0$
\begin{align*}
\left\| \beta_2w_0^\frac{1}{2}Z^I\left(g^{(2)}_{LL}-g_{LL}\right)\partial_q^2Z^J \phi\right\|_{L^2}&\lesssim
\left\| \frac{\ep^2(1+|q|)^{\frac{3}{2}+\kappa}}{\sqrt{T}(1+s)^\frac{3}{2}(1+|q|)^{1+2\kappa}}\partial Z^{J+1} \phi\right\|_{L^2}\\
&\lesssim \frac{\ep^2}{\sqrt{T}(1+t)^{1+\kappa}}\left \| w^\frac{1}{2}\partial Z^{J+1} \phi\right\|_{L^2}\\
&\lesssim \frac{\ep^3}{\sqrt{T}(1+t)^{1+\kappa}}.
\end{align*} 
The case $J\leq \frac{N-6}{2}$ can be treated as in  Section \ref{subsecl2}.

We now evaluate
\[\left\| \beta_2w_0^\frac{1}{2}Z^Ig_{LL}\partial_q^2Z^J \left(\phi^{(2)}-\phi\right)\right\|_{L^2}\]
for $I+J\leq N-6$ and $J\leq \frac{N-6}{2}$. We have, since $\frac{N-6}{2}+2\leq N-20$
\[\left|\partial_q^2Z^J \left(\phi^{(2)}-\phi\right)\right|\lesssim \frac{\ep^2}{\sqrt{T}\sqrt{1+s}(1+|q|)^{\frac{5}{2}-5\rho-2\kappa}}.\]
Therefore we can estimate
\begin{align*}
\left\| \beta_2w_0^\frac{1}{2}Z^Ig_{LL}\partial_q^2Z^J \left(\phi^{(2)}-\phi\right)\right\|_{L^2}
&\lesssim \left\| \frac{\ep^2}{\sqrt{T}\sqrt{1+s}(1+|q|)^{\frac{5}{2}-5\rho}}Z^Ig_{LL}\right\|_{L^2}\\
&\lesssim \left\| \frac{\ep^2}{\sqrt{T}\sqrt{1+s}(1+|q|)^{\frac{3}{2}-5\rho}}\partial Z^Ig_{LL}\right\|_{L^2}\\
&\lesssim \left\| \frac{\ep^2}{\sqrt{T}(1+s)^\frac{3}{2}(1+|q|)^{\frac{3}{2}-5\rho}}Z^{I+1}\wht g_3\right\|_{L^2}\\
&\lesssim  \left\| \frac{\ep^2}{\sqrt{T}(1+s)^\frac{3}{2}(1+|q|)^{\frac{1}{2}-5\rho}}\partial Z^{I+1}\wht g_3\right\|_{L^2}\\
&\lesssim \frac{\ep^2}{\sqrt{T}(1+t)^{\frac{3}{2}-5\rho-\mu}}\left\|w_2^\frac{1}{2}\partial Z^{I+1}\wht g_3\right\|_{L^2}
\lesssim \frac{\ep^3}{\sqrt{T}(1+t)^{\frac{3}{2}-5\rho-\mu}}.
\end{align*}
The case $I\leq \frac{N-6}{2}$ can be treated similarly than in  Section \ref{subsecl2}.
The weighted energy estimate yields
\begin{equation}\label{enphi}
\frac{d}{dt}\left\|\beta_2w_0^\frac{1}{2}\partial Z^I \left(\phi^{(2)}_2-\phi\right)\right\|^2_{L^2}
+\left\|\beta_2w_0'(q)^\frac{1}{2}\bar{\partial} Z^I \left(\phi^{(2)}_2-\phi\right)\right\|^2_{L^2}\\
\lesssim \frac{\ep^3}{\sqrt{T}(1+t)^{1+\kappa}}\left\|\beta_2w^\frac{1}{2}\partial Z^I \left(\phi^{(2)}_2-\phi\right)\right\|_{L^2}.
\end{equation}
Consequently, since the initial data for $\phi^{(2)}_2-\phi$ are zero we have
\begin{equation}
\label{l2phi}
\left\|\beta_2w_0^\frac{1}{2}\partial Z^I \left(\phi^{(2)}_2-\phi\right)\right\|_{L^2}\lesssim \frac{\ep^3}{\sqrt{T}}.
\end{equation}
\paragraph{$L^2$ estimates for $\partial Z^I \left(h^{(2)}-h\right)$ with $I\leq N-6$.}
We write the equation satisfied by $h^{(2)}-h$
\begin{equation*}
\Box_{g}\left(h- h^{(2)}\right)= 2(\partial_q \phi^{(2)})^2-2(\partial_q \phi)^2 +2(R_{b})_{qq}- 2(R_{b^{(2)}})_{qq}
         +Q_{\ba L\ba L}(h,\wht g)-Q_{\ba L \ba L}(h^{(2)},\wht g^{(2)}).
\end{equation*}
We first estimate for $I+J\leq N-6$ and $I\leq \frac{N-6}{2}$. We recall that we restrict all the quantities to $q<0$ (therefore $w_3=w_0$).
\begin{align*}
\left\|\beta_2w_0^\frac{1}{2}\partial_q Z^I\left(\phi-\phi^{(2)}\right)\partial_q Z^J \phi\right\|_{L^2}
\lesssim \left\|\frac{\ep^2}{\sqrt{T}\sqrt{1+s}(1+|q|)^{\frac{3}{2}-5\rho}}\partial_q Z^J \phi\right\|_{L^2}\\
\lesssim \frac{\ep^2}{\sqrt{T}\sqrt{1+t}}\|w_0^\frac{1}{2}\partial_q Z^J \phi\|_{L^2}\lesssim \frac{\ep^3}{\sqrt{T}\sqrt{1+t}}.
\end{align*}
We now estimate the quasilinear term
\[\left\| \beta_2w_0^\frac{1}{2}Z^I\left(g^{(2)}_{LL}-g_{LL}\right)\partial_q^2Z^J h\right\|_{L^2}\]
for $I+J\leq N-6$ and $I\leq \frac{N-6}{2}$. We have
\begin{align*}
\left\| \beta_2w_0^\frac{1}{2}Z^I\left(g^{(2)}_{LL}-g_{LL}\right)\partial_q^2Z^J h\right\|_{L^2}&\lesssim
\left\| \frac{\ep^2(1+|q|)^{\frac{3}{2}+\kappa}}{\sqrt{T}(1+s)^\frac{3}{2}(1+|q|)^{1+2\kappa}}\partial Z^{J+1} h\right\|_{L^2}\\
&\lesssim \frac{\ep^2}{\sqrt{T}(1+t)^{1+\kappa}}\left \| w_0^\frac{1}{2}\partial Z^{J+1} h\right\|_{L^2}\\
&\lesssim \frac{\ep^3}{\sqrt{T}(1+t)^{\frac{1}{2}+\kappa}}.
\end{align*} 
The other terms can be treated as in the proof of Proposition \ref{prphn}.
The energy inequality yields
\begin{equation*}
\frac{d}{dt}\left\|\beta_2w_0^\frac{1}{2}\partial Z^I \left(h^{(2)}_2-h\right)\right\|^2_{L^2}
+\left\|\beta_2w_0'(q)^\frac{1}{2}\bar{\partial} Z^I \left(h^{(2)}_2-h\right)\right\|^2_{L^2}\\
\lesssim \frac{\ep^3}{\sqrt{T}(1+t)^{\frac{1}{2}}}\left\|\beta_2w_0^\frac{1}{2}\partial Z^I \left(h^{(2)}_2-h\right)\right\|_{L^2}.
\end{equation*}
Consequently, since the initial data for $h^{(2)}_2-h$ are zero, we have
\begin{equation}
\label{l2h}
\left\|\beta_2w_0^\frac{1}{2}\partial Z^I \left(h^{(2)}_2-h\right)\right\|_{L^2}\lesssim \frac{\ep^3\sqrt{1+t}}{\sqrt{T}}.
\end{equation}
\paragraph{$L^2$ estimates for $\partial Z^I \left(\wht g^{(2)}-\wht g \right)$ with $I\leq N-6$.}
As usual we estimate the following contributions
\begin{itemize}
\item the terms coming from the non commutation of the null decomposition with the wave operator: it is sufficient to study the term 
$\Up\left(\frac{r}{t}\right)\frac{1}{r^2}\partial_\theta \left(h^{(2)}-h\right)$,
\item the semi-linear terms: it is sufficient to study
$\partial_{\ba L} (h)\partial_U \left(g_{LL}-g^{(2)}_{LL}\right)$ and $\partial_U g_{LL}\partial_{\ba L} \left(h^{(2)}-h\right)$,
\item the quasilinear terms: it is sufficient to study the terms
$g_{LL}\partial^2_{\ba L} \left( \wht g^{(2)}-\wht g\right)$ and $\left( g^{(2)}_{LL}-g_{LL}\right)\partial^2_{\ba L} \wht g$.
\end{itemize}
We estimate the first term.  We recall that we restrict all the quantities to $q<0$.
\begin{align*}
\left\|\beta_2w_2^\frac{1}{2}\chi\left(\frac{r}{t}\right)\frac{1}{r^2}\partial_\theta Z^I \left(h^{(2)}-h\right)\right\|_{L^2}&\lesssim
\left\|\beta_2\frac{1}{(1+s)^2(1+|q|)^{\frac{1}{2}+\mu}} Z^{I+1} \left(h^{(2)}-h\right)\right\|_{L^2}\\
&\lesssim \left\|\beta_2\frac{1}{(1+s)^{\frac{3}{2}+\sigma}(1+|q|)^{\mu-\sigma}} \partial Z^{I+1} \left(h^{(2)}-h\right)\right\|_{L^2}\\
&\lesssim \frac{1}{(1+t)^{\frac{3}{2}+\sigma}}\left\|\beta_2w_0^\frac{1}{2}\partial Z^{I+1} \left(h^{(2)}-h\right)\right\|_{L^2}.
\end{align*}
We estimate the second term
\[\left\|\beta_2w_1^\frac{1}{2}Z^I\partial_{\ba L} (h)Z^J\partial_U \left(g_{LL}-g^{(2)}_{LL}\right)\right\|_{L^2}\]
for $I+J\leq N-6$ and $J\leq \frac{N-6}{2}$. We have, thanks to \eqref{2wc1}
\[\left|Z^J\partial_U \left(g_{LL}-g^{(2)}_{LL}\right)\right|\lesssim \frac{1}{1+s}\left|Z^{J+1} \left(g_{LL}-g^{(2)}_{LL}\right)\right|\lesssim \frac{\ep^2(1+|q|)^{\frac{3}{2}+\kappa}}{\sqrt{T}(1+s)^{\frac{5}{2}}}.\]
Therefore we can estimate
\begin{align*}
\left\|\beta_2w_2^\frac{1}{2}Z^I\partial_{\ba L} (h)Z^J\partial_U \left(g_{LL}-g^{(2)}_{LL}\right)\right\|_{L^2}
&\lesssim \left\| \frac{\ep^2(1+|q|)^{\frac{3}{2}+\kappa}}{\sqrt{T}(1+s)^{\frac{5}{2}}(1+|q|)^{\frac{1}{2}+\mu +2\kappa}}\partial_{\ba L} Z^I h\right\|_{L^2}\\
&\lesssim \left\| \frac{\ep^2}{\sqrt{T}(1+s)^{\frac{3}{2}+\mu+\kappa}}\partial_{\ba L} Z^I h \right\|_{L^2}\\
&\lesssim \frac{\ep^2}{\sqrt{T}(1+t)^{\frac{3}{2}+\mu+\kappa}}\|w^\frac{1}{2}\partial_{\ba L} Z^I h\|_{L^2}
\lesssim \frac{\ep^3}{\sqrt{T}(1+t)^{1+\mu +\kappa}}.
\end{align*}
The other terms are treated as in the proof of Proposition \ref{prpwhtg}. We have proved, when we restrict ourselves to $q<0$
\begin{equation}\label{eng3}\begin{split}
&\frac{d}{dt}\left\|\beta_2w_2^\frac{1}{2}\partial Z^I \left(\wht g^{(2)}_3-\wht g_3\right)\right\|^2_{L^2}
+\left\|\beta_2w'_2(q)^\frac{1}{2}\bar{\partial} Z^I \left(\wht g^{(2)}_3-\wht g_3\right)\right\|^2_{L^2}\\
\lesssim& \Bigg(\frac{\ep^3}{\sqrt{T}(1+t)^{1+\sigma}}+\frac{\ep}{(1+t)^{\frac{1}{2}+\sigma}}\left\|\beta_2w'_2(q)^\frac{1}{2}\bar{\partial}Z^{J+1}\left(\wht g_4-\wht g^{(2)}_4\right)\right\|_{L^2}\\
&+ \frac{1}{(1+t)^{\frac{3}{2}+\sigma}}\left\|\beta_2w^\frac{1}{2}\partial Z^{I+1} \left(h^{(2)}-h\right)\right\|_{L^2}
\Bigg)\left\|\beta_2w_2^\frac{1}{2}\partial Z^I \left(\wht g^{(2)}_3-\wht g_3\right)\right\|_{L^2}.
\end{split}
\end{equation}

\paragraph{$L^2$ estimates for $I\leq N-4$}
We can prove, following Section \ref{subsecl21} that, since we do as if no quantity was present for $q>0$,
\begin{equation}\label{total}
\begin{split}
&\frac{d}{dt}\Bigg(\left\| \beta_2w^\frac{1}{2} \partial Z^I \left(\phi-\phi^{(2)}\right)\right\|^2_{L^2}
+\left\| \beta_2 w_2^\frac{1}{2}\partial Z^I \left(\wht g_4-\wht g^{(2)}_4\right)\right\|^2_{L^2}\\
+&\frac{1}{\ep(1+t)}\left\| \beta_2 w_3^\frac{1}{2}\partial Z^I \left( h-h^{(2)}\right)\right\|^2_{L^2}
+\frac{1}{\ep(1+t)}\left\| \beta_2 w_3^\frac{1}{2}\partial Z^I \left(k-k^{(2)}\right)\right\|^2_{L^2}\Bigg)\\
+&\left\| \beta_2w_0'(q)^\frac{1}{2} \bar{\partial} Z^I \left(\phi-\phi^{(2)}\right)\right\|^2_{L^2}
+\left\| \beta_2 w'_2(q)^\frac{1}{2}(q)'\bar{\partial} Z^I \left(\wht g_4-\wht g^{(2)}_4\right)\right\|^2_{L^2}\\
+&\frac{1}{\ep(1+t)}\left\| \beta_2 w'_3(q)^\frac{1}{2}\bar{\partial} Z^I \left( h-h^{(2)}\right)\right\|^2_{L^2}
+\frac{1}{\ep(1+t)}\left\| \beta_2 w'_3(q)^\frac{1}{2}\bar{\partial} Z^I \left(k-k^{(2)}\right)\right\|^2_{L^2}.\\
&\lesssim \frac{\sqrt{\ep}}{1+t}\Bigg(\left\| \beta_2w^\frac{1}{2} \partial Z^I \left(\phi-\phi^{(2)}\right)\right\|^2_{L^2}
+\left\| \beta_2 w_2^\frac{1}{2}\partial Z^I \left(\wht g_4-\wht g^{(2)}_4\right)\right\|^2_{L^2}\\
+&\frac{1}{\ep(1+t)}\left\| \beta_2 w_3^\frac{1}{2}\partial Z^I \left( h-h^{(2)}\right)\right\|^2_{L^2}
+\frac{1}{\ep(1+t)}\left\| \beta_2 w_3^\frac{1}{2}\partial Z^I \left(k-k^{(2)}\right)\right\|^2_{L^2}\Bigg)
+O\left(\frac{\ep^\frac{9}{2}}{T(1+t)}\right)
\end{split}
\end{equation}

\subsection{Conclusion of the proof of Proposition \ref{prpfin}}
Estimate \eqref{amelioh} gives us for $I\leq N-8$
\[\left|Z^I \left(h_0-h_0^{(2)}\right)\right|\leq \frac{C\ep^3}{\sqrt{T}}. \]
Estimate \eqref{phiN} gives us for $I\leq N-18$
\[\left|Z^I\left( \phi-\phi^{(2)}\right)\right|\leq \frac{C\ep^3}{\sqrt{T}(1+s)^\frac{1}{2}(1+|q|)^{\frac{1}{2}-5\rho-2\kappa}}.\]
Estimate \eqref{phiN2} gives us for $I\leq N-16$
\[\left| Z^I\left(\phi-\phi^{(2)}\right)\right|\leq \frac{C\ep^3}{\sqrt{T}(1+s)^{\frac{1}{2}-2\rho-2\kappa}}.\]
Therefore, if $C\ep \leq C_0$ we have ameliorated the $L^\infty$ estimates
\eqref{2booth2}, \eqref{2bootphi1} and \eqref{2bootphi2}.
Estimate \eqref{total} implies, following the proof of Proposition \ref{estn},
\begin{align*}
&\left\| \beta_2w^\frac{1}{2} \partial Z^I \left(\phi-\phi^{(2)}\right)\right\|_{L^2}
+\left\| \beta_2 w_2^\frac{1}{2}\partial Z^I \left(\wht g_3-\wht g^{(2)}_3\right)\right\|_{L^2}\\
+&\frac{1}{\sqrt{\ep(1+t)}}\left\| \beta_2 w_3^\frac{1}{2}\partial Z^I \left( h-h^{(2)}\right)\right\|_{L^2}
+\frac{1}{\sqrt{\ep(1+t)}}\left\| \beta_2 w_3^\frac{1}{2}\partial Z^I \left(k-k^{(2)}\right)\right\|_{L^2}\\
&\leq\frac{1}{\sqrt{T}}\left(C_0\ep^2+\ep^2 \right)(1+t)^{C\sqrt{\ep}}.
\end{align*}
Therefore, if we had chosen $C_0\geq 2$ and $C\sqrt{\ep}\leq \rho$ we have ameliorated this estimate \eqref{2bootl23} and \eqref{2bootl22}. Moreover we have
\[\left\| \beta_2 w_3^\frac{1}{2}\partial Z^I \left( h-h^{(2)}\right)\right\|_{L^2}\lesssim \frac{\ep^\frac{5}{2}}{\sqrt{T}}
(1+t)^{\frac{1}{2}+C\sqrt{\ep}}.\]
Estimate \eqref{total} also implies
\[\int_0^t \frac{1}{(1+t)^\sigma}\left\| \beta_2 w_2'(q)^\frac{1}{2}\bar{\partial} Z^I \left(\wht g_4-\wht g^{(2)}_4\right)\right\|^2_{L^2} \lesssim \frac{\ep^4}{T}\]
and consequently, estimate \eqref{eng3}, together with the bootstrap assumption \eqref{2bootl21} yields
\begin{equation*}\begin{split}
&\frac{d}{dt}\left\|\beta_2w_2^\frac{1}{2}\partial Z^I \left(\wht g^{(2)}_3-\wht g_3\right)\right\|^2_{L^2}
+\left\|\beta_2w'_2(q)^\frac{1}{2}\bar{\partial} Z^I \left(\wht g^{(2)}_3-\wht g_3\right)\right\|^2_{L^2}\\
\lesssim&
\frac{\ep^{\frac{9}{2}}}{T(1+t)^{1+\sigma}}
+\frac{\ep}{(1+t)^\sigma}\left\|(\beta_2w_2^\frac{1}{2})'\bar{\partial}Z^{J+1}\left(\wht g_4-\wht g^{(2)}_4\right)\right\|^ 2_{L^2},
\end{split}
\end{equation*}
Therefore, when we integrate we obtain
\begin{align*}
&\left\|\beta_2w_2^\frac{1}{2}\partial Z^I \left(\wht g^{(2)}_3-\wht g_3\right)\right\|^2_{L^2}
+\int_0^t \left\|\left(\beta_2w_2^\frac{1}{2}\right)'\bar{\partial} Z^I \left(\wht g^{(2)}_3-\wht g_3\right)\right\|^2_{L^2}\\
\lesssim& \frac{C_0^2\ep^4}{T}+C^2\frac{\ep^{\frac{9}{2}}}{T}.
\end{align*}
Therefore, for $C\ep^\frac{1}{2}\leq \frac{C_0}{2}$, this, together with \eqref{l2h} and \eqref{l2phi} improve the estimate \eqref{2bootl21}.
We proceed in the same way to ameliorate the remaining estimates, using  \eqref{en12} and \eqref{en13}.
Consequently, the solution $(\phi^{(2)}, \wht g^{(2)})$ exists in $[0,T]$ and we have the following estimate for $\phi-\phi^{(2)}$
\begin{align}
\label{phi1}|Z^I(\phi-\phi^{(2)})|&\leq\frac{C\ep^3}{\sqrt{T}\sqrt{1+s}(1+q)^{\frac{1}{2}-2\kappa-5\rho}},\;for\;I\leq N-20\\
\label{phi2}\|\alpha_2w_0^\frac{1}{2}\partial Z^I(\phi-\phi^{(2)})\|&\leq \frac{C\ep^3(1+t)^{C\sqrt{\ep}}}{\sqrt{T}}, \; for \; I\leq N-4.
\end{align}

We now go to the amelioration of the estimate for $\wht b$. In view of the definition \eqref{defb2} of $\wht b^{(2)}$ we have for $I\leq N-4$.
\begin{align*}
& \partial_\theta^I\left(\wht b^{(2)}(\theta)-\Pi\int_{\Sigma_{T,\theta}} (\partial_q \phi^{(2)})^2 r dr\right)\\
=&\partial_\theta^I\Pi\int_{\Sigma_{T,\theta}}\left((\partial_q \phi)^2-(\partial_q \phi^{(2)})^2 \right)r dr\\
=&\sum_{I_1+I_2\leq J}\Pi\int_{\Sigma_{T,\theta}}\partial_\theta^{I_1}(\partial_q \phi+\partial_q \phi^{(2)})\partial_\theta^{I_2}(\partial_q \phi-\partial_q \phi^{(2)})rdr.
\end{align*}
We estimate, for $I_1\leq \frac{N}{2}$
\begin{align*}
&\left\|\int_{\Sigma_{T,\theta}}\partial_\theta^{I_1}(\partial_q \phi+\partial_q \phi^{(2)})\partial_\theta^{I_2}(\partial_q \phi-\partial_q \phi^{(2)})rdr\right\|_{L^2(\m S^1)}\\
&\lesssim 
\int_{\Sigma_{T,\theta}}\frac{\ep}{\sqrt{1+s}(1+|q|)^{\frac{3}{2}-4\rho}}\left\|\partial_\theta^{I_2}(\partial_q \phi-\partial_q \phi^{(2)})\right \|_{L^2(\m S^1)}rdr\\
&\lesssim \left(\int_0^\infty \frac{\ep^2}{(1+s)(1+|q|)^{3-8\rho-4\kappa}}rdr\right)^\frac{1}{2}\left\|\beta_2\partial_\theta^{I_2}(\partial_q \phi-\partial_q \phi^{(2)})\right\|_{L^2}
\end{align*}
Then the estimate \eqref{phi2}, with the condition
$(1+T)^{C\sqrt{\ep}}\leq 1$
 yields for $I\leq N-4$
\begin{equation}
\label{enfin}
\left|\partial_\theta^I\left(\wht b^{(2)}(\theta)-\Pi\int_{\Sigma_{T,\theta}} (\partial_q \phi^{(2)})^2 r dr\right)\right|_{L^2(\m S^1)}\lesssim \frac{\ep^4}{\sqrt{T}}.
\end{equation}
The case $I_2\leq \frac{N}{2}$ can be treated similarly thanks to \eqref{phi1}.
To conclude, we estimate
\begin{align*}
\left\|\partial_\theta^I \wht b^{(2)}\right\|_{L^2(\m S^2)}
&=\left\|\int_0^\infty \sum_{I_1+I_2= I}\partial_q \partial^{I_1}_\theta \phi \partial_q \partial^{I_2}_\theta \phi rdr\right\|_{L^2(\m S^1)}\\
&\leq \int_0^\infty \frac{C_0\ep}{\sqrt{1+s}(1+|q|)^{\frac{3}{2}-4\rho}}\|\partial_q \partial^{I_2}_\theta \phi\|_{L^2(\m S^1)}rdr\\
&\leq \left(\int \frac{C_0^2\ep^2}{(1+s)(1+|q|)^{3-8\rho}}rdr\right)^\frac{1}{2}\|\partial_q \partial^{I_2}_\theta \phi\|_{L^2}\\
&\leq 2C_0^2\ep^2
\end{align*}
where we have used again
$(1+T)^{C\sqrt{\ep}}\leq 1$.
 This concludes the proof of Proposition \ref{prpfin}, and the proof of Theorem \ref{main}.

\appendix

\section{Reduction of the Einstein equations}\label{reduc}
 we recall the form of the Einstein equations in the presence of a space-like translational 
Killing field. We follow here the exposition
in \cite{livrecb}.
A metric $^{(4)}\mathbf{g}$ on  $\mathbb{R}^2 \times \m R  \times \mathbb{R}$ admitting $\partial_3$ as a Killing field can be written
$$^{(4)}\mathbf{g} = \grat g + e^{2\gamma}(dx^3 +A_\alpha dx^\alpha)^2, $$
where $\grat g$ is a Lorentzian metric on $\mathbb{R}^{1+2}$, $\gamma$ is a scalar function on $\mathbb{R}^{1+2}$, $A$ is a $1$-form on
$\mathbb{R}^{1+2}$ and $x^\alpha$, $\alpha=0,1,2$, are the coordinates on $\mathbb{R}^{1+2}$. Since $\partial_3$ is a Killing field, $\mathbf{g}$, $\gamma$ and $A$ do not depend on $x^3$. The polarized case consists in choosing $A=0$.
Let  $^{(4)}\mathbf{R}_{\mu \nu}$ denote the Ricci tensor associated to $^{(4)}\mathbf{g}$.
$\grat R_{\alpha \beta}$ and $\grat D$ are respectively the Ricci tensor and the covariant derivative associated to $\grat g$.

With this metric, the vacuum Einstein equations
$$^{(4)}\mathbf{R}_{\mu \nu} = 0, \; \mu, \nu = 0,1,2,3$$
can be written in the basis $(dx^\alpha, dx^3+A_\alpha dx^\alpha)$  (see \cite{livrecb} appendix VII)
\begin{align}
\label{rab}0= ^{(4)}\gra R_{\alpha \beta} &= \grat R_{\alpha \beta}
 -\grat D_\alpha \partial_\beta \gamma -\partial_\alpha \gamma \partial_\beta \gamma,\\
\label{r33} 0=^{(4)}\gra R_{33}&= -e^{-2\gamma} \left( \grat g^{\alpha \beta}\partial_\alpha \gamma \partial_\beta \gamma
 + \grat g^{\alpha \beta}\grat D_\alpha \partial_\beta \gamma\right),
\end{align} 
and the equation $ 0=^{(4)}\gra R_{\alpha 3}$ is automatically satisfied. By doing the conformal change of metric $\grat g = e^{-2\gamma}\gra g$, 
 \eqref{rab} and \eqref{r33}, yield the following system,

\begin{align*}
& \Box_\mathbf{g} \gamma  = 0, \\
 &\mathbf{R}_{\alpha \beta} = 2\partial_\alpha \gamma\partial_\beta \gamma  \; \alpha,\beta = 0,1,2.
\end{align*}
By setting $\phi= \sqrt{2}\gamma$ we obtain the system \ref{s1}.

\section{Construction of the initial data}\label{apini}
Theorem \ref{thinitial} is a consequence of the following result on the constraint equations, proved in \cite{expcont}.
The method of solving is inspired from the conformal method in three dimension. 
We look for space-like metrics $\bar{g}$ of the form $\bar{g}= e^{2\lambda}\delta$. 
We introduce the traceless part of $K$, 
$$H_{ij} = K_{ij} - \frac{1}{2}\tau \bar{g}_{ij},$$
and the following rescaling
$$\dot{\phi} = \frac{e^{\lambda}}{N}\partial_0 u, \quad \breve{H}=e^{-\lambda}H, \quad \breve{\tau} = e^{\lambda}\tau.$$
We also introduce the notation
$$M_\theta = \left(\begin{array}{ll}\cos(2\theta)&\sin(2\theta)\\\sin(2\theta)&-\cos(2\theta)\end{array}\right), \quad N_\theta=\left(\begin{array}{ll}-\sin(2\theta)&\cos(2\theta)\\\cos(2\theta)&\sin(2\theta)\end{array}\right).$$
\begin{thm}\label{contrainte}
Let $0<\delta<1$. Let $\dot{\phi}^2, |\nabla \phi|^2 \in H^{N-1}_{\delta+2}$ and $\bar{ b}\in W^{N,2}(\m S^1)$ such that
$$\int_{\m S^1} \bar{ b}(\theta)\cos(\theta)d\theta=\int_{\m S^1}  \bar{b}(\theta)\sin(\theta)d\theta =0.$$
We note
$$\varepsilon^2 = \int \dot{\phi}^2 +|\nabla \phi|^2.$$
We assume 
$$\|\dot{\phi}^2\| _{H^{N-1}_{\delta+2}} + \||\nabla \phi|^2\|_{H^{N-1}_{\delta+2}} +\|\bar{ b}\|_{W^{N,2}}\lesssim \varepsilon^2.$$
Let $B\in W^{N,2}(\m S^1)$. We assume
$$\|B\|_{W^{N,2}}\lesssim \varepsilon^4.$$
Let $\Psi \in H^{N+1}_{\delta+1}$ be such that 
$\int \Psi = 2\pi.$
If $\ep>0$ is small enough, there exist $\alpha,\rho,\eta,A,J,c_1,c_2$ in $\m R$, a scalar function $\wht \lambda \in H^{N+1}_{\delta}$ and a symmetric traceless tensor $\wht H \in H^{N}_{\delta+1}$ such that, if
$r,\theta$ are the polar coordinates centered in $c_1,c_2$, and if we note
\begin{align*}
\lambda&= -\alpha \chi(r)\ln(r)+\wht \lambda,\\
\breve{H}&=-(\bar{b}(\theta)+\rho\cos(\theta-\eta))\frac{\chi(r)}{2r}M_\theta + e^{-\lambda}\frac{\chi(r)}{r^2}\left((J-(1-\alpha)B(\theta))N_\theta -\frac{B'(\theta)}{2}M_\theta \right) + \wht H,
\end{align*}
then $\lambda, e^{\lambda}\breve{H}$ are solutions of the constraint equations with
$$\breve{\tau}=(\bar{ b}(\theta)+\rho \cos(\theta-\eta))\frac{\chi(r)}{r}+e^{-\lambda} B'(\theta)\frac{\chi(r)}{r^2}+A\psi.$$
Moreover we have the estimates
\begin{align*}
\alpha&=\frac{1}{4\pi}\int \left(\dot{\phi}^2+|\nabla \phi|^2\right) +O(\ep^4),\\
\rho \cos(\eta)&=\frac{1}{\pi}\int \dot{\phi}\partial_1 \phi +O(\ep^4),\\
\rho \sin(\eta)&=\frac{1}{\pi}\int \dot{\phi}\partial_2 \phi +O(\ep^4),\\
c_1&=-\frac{1}{4\pi}\int x_1\left(\dot{\phi}^2+|\nabla \phi|^2\right) +O(\ep^4),\\
c_2&=-\frac{1}{4\pi}\int x_2\left(\dot{\phi}^2+|\nabla \phi|^2\right) +O(\ep^4),\\
J&=-\frac{1}{2\pi}\int \dot{\phi}\partial_\theta \phi+\frac{\rho }{\alpha}(c_2\cos(\eta)-c_1\sin(\eta))+O(\ep^4),\\
A&=-\frac{1}{2\pi}\int \dot{\phi}r\partial_r \phi
+\frac{1}{2\pi}\left(\int \chi'(r)rdr\right)\int \bar{b}(\theta)d\theta + O(\ep^4),\\
\end{align*}
and
$$\|\wht \lambda\|_{H^{N+1}_\delta}+\|\wht H\|_{H^N_{\delta+1}} \lesssim \ep^2.$$
\end{thm}
We will use the notation
\begin{equation}
\label{defb1}
b^{(1)}=\rho \cos(\theta-\eta)+\bar{b}(\theta).
\end{equation}
The end of this section is devoted to the proof of Theorem \ref{thinitial}.
\begin{lm}\label{lmgas}
The second fundamental form of the space-time metric
\begin{equation}
\label{gas}
g_{a}=-dt^2-2Jdtd\theta+r^{-2\alpha}(dr^2+(r-b^{(1)}(\theta)r^{\alpha}t)^2d\theta^2)
-2B'(\theta)td\theta^2+4(1-\alpha)B(\theta)\frac{t}{r}drd\theta.
\end{equation}
is given at $t=0$ by
$$K_{ij}=H_{ij}+\frac{1}{2}(g_{a})_{ij}\tau,$$
with
$$\tau=r^{\alpha}\frac{b^{(1)}(\theta)}{r}+r^{2\alpha}\frac{B'(\theta)}{r},$$
$$H=-r^{-\alpha}b^{(1)}(\theta)\frac{\chi(r)}{2r}M_\theta + (J-(1-\alpha)B(\theta))\frac{\chi(r)}{r^2}N_\theta-B'(\theta)\frac{\chi(r)}{2r^2}M_\theta.$$
\end{lm}
\begin{proof}[Proof of Lemma \ref{lmgas}]
The metric induced by $g_{a}$ on the space-like hypersurface $t=0$ is $r^{-2\alpha}\delta$. The shift is given by
$\beta_\theta=-J$ and the lapse is given by $N=1$. Therefore we calculate
$$K_{ij}=-\frac{1}{2N}(\partial_t \bar{g}_{ij}-\partial_i \beta_j-\partial_j \beta_i).$$
We infer
\begin{align*}
K_{11}=&-\frac{1}{2}\left(-\left(\frac{2r^{-\alpha}b^{(1)}(\theta)}{r}+\frac{2B'(\theta)}{r^2}\right)\sin^2(\theta)
-\frac{4(1-\alpha)B(\theta)}{r^2}\cos(\theta)\sin(\theta)
+\frac{2J}{r^2}\cos(\theta)\sin(\theta)\right),\\
K_{22}=&-\frac{1}{2}\left(-\left(\frac{2r^{-\alpha}b^{(1)}(\theta)}{r}+\frac{2B'(\theta)}{r^2}\right)\cos^2(\theta)
+\frac{4(1-\alpha)B(\theta)}{r^2}\cos(\theta)\sin(\theta)
-\frac{2J}{r^2}\cos(\theta)\sin(\theta)\right),\\
K_{12}=&-\frac{1}{2}\Bigg(\left(\frac{2r^{-\alpha}b^{(1)}(\theta)}{r}+\frac{2B'(\theta)}{r^2}\right)\cos(\theta)\sin(\theta)
+\frac{4(1-\alpha)B(\theta)}{2r^2}(\cos^2(\theta)-\sin^2(\theta))\\
&-\frac{2J}{r^2}(\cos^2(\theta)-\sin^2(\theta))\Bigg).\\
\end{align*}
We calculate
$$\tau = \bar{g}^{ij}K_{ij}=r^{\alpha}\frac{b^{(1)}(\theta)}{r}+r^{2\alpha}\frac{B'(\theta)}{r^2}
$$
so we obtain exactly
$$H=-r^{-\alpha}b^ {(1)}(\theta)\frac{\chi(r)}{2r}M_\theta + (J-(1-\alpha)B(\theta))\frac{\chi(r)}{r^2}N_\theta-B'(\theta)\frac{\chi(r)}{2r^2}M_\theta.$$
\end{proof}

\begin{lm}
\label{lmiso}
The metric $g_a$, defined by \eqref{gas} is isometric to 
$g_b +g^{(1)}$
where at $t=0$ we have
$$g^{(1)}=O\left(\frac{1}{r^2}\right), \quad \partial_t g^{(1)}=O\left(\frac{1}{r^2}\right),$$
and $g_b$ is defined by \eqref{gb}, 
where
\begin{align}
\label{defb}b(\theta)&=\frac{b^{(1)}(F(\theta))}{1-\alpha-b^{(1)}(F(\theta))},\\
\label{defJ}J(\theta)&=2JF'(\theta),
\end{align}
with $F$ the inverse function of 
$$\theta \mapsto \theta+\int_0^\theta (\alpha-b^{(1)}(\theta')d\theta';$$
provided the following relations hold
\begin{align}
\label{rel1}\alpha&=-\int \bar{b}(\theta),\\
\label{rel2}B(\theta)&=\frac{Jb^{(1)}(\theta)}{1-\alpha}.
\end{align}
\end{lm}
\begin{proof}
During all the proof, the notation $g\sim g'$ stands for $g$ is isometric to $g'+\wht g$ where $\wht g = O\left(\frac{1}{r}\right)$ and
$\partial_t \wht g= O\left(\frac{1}{r^3}\right).$
 In polar coordinate $r,\theta$, this means neglecting the metric terms of the form
$$\frac{dr^2}{r^2}, \quad \frac{d\theta dr}{r}, \quad d\theta^2,$$
$$\frac{tdr^2}{r^3}, \quad \frac{td\theta dr}{r^2}, \quad \frac{td\theta^2}{r}.$$
We  perform some changes of variable in $g_{as}$. First of all we introduce $r'$ such that 
$$r'=\frac{r^{1-\alpha}}{1-\alpha}, \quad dr'=r^{-\alpha}dr.$$
The metric $g_{a}$ becomes
$$
g_{a}\sim-dt^2-2Jdtd\theta+(dr)^2+(r(1-\alpha)-b^{(1)}(\theta)t)^2d\theta^2
-2B'(\theta)td\theta^2+4B(\theta)\frac{t}{r}drd\theta,
$$
where we keep writing $r$ instead of $r'$.
We now make the change of variable
$$\theta=\theta'-\frac{J}{(1-\alpha)^2r},\quad d\theta=d\theta'+\frac{J}{(1-\alpha)^2r^2}.$$
Since we will neglect the contributions to the metric decaying like $\frac{1}{r^2}$ we obtain
$$d\theta^2\sim (d\theta')^2 +2\frac{J}{(1-\alpha)^2r^2}d\theta'dr, \quad b^{(1)}(\theta)\sim b(\theta')-b'(\theta')\frac{J}{r(1-\alpha)^2}.$$
We keep also writing $\theta$ instead of $\theta'$. We infer
\begin{align*}
g_{a}\sim& -dt^2-2J(dt-dr)d\theta+dr^2+(r(1-\alpha)-b^{(1)}(\theta)t)^2d\theta^2\\
&+\left(2\frac{Jb'(\theta)}{(1-\alpha)}-2B'(\theta)\right)td\theta^2+\left(-4b^{(1)}(\theta)\frac{J}{1-\alpha}+4B(\theta)\right)\frac{t}{r}drd\theta.
\end{align*}
We choose 
$$B(\theta)=\frac{Jb^{(1)}(\theta)}{1-\alpha}.$$ 
With this choice we obtain
\begin{align*}
g_{a}\sim&-dt^2-2J(dt-dr)d\theta+dr^2+(r(1-\alpha)-b^{(1)}(\theta)t)^2d\theta^2\\
\sim&-dt^2-2J(dt-dr)d\theta+dr^2+(r-(b^{(1)}(\theta)+\alpha)r+b^{(1)}(\theta)(r-t))^2d\theta^2.
\end{align*}
We impose 
$$\alpha=-\int b^{(1)}(\theta)=-\int \bar{b}(\theta).$$
 Therefore we can find $f(\theta)$ such that
$$f'(\theta)=-(b^{(1)}(\theta)+\alpha).$$
 We perform the change of variable
$$\theta'=\theta+f(\theta).$$
We note $F$ the inverse function of
$$\theta \mapsto \theta+f(\theta),$$
so that $\theta=F(\theta').$
Then $g_{a}$ becomes
$$g_{a}\sim-dt^2-2JF'(\theta')(dt-dr)d\theta'+dr^2+\left(r+\frac{b^{(1)}(F(\theta'))}{1-\alpha-b^{(1)}(F(\theta'))}(r-t)\right)^2d(\theta')^2.$$
We set
$$b(\theta')=\frac{b^{(1)}(F(\theta'))}{1-\alpha-b^{(1)}(F(\theta'))},$$
$$J(\theta')=2JF'(\theta').$$
Let us note that $J$ is at the same level of regularity than $b$.
\end{proof}
We are now ready to prove Theorem \ref{thinitial}.
\begin{proof}[Proof of Theorem \ref{thinitial}]
We consider the map 
$$\Phi: \bar{b}\mapsto \Pi b,$$
where
\begin{itemize}
\item $\bar{b} \in W^{N,2}$ is such that 
$$\int \bar{b}\cos(\theta)=\int \bar{b}\sin(\theta)=0,\quad \alpha=-\int \bar{b}(\theta),$$
where $\alpha$ is given by Theorem \ref{contrainte},
\item $b$ is given by formula \eqref{defb}, where $b^{(1)}=\rho\cos(\theta-\eta)+\bar{b},$ and $\rho,\eta$ are given by Theorem \ref{contrainte}.
\item $\Pi$ is the projection
\begin{equation}
\label{projection}
\Pi: W^{2,N}(\m S^1) \rightarrow \{u \in W^{2,N}(\m S^1), \; \int u= \int \cos(\theta)u=\int \sin(\theta)u=0\}.
\end{equation}
\end{itemize} 
It is easy to see that $\Phi$ is invertible for $\ep$ small enough.
Therefore, for $\wht b \in W^{2,N}$ such that 
$$\int \wht b= \int \wht b\cos(\theta)= \int \wht b\sin(\theta)=0,$$
we apply Theorem \ref{contrainte} to $\Phi^{-1}(\wht b)$.
Thanks to Lemma \ref{lmgas} and \ref{lmiso} we can find 
$(g_0)_{ij} \in H^{N+1}_\delta$ and $(K_0)_{ij} \in H^N_{\delta+1}$
such that $(g_b)_{ij}+(g_{0})_{ij}$ and $(K_b)_{ij}+(K_0)_{ij}$ satisfy the constraint equations, where we have noted $K_b$ the second fundamental form associated to $g_b$. We complete the initial data as follow.
We write our metric in the form $g=g_b + \wht g$. The initial data for $\wht g$ are the following
\begin{itemize}
\item $\wht g_{ij}$ is given by $\wht g_{ij}=(\wht g_0)_{ij}$, 
\item $\wht g_{00}$ and $\wht g_{0i}$ are taken to be $0$\footnote{The lapse ans shift are given by $g_b$: we have $N=1$ and $\beta_r=0$ and $\beta_\theta=-J$.}
\item $\partial_t \wht g_{ij}$ is given by the relation
$\partial_0 g_{ij}=-2NK_{ij}$ and $K_{ij}=(K_b)_{ij}+(K_0)_{ij}$.
\item $\partial_t \wht g_{00}$ and $\partial_t \wht g_{0i}$ are chosen such that the generalized wave coordinate condition is satisfied at $t=0$. 
\end{itemize}
Let us describe the last point. The generalized wave coordinate condition writes
$$g^{\lambda \beta}\Gamma^\alpha_{\lambda \beta}=H_b^\alpha=(g_b)^{\lambda \beta}(\Gamma_b)^\alpha_{\lambda \beta}+F^\alpha,$$
Therefore, if we write it for $\alpha=i$ we obtain a relation for $\partial_tg_{0i}$ and if we write it for $\alpha=0$, we obtain a relation for
$\partial_t g_{00}$. However, if we write $g=g_b +\wht g$, the term
$$g^{\lambda \beta}\Gamma^\alpha_{\lambda \beta}-(g_b)^{\lambda \beta}(\Gamma_b)^\alpha_{\lambda \beta}
$$
contains crossed terms of the form
$$\wht g \partial_U g_b \sim \wht g \frac{\partial_\theta b(\theta)}{r}.$$
which do not belong in $H^N_{\delta+1}$ because we are missing a derivative on $b$, since $b\in W^{2,N}$. Therefore, we will take $F^\alpha$ as defined in \eqref{defHalpha}. With this choice, the generalized wave coordinate condition imply that $\partial_t \wht g_{00}$ and $\partial_t \wht g_{0i}$ are given by a sum of terms the form
$$K_0, \;\nabla g_0,\;g_bK_0,\; g_b \nabla g_0, \; \frac{\chi(r)g_b}{r}g_0.$$ 
With this choice, $\partial_t \wht g_{0i}$ and $\partial_t \wht g_{00}$
belong to $H^N_{\delta+1}$.
\end{proof}

\section{The generalised wave coordinates}\label{genw}
In a coordinate system, the Ricci tensor is given by
\begin{equation}\label{defric}
R_{\mu \nu} = \partial_\alpha \Gamma^\alpha_{\mu \nu} -\partial_\mu \Gamma^\alpha_{\alpha \nu}
+\Gamma^\alpha_{\mu \nu}\Gamma^\lambda_{\alpha \lambda} -\Gamma^\alpha_{\mu \lambda} \Gamma^\lambda_{\nu \alpha},
\end{equation}
where the $\Gamma^\lambda_{\alpha \beta}$ are the Christoffel symbols given by
\begin{equation}\label{defcri}
\Gamma^\lambda_{\alpha \beta}=\frac{1}{2}g^{\lambda \rho}\left(\partial_\alpha g_{\rho \beta} + \partial_\beta g_{\rho \alpha }
-\partial_\rho g_{\alpha \beta}\right).
\end{equation}
$R_{\mu \nu}$ an operator of order two for $g$. In order to single out the hyperbolic part, we will write
\begin{equation}\label{defh}
H^\alpha = g^{\lambda \beta}\Gamma^\alpha_{\lambda \beta},
\end{equation}
which can also be written
$$H^\alpha = -\partial_\lambda g^{\lambda \alpha}-\frac{1}{2}g^{\lambda \mu} \partial^\alpha g_{\lambda \mu}.$$
We compute $R_{\mu \nu}$ in terms of $g$ and $H$.
\begin{align*}
 R_{\mu \nu}=& \frac{1}{2} \partial_\alpha\left(g^{\alpha \rho}(\partial_\mu g_{\rho  \nu} +\partial_\nu g_{\rho \mu} -\partial_\rho g_{\mu \nu})\right)
-\frac{1}{2}\partial_\mu \left(g^{\alpha \rho}(\partial_\nu g_{\rho \alpha} + \partial_\alpha g_{\rho \nu} -\partial_\rho g_{\nu \alpha})\right)\\
&+\frac{1}{4}g^{\alpha \rho}g^{\lambda\beta}(\partial_\mu g_{\rho  \nu} +\partial_\nu g_{\rho \mu} -\partial_\rho g_{\mu \nu})
(\partial_\lambda g_{\beta \alpha} + \partial_\alpha g_{\beta \lambda} -\partial_\beta g_{\alpha \lambda})\\
&-\frac{1}{4}g^{\alpha \rho}g^{\lambda \beta}(\partial_\nu g_{\rho \lambda} + \partial_\lambda g_{\rho \nu} -\partial_\rho g_{\nu \lambda})
(\partial_\mu g_{\alpha \beta} + \partial_\alpha g_{\beta \mu} -\partial_\beta g_{\alpha \mu}),\end{align*}

\begin{equation}\label{calcricci}
R_{\mu \nu}=-\frac{1}{2}g^{\alpha \rho}\partial_\alpha \partial_\rho g_{\mu \nu} +\frac{1}{2} H^\rho \partial_\rho g_{\mu \nu}
+\frac{1}{2}\left( g_{\mu \rho}\partial_\nu H^\rho + g_{\nu \rho}\partial_\mu H^\rho \right) + \frac{1}{2}P_{\mu \nu}(g)(\partial g, \partial g),
\end{equation}
with
\begin{equation}\label{quad}
\begin{split}
 P_{\mu \nu}(g)(\partial g, \partial g)
=&\frac{1}{2}g^{\alpha \rho}g^{\beta \sigma}\left(\partial_\mu g_{\rho \sigma}\partial_\alpha g_{\beta \nu}+\partial_\nu g_{\rho \sigma}\partial_\alpha g_{\beta \mu}
-\partial_\beta g_{\mu \rho}\partial_\alpha g_{\nu \sigma}-\frac{1}{2}\partial_\mu g_{\alpha \beta} \partial_\nu g_{\rho \sigma}\right)\\
&+\frac{1}{2}g^{\alpha \beta}g^{\lambda \rho}\partial_\alpha g_{\nu \rho}\partial_\beta g_{\mu \rho}
\end{split}
\end{equation}

\begin{prp}
 If the coupled system of equations
$$\left\{\begin{array}{l}
-\frac{1}{2}g^{\alpha \rho}\partial_\alpha \partial_\rho g_{\mu \nu} +\frac{1}{2} F^\rho \partial_\rho g_{\mu \nu}
+\frac{1}{2}\left( g_{\mu \rho}\partial_\nu F^\rho + g_{\nu \rho}\partial_\mu F^\rho \right) +\frac{1}{2} P_{\mu \nu}(g)(\partial g, \partial g)=\partial_\mu \phi \partial_\nu \phi\\
g^{\alpha \rho}\partial_\alpha\partial_\rho \phi - F^\rho\partial_\rho \phi = 0
  \end{array}\right.
$$
with $F$ a function which may depend on $\phi, g$,
is satisfied on a time interval $[0,T]$ with $T>0$, if the initial induced Riemannian metric and second fundamental form $(\bar{g}, K)$ satisfy the constraint equations, and if the initial compatibility condition
\begin{equation}\label{ft0}
F^\alpha|_{t=0} = H^\alpha|_{t=0},
\end{equation}
is satisfied, then for all time, the equations \eqref{s1} are satisfied on $[0,T]$, together with the wave coordinate condition
$$F^\alpha = H^\alpha.$$
\end{prp}

\begin{proof}
 We use the twice contracted Bianchi Identity
$$D^\mu \left(R_{\mu \nu} -\frac{1}{2}Rg_{\mu \nu} \right)=0.$$
with $H$ defined by \eqref{defh}. Since in $[0,T]$, we have
$$-\frac{1}{2}g^{\alpha \rho}\partial_\alpha \partial_\rho g_{\mu \nu} +\frac{1}{2} F^\rho \partial_\rho g_{\mu \nu}
+\frac{1}{2}\left( g_{\mu \rho}\partial_\nu F^\rho + g_{\nu \rho}\partial_\mu F^\rho \right) + P_{\mu \nu}(g)(\partial g, \partial g)=\partial_\mu \phi \partial_\nu \phi.$$
Thanks to \eqref{calcricci} we obtain
$$\frac{1}{2} (F^\rho-H^\rho) \partial_\rho g_{\mu \nu}
+\frac{1}{2}\left( g_{\mu \rho}\partial_\nu (F^\rho-H^\rho) + g_{\nu \rho}\partial_\mu (F^\rho-H^\rho) \right)= \partial_\mu \phi \partial_\nu \phi-R_{\mu\nu}.$$
Consequently, since
$D^\mu \left(R_{\mu \nu} -\frac{1}{2}Rg_{\mu \nu} \right)=0$ and
$D^\mu \left(\partial_\mu \phi \partial_\nu \phi -\frac{1}{2}g_{\mu \nu}\partial^\alpha \phi \partial_\alpha \phi \right)=0$
and we obtain the following equation on $F^\rho-H^\rho$
\begin{align*}
0=& D^\mu \bigg(\frac{1}{2}\left( g_{\mu \rho}\partial_\nu (F^\rho-H^\rho) + g_{\nu \rho}\partial_\mu (F^\rho-H^\rho) \right)
-\frac{1}{4}g_{\mu \nu} g^{\alpha \beta}\left( g_{\alpha \rho}\partial_\beta (F^\rho-H^\rho) + g_{\alpha \rho}\partial_\beta (F^\rho-H^\rho) \right)\\
&+\frac{1}{2}\left(\partial_\rho g_{\mu \nu}-\frac{1}{2}g^{\alpha\beta}\partial_\rho g_{\alpha \beta}\right) (F^\alpha - H^\alpha) \bigg)
\end{align*}
Multiplying by $g^{\nu \alpha}$ we obtain
$$\Box_g (F^\alpha-H^\alpha) + B^{\alpha,\beta}_\rho \partial_\beta (F^\rho-H^\rho) + C^\alpha_\rho (F^\rho-H^\rho)=0,$$
with $ B^{\alpha,\beta}_\rho $, $C^\alpha_\rho$ coefficients depending on $g,\phi$, well defined in $[0,T]$. 
This is an equation in hyperbolic form, therefore if the initial data  $(F^\alpha-H^\alpha)|_{t=0} $
 and $\partial_t(F^\alpha-H^\alpha)|_{t=0} $ are zero, then the solution is identically zero on $[0,T]$.  
Since we assume \eqref{ft0}, we only have to check 
$$\partial_t(F^\alpha-H^\alpha)|_{t=0} = 0.$$
Since the constraint equations are satisfied, we have
\begin{align*}
 R_{0i} &= \partial_0 \phi \partial_i \phi,\\
R_{00} - \frac{1}{2}g_{00} R &= \partial_0 \phi \partial_0 \phi-\frac{1}{2}g_{00} \partial^\mu \phi \partial_\mu \phi.
\end{align*}
Therefore, using once again equation \eqref{calcricci} and \eqref{ft0} we obtain
\begin{align*}
0=& g_{i \rho}\partial_t (F^\rho-H^\rho), \\
0=& 2g_{0 \rho}\partial_t (F^\rho-H^\rho) -g_{00}  \partial_t (F^0 -H^0). 
\end{align*}
This system can be written as
$$\left(\begin{array}{lll}
g_{00}& 2g_{01} &2g_{02}\\
g_{01} & g_{11} & g_{12}\\
g_{02} & g_{12} & g_{22}
        \end{array}\right)
\left( \begin{array}{l}
        \partial_t (F^0 -H^0)\\
\partial_t (F^1 -H^1)\\
\partial_t (F^2 -H^2)
       \end{array}\right)=0.$$
It is invertible so $\partial_t(F^\rho -H^\rho)_{t=0}=0$. Therefore in $[0,T]$ we have
$F^\rho = H^\rho$ and equation \eqref{calcricci} implies that the Einstein Equations \eqref{s1} are satisfied.
\end{proof}

\section{The $L^\infty-L^\infty$ estimate}\label{linflinf}
For the sake of completeness, we give here the proof of the $L^\infty-L^\infty$ estimate by Kubo and Kubota (see \cite{kubo}).
\begin{prp}
Let u be a solution of
\begin{equation*}
 \left\{ \begin{array}{l} 
         \Box u = F, \\
	 (u, \partial_t u)|_{t=0}=(0,0),
        \end{array}
\right.
\end{equation*}
The $L^\infty-L^\infty$ estimate writes: for $\mu>\frac{3}{2} , \nu >1$
$$|u(t,x)|( 1+t+|x|)^\frac{1}{2} \leq C(\mu, \nu) M_{\mu, \nu}(f) (1+|t-|x|||)^{-\frac{1}{2} + [2-\mu]_{+}},$$
where 
$$M_{\mu, \nu}(f) = \sup (1+|y|+s)^\mu (1+|s-|y||)^\nu |F(y,s)|,$$
and we have the convention $A^{[0]_{+}}=\ln(A)$.

\end{prp}

\begin{proof}

We write the solution $u$ of
\begin{equation*}
 \left\{ \begin{array}{l} 
         \Box u = F, \\
	 (u, \partial_t u)|_{t=0}=(0,0),
        \end{array}
\right.
\end{equation*}
with the representation formula
\begin{equation*}
u(x,t) = \int_0^t \int_{|y| \leq t-s} \frac{1}{\sqrt{(t-s)^2-|y|^2}}F(s, x-y)dy ds.
\end{equation*}
With $M_{\mu, \nu}(f) = \sup (1+|y|+s)^\mu (1+|s-|y||)^\nu |F(y,s)|$, we can write
$$|u(x,t)| \leq M_{\mu,\nu}(f) \int_0^t \int_{|y| \leq t-s} \frac{1}{\sqrt{(t-s)^2-|y|^2}}\frac{1}{(1+|x-y|+s)^\mu (1+|s-|x-y||)^\nu}dy ds.$$
It is therefore sufficient to study the quantity
$$I(x,t)= \int_0^t \int_{|y| \leq t-s} \frac{1}{\sqrt{(t-s)^2-|y|^2}}\frac{1}{(1+|x-y|+s)^\mu (1+|s-|x-y||)^\nu}dy ds.$$
We begin with a lemma on spherical means.

\begin{lm}\label{moysp}
Let $b \in \q C^0(\m R^2)$. We have the following equality for $\rho \geq 0$
$$\int_{|\omega|=1} b(|x+\rho \omega|)d\omega = 4\int_{|\rho-r|}^{\rho+r}\lambda b(\lambda)h(\lambda, \rho,r)d\lambda,$$
where we note $r=|x|$ and 
\begin{align*}
 h(\lambda, \rho,r)&=\left(\lambda^2-(\rho-r)^2\right)^{-\frac{1}{2}}\left((\rho+r)^2-\lambda^2\right)^{-\frac{1}{2}}\\
&= \left((\lambda+r)^2-\rho^2\right)^{-\frac{1}{2}}\left(\rho^2-(\lambda-r)^2\right)^{-\frac{1}{2}}.
\end{align*}
\end{lm}

\begin{proof}
By eventually rotating the axis, we can assume $x=(r,0)$ in $(x_1,x_2)$ coordinates. Therefore we have 
\[\int_{|\omega|=1} b(|x+\rho \omega|)d\omega= \int_0^{2\pi}b\left( (r^2+\rho^2+ 2r\rho \cos(\theta))^\frac{1}{2}\right)d\theta
=2\int_0^{\pi}b\left( (r^2+\rho^2+ 2r\rho \cos(\theta))^\frac{1}{2}\right)d\theta.\]
 We make the change of variable $\lambda = (r^2+\rho^2+2\rho r \cos(\theta))^{\frac{1}{2}}$, for $\theta \in [0, \pi[$. Then we have
\begin{align*}
 d\lambda &= -\frac{1}{\lambda}\rho r \sin(\theta) d\theta \\
&= -\frac{1}{\lambda}\rho r \left(1- \frac{(\lambda^2-r^2-\rho^2)^2}{(2\rho r)^2}\right)^{\frac{1}{2}}d\theta\\
&= -\frac{1}{2\lambda} \left( (2\rho r)^2-(\lambda^2 -r^2-\rho^2)^2\right)^{\frac{1}{2}} d\theta\\
&=-\frac{1}{2\lambda}\left(2\rho r -\lambda^2+r^2+\rho^2\right)^\frac{1}{2}\left( 2\rho r + \lambda^2 -\rho^2 -r^2\right)^\frac{1}{2}d\theta.
\end{align*}
We have therefore 
$d\theta = -2\lambda h(\lambda,\rho,r) d\lambda$, which concludes the proof of Lemma \ref{moysp}.
\end{proof}

We use Lemma \ref{moysp} to  calculate $I$
\begin{align*}
I(x,t)&= \int_0^t \int_{\rho \leq t-s}\frac{\rho}{\sqrt{(t-s)^2-\rho^2}}\int_{|\omega|=1}\frac{1}{(1+|x+\rho \omega|+s)^\mu (1+|s-|x+\rho \omega||)^\nu}d\omega d\rho ds\\
&=  4\int_0^t \int_{\rho \leq t-s}\frac{\rho}{\sqrt{(t-s)^2-\rho^2}}\int_{|\rho-r|}^{\rho +r} \frac{h(\lambda,\rho,r)}{(1+\lambda+s)^\mu (1+|s-\lambda|)^\nu} \lambda d\lambda d\rho ds.
\end{align*}
We exchange the integration in $\rho$ with the integration in $\lambda$, noticing that 
$$\ch 1_{|\rho-r|\leq \lambda \leq \rho+r} = \ch 1_{|\lambda-r|\leq \rho \leq \lambda+r},$$
and we make the decomposition $ I= I_1 + I_2$, separating the region $\lambda+r\leq t-s$ from $\lambda+r \geq t-s$.
\begin{align*}
 I_1 &= \int_0^{t-r} \int_{\lambda = 0}^{t-s-r} \frac{\lambda}{z(s,\lambda)}\int_{|\lambda-r|}^{\lambda +r} \frac{h(\lambda, \rho,r)}{\sqrt{(t-s)^2-\rho^2}} \rho d\rho d\lambda ds,\\
I_2 &= \int_0^{t} \int_{\lambda = \max(t-s-r,0)}^{t-s+r} \frac{\lambda}{z(s,\lambda)}\int_{|\lambda-r|}^{t-s} \frac{h(\lambda, \rho,r)}{\sqrt{(t-s)^2-\rho^2}} \rho d\rho d\lambda ds,
\end{align*}
where $z(s,\lambda)=(1+\lambda+s)^\mu (1+|s-\lambda|)^\nu$.

\subsection{Estimate of $I_1$}
We write
\begin{align*}
\int_{|\lambda-r|}^{\lambda +r} \frac{h(\lambda, \rho,r)}{\sqrt{(t-s)^2-\rho^2}} \rho d\rho &=\int_{|\lambda-r|}^{\lambda +r}\frac{1}{\sqrt{(t-s)^2-\rho^2}
\sqrt{(\lambda+r)^2-\rho^2}\sqrt{\rho^2-(\lambda-r)^2}}\rho d\rho\\
&=\frac{1}{2}\int_a^b \frac{du}{\sqrt{d-u}\sqrt{b-u}\sqrt{u-a}},
\end{align*}
with $a=(\lambda-r)^2$, $b=(\lambda+r)^2$ and $d=(t-s)^2$. Recall that in the integration region of $I_1$, 
we have $\lambda+r\leq t-s$ so $b\leq d$. This yields
\begin{equation}\label{astuce} 
\int_a^b \frac{du}{\sqrt{d-u}\sqrt{b-u}\sqrt{u-a}} \leq
\frac{1}{\sqrt{d-b}} \int_a^b \frac{du}{\sqrt{b-u}\sqrt{u-a}}
\leq \frac{1}{\sqrt{d-b}} \int_0^1 \frac{dv}{\sqrt{v}\sqrt{1-v}} \leq \frac{\pi}{\sqrt{d-b}}.
\end{equation}

Consequently we have
$$I_1 \lesssim \int_0^{t-r} \int_0^{t-s-r}\frac{\lambda}{\sqrt{(t-s)^2-(\lambda+r)^2}(1+\lambda+s)^\mu (1+|s-\lambda|)^\nu}d\lambda ds.$$
We make the change of variable $\alpha = s-\lambda$, $\beta = \lambda +s$. We obtain
$$I_1 \lesssim \left( \int_0^{t-r}\frac{\beta d\beta}{\sqrt{t-r-\beta}(1+\beta)^\mu }\right) \left( \int_{r-t}^{t-r}\frac{d\alpha}{\sqrt{t+r-\alpha}(1+|\alpha|)^\nu }\right).$$

We estimate the first factor. We note that if $t-r \leq 1$, this factor is bounded. We assume therefore that $t-r \geq 1$.
\begin{align*}
 \int_0^{t-r}\frac{\beta d\beta}{\sqrt{t-r-\beta}(1+\beta)^\mu}& = \int_0^\frac{t-r}{2}\frac{\beta d\beta}{\sqrt{t-r-\beta}(1+\beta)^\mu}
+\int_\frac{t-r}{2}^{t-r}\frac{\beta d\beta}{\sqrt{t-r-\beta}(1+\beta)^\mu}\\
&\lesssim \frac{1}{\sqrt{t-r}} \int_0^\frac{t-r}{2}\frac{\beta d\beta}{(1+\beta)^\mu} +(t-r)^{1-\mu} \int_\frac{t-r}{2}^{t-r}\frac{d\beta}{\sqrt{t-r-\beta}}\\
&\lesssim \frac{(t-r)^{[2-\mu]_{+}}}{\sqrt{t-r}}.
\end{align*}

We estimate the second factor
\begin{align*}
 &\int_{r-t}^{t-r}\frac{d\alpha}{\sqrt{t+r-\alpha}(1+|\alpha|)^\nu }\\&= \int_{r-t}^{\min(\frac{t+r}{2},t-r)}\frac{d\alpha}{\sqrt{t+r-\alpha}(1+|\alpha|)^\nu }
+\int_{\min(\frac{t+r}{2},t-r)} ^{t-r} \frac{d\alpha}{\sqrt{t+r-\alpha}(1+|\alpha|)^\nu }\\
&\lesssim \frac{1}{\sqrt{t+r}}\int_{r-t}^{\min(\frac{t+r}{2},t-r)}\frac{d\alpha}{(1+|\alpha|)^\nu }+\frac{1}{(1+t+r)^\nu}\int_{\min(\frac{t+r}{2},t-r)} ^{t-r} \frac{d\alpha}{\sqrt{t+r-\alpha} }\\
&\lesssim \frac{1}{\sqrt{t+r}},
\end{align*}
where we have used in the last inequality the fact that $\nu>1$. We have proved
$$I_1\lesssim \frac{(1+|t-r|)^{[2-\mu]_{+}}}{\sqrt{1+t+r}\sqrt{|t-r|}}.$$

\subsection{Estimate of $I_2$}
As in the estimate of $I_1$, we write
\[
\int_{|\lambda-r|}^{t-s} \frac{h(\lambda, \rho,r)}{\sqrt{(t-s)^2-\rho^2}} \rho d\rho 
=\frac{1}{2}\int_a^d \frac{du}{\sqrt{d-u}\sqrt{b-u}\sqrt{u-a}},
\]
with $a=(\lambda-r)^2$, $b=(\lambda+r)^2$ and $d=(t-s)^2$.
In the region $\lambda+r\geq t-s$, we have $b\geq d$, therefore as for \eqref{astuce} we get
\[\frac{1}{2}\int_a^d \frac{du}{\sqrt{d-u}\sqrt{b-u}\sqrt{u-a}}
\lesssim \frac{1}{\sqrt{b-d}}\int_a^d \frac{du}{\sqrt{d-u}\sqrt{u-a}}\]
and so
\[\int_{|\lambda-r|}^{t-s} \frac{h(\lambda, \rho,r)}{\sqrt{(t-s)^2-\rho^2}} \rho d\rho
\lesssim \frac{1}{\sqrt{(\lambda+r)^2-(t-s)^2}}.\]
Therefore we have
$$I_2 \lesssim \int_0^{t} \int_{\lambda = \max(t-s-r,0)}^{t-s+r} \frac{\lambda}{\sqrt{(\lambda +r)^2-(t-s)^2}(1+\lambda+s)^\mu (1+|s-\lambda|)^\nu}d\lambda ds.$$
We make the same change of variable $\alpha = s-\lambda$, $\beta = \lambda +s$. We obtain
$$I_2 \lesssim  \left( \int_{\max(0, t-r)}^{t+r}\frac{\beta d\beta}{\sqrt{\beta-(t-r)}(1+\beta)^\mu }\right) \left( \int_{-r-t}^{t}\frac{d\alpha}{\sqrt{t+r-\alpha}(1+|\alpha|)^\nu }\right).$$

We estimate the first factor. We first assume $t-r>0$.
\begin{align*}
 \int_{t-r}^{t+r}\frac{\beta d\beta}{\sqrt{\beta-(t-r)}(1+\beta)^\mu}
&\lesssim \int_0^{2r} \frac{ (\rho +1+t-r)^{1-\mu}}{\sqrt{\rho}}d\rho\\
&\lesssim (1+|t-r|)^{\frac{3}{2}-\mu}\int_0^{\frac{2r}{1+t-r}} \frac{(1+u)^{1-\mu}}{\sqrt{u}}du\\
&\lesssim (1+|t-r|)^{\frac{3}{2}-\mu},
\end{align*}
where we have made consecutively the changes of variable $\rho = \beta-|t-r|$ and $u= \frac{\rho}{1+|t-r|}$, and where we use in the last inequality the fact that
$\frac{(1+u)^{1-\mu}}{\sqrt{u}}$ is integrable.

We now assume $t-r<-1$. Then
\begin{align*}
 \int_{0}^{t+r}\frac{\beta d\beta}{\sqrt{\beta+|t-r|}(1+\beta)^\mu}
&\lesssim |t-r|^\frac{3}{2}\int_0^{\frac{t+r}{|t-r|} }\frac{ \rho}{\sqrt{1+\rho}(1+|t-r|\rho)^\mu}d\rho\\
&\lesssim (1+|t-r|)^{\frac{3}{2}-\mu}\int_0^{\frac{t+r}{|t-r|}} \frac{\rho}{\sqrt{\rho}\left(\frac{1}{|t-r|}+\rho\right)^\mu}d\rho\\
&\lesssim (1+|t-r|)^{-\frac{1}{2}+[2-\mu]_{+}},
\end{align*}
where we have made the change of variable $\rho= \frac{\beta}{|t-r|}$, and also used the fact that 
$\mu >\frac{3}{2}$.

We estimate the second factor
\begin{align*}
&\int_{-r-t}^{t}\frac{d\alpha}{\sqrt{t+r-\alpha}(1+|\alpha|)^\nu }\\
&\lesssim \int_{-r-t}^{\min(t,\frac{t+r}{2})} \frac{d\alpha}{\sqrt{t+r-\alpha}(1+|\alpha|)^\nu }
+\int_{\min(t,\frac{t+r}{2})}^{t}  \frac{d\alpha}{\sqrt{t+r-\alpha}(1+|\alpha|)^\nu }\\
&\lesssim \frac{1}{\sqrt{t+r}} \int_{-r-t}^{\min(t,\frac{t+r}{2})}  \frac{d\alpha}{(1+|\alpha|)^\nu } + \frac{1}{(1+t+r)^\nu} \int_{\min(t,\frac{t+r}{2})}^{t}\frac{d\alpha}{\sqrt{t+r-\alpha} }\\
&\lesssim \frac{1}{\sqrt{t+r}},
\end{align*}
where we have used the fact that $\nu>1$. We have proved therefore that
\[I_2\lesssim \frac{(1+|t-r|)^{-\frac{1}{2}+[2-\mu]_{+}}}{\sqrt{1+t+r}}, \]
so 
\[I \leq I_1+I_2 \lesssim  \frac{(1+|t-r|)^{[2-\mu]_{+}}}{\sqrt{1+t+r}\sqrt{1+|t-r|}}\]
The proof of the $L^\infty-L^\infty$ estimate is now complete.
\end{proof}

\section{Hardy inequality with weight}\label{hardyin}
\begin{prp}\label{prphard}
 Let $\alpha<1$ and $\beta>1$. We have, with $q=r-t$,
$$\int u^2f(q)rdrd\theta \leq C(\alpha,\rho)\int (\partial_r u)^2 g(q)rdrd\theta$$
where
\begin{align*}
 f(q)&= (1+|q|)^{\beta-2}, \quad q>0\\
&=(1+|q|)^{\alpha-2},\quad q<0
\end{align*}
\begin{align*}
 g(q)&= (1+|q|)^{\beta}, \quad q>0\\
&=(1+|q|)^{\alpha},\quad q<0
\end{align*}
\end{prp}
\begin{proof}
We look first at the region $r>t$. We can assume, by a density argument that $u$ is compactly supported. 
We calculate
$$\partial_r\left( r(1+r-t)^{\beta-1}\right)=(1+r-t)^{\beta-1}+(\beta-1)r(1+r-t)^{\beta-2}
=r(1+r-t)^{\beta-2}\left( \frac{1+r-t}{r}+\beta-1 \right)$$
We want to find $c>0$ such that 
$$\frac{1+r-t}{r}+\beta-1>c.$$
This condition is satisfied if
$$t<1+r(\beta-c)$$
which is the case if $\beta-c>1$. Since $\beta>1$ we can find such a $c>0$.
Therefore
\begin{align*}
&\int_t^\infty\int_0^{2\pi}  u^2(1+r-t)^{\beta-2} rdrd\theta \\
\leq& \frac{1}{c} \int_t^\infty \int_0^{2\pi} u^2 \partial_r \left( r(1+r-t)^{\beta-1}\right)drd\theta\\
\leq &\frac{1}{c}\left( -\int_t^\infty \int_0^{2\pi} (\partial_r u^2 ) (1+r-t)^{\beta-1}rdrd\theta+\left[ \int_0^{2\pi}  u^2(r,\theta)(1+r-t)^{\beta-1}rd\theta\right]_t^\infty\right).
\end{align*}
Since $u$ is compactly supported, 
$$\left[ \int_0^{2\pi} u^2(r,\theta)(1+r-t)^{\beta-1}rd\theta\right]_t^\infty\leq 0$$
therefore
\begin{align*}
 &\int_t^\infty\int_0^{2\pi}  u^2(1+r-t)^{\beta-2} rdrd\theta \\
 \leq &\frac{2}{c}
\int_t^\infty \int_0^{2\pi} |u\partial_r u|(1+r-t)^{\beta-1}rdrd\theta\\
\leq &\frac{2}{c} \left( \int_t^\infty\int_0^{2\pi}  u^2(1+r-t)^{\beta-2} rdrd\theta\right)^\frac{1}{2}
\left( \int_t^\infty \int_0^{2\pi} (\partial_r u)^2(1+r-t)^{\beta} rdrd\theta\right)^\frac{1}{2}
\end{align*}
We have proved
\begin{equation}\label{har1}
 \int_t^\infty \int_0^{2\pi} u^2(1+r-t)^{\beta-2} rdrd\theta \leq C(\alpha)
 \int_t^\infty \int_0^{2\pi} (\partial_r u)^2(1+r-t)^{\beta} rdrd\theta.
\end{equation}
We now look at the region $r<t$. We calculate
$$\partial_r \left(r(1+t-r)^{\alpha-1}\right)= (1+t-r)^{\alpha-1} + (1-\alpha) r(1+t-r)^{\alpha-2}.$$
Therefore
\begin{align*}&\int_0^t\int_0^{2\pi} u^2(1+t-r)^{\alpha-2}rdrd\theta\\
&\leq \frac{1}{1-\alpha} \int_0^t\int_0^{2\pi} u^2\partial_r \left(r(1+t-r)^{\alpha-1}\right)drd\theta\\
&\leq \frac{1}{1-\alpha} \left(\int_0^t\int_0^{2\pi} -\left(\partial_r u^2\right) (1+t-r)^{\alpha-1}rdrd\theta
+\left[\int_0^{2\pi}  u^2 (1+t-r)^{-\rho}r \right]_0^t\right)\\
&\leq \frac{1}{1-\alpha} \left( 2\int_0^t \int_0^{2\pi} |u\partial_r u| (1+t-r)^{\alpha-1}rdrd\theta
+ t\int_0^{2\pi} u^2(t,\theta)d\theta
\right).
\end{align*}
We have
\begin{align*}
&\int_0^t\int_0^{2\pi}  |u\partial_r u| (1+t-r)^{\alpha-1}rdrd\theta\\
 \leq&
 \left( \int_0^t\int_0^{2\pi}  u^2(1+t-r)^{\alpha-2} rdrd\theta\right)^\frac{1}{2}
\left( \int_0^t \int_0^{2\pi} (\partial_r u)^2(1+t-r)^{\alpha} rdrd\theta\right)^\frac{1}{2}
\end{align*}
and
\begin{align*}
 t\int_0^{2\pi}  u^2(t,\theta)d\theta
&\leq t \int_t^\infty \int_0^{2\pi} |\partial_r (u^2)|drd\theta\\
&\leq 2t \int_t^\infty \int_0^{2\pi} |u\partial_r u| \frac{(1+t-r)^\frac{\beta}{2}}{(1+t-r)^\frac{\beta}{2}}\frac{r}{t}drd\theta\\
&\leq 2\left( \int_t^\infty \int_0^{2\pi} u^2(1+r-t)^{-\beta} rdrd\theta\right)^\frac{1}{2}
\left( \int_t^\infty \int_0^{2\pi} (\partial_r u)^2(1+r-t)^{\beta} rdrd\theta\right)^\frac{1}{2}.
\end{align*}
Since $\beta >1$, we have $\beta >2-\beta$. Thanks to the estimate \eqref{har1} in the region $r>t$, we obtain
\begin{align*}
&\int_0^t \int_0^{2\pi} u^2(1+t-r)^{\alpha-2}rdrd\theta\\
 \leq& C(\rho,\alpha)\left(
 \int_0^t \int_0^{2\pi} (\partial_r u)^2(1+t-r)^{\alpha} rdrd\theta
+  \int_t^\infty \int_0^{2\pi} (\partial_r u)^2(1+r-t)^{\beta} rdrd\theta\right)
\end{align*}
This concludes the proof of Proposition \ref{prphard}.
\end{proof}

\section{Weighted Klainerman-Sobolev inequality}\label{weigklai}
\begin{prp}\label{prpks}
 We have the inequality
$$|f(t,x)v^\frac{1}{2}(|x|-t)|\lesssim \frac{1}{\sqrt{1+t+|x|}\sqrt{1+||x|-t|}}\sum_{I\leq 2}\|v^\frac{1}{2}(.-t) Z^I f\|_{L^2}.$$
\end{prp}
\begin{proof}
We introduce the decomposition
$$f=f_1+f_2,$$
where
$$f_1=\chi\left(\frac{r}{t}\right)f, \quad f_2= \left(1-\chi\left(\frac{r}{t}\right) \right)f,$$
and $\chi$ is a cut-off such that $\chi(\rho)=1$ for $\rho \leq \frac{1}{2}$ and $\chi(\rho)=0$ for $\rho\geq \frac{2}{3}.$
Since the quantities $Z^I \chi$ are bounded, it is sufficient to prove the proposition for $f_1$ and $f_2$.

For $f_1$, we introduce the function $f_t=f_1(t,tx).$ The Sobolev embedding $H^2 \hookrightarrow L^\infty$ gives
\begin{align*}
\|f_t\|_{L^\infty}&\lesssim \sum_{|\alpha| \leq 2}\|\nabla^\alpha f_t\|_{L^2}\\
&\lesssim  \frac{1}{t} \sum_{|\alpha| \leq 2}  \|t^\alpha \nabla^\alpha f_1\|_{L^2}.
\end{align*}
In the region $r\leq \frac{2t}{3}$ we have $-t\leq r-t \leq -\frac{t}{3}$, therefore
$$|t\nabla \phi|\lesssim |(r-t)\nabla \phi|\lesssim \sum_{Z \in \q Z}|Z \phi|.$$ 
Moreover, in this region $v(|x|-t)\sim v(t)$, so
\begin{align*}
|f_1(t,x)v^\frac{1}{2}(|x|-t)|&\lesssim \frac{1}{t} \sum_{I \leq 2}  \|v^\frac{1}{2}(t)Z^I f_1\|_{L^2}\\
&\lesssim \frac{1}{\sqrt{1+t+|x|}\sqrt{1+||x|-t|}}\sum_{I\leq 2}\|v^\frac{1}{2}(.-t) Z^I f_1\|_{L^2}.
\end{align*}

For $f_2$ we write
\begin{align*}
& (1+t+r)(1+|t-r|)v(r-t)(f_2(t,r,\theta))^2 \\
&\lesssim
\int_{\frac{t}{2}}^r \partial_\rho \left((1+t+\rho)(1+|t-\rho|)v(\rho-t)f_2(t,\rho,\theta)^2\right)d\rho\\
&\lesssim \sum_{0\leq \alpha \leq 1}
\int_{\frac{t}{2}}^r \int_0^{2\pi}|\partial^\alpha_\theta \partial_\rho \left((1+t+\rho)(1+|t-\rho|)v(\rho-t)f_2(t,\rho,\theta)^2\right)|d\rho d\theta
\end{align*}
where we have used the Sobolev embedding $W^{1,1}(\m S^1) \hookrightarrow L^\infty(\m S^1)$ .
We estimate the terms appearing when we distribute the derivation $\partial_\rho$ from left to right.
\begin{align*}
|(1+|t-\rho|) v(\rho-t) \partial^\alpha_\theta f_2^2| &\lesssim \rho |v(\rho-t) \partial^\alpha_\theta f_2^2| ,\\
|(1+t+\rho) v(\rho-t) \partial^\alpha_\theta f_2^2| &\lesssim \rho |v(\rho-t) \partial^\alpha_\theta f_2^2| ,\\
|(1+t+\rho) (1+|t-\rho|)v'(\rho-t) \partial^\alpha_\theta f_2^2|
&\lesssim \rho |(1+|t-\rho|)v'(\rho-t) ||\partial^\alpha_\theta f_2^2| &\lesssim \rho |v(\rho-t) \partial^\alpha_\theta f_2^2| ,\\
|(1+t+\rho) v(\rho-t) (1+|t-\rho|)\partial_\rho \partial^\alpha_\theta f_2^2|
&\lesssim \rho |v(\rho-t)|\sum_{Z \in \q Z} |Z\partial^\alpha_\theta f_2^2|,
\end{align*}
where we have used in the third inequality $|sv'(s)|\leq v(s)$.
Therefore
$$|(1+t+r)(1+|t-r|)v(r-t)(f_2(t,r,\theta))^2|
\lesssim \sum_{0\leq \alpha \leq 1} \sum_{Z \in \q Z}\|v^\frac{1}{2} \partial^\alpha_\theta Z f_2\|^2_{L^2}
\lesssim \sum_{I\leq 2}   \|v^\frac{1}{2}Z^I f_2\|^2_{L^2}.$$
This concludes the proof of Proposition \ref{prpks}.
\end{proof}

\paragraph{Acknowledgement} The author would like to thank Jérémie Szeftel for his attentive reading of this paper and 
his supervising.
\bibliographystyle{plain}
\bibliography{stab}

\end{document}